\documentclass[10pt,a4paper]{article}
\usepackage[centertags]{amsmath}
\usepackage{mathrsfs}
\usepackage{amsfonts}
\usepackage{amssymb}
\usepackage{amsthm}
\usepackage[utopia]{mathdesign}
\newcommand{\Real}{\mathbb R}
\newcommand{\Complex}{\mathbb C}
\newcommand{\integer}{\mathbb Z}
\newcommand{\Pstint}{\mathbb N}

\newcommand{\dbar}{\overline\partial}
\newcommand{\pr}{\partial}
\newcommand{\ol}{\overline}

\newcommand{\Corner}{{\displaystyle\lrcorner}}
\newcommand{\Td}{\widetilde}
\newcommand{\norm}[1]{\left\Vert#1\right\Vert}
\newcommand{\abs}[1]{\left\vert#1\right\vert}
\newcommand{\set}[1]{\left\{#1\right\}}
\newcommand{\seq}[1]{\left<#1\right>}
\newcommand{\eps}{\varepsilon}
\newcommand{\To}{\rightarrow}
\newcommand{\La}{\triangle}
\theoremstyle{plain}
\newtheorem{thm}{Theorem}[section]
\newtheorem{cor}[thm]{Corollary}
\newtheorem{lem}[thm]{Lemma}
\newtheorem{prop}[thm]{Proposition}
\newtheorem{ass}[thm]{Assumption}

\theoremstyle{definition}
\newtheorem{defn}[thm]{Definition}
\theoremstyle{remark}
\newtheorem{rem}[thm]{Remark}
\numberwithin{equation}{section}

\usepackage[center]{titlesec}

\title{Projections in several complex variables}
\date{}
\author{Chin-Yu Hsiao  \\
CMLS, Ecole Polytechnique, FR-91128 Palaiseau cedex}
\begin{document}
\maketitle

\fontsize{12}{16pt}\selectfont
\begin{quotation}
\begin{center} \textbf{Abstract} \end{center}\parskip=6pt

This work consists two parts. In the first part, we completely study the heat equation method of Menikoff-Sj\"{o}strand and
apply it to the Kohn Laplacian defined on a compact orientable connected CR manifold. We then get the
full asymptotic expansion of the Szeg\"{o} projection for $(0, q)$ forms when the Levi form is non-degenerate.
This generalizes a result of Boutet de Monvel and Sj\"{o}strand for $(0,0)$ forms.
Our main tool is Fourier integral operators with complex valued phase functions of Melin and Sj\"{o}strand.

In the second part, we obtain the full asymptotic expansion of the Bergman projection
for $(0, q)$ forms when the Levi form is non-degenerate.
This also generalizes a result of Boutet de Monvel and
Sj\"{o}strand for $(0,0)$ forms. We introduce a new operator analogous to the
Kohn Laplacian defined on the boundary of a domain and we apply the heat equation method
of Menikoff and Sj\"{o}strand to this operator. We obtain a description of a
new Szeg\"{o} projection up to smoothing operators. Finally, by using the Poisson
operator, we get our main result.
\end{quotation}

\begin{quotation}
\begin{center} \textbf{R\'{e}sum\'{e}} \end{center}\parskip=6pt

Ce travail est constitu\'{e}e de deux parties. Dans la premi\`{e}re partie, nous appliquons
la m\'{e}thode de Menikoff-Sj\"{o}strand au laplacien de Kohn, d\'{e}fini sur une variet\'{e}
CR compacte orient\'{e}e connexe et nous obtenons un de'veloppement asymptotique complet du projecteur de
Szeg\"{o} pour les $(0,q)$ formes quand la forme de Levi est non-d\'{e}g\'{e}n\'{e}r\'{e}e. Cela g\'{e}n\'{e}ralise un
r\'{e}sultat de Boutet de Monvel et Sj\"{o}strand pour les $(0,0)$ formes. Nous utilisons des op\'{e}rateurs int\'{e}graux de Fourier \`{a} phases
complexes de Melin et Sj\"{o}strand.

Dans la deuxi\'{e}me partie, nous obtenons un d\'{e}veloppement asymptotique de la
singularit\'e du noyau de Bergman pour les $(0,q)$ formes quand la
forme de Levi est non-d\'eg\'en\'er\'ee.  Cela g\'{e}n\'{e}ralise un
r\'esultat de Boutet de Monvel et Sj\"{o}strand pour les $(0,0)$
formes. Nous introduisons un nouvel op\'{e}rateur analogue au
laplacien de Kohn, d\'efini sur le bord du domaine et nous y
appliquons la m\'ethode de Menikoff-Sj\"{o}strand.  Cela donne une
description modulo des op\'ereateurs r\'egularisants d'un nouvel
projecteur de Szeg\"{o}.  Finalement, en utilisant l'op\'{e}rateur de
Poisson, nous obtenons notre r\'{e}sultat principal.
\end{quotation}

\newpage

\tableofcontents
\fontsize{12}{16pt}\selectfont
\newpage

\addcontentsline{toc}{section}{Introduction}
\section*{Introduction}

The Bergman and Szeg\"{o} projections are classical subjects in several complex variables and complex geometry.
By Kohn's regularity theorem for the $\dbar$-Neumann problem (1963,~\cite{Kohn63}), the boundary behavior of the Bergman kernel is
highly dependent on the Levi curvature of the boundary. The study of the boundary behavior of the Bergman kernel
on domains with positive Levi curvature (strictly pseudoconvex domains) became an important topic in the field then.
In 1965, L. H\"{o}rmander (\cite{Hor65}) determined the boundary behavior of the Bergman kernel.
C. Fefferman (1974,~\cite{Fer74}) established an asymptotic expansion at the diagonal of the Bergman kernel.
More complete asymptotics of the Bergman kernel was obtained by Boutet de Monvel and Sj\"{o}strand (1976,~\cite{BS76}).
They also established an asymptotic expansion of the Szeg\"{o} kernel on strongly pseudoconvex boundaries.
All these developments concerned pseudoconvex domains. For the nonpseudoconvex domain, there are few results.
R. Beals and P. Greiner (1988,~\cite{BG88}) proved that the Szeg\"{o} projection is a Heisenberg pseudodifferential operator, under certain
Levi curvature assumptions.
H\"{o}rmander (2004,~\cite{Hor04}) determined the boundary behavior of the Bergman kernel when the Levi form is negative definite
by computing the leading term of the Bergman kernel on a spherical shell in $\Complex^n$.

Other developments recently concerned the Bergman kernel for a high power of a holomorphic line bundle.
D. Catlin (1997,~\cite{Cat97}) and S. Zelditch (1998,~\cite{Zel98}) adapted
a result of Boutet de Monvel-Sj\"{o}strand for the asymptotics of the Szeg\"{o} kernel on a strictly pseudoconvex boundary to
establish the complete asymptotic expansion of the Bergman kernel for a high power of a holomorphic line bundle with positive curvature.
Recently, a new proof of the existence of the complete asymptotic expansion was obtained by B. Berndtsson, R. Berman and
J. Sj\"{o}strand (2004,~\cite{BBS04}). Without the positive curvature assumption,
R. Berman and J. Sj\"{o}strand (2005,~\cite{BS05}) obtained a full asymptotic expansion of the Bergman kernel for a high power of a line bundle
when the curvature is non-degenerate. The approach of Berman and Sj\"{o}strand builds on the heat equation method of
Menikoff-Sj\"{o}strand (1978,~\cite{MS78}).
The expansion was obtained independently by X. Ma and G. Marinescu (2006,~\cite{MM06}) (without a phase function)
by using a spectral gap estimate for the Hodge Laplacian.

Recently, H\"{o}rmander (2004,~\cite{Hor04}) studied the Bergman projection for $(0,q)$ forms.
In that paper (page 1306), H\"{o}rmander suggested:
"A carefull microlocal analysis along the lines of Boutet de Monvel-Sj\"{o}strand should give the asymptotic expansion of the Bergman
projection for $(0,q)$ forms when the Levi form is non-degenerate."

The main goal for this work is to achieve H\"{o}rmander's wish-more precisely, to obtain an asymptotic expansion of the Bergman
projection for $(0,q)$ forms. The first step of my research is to establish an asymptotic expansion
of the Szeg\"{o} projection for $(0,q)$ forms. Then, find a suitable operator defined on the boundary of domain  which
plays the same role as the Kohn Laplacian in the approach of Boutet de Monvel-Sj\"{o}strand.

This work consists two parts. In the first paper, we completely study the heat equation method of Menikoff-Sj\"{o}strand and
apply it to the Kohn Laplacian defined on a compact orientable connected CR manifold. We then get the
full asymptotic expansion of the Szeg\"{o} projection for $(0, q)$ forms when the Levi form is non-degenerate.
We also compute the leading term of the Szeg\"{o} projection.

In the second paper, we introduce a new operator analogous to the
Kohn Laplacian defined on the boundary of a domain and we apply the method
of Menikoff-Sj\"{o}strand to this operator. We obtain a description of a
new Szeg\"{o} projection up to smoothing operators. Finally, by using the Poisson
operator, we get the full asymptotic expansion of the Bergman projection for $(0, q)$ forms
when the Levi form is non-degenerate.

We recall briefly some microlocal analysis that we used in this work in Appendix A and B.
These two papers can be read independently. We hope that this work can serve as an introduction
to certain microlocal techniques with applications to complex geometry and CR geometry.

\bigskip
\noindent
\textbf{Acknowledgements.} \ \ This work is part of a Ph.D thesis at Ecole Polytechnique in $2008$.
The author would like to thank his advisor Johannes Sj\"{o}strand for
his patience, guidance and inspiration. Furthermore, the auther is grateful to Bo Berndtsson, Louis Boutet de Monvel and
Xiaonan Ma for comments and useful suggestions on an early draft of the manuscript.
\newpage

\addcontentsline{toc}{part}{ Part I:\ \
On the singularities of the Szeg\"{o} projection for $(0, q)$ forms}
\part*{\mbox{} \\ \mbox{} \\ \mbox{} \mbox{} \\ \mbox{} \\ \mbox{} \mbox{} \\ \mbox{} \\ \mbox{}
\quad\quad\quad\quad\quad\quad\quad\quad Part {\rm I\,} \\
\mbox{} \\ \mbox{} \\ \mbox{} \\
\quad\quad On the singularities of the Szeg\"{o}\\ \quad\quad projection for $(0, q)$ forms}
\mbox{}\newpage


\setcounter{section}{0}
\section{Introduction and statement of the main results}

Let $(X, \Lambda^{1, 0}T(X))$ be a compact orientable connected CR manifold of dimension $2n-1$, $n\geq2$,
(see Definition~\ref{d:CR-2.1}) and take a smooth Hermitian metric $(\ |\ )$ on $\Complex T(X)$ so that
$\Lambda^{1, 0}T(X)$ is orthogonal to
$\Lambda^{0, 1}T(X)$ and $(u\ |\ v)$ is real if $u$, $v$ are real tangent vectors,
where $\Lambda^{0, 1}T(X)=\ol{\Lambda^{1, 0}T(X)}$ and $\Complex T(X)$ is the complexified tangent bundle.
For $p\in X$, let $L_p$ be the Levi form of $X$ at $p$. (See (\ref{e:i-070502-*}) and Definition~\ref{d:d-Leviform}.)
Given $q$, $0\leq q\leq n-1$, the Levi form is said to satisfty condition $Y(q)$ at $p\in X$ if for any $\abs{J}=q$, $J=(j_1,j_2,\ldots,j_q)$,
$1\leq j_1<j_2<\cdots<j_q\leq n-1$, we have
\begin{equation} \label{e:0807111454}
\abs{\sum_{j\notin J}\lambda_j-\sum_{j\in J}\lambda_j}<\sum^{n-1}_{j=1}\abs{\lambda_j},
\end{equation}
where $\lambda_j$, $j=1,\ldots,(n-1)$, are the eigenvalues of $L_p$. (For the precise meaning of the eigenvalues of the Levi form, see
Definition~\ref{d:0710211410}.) If the Levi form is non-degenerate at $p$, then $Y(q)$ holds at $p$ if and only if
$q\neq n_-, n_+$, where $(n_-, n_+)$ is the signature of $L_p$, i.e. the number of negative eigenvalues
of $L_p$ is $n_-$ and $n_++n_-=n-1$. Let $\Box_{b}$ be the Kohn Laplacian on $X$ (see Chen-Shaw~\cite{CS01} or section $2$) and let $\Box^{(q)}_b$
denote the restriction to $(0, q)$ forms.
When condition $Y(q)$ holds, Kohn's $L^2$ estimates give the hypoellipicity with loss of one dervative for the solutions
of $\Box^{(q)}_b u=f$. (See Folland-Kohn~\cite{FK72}, \cite{CS01} and section $3$.) The Szeg\"{o} projection is the
orthogonal projection onto the kernel of $\Box^{(q)}_b$ in the $L^2$ space. When condition $Y(q)$ fails, one is
interested in the Szeg\"{o} projection on the level of $(0, q)$ forms. Beals and Greiner~(see \cite{BG88}) proved that
the Szeg\"{o} projection is a Heisenberg pseudodifferential operator. Boutet de Monvel and Sj\"{o}strand~(see \cite{BS76})
obtained the full asymptotic expansion for the Szeg\"{o} projection in the case of functions. We have
been influenced by these works. The main inspiration for the present paper comes from Berman and
Sj\"{o}strand \cite{BS05}.

We now start to formulate the main results. First, we introduce some standard notations.
Let $\Omega$ be a $C^\infty$ paracompact manifold equipped with a smooth density of integration.
We let $T(\Omega)$ and $T^*(\Omega)$ denote the tangent bundle of $\Omega$ and the cotangent bundle of $\Omega$ respectively.
The complexified tangent bundle of $\Omega$ and the complexified cotangent bundle of $\Omega$ will be denoted by $\Complex T(\Omega)$
and $\Complex T^*(\Omega)$ respectively. We write $\seq{\ ,}$ to denote the pointwise duality between $T(\Omega)$ and $T^*(\Omega)$.
We extend $\seq{\ ,}$ bilinearly to $\Complex T(\Omega)\times\Complex T^*(\Omega)$.
Let $E$ be a $C^\infty$ vector bundle over $\Omega$. The fiber of $E$ at $x\in\Omega$ will be denoted by $E_x$.
Let $Y\subset\subset\Omega$ be an open set. From now on, the spaces of
smooth sections of $E$ over $Y$ and distribution sections of $E$ over $Y$ will be denoted by $C^\infty(Y;\, E)$ and $\mathscr D'(Y;\, E)$ respectively.
Let $\mathscr E'(Y;\, E)$ be the subspace of $\mathscr D'(Y;\, E)$ whose elements have compact support in $Y$.
For $s\in\Real$, we let $H^s(Y;\, E)$ denote the Sobolev space
of order $s$ of sections of $E$ over $Y$.

The Hermitian metric $(\ |\ )$ on $\Complex T(X)$ induces, by duality, a Hermitian metric on $\Complex T^*(X)$ that
we shall also denote by $(\ |\ )$ (see (\ref{e:CR-Gamma2})). Let $\Lambda^{0, q}T^*(X)$ be the bundle of $(0, q)$ forms of $X$.
(See (\ref{e:0710211520}).) The Hermitian metric $(\ |\ )$ on $\Complex T^*(X)$ induces a Hermitian metric on
$\Lambda^{0, q}T^*(X)$ (see (\ref{e:0710211515})) also denoted by $(\ |\ )$.

We take $(dm)$ as the induced volume form on $X$. Let $(\ |\ )$ be the inner product on $C^\infty(X;\, \Lambda^{0,q}T^*(X))$
defined by
\begin{equation} \label{e:0807161950}
(f\ |\ g)=\int_X(f(z)\ |\ g(z))(dm),\ f,g\in C^\infty(X;\, \Lambda^{0,q}T^*(X)).
\end{equation}
Let
\[\pi^{(q)}: L^2(X;\, \Lambda^{0,q}T^*(X))\To {\rm Ker\,}\Box^{(q)}_b\]
be the Szeg\"{o} projection, i.e.
the orthogonal projection onto the kernel of $\Box^{(q)}_b$. Let
\[K_{\pi^{(q)}}(x, y)\in\mathscr D'(X\times X;\, \mathscr L(\Lambda^{0,q}T^*_y(X),\Lambda^{0,q}T^*_x(X)))\]
be the distribution kernel of $\pi^{(q)}$ with respect to the induced volume form $(dm)$. Here
$\mathscr L(\Lambda^{0,q}T^*_y(X), \Lambda^{0,q}T^*_x(X))$ is the
vector bundle with fiber over $(x, y)$ consisting of the linear maps
from $\Lambda^{0,q}T^*_y(X)$ to $\Lambda^{0,q}T^*_x(X)$.

We pause and recall a general fact of distribution theory.
Let $E$ and $F$ be $C^\infty$ vector
bundles over a paracompact $C^\infty$ manifold $\Omega$ equipped with a smooth density of integration. If
$A: C^\infty_0(\Omega;\, E)\To \mathscr D'(\Omega;\, F)$
is continuous, we write $K_A(x, y)$ or $A(x, y)$ to denote the distribution kernel of $A$.
The following two statements are equivalent
\begin{enumerate}
\item $A$ is continuous: $\mathscr E'(\Omega;\, E)\To C^\infty(\Omega;\, F)$,
\item $K_A\in C^\infty(\Omega\times\Omega;\, \mathscr L(E_y, F_x))$.
\end{enumerate}
If $A$ satisfies (a) or (b), we say that $A$ is smoothing. Let
$B: C^\infty_0(\Omega;\, E)\to \mathscr D'(\Omega;\, F)$.
We write $A\equiv B$ if $A-B$ is a smoothing operator. $A$ is smoothing if and only if
$A: H^s_{\rm comp\,}(\Omega;\, E)\To H^{s+N}_{\rm loc\,}(\Omega;\, F)$
is continuous, for all $N\geq0$, $s\in\Real$, where
\[H^s_{\rm loc\,}(\Omega;\, F)=\set{u\in\mathscr D'(\Omega;\, F);\, \varphi u\in H^s(\Omega;\, F);\,
      \forall\varphi\in C^\infty_0(\Omega)}\]
and $H^s_{\rm comp\,}(\Omega;\, E)=H^s_{\rm loc}(\Omega;\, E)\bigcap\mathscr E'(\Omega;\, E)$.
(See H\"{o}rmander~\cite{Hor03}.)

Let $\Lambda^{1, 0}T^*(X)$ denote the bundle with fiber
$\Lambda^{1, 0}T^*_z(X):=\ol{\Lambda^{0, 1}T^*_z(X)}$ at $z\in X$. Locally we can choose an orthonormal frame
$\omega_1(z),\ldots,\omega_{n-1}(z)$
for the bundle $\Lambda^{1,0}T^*_z(X)$, then
$\ol\omega_1(z),\ldots,\ol\omega_{n-1}(z)$
is an orthonormal frame for the bundle $\Lambda^{0,1}T^*_z(X)$. The real $(2n-2)$ form
$\omega=i^{n-1}\omega_1\wedge\ol\omega_1\wedge\cdots\wedge\omega_{n-1}\wedge\ol\omega_{n-1}$
is independent of the choice of the orthonormal frame. Thus $\omega$ is globally
defined. Locally there is a real $1$ form $\omega_0(z)$ of length one which is orthogonal to
$\Lambda^{1,0}T^*_z(X)\oplus\Lambda^{0,1}T^*_z(X)$. $\omega_0(z)$ is unique up to the choice of sign.
Since $X$ is orientable, there is a nowhere vanishing $(2n-1)$ form $Q$ on $X$.
Thus, $\omega_0$ can be specified uniquely by requiring that
\begin{equation} \label{e:0807152304}
\omega\wedge\omega_0=fQ,
\end{equation}
where $f$ is a positive function. Therefore $\omega_0$, so chosen, is globally defined. We call $\omega_0$
the uniquely determined global real $1$ form.

We recall that the Levi form $L_p$, $p\in X$, is the Hermitian quadratic form on $\Lambda^{1,0}T_p(X)$ defined as follows:
\begin{equation} \label{e:i-070502-*} \begin{split}
&\mbox{For any $Z$, $W\in \Lambda^{1,0}T_p(X)$, pick $\Td Z$, $\Td W\in
C^\infty(X;\, \Lambda^{1,0}T(X))$ that satisfy}  \\
&\mbox{$\Td Z(p)=Z$, $\Td W(p)=W$. Then }
L_p(Z,\ol W)=\frac{1}{2i}\seq{[\Td Z\ ,\ol{\Td W}](p)\ ,
\omega_0(p)}.
\end{split}
\end{equation}
In this work, we assume that

\begin{ass} \label{a:0807120935}
The Levi form is non-degenerate at each point of\, $X$.
\end{ass}

The characteristic manifold of $\Box^{(q)}_b$ is given by $\Sigma=\Sigma^+\bigcup\Sigma^-$, where
\begin{align} \label{e:0807120937}
&\Sigma^+=\set{(x, \xi)\in T^*(X)\setminus0;\, \xi=\lambda\omega_0(x), \lambda>0}, \nonumber \\
&\Sigma^-=\set{(x, \xi)\in T^*(X)\setminus0;\, \xi=\lambda\omega_0(x), \lambda<0}.
\end{align}
Let $(n_-, n_+)$, $n_-+n_+=n-1$, be the signature of the Levi form. We define
\begin{align*}
&\hat{\Sigma}=\Sigma^+\ \ \mbox{if}\ \ n_+=q\neq n_-,  \\
&\hat{\Sigma}=\Sigma^-\ \ \mbox{if}\ \ n_-=q\neq n_+,  \\
&\hat{\Sigma}=\Sigma^+\bigcup\Sigma^-\ \ \mbox{if}\ \ n_+=q=n_-.
\end{align*}
Recall that (see \cite{FK72}, \cite{CS01}) if $q\notin\set{n_-, n_+}$, then
\[K_{\pi^{(q)}}(x, y)\in C^\infty(X\times X;\, \mathscr L(\Lambda^{0,q}T^*_y(X),\Lambda^{0,q}T^*_x(X))).\]

The main result of this work is the following

\begin{thm} \label{t:i-szegomain}
Let $(X, \Lambda^{1, 0}T(X))$ be a compact orientable connected CR manifold of dimension $2n-1$, $n\geq2$,
with a Hermitian metric $(\ |\ )$. (See Definition~\ref{d:CR-2.1} and Definition~\ref{d:CR-2.1more}.)
Let $(n_-, n_+)$, $n_-+n_+=n-1$, be the signature of the Levi form and let $q=n_-$ or $n_+$.
Suppose $\Box^{(q)}_b$ has closed range. Then,
\[\pi^{(q)}:H^s(X;\, \Lambda^{0, q}T^*(X))\To H^s(X;\, \Lambda^{0, q}T^*(X))\]
is continuous, for all $s\in\Real$, and ${\rm WF\,}'(K_{\pi^{(q)}})={\rm diag\,}(\hat{\Sigma}\times\hat{\Sigma})$,
where
${\rm WF\,}'(K_{\pi^{(q)}})=\set{(x, \xi, y, \eta)\in T^*(X)\times T^*(X);\, (x, \xi, y, -\eta)\in{\rm WF\,}(K_{\pi^{(q)}})}$.
Here ${\rm WF\,}(K_{\pi^{(q)}})$ is the wave front set of $K_{\pi^{(q)}}$ in the sense of H\"{o}rmander~\cite{Hor71}.
(See also chapter $8$ of~\cite{Hor03}.) Moreover, we have
\begin{align*}
&K_{\pi^{(q)}}=K_{\pi^{(q)}_+}\ \ {\rm if\,}\ \ n_+=q\neq n_-,  \\
&K_{\pi^{(q)}}=K_{\pi^{(q)}_-}\ \ {\rm if\,}\ \ n_-=q\neq n_+,  \\
&K_{\pi^{(q)}}=K_{\pi^{(q)}_+}+K_{\pi^{(q)}_-}\ \ {\rm if}\ \ n_+=q=n_-,
\end{align*}
where $K_{\pi^{(q)}_+}(x, y)$ satisfies
\begin{align*}
K_{\pi^{(q)}_+}(x, y)\equiv\int^{\infty}_{0} e^{i\phi_+(x, y)t}s_+(x, y, t)dt
\end{align*}
with
\begin{align*}
&s_+(x, y, t)\in S^{n-1}_{1, 0}(X\times X\times]0, \infty[;\, \mathscr L(\Lambda^{0,q}T^*_y(X), \Lambda^{0,q}T^*_x(X))),\\
&s_+(x, y, t)\sim\sum^\infty_{j=0}s^j_+(x, y)t^{n-1-j}
\end{align*}
in $S^{n-1}_{1, 0}(X\times X\times]0, \infty[;\, \mathscr L(\Lambda^{0,q}T^*_y(X), \Lambda^{0,q}T^*_x(X)))$,
where $S^m_{1, 0}$, $m\in\Real$, is the H\"{o}rmander symbol space~\cite{Hor71} (see also chapter $1$ of Grigis-Sj\"{o}strand~\cite{GS94}),
\[s^j_+(x, y)\in C^\infty(X\times X;\, \mathscr L(\Lambda^{0, q}T^*_y(X), \Lambda^{0, q}T^*_x(X))),\ j=0,1,\ldots,\]
and
\begin{align}
&\phi_+(x, y)\in C^\infty(X\times X),\ \ {\rm Im\,}\phi_+(x, y)\geq0, \label{e:0709241541} \\
&\phi_+(x, x)=0,\ \ \phi_+(x, y)\neq0\ \ \mbox{if}\ \ x\neq y, \label{e:0803301702} \\
&d_x\phi_+\neq0,\ \ d_y\phi_+\neq0\ \ \mbox{where}\ \ {\rm Im\,}\phi_+=0, \label{e:0803271451} \\
&d_x\phi_+(x, y)|_{x=y}=\omega_0(x),\ \ d_y\phi_+(x, y)|_{x=y}=-\omega_0(x), \label{e:t-i-bis1} \\
&\phi_+(x, y)=-\ol\phi_+(y, x). \label{e:0709241544}
\end{align}

Similarly,
\begin{align*}
K_{\pi^{(q)}_-}(x, y)\equiv\int^{\infty}_{0} e^{i\phi_-(x, y)t}s_-(x, y, t)dt
\end{align*}
with
\[s_-(x, y, t)\sim\sum^\infty_{j=0}s^j_-(x, y)t^{n-1-j}\]
in $S^{n-1}_{1, 0}(X\times X\times]0, \infty[;\, \mathscr L(\Lambda^{0,q}T^*_y(X), \Lambda^{0,q}T^*_x(X)))$,
where
\[s^j_-(x, y)\in C^\infty(X\times X;\, \mathscr L(\Lambda^{0, q}T^*_y(X), \Lambda^{0, q}T^*_x(X))),\ j=0,1,\ldots,\]
and when $q=n_-=n_+$, $\phi_-(x, y)=-\ol\phi_+(x, y)$.

A formula for $s^0_+(x, x)$ will be given in Proposition~\ref{p:i-leading1}.
More properties of the phase $\phi_+(x, y)$ will be given in
Theorem~\ref{t:i-phase} and Remark~\ref{r:0806222129} below.
\end{thm}

\begin{rem} \label{r:i-1}
If $Y(q-1)$ and $Y(q+1)$ hold then $\Box^{(q)}_b$ has closed range. (See section $7$ for a review and references.)
\end{rem}

\begin{rem} \label{r:i-2070826}
If $(X, \Lambda^{1, 0}T(X))$ is non-orientable, we also have generalization of
Theorem~\ref{t:i-szegomain}. (See section $10$.)
\end{rem}

The phase $\phi_+(x, y)$ is not unique. We can replace $\phi_+(x, y)$ by
\begin{equation} \label{e:i-0709241626}
\Td\phi(x, y)=f(x, y)\phi_+(x, y),
\end{equation}
where $f(x, y)\in C^\infty(X\times X)$ is real and $f(x, x)=1$, $f(x, y)=f(y, x)$. Then $\Td\phi$ satisfies
(\ref{e:0709241541})-(\ref{e:0709241544}).
We work with local coordinates $x=(x_1,\ldots,x_{2n-1})$ defined on an open set $\Omega\subset X$. Let $p\in\Omega$.
We can check that
\begin{align*}
\seq{\Td\phi''(p, p)U, V}&=\seq{(\phi_+)''(p, p)U, V}+\seq{df(p, p), U}\seq{d\phi_+(p, p), V} \\
                         &\quad+\seq{df(p, p), V}\seq{d\phi_+(p, p), U},\ \ U, V\in\Complex T(X),
\end{align*}
where
$(\phi_+)''=\left[
\begin{array}[c]{cc}
  (\phi_+)''_{xx} & (\phi_+)''_{xy} \\
  (\phi_+)''_{yx} & (\phi_+)''_{yy}
\end{array}\right]$
and similarly for $\Td\phi''$. The Hessian $(\phi_+)''$ of $\phi_+$ at $(p, p)$ is well-defined on the space
\[T_{(p,p)}H_+:=\set{W\in\Complex T_p(X)\times\Complex T_p(X);\, \seq{d\phi_+(p, p), W}=0}.\]
In section $8$, we will define $T_{(p, p)}H_+$ as the tangent space of the formal hypersurface $H_+$ (see (\ref{e:0708281710}))
at $(p, p)\in X\times X$.

We define the tangential Hessian of $\phi_+(x, y)$ at $(p, p)$ as the bilinear map:
\begin{align*}
T_{(p, p)}H_+\times T_{(p, p)}H_+&\To\Complex, \\
(U, V)&\To\seq{(\phi_+'')(p, p)U, V},\ \ U, V\in T_{(p, p)}H_+.
\end{align*}

In section $9$, we compute the tangential Hessian of $\phi_+(x, y)$ at $(p, p)$

\begin{thm} \label{t:i-phase}
For a given point $p\in X$, let $U_1(x),\ldots, U_{n-1}(x)$
be an orthonormal frame of $\Lambda^{1, 0}T_x(X)$ varying smoothly with $x$ in a neighborhood of\, $p$,
for which the Levi form is diagonalized at $p$.
We take local coordinates
$x=(x_1,\ldots,x_{2n-1})$, $z_j=x_{2j-1}+ix_{2j}$, $j=1,\ldots,n-1$,
defined on some neighborhood of $p$ such that $\omega_0(p)=\sqrt{2}dx_{2n-1}$, $x(p)=0$,
$(\frac{\pr}{\pr x_j}(p)\ |\ \frac{\pr}{\pr x_k}(p))=2\delta_{j,k}$, $j, k=1,\ldots,2n-1$
and
\[U_j=\frac{\pr}{\pr z_j}-\frac{1}{\sqrt{2}}i\lambda_j\ol z_j\frac{\pr}{\pr x_{2n-1}}-
\frac{1}{\sqrt{2}}c_jx_{2n-1}\frac{\pr}{\pr x_{2n-1}}+O(\abs{x}^2),\ \ j=1,\ldots,n-1,\]
where
$c_j\in\Complex$, $\frac{\pr}{\pr z_j}=\frac{1}{2}(\frac{\pr}{\pr x_{2j-1}}-i\frac{\pr}{\pr x_{2j}})$, $j=1,\ldots,n-1$,
and $\lambda_j$, $j=1,\ldots,n-1$, are the eigenvalues of\, $L_p$.
(This is always possible. See page $157$-page $160$ of~\cite{BG88}.) We also write
$y=(y_1,\ldots,y_{2n-1})$, $w_j=y_{2j-1}+iy_{2j}$, $j=1,\ldots,n-1$.
Then,
\begin{align} \label{e:i-t-final*****}
&\phi_+(x, y)=\sqrt{2}(x_{2n-1}-y_{2n-1})+i\sum^{n-1}_{j=1}\abs{\lambda_j}\abs{z_j-w_j}^2
+\sum^{n-1}_{j=1}\Bigr(i\lambda_j(z_j\ol w_j-\ol z_jw_j)\nonumber \\
&\quad+c_j(z_jx_{2n-1}-w_jy_{2n-1})+\ol c_j(\ol z_jx_{2n-1}-\ol w_jy_{2n-1})\Bigr)+(x_{2n-1}-y_{2n-1})f(x, y)\nonumber \\
&\quad+O(\abs{(x, y)}^3),\ \ f\in C^\infty,\ \ f(0,0)=0,\ \ f(x, y)=\ol f(y, x).
\end{align}
\end{thm}

\begin{rem} \label{r:0806222129}
With the notations used in Theorem~\ref{t:i-phase}, since $\frac{\pr\phi_+}{\pr x_{2n-1}}(0, 0)\neq0$,
from the Malgrange preparation theorem (see Theorem~$7.57$ of~\cite{Hor03}), we have
\[\phi_+(x, y)=g(x, y)(\sqrt{2}x_{2n-1}+h(x', y))\]
in some neighborhood of $(0, 0)$, where $g$, $h\in C^\infty$, $g(0, 0)=1$, $h(0, 0)=0$ and $x'=(x_1,\ldots,x_{2n-2})$. Put
$\hat\phi(x, y)=\sqrt{2}x_{2n-1}+h(x', y)$.
From the global theory of Fourier integral operators (see Theorem~$4.2$ of Melin-Sj\"{o}strand~\cite{MS74} or
Proposition~\ref{p:c-mesjmore}), we see that
$\phi_+$ and $\hat\phi$ are equivalent at $(p, \omega_0(p))$
in the sense of Melin-Sj\"{o}strand~\cite{MS74} (see Definition~\ref{d:0709171814}). Since
$\phi_+(x, y)=-\ol\phi_+(y, x)$,
we can replace the phase $\phi_+(x, y)$ by
\[\frac{\hat\phi(x, y)-\ol{\hat\phi}(y, x)}{2}.\]
Then the new function $\phi_+(x, y)$ satisfies  (\ref{e:i-t-final*****}) with $f=0$.
\end{rem}

We have the following corollary of Theorem~\ref{t:i-szegomain}. (See section $9$.)

\begin{cor}
There exist
\[F_+, G_+, F_-, G_-\in C^\infty(X\times X;\ \mathscr L(\Lambda^{0, q}T^*_y(X), \Lambda^{0, q}T^*_x(X)))\]
such that
\[K_{\pi^{(q)}_+}=F_+(-i(\phi_+(x, y)+i0))^{-n}+G_+\log(-i(\phi_+(x, y)+i0)),\]
\[K_{\pi^{(q)}_-}=F_-(-i(\phi_-(x, y)+i0))^{-n}+G_-\log(-i(\phi_-(x, y)+i0)).\]

Moreover, we have
\begin{align} \label{e:i-000}
&F_+=\sum^{n-1}_0(n-1-k)!s^k_+(x, y)(-i\phi_+(x, y))^k+f_+(x, y)(\phi_+(x, y))^n,\nonumber \\
&F_-=\sum^{n-1}_0(n-1-k)!s^k_-(x, y)(-i\phi_-(x, y))^k+f_-(x, y)(\phi_-(x, y))^n,\nonumber \\
&G_+\equiv\sum^\infty_0\frac{(-1)^{k+1}}{k!}s^{n+k}_+(x, y)(-i\phi_+(x, y))^k,\nonumber \\
&G_-\equiv\sum^\infty_0\frac{(-1)^{k+1}}{k!}s^{n+k}_-(x, y)(-i\phi_-(x, y))^k,
\end{align}
where $f_+(x, y), f_-(x, y)\in C^\infty(X\times X;\, \mathscr L(\Lambda^{0,q}T^*_y(X), \Lambda^{0,q}T^*_x(X)))$.
\end{cor}

If $w\in\Lambda^{0,1}T^*_z(X)$, let
$w^{\wedge, *}: \Lambda^{0,q+1}T^*_z(X)\To \Lambda^{0,q}T^*_z(X),\ q\geq0$,
be the adjoint of left exterior multiplication
$w^\wedge: \Lambda^{0,q}T^*_z(X)\To \Lambda^{0,q+1}T^*_z(X)$.
That is,
\begin{equation} \label{e:0807152331}
(w^\wedge u\ |\ v)=(u\ |\ w^{\wedge, *}v),
\end{equation}
for all $u\in\Lambda^{0,q}T^*_z(X)$, $v\in\Lambda^{0,q+1}T^*_z(X)$.
Notice that $w^{\wedge, *}$ depends anti-linearly on $w$.

In section $9$, we compute $F_+(x, x)$

\begin{prop} \label{p:i-leading1}
For a given point $p\in X$, let $U_1(x),\ldots,U_{n-1}(x)$
be an orthonormal frame of $\Lambda^{1,0}T_x(X)$, for which
the Levi form is diagonalized at $p$. Let $e_j(x)$, $j=1,\ldots,n-1$,
denote the basis of $\Lambda^{0,1}T^*_x(X)$, which is dual to $\ol U_j(x)$, $j=1,\ldots,n-1$. Let
$\lambda_j(x)$, $j=1,\ldots,n-1$, be the eigenvalues of the Levi form $L_x$. We assume that $q=n_+$ and that
$\lambda_j(p)>0$ if\, $1\leq j\leq n_+$. Then
\[F_+(p, p)=(n-1)!\frac{1}{2}\abs{\lambda_1(p)}\cdots
\abs{\lambda_{n-1}(p)}\pi^{-n}\prod_{j=1}^{j=n_+}e_j(p)^\wedge e_j(p)^{\wedge, *}.\]
\end{prop}

In the rest of this section, we will give the outline of the proof of Theorem~\ref{t:i-szegomain}.
Let $M$ be an open set in $\Real^n$ and let $f$, $g\in C^\infty(M)$. We write $f\asymp g$
if for every compact set $K\subset M$ there is a constant $c_K>0$ such that
$f\leq c_Kg$ and $g\leq c_Kf$ on $K$.

We will use the heat equation method. We work with some
real local coordinates $x=(x_1,\ldots,x_{2n-1})$
defined on an open set $\Omega\subset X$. Let $(n_-, n_+)$, $n_-+n_+=n-1$, be the signature of the Levi form.
We will say that
$a\in C^\infty(\ol\Real_+\times\Omega\times\Real^{2n-1})$ is quasi-homogeneous of
degree $j$ if $a(t,x,\lambda\eta)=\lambda^ja(\lambda t,x,\eta)$
for all $\lambda>0$. We consider the problem
\begin{equation} \label{e:i-heat}
\left\{ \begin{array}{ll}
(\pr_t+\Box^{(q)}_b)u(t,x)=0  & {\rm in\;}\Real_+\times\Omega  \\
u(0,x)=v(x) \end{array}\right..
\end{equation}
We shall start by making only a formal construction. We look for an approximate solution of
(\ref{e:i-heat}) of the form $u(t,x)=A(t)v(x)$,
\begin{equation} \label{e:i-fourierheat}
A(t)v(x)=\frac{1}{(2\pi)^{2n-1}}\int e^{i(\psi(t,x,\eta)-\seq{y,\eta})}a(t,x,\eta)v(y)dyd\eta
\end{equation}
where formally
\[a(t,x,\eta)\sim\sum^\infty_{j=0}a_j(t,x,\eta),\]
$a_j(t,x,\eta)$ is a matrix-valued quasi-homogeneous function of degree $-j$.

We let the full symbol of $\Box^{(q)}_b$ be:
\[\mbox{full symbol of }\ \Box^{(q)}_b=\sum^2_{j=0}p_j(x,\xi)\]
where $p_j(x,\xi)$ is positively homogeneous of order $2-j$ in the sense that
\[p_j(x, \lambda\eta)=\lambda^{2-j}p_j(x, \eta),\ \abs{\eta}\geq1,\ \lambda\geq1.\]
We apply $\pr_t+\Box^{(q)}_b$ formally
inside the integral in (\ref{e:i-fourierheat}) and then introduce the asymptotic expansion of
$\Box^{(q)}_b(ae^{i\psi})$. Set $(\pr_t+\Box^{(q)}_b)(ae^{i\psi})\sim 0$ and regroup
the terms according to the degree of quasi-homogeneity. The phase $\psi(t, x, \eta)$ should solve
\begin{equation} \label{e:i-chara}
\left\{ \begin{split}
& \frac{\displaystyle\pr\psi}{\displaystyle\pr t}-ip_0(x,\psi'_x)=
   O(\abs{{\rm Im\,}\psi}^N),\ \forall N\geq0   \\
& \psi|_{t=0}=\seq{x, \eta} \end{split}\right..
\end{equation}
This equation can be solved with ${\rm Im\,}\psi(t, x,\eta)\geq0$ and the phase $\psi(t, x, \eta)$ is quasi-homogeneous of
degree $1$. Moreover,
\[\psi(t,x,\eta)=\seq{x,\eta}{\;\rm on\;}\Sigma,\; d_{x,\eta}(\psi-\seq{x,\eta})=0{\;\rm on\;}
\Sigma,\]
\[{\rm Im\,}\psi(t, x,\eta)\asymp(\abs{\eta}\frac{t\abs{\eta}}{1+t\abs{\eta}}){\rm dist\,}((x,\frac{\eta}{\abs{\eta}}),\Sigma))^2,
\ \ \abs{\eta}\geq1.\]
Furthermore, there exists $\psi(\infty,x,\eta)\in C^\infty(\Omega\times\dot\Real^{2n-1})$ with a
uniquely determined Taylor expansion at each point of $\Sigma$ such that for every compact set
$K\subset\Omega\times\dot\Real^{2n-1}$ there is a constant $c_K>0$ such that
\[{\rm Im\,}\psi(\infty,x,\eta)\geq c_K\abs{\eta}({\rm dist\,}((x,\frac{\eta}{\abs{\eta}}),\Sigma))^2,
\ \ \abs{\eta}\geq1.\]
If $\lambda\in C(T^*\Omega\smallsetminus 0)$, $\lambda>0$ is positively homogeneous of degree $1$ and $\lambda|_\Sigma<\min\lambda_j$,
$\lambda_j>0$, where
$\pm i\lambda_j$ are the non-vanishing eigenvalues of the fundamental matrix of $\Box^{(q)}_b$,
then the solution $\psi(t,x,\eta)$ of (\ref{e:i-chara}) can be chosen so that for every
compact set $K\subset\Omega\times\dot\Real^{2n-1}$ and all indices $\alpha$, $\beta$, $\gamma$,
there is a constant $c_{\alpha,\beta,\gamma,K}$ such that
\[ \abs{\pr^\alpha_x\pr^\beta_\eta\pr^\gamma_t(\psi(t,x,\eta)-\psi(\infty,x,\eta))}
\leq c_{\alpha,\beta,\gamma,K}e^{-\lambda(x,\eta)t}{\;\rm on\;}\ol\Real_+\times K.\]
(For the details, see Menikoff-Sj\"{o}strand\cite{MS78} or section $4$.)

We obtain the transport equations
\begin{equation} \label{e:i-heattransport}
\left\{ \begin{array}{ll}
 T(t,x,\eta,\pr_t,\pr_x)a_0=O(\abs{{\rm Im\,}\psi}^N),\; \forall N   \\
 T(t,x,\eta,\pr_t,\pr_x)a_j+l_j(t,x,\eta,a_0,\ldots,a_{j-1})= O(\abs{{\rm Im\,}\psi}^N),\; \forall N.
 \end{array}\right.
\end{equation}

Following the method of Menikoff-Sj\"{o}strand~\cite{MS78}, we see that we can solve (\ref{e:i-heattransport}). Moreover,
$a_j$ decay exponetially fast in $t$ when $q\neq n_-$, $n_+$, and has subexponetially growth in general.
(Subexponetially growth means that $a_j$ satifies (\ref{e:hf-estimate3}).) We assume that $q=n_-$ or $n_+$.
To get further, we use a trick from Berman-Sj\"{o}strand~\cite{BS05}.
We use $\dbar_b\Box^{(q)}_b=\Box^{(q+1)}_b\dbar_b$, $\ol{\pr_b}^*\Box^{(q)}_b=\Box^{(q-1)}_b\ol{\pr_b}^*$
and get in a formula asymptotic sense
\begin{align*}
&\pr_t(\dbar_b(e^{i\psi}a))+\Box^{(q+1)}_b(\dbar_b(e^{i\psi}a))\sim0 \\
&\pr_t(\ol{\pr_b}^*(e^{i\psi}a))+\Box^{(q-1)}_b(\ol{\pr_b}^*(e^{i\psi}a))\sim0.
\end{align*}
Put
\[\dbar_b(e^{i\psi}a)=e^{i\psi}\hat a,\ \ol{\pr_b}^*(e^{i\psi}a)=e^{i\psi}\Td a.\]
We have
\begin{align*}
&(\pr_t+\Box^{(q+1)}_b)(e^{i\psi}\hat a)\sim0, \\
&(\pr_t+\Box^{(q-1)}_b)(e^{i\psi}\Td a)\sim0.
\end{align*}
The corresponding degrees of $\hat a$ and $\Td a$ are $q+1$ and $q-1$. We deduce as above that $\hat a$ and $\Td a$ decay
exponetially fast in $t$. This also applies to
\begin{align*}
\Box^{(q)}_b(ae^{i\psi})&=\dbar_b(\ol{\pr_b}^*ae^{i\psi})+\ol{\pr_b}^*(\dbar_bae^{i\psi}) \\
                  &=\dbar_b(e^{i\psi}\Td a)+\ol{\pr_b}^*(e^{i\psi}\hat a).
\end{align*}
Thus, $\pr_t(ae^{i\psi})$ decay exponetially fast in $t$. Since $\pr_t\psi$ decay exponetially fast in $t$ so does $\pr_ta$.
Hence, there exist
\[a_j(\infty, x, \eta)\in
C^{\infty}(T^*(\Omega);\, \mathscr L(\Lambda^{0,q}T^*(\Omega)\ ,\Lambda^{0,q}T^*(\Omega))),\ j=0,1,\ldots,\]
positively homogeneous of degree $-j$ such that $a_j(t, x, \eta)$ converges exponentially
fast to $a_j(\infty, x, \eta)$, for all $j=0,1,\ldots$.

Choose $\chi\in C^\infty_0(\Real^{2n-1})$ so that $\chi(\eta)=1$ when $\abs{\eta}<1$ and $\chi(\eta)=0$ when $\abs{\eta}>2$. We formally set
\begin{align*}
G &= \frac{1}{(2\pi)^{2n-1}}\int\Bigl(\int^{\infty}_0\bigl(e^{i(\psi(t,x,\eta)-\seq{y,\eta})}a(t,x,\eta) \\
  &\quad -e^{i(\psi(\infty,x,\eta)-\seq{y,\eta})}a(\infty,x,\eta)\bigr)(1-\chi(\eta))dt\Bigl)d\eta
\end{align*}
and
\[S=\frac{1}{(2\pi)^{2n-1}}\int (e^{i(\psi(\infty,x,\eta)-\seq{y,\eta})}a(\infty,x,\eta))d\eta.\]
In section $6$, we will show that $G$ is a pseudodifferential operator of order $-1$ type $(\frac{1}{2}, \frac{1}{2})$.
In section $7$, we will show that $S+\Box^{(q)}_b\circ G\equiv I$, $\Box^{(q)}_b\circ S\equiv0$.
From this, it is not difficult to see that $\pi^{(q)}\equiv S$ if
$\Box^{(q)}_b$ has closed range. We deduce that $\pi^{(q)}$ is a Fourier integral operator if
$\Box^{(q)}_b$ has closed range.
From the global theory of Fourier integral operators (see~\cite{MS74} and section $8$), we get Theorem~\ref{t:i-szegomain}.


\section{Cauchy-Riemann manifolds, $\dbar_b\text{-Complex}$ and $\Box_b$, a review}

We will give a brief discussion of the basic elements of CR geometry in a setting
appropriate for our purpose. General references for this section are the books~\cite{BG88}, Boggess~\cite{Bog91} and~\cite{CS01}.

\begin{defn} \label{d:CR-2.1}
Let $X$ be a real $C^\infty$ manifold of dimension $2n-1$, $n\geq2$, and let $\Lambda^{1, 0}T(X)$ be a
subbundle of $\Complex T(X)$. The pair $(X, \Lambda^{1, 0}T(X))$ is called a CR manifold or a CR structure if
\begin{enumerate}
\item ${\rm dim\,}_{\Complex}\Lambda^{1, 0}T_p(X)=n-1$, $p\in X$,
\item $\Lambda^{1, 0}T(X)\bigcap\Lambda^{0, 1}T(X)=0$, where $\Lambda^{0, 1}T(X)=\ol{\Lambda^{1, 0}T(X)}$,
\item For any $V_1$, $V_2\in C^\infty(U;\, \Lambda^{1, 0}T(X))$,
$[V_1, V_2]\in C^\infty(U;\, \Lambda^{1, 0}T(X))$, where $U$ is any open subset of $X$.
\end{enumerate}
\end{defn}

\begin{defn} \label{d:CR-2.1more}
Let $(X, \Lambda^{1, 0}T(X))$ be a CR manifold. A Hermitian metric $(\ |\ )$ on $\Complex T(X)$ is a complex inner product $(\ |\ )$ on each
$\Complex T_p(X)$ depending smoothly on $p$ with
the properties that
$\Lambda^{1, 0}T_p(X)$ is orthogonal to $\Lambda^{0, 1}T_p(X)$ and $(u\ |\ v)$ is real if $u$, $v$ are real tangent
vectors.
\end{defn}

Until further notice, we assume that $(X, \Lambda^{1, 0}T(X))$ is a compact orientable connected CR manifold of
dimension $2n-1$, $n\geq2$, and
we fix a Hermitian metric $(\ |\ )$ on $\Complex T(X)$. Then there is a real non-vanishing vector field $Y$ on $X$ which
is pointwise orthogonal to $\Lambda^{1, 0}T(X)\oplus\Lambda^{0, 1}T(X)$.

The Hermitian metric $(\ |\ )$ on $\Complex T(X)$ induces, by duality, a Hermitian metric on $\Complex T^*(X)$ that we
shall also denote by $(\ |\ )$ in the following way. For a given point $z\in X$, let
$\Gamma:\Complex T_z(X)\To\Complex T^*_z(X)$ be the anti-linear map
defined by
\begin{equation} \label{e:CR-Gamma1}
(u\ |\ v)=\seq{u, \Gamma v},\ \  u, v\in\Complex T_z(X).
\end{equation}
For $\omega$, $\mu\in\Complex T^*_z(X)$, we put
\begin{equation} \label{e:CR-Gamma2}
(\omega\ |\ \mu)=(\Gamma^{-1}\mu\ |\ \Gamma^{-1}\omega).
\end{equation}

Let $\Lambda^r(\Complex T^*(X))$, $r\in\Pstint$, be the vector bundle of $r$ forms of $X$.
That is, the fiber of $\Lambda^r(\Complex T^*(X))$ at $z\in X$ is the
vector space $\Lambda^r(\Complex T^*_z(X))$ of all finite sums of
$V_1\wedge\cdots\wedge V_r,\ V_j\in\Complex T^*_z(X),\ j=1,\ldots,r$.
Here $\wedge$ denotes the wedge product. The Hermitian metric $(\ |\ )$ on
$\Lambda^r(\Complex T^*(X))$ is defined by
\begin{align} \label{e:0710211515}
&(u_{1}\wedge\cdots\wedge u_r\ |\ v_{1}\wedge\cdots\wedge v_r)=\det\left((u_{j}\ |\ v_{k})\right)_
{1\leq j,k\leq r},\nonumber \\
&\quad u_{j}, v_{k}\in\Complex T^*(X),\ j, k=1,\ldots,r,
\end{align}
and we extend the definition to arbitrary forms by sesqui-linearity.

Let $\Lambda^r(\Complex T(X))$, $r\in\Pstint$, be the bundle with fiber $\Lambda^r(\Complex T_z(X))$
at $z\in X$. The duality between $\Lambda^r(\Complex T(X))$ and
$\Lambda^r(\Complex T^*(X))$ is defined by
\[\seq{v_{1}\wedge\cdots\wedge v_r,\\ u_{1}\wedge\cdots\wedge u_r}=\det\left(\seq{v_{j}, u_{k}}\right)_
{1\leq j,k\leq r},\]
$u_{j}\in\Complex T^*(X)$, $v_{j}\in\Complex T(X), j=1,\ldots,r$,
and we extend the definition by bilinearity.

For $z\in X$, let $v\in\Complex T_z(X)$. The contraction
\[v^{\Corner}: \Lambda^{r+1}(\Complex T^*_z(X))\To \Lambda^{r}(\Complex T^*_z(X))\]
is defined by
$\langle v_1\wedge\cdots\wedge v_r, v^\Corner u\rangle=\seq{v\wedge v_1\wedge\cdots\wedge v_r, u}$
for all $u\in\Lambda^{r+1}(\Complex T^*_z(X))$, $v_j\in\Complex T_z(X)$, $j=1,\ldots,r$.

We have the pointwise orthogonal decomposition
\begin{equation} \label{e:d-decomposition2}
\Complex T(X)=\Lambda^{1, 0}T(X)\oplus\Lambda^{0, 1}T(X)\oplus\Complex Y.
\end{equation}
Define
$\Lambda^{1, 0}T^*(X)=(\Lambda^{0, 1}T(X)\oplus\Complex Y)^{\bot}\subset\Complex T^*(X)$ and
$\Lambda^{0, 1}T^*(X)=(\Lambda^{1, 0}T(X)\oplus\Complex Y)^{\bot}\subset\Complex T^*(X)$.
We can check that $\Lambda^{1, 0}T^*_z(X)=\Gamma\Lambda^{1, 0}T_z(X)$, $\Lambda^{0, 1}T^*_z(X)=\Gamma\Lambda^{0, 1}T_z(X)$
and  $\Lambda^{1, 0}T^*(X)$ is pointwise orthogonal to $\Lambda^{0, 1}T^*(X)$,
where $\Gamma$ is as in (\ref{e:CR-Gamma1}). For $q\in\Pstint$, the bundle of $(0, q)$ forms of $X$ is given by
\begin{equation} \label{e:0710211520}
\Lambda^{0, q}T^*(X)=\Lambda^{q}(\Lambda^{0, 1}T^*(X)).
\end{equation}
We use the Hermitian metric $(\ |\ )$ on $\Lambda^{0,q}T^*_z(X)$, that is naturally obtained from
$\Lambda^q(\Complex T^*(X))$.

Let $d:C^\infty(X;\, \Lambda^r(\Complex T^*(X)))\To C^\infty(X;\, \Lambda^{r+1}(\Complex T^*(X)))$
be the usual exterior derivative. We recall that the exterior derivative $d$ has the following properties, where
$(b)$, $(c)$ are special cases of Cartan's formula:
$\mathscr L_\nu\omega=\nu^\Corner d\omega+d(\nu^\Corner\omega)$.
Here $\nu$ is a smooth vector field, $\omega$ is a $q$ form and $\mathscr L_\nu\omega$ is the Lie derivative of
$\omega$ along $\nu$.
\begin{enumerate}
\item If $f\in C^\infty(X)$ then $\seq{V, df}=V(f), V\in C^\infty(X;\, \Complex T(X))$.
\item If $\phi\in C^\infty(X;\, \Complex T^*(X))$ then
\begin{align} \label{e:CR-1.3add}
\seq{V_1\wedge V_2, d\phi}=V_1(\seq{V_2, \phi})-V_2(\seq{V_1, \phi})-\seq{[V_1, V_2], \phi},
\end{align}
where $V_1$, $V_2\in C^\infty(X;\, \Complex T(X))$.
\item If $\phi\in C^\infty(X;\, \Lambda^{q-1}(\Complex T^*(X)))$, $q\geq2$, then
\begin{align} \label{e:CR-1.3addbis}
\seq{V_1\wedge\cdots\wedge V_q, d\phi}=
&-\seq{V_2\wedge\cdots\wedge V_q, d(V_1^\Corner\phi)}  \nonumber \\
&+V_1(\seq{V_2\wedge\cdots\wedge V_q, \phi}) \nonumber \\
&-\seq{[V_1, V_2]\wedge V_3\wedge\cdots\wedge V_q, \phi} \nonumber \\
&-\cdots-\seq{V_2\wedge V_3\wedge\cdots\wedge[V_1, V_q], \phi},
\end{align}
where $V_j\in C^\infty(X;\, \Complex T(X))$, $j=1,\ldots,q$.
\end{enumerate}

We have the pointwise orthogonal decomposition:
\begin{equation} \label{e:d-decomposition1}
\Complex T^*(X)=\Lambda^{1,0}T^*(X)\oplus \Lambda^{0,1}T^*(X)\oplus\set{\lambda\omega_0;\,
\lambda\in\Complex}.
\end{equation}
We recall that $\omega_0$ is the uniquely determined global real $1$ form (see (\ref{e:0807152304})).
We take $Y$ (already introduced after Definition~\ref{d:CR-2.1more} and (\ref{e:d-decomposition2})) so that
$\norm{Y}=1$, $\seq{Y, \omega_0}=-1$.
Therefore $Y$ is uniquely determined. We call $Y$ the uniquely determined global real vector field.

\begin{defn} \label{d:d-Leviform}
For $p\in X$, the Levi form $L_p$ is the Hermitian quadratic form on $\Lambda^{1,0}T_p(X)$ defined as follows:
\begin{equation} \label{e:070502-*} \begin{split}
&\mbox{For any $Z$, $W\in \Lambda^{1,0}T_p(X)$, pick $\Td Z$, $\Td W\in
C^\infty(X;\, \Lambda^{1,0}T(X))$ that satisfy}  \\
&\mbox{$\Td Z(p)=Z$, $\Td W(p)=W$. Then }
L_p(Z,\ol W)=\frac{1}{2i}\seq{[\Td Z\ ,\ol{\Td W}](p)\ ,
\omega_0(p)}.
\end{split}
\end{equation}
Here $[\Td Z\ ,\ol{\Td W}]=\Td Z\ol{\Td W}-
\ol{\Td W}\Td Z$
denotes the commutator of $\Td Z$ and $\ol{\Td W}$.
\end{defn}

Recall that $L_p$ does not depend of the choices of $\Td Z$ and $\Td W$ (see page $153$ of~\cite{BG88} or using (\ref{e:CR-1.3add}))

\begin{lem} \label{l:d-Leviform}
Let $\Td Z, \Td W\in C^\infty(X;\, \Lambda^{1,0}T(X))$. We have
\begin{equation} \label{e:*}
\frac{1}{2i}\seq{[\Td Z\ ,\ol{\Td W}](p)\ ,\omega_0(p)}=-\frac{1}{2i}\seq{\Td Z(p)\wedge\ol{\Td W}(p), d\omega_0(p)}.
\end{equation}
\end{lem}

\begin{defn} \label{d:0710211410}
The eigenvalues of the Levi form at $p\in X$ are the eigenvalues of the Hermitian form $L_p$ with
respect to the inner product $(\ |\ )$ on $\Lambda^{1, 0}T_p(X)$.
\end{defn}

For $U$, $V\in C^\infty(X;\, \Lambda^{1,0}T(X))$, from (\ref{e:070502-*}), we can check that
\begin{equation} \label{e:d-Levidecomp}
[U\ ,\ol V](p)=-(2i)L_p(U(p)\ ,\ol{V(p)})Y(p)+h,\ \ h\in\Lambda^{1,0}T_p(X)\oplus\Lambda^{0,1}T_p(X).
\end{equation}

Next we recall the tangential Cauchy-Riemann operator $\dbar_b$ and $\ol{\pr_b}^*$. Let
\[\pi^{0, q}:\Lambda^q(\Complex T^*(X))\To\Lambda^{0,q}T^*(X)\]
be the orthogonal projection map.

\begin{defn} \label{d:CR-2.2}
The tangential Cauchy-Riemann operator $\dbar_b$ is defined by
\[\dbar_b=\pi^{0, q+1}\circ d:C^\infty(X;\, \Lambda^{0,q}T^*(X))\to C^\infty(X;\, \Lambda^{0,q+1}T^*(X)).\]
\end{defn}

The following is well-known (see page $152$ of~\cite{BG88} or using (\ref{e:CR-1.3addbis}))

\begin{prop}
We have\, $\dbar_b^2=0$.
\end{prop}

Let $\ol{\pr_b}^*$ be the formal adjoint of $\dbar_b$, that is
$(\dbar_{b}f\ |\ h)=(f\ |\ \ol{\pr_b}^*h)$, where $f\in C^\infty(X;\, \Lambda^{0,q}T^*(X)),
h\in C^\infty(X;\, \Lambda^{0,q+1}T^*(X))$ and $(\ |\ )$ is given by (\ref{e:0807161950}).
$\ol{\pr_b}^*$ is a first order differential operator and $(\ol{\pr_b}^*)^2=0$.
The Kohn Laplacian $\Box_b$ is given by
\[\Box_b=\dbar_b\ol{\pr_b}^*+\ol{\pr_b}^*\dbar_b.\]
From now on, we write $\Box^{(q)}_b$ to denote the restriction to $(0, q)$ forms.

For $z_0\in X$, we can choose an orthonormal frame
$e_1(z),\ldots,e_{n-1}(z)$
for $\Lambda^{0,1}T^*_z(X)$ varying smoothly with $z$ in a neighborhood of $z_0$.
Let $Z_j(z)$, $j=1,\ldots,n-1$, denote the basis of $\Lambda^{0,1}T_z(X)$,
which is dual to $e_j(z)$, $j=1,\ldots,n-1$. Let $Z^*_j$ be the formal adjoint of $Z_j$, $j=1,\ldots,n-1$. That is,
$(Z_jf\ |\ h)=(f\ |\ Z_j^*h), f, h\in C^\infty(X)$.
We have the following (for a proof, see page $154$-page $156$ of~\cite{BG88})

\begin{prop} \label{p:d-kohn}
With the notations used before, the Kohn Laplacian $\Box^{(q)}_b$ is given by
\begin{align} \label{e:h-principal}
\Box^{(q)}_b &=\dbar_b\ol{\pr_b}^*+\ol{\pr_b}^*\dbar_b\nonumber  \\
       &=\sum^{n-1}_{j=1}Z^*_jZ_j+
         \sum^{n-1}_{j,k=1}e^\wedge_je^{\wedge, *}_k\circ [Z_j\ ,Z^*_k] \nonumber \\
         &\quad +\eps(Z)+\eps(Z^*)+\textup{zero order terms},
\end{align}
where $\eps(Z)$ denotes remainder terms of the form $\sum a_k(z)Z_k$ with $a_k$ smooth,
matrix-valued and similarly for $\eps(Z^*)$ and the map
\[e^\wedge_je^{\wedge,*}_k\circ[Z_j, Z^*_k]: C^\infty(X;\, \Lambda^{0,q}T^*(X))\To C^\infty(X;\, \Lambda^{0,q}T^*(X))\]
is defined by
\[(e^\wedge_je^{\wedge,*}_k\circ[Z_j, Z^*_k])(f(z)e_{j_1}\wedge\cdots\wedge e_{j_q})=
[Z_j, Z^*_k](f)(e_j^\wedge e^{\wedge,*}_k)\circ e_{j_1}\wedge\cdots\wedge e_{j_q}\]
and we extend the definition by linearity. We recall that $e^{\wedge,*}_k$ is given by (\ref{e:0807152331}).
\end{prop}


\section{The hypoellipicity of $\Box_b$}

We work with some real local coordinates $x=(x_1,\ldots,x_{2n-1})$ defined on an open set
$\Omega\subset X$. We let the full symbol of $\Box^{(q)}_b$ be:
\begin{equation} \label{e:0807171053}
\mbox{full symbol of }\ \Box^{(q)}_b=\sum^2_{j=0}p_j(x,\xi)
\end{equation}
where $p_j(x,\xi)$ is positively homogeneous of order $2-j$. Let $q_j$, $j=1,\ldots,n-1$, be the
principal symbols of $Z_j$, $j=1,\ldots,n-1$, where $Z_j$, $j=1,\ldots,n-1$, are as in Proposition~\ref{p:d-kohn}. Then,
\begin{equation} \label{e:h-principalmore}
p_0=\sum^{n-1}_{j=1}\ol q_jq_j.
\end{equation}
The characteristic manifold $\Sigma$ of $\Box^{(q)}_b$ is
\begin{align} \label{e:h-Sj}
\Sigma=\set{(x, \xi)\in T^*(X)\smallsetminus0;\,  \xi=\lambda\omega_0(x), \lambda\neq 0}.
\end{align}
From (\ref{e:h-principalmore}), we see that $p_0$ vanishes to second order at $\Sigma$.
Thus, $\Sigma$ is a doubly characteristic manifold of $\Box^{(q)}_b$. The subprincipal symbol of $\Box^{(q)}_b$ at $(x_0, \xi_0)\in\Sigma$ is given by
\begin{equation} \label{e:0807171410}
p^s_0(x_0, \xi_0)=p_1(x_0, \xi_0)+\frac{i}{2}\sum^{2n-1}_{j=1}\frac{\pr^2p_0(x_0,\xi_0)}{\pr x_j\pr\xi_j}\in
\mathscr L(\Lambda^{0,q}_{x_0}T^*(X),\Lambda^{0,q}_{x_0}T^*(X)).
\end{equation}
It is well-known (see page $83$ of H\"{o}rmander~\cite{Hor85}) that
the subprincipal symbol of $\Box^{(q)}_b$ is invariantly defined on $\Sigma$.

For an operator of the form $Z^*_jZ_j$ this subprincipal symbol is given by
\[\frac{1}{2i}\set{\ol q_j,q_j}\]
and the contribution from the double sum in
(\ref{e:h-principal}) to the subprincipal symbol of $\Box^{(q)}_b$ is
\[\frac{1}{i}\sum^{n-1}_{j,k=1}e^\wedge_je^{\wedge, *}_k\circ\set{q_j,\ol q_k},\]
where $\set{q_j, \ol q_k}$ denotes the Poisson bracket of $q_j$ and $\ol q_k$. We recall that
$\set{q_j, \ol q_k}=\sum^{2n-1}_{s=1}(\frac{\pr q_j}{\pr \xi_s}\frac{\pr\ol q_k}{\pr x_s}-\frac{\pr q_j}{\pr x_s}\frac{\pr\ol q_k}{\pr\xi_s})$.
We get the subprincipal symbol of $\Box^{(q)}_b$ on $\Sigma$,
\begin{equation} \label{e:h-sub}
p^s_0=(\sum^{n-1}_{j=1}-\frac{1}{2i}\set{q_j,\ol q_j})+\sum^{n-1}_{j,k=1}
e^\wedge_je^{\wedge, *}_k\frac{1}{i}\set{q_j,\ol q_k}.
\end{equation}
From (\ref{e:d-Levidecomp}), we see that
$[\ol Z_k, Z_j]=-(2i)L(\ol Z_k, Z_j)Y+h$, $h\in C^\infty(X;\, \Lambda^{1,0}T(X)\oplus \Lambda^{0,1}T(X))$.
Note that the principal symbol of $\ol Z_k$ is $-\ol q_k$. Hence,
\begin{equation} \label{e:h-Levi}
\set{\ol q_k, q_j}=(2i)L(\ol Z_k, Z_j)\sigma_{iY}{\;\rm on\;}\Sigma,
\end{equation}
where $\sigma_{iY}$ is the principal symbol of $iY$. Thus,
\begin{equation} \label{e:h-submore}
p^s_0=(\sum^{n-1}_{j=1}L(\ol Z_j,Z_j)-
\sum^{n-1}_{j,k=1}2e^\wedge_je^{\wedge, *}_kL(\ol Z_k, Z_j))\sigma_{iY}
{\;\rm on\;}\Sigma.
\end{equation}

In the rest of this section, we will assume the reader is familiar with some basic notions of symplectic geometry.
For basic notions and facts of symplectic geometry,
see chapter ${\rm XVIII\,}$ of~\cite{Hor85} or chapter $3$ of Duistermaat~\cite{Dui96}.

From now on, for any $f\in C^\infty(T^*(X))$, we write $H_f$ to denote the Hamilton field of $f$. We need the following

\begin{lem} \label{l:h-symplectic}
We recall that we work with Assumption~\ref{a:0807120935}.
$\Sigma$ is a symplectic submanifold of $T^*(X)$.
\end{lem}

\begin{proof}
Let $\rho\in\Sigma$. Note that
\[\Sigma=\set{(x, \xi)\in T^*(X)\smallsetminus 0;\, q_j(x, \xi)=\ol q_j(x, \xi)=0,\ \ j=1,\ldots,n-1}.\]
Let $\Complex T_\rho(\Sigma)$ and $\Complex T_\rho(T^*(X))$ be the
complexifications of $T_\rho(\Sigma)$ and $T_\rho(T^*(X))$ respectively.
We can choose the basis
$H_{q_1},\ldots,H_{q_{n-1}},H_{\ol q_1},\ldots,H_{\ol q_{n-1}}$
for $T_\rho(\Sigma)^\bot$, where $T_\rho(\Sigma)^{\bot}$ is the orthogonal to
$\Complex T_\rho(\Sigma)$ in $\Complex T_\rho(T^*(X))$ with
respect to the canonical two form,
\begin{equation} \label{e:0807172212}
\sigma=d\xi\wedge dx.
\end{equation}
In view of (\ref{e:h-Levi}), we have $\sigma(H_{q_j},\ H_{\ol q_{k}})=\set{q_j,\ol q_k}=\frac{2}{i}L(\ol Z_k, Z_j)\sigma_{iY}$ on $\Sigma$.
We notice that $\set{q_j, q_k}=0$ on $\Sigma$. Thus, if the Levi form is
non-degenerate at each point of $X$, then $\sigma$ is non-degenerate on
$T_\rho(\Sigma)^\bot$,
hence also on $\Complex T_\rho(\Sigma)$ and $\Sigma$ is therefore symplectic.
\end{proof}

The fundamental matrix of $p_0$ at $\rho=(p, \xi_0)\in\Sigma$ is the linear map $F_\rho$ on $T_\rho(T^*(X))$ defined by
\begin{equation} \label{e:0807172346}
\sigma(t, F_\rho s)=\seq{t, p_0''(\rho)s},\ \ t, s\in T_\rho(T^*(X)),
\end{equation}
where $\sigma$ is the canonical two form (see (\ref{e:0807172212})) and
\[p''_0(\rho)=\left( \begin{array}{cc}
\frac{\pr^2p_0}{\pr x\pr x}(\rho)  & \frac{\pr^2p_0}{\pr\xi\pr x}(\rho)  \\
\frac{\pr^2p_0}{\pr x\pr\xi}(\rho)  & \frac{\pr^2p_0}{\pr\xi\pr\xi}(\rho)
\end{array} \right).\]

We can choose the basis
$H_{q_1},\ldots,H_{q_{n-1}},H_{\ol q_1},\ldots,H_{\ol q_{n-1}}$
for $T_\rho(\Sigma)^\bot$, where $T_\rho(\Sigma)^{\bot}$ is the orthogonal to
$\Complex T_\rho(\Sigma)$ in $\Complex T_\rho(T^*(X))$ with
respect to canonical two form. We notice that
$H_{p_0}=\sum_j(\ol q_jH_{q_j}+q_jH_{\ol q_j})$.
We compute the linearization of $H_{p_0}$ at $\rho$
\begin{align*}
H_{p_0}(\rho+\sum(t_kH_{q_k}+s_kH_{\ol q_k}))
&= O(\abs{t,s}^2)+\sum_{j,k}t_k\set{q_k,\ol q_j}H_{q_j} \\
&\quad +\sum_{j,k}s_k\set{\ol q_k,q_j}H_{\ol q_j}.
\end{align*}
So $F_\rho$ is expressed in the basis
$H_{q_1},\ldots,H_{q_{n-1}},H_{\ol q_1},\ldots,H_{\ol q_{n-1}}$
by
\begin{equation} \label{e:h-fundamental}
F_\rho=\left(
 \begin{array}{cc}
   \set{q_k,\ol q_j}  &  0  \\
     0     &  \set{\ol q_k,q_j}
 \end{array}\right).
\end{equation}
Again, from (\ref{e:h-Levi}), we see that the non-vanishing eigenvalues of $F_\rho$ are
\begin{equation} \label{e:h-eigenvalues}
\pm 2i\lambda_j\sigma_{iY}(\rho),
\end{equation}
where $\lambda_j$, $j=1,\ldots,n-1$, are the eigenvalues of $L_p$.

To compute further, we assume that the Levi form is diagonalized at the given point $p\in X$. Then
\[\sum_{j,k}2e^\wedge_je^{\wedge, *}_k L_p(\ol Z_k, Z_j)\sigma_{iY}=
\sum_j2e^\wedge_je^{\wedge, *}_j L_p(\ol Z_j, Z_j)\sigma_{iY}.\]
From this, we see that on $\Sigma$ and on the space of $(0,q)$ forms,
$p^s_0+\frac{1}{2}\Td{\rm tr\,}F$ has the eigenvalues
\begin{equation} \label{e:h-eigenmore}
\begin{split}
& \sum^{n-1}_{j=1}\abs{\lambda_j}\abs{\sigma_{iY}}+\sum_{j\notin J}\lambda_j\sigma_{iY}-\sum_{j\in J}
  \lambda_j\sigma_{iY}, \; \abs{J}=q,  \\
& J=(j_1,j_2,\ldots,j_q),\; 1\leq j_1<j_2<\cdots<j_q\leq n-1,
\end{split}
\end{equation}
where $\Td{\rm tr\,}F$ denotes $\sum\abs{\mu_j}$, $\pm\mu_j$ are the non-vanishing eigenvalues of $F_\rho$.

Let $(n_-, n_+)$, $n_-+n_+=n-1$, be the signature of the Levi form.
Since $\seq{Y, \omega_0}=-1$, we have $\sigma_{iY}>0$ on $\Sigma^+$, $\sigma_{iY}<0$
on $\Sigma^-$. (We recall that $\Sigma^+$ and $\Sigma^-$ are given by (\ref{e:0807120937}).) Let
\[\inf\, (p^s_0+\frac{1}{2}\Td{\rm tr\,}F)=\inf\set{\lambda;\, \lambda:
{\rm eigenvalue\ of\ }p^s_0+\frac{1}{2}\Td{\rm tr\,}F}.\]
From (\ref{e:h-eigenmore}), we see that on $\Sigma^+$
\begin{equation} \label{e:h-hsiao1}
\inf\,(p^s_0+\frac{1}{2}\Td{\rm tr\,}F) \left\{ \begin{array}{ll}
=0,   &  q=n_+   \\
>0,   &  q\neq n_+
\end{array}\right..
\end{equation}
On $\Sigma^-$
\begin{equation} \label{e:h-hsiao2}
\inf\,(p^s_0+\frac{1}{2}\Td{\rm tr\,}F) \left\{ \begin{array}{ll}
=0,   &  q=n_-   \\
>0,   &  q\neq n_-
\end{array}\right..
\end{equation}

Let $\Omega$ be an open set in $\Real^N$.
Let $P$ be a classical pseudodifferential operator on $\Omega$ of order $m>1$. $P$ is said to be
hypoelliptic with loss of one derivative if $u\in\mathscr E'(\Omega)$ and $Pu\in H^s_{{\rm loc}}(\Omega)$ implies
$u\in H^{s+m-1}_{{\rm comp}}(\Omega)$.

We recall classical works by Boutet de Monvel \cite{Bou74} and Sj\"{o}strand \cite{Sjo74}.

\begin{prop} \label{p:h-Boutet}
Let\, $\Omega$ be an open set in $\Real^N$.
Let $P$ be a classical pseudodifferential operator on $\Omega$ of order $m>1$. The symbol of $P$ takes the form
\[\sigma_P(x,\xi)\sim p_m(x,\xi)+p_{m-1}(x,\xi)+p_{m-2}(x,\xi)+\cdots,\]
where $p_j$ is positively homogeneous of degree $j$. We assume that
$\Sigma=p^{-1}_m(0)$ is a symplectic submanifold of codimension $2d$, $p_m\geq 0$ and $p_m$
vanishes to precisely second order on $\Sigma$. Let $F$ be the fundamental matrix of $p_m$. Let $p^s_m$ be
the subprincipal symbol of $P$. Then $P$ is hypoelliptic with loss of one derivative if and only
if $p^s_m(\rho)+\sum^d_{j=1}(\frac{1}{2}+\alpha_j)\abs{\mu_j}\neq 0$
at every point $\rho\in\Sigma$\, for all $(\alpha_1,\alpha_2,\ldots,\alpha_d)\in\Pstint^d$,
where $\pm i\mu_j$ are the eigenvalues of $F$ at $\rho$.
\end{prop}

Proposition~\ref{p:h-Boutet} also holds if $P$ is a matrix-valued
classical pseudodifferential operator on $\Omega$ of order $m>1$ with scalar principal symbol.

From (\ref{e:h-hsiao1}), (\ref{e:h-hsiao2}) and  Proposition~\ref{p:h-Boutet}, we have the following

\begin{prop} \label{p:h-HorBoutet}
$\Box^{(q)}_b$ is hypoelliptic
with loss of one derivative if and only if\, $Y(q)$ holds at each point of $X$.
(We recall that the definition of\, $Y(q)$ is given by (\ref{e:0807111454}).)
\end{prop}

\begin{rem}
Kohn's $L^2$ estimates give the hypoellipicity with loss of one dervative for the solutions
$\Box^{(q)}_b u=f$ under condition $Y(q)$. (See Folland-Kohn \cite{FK72}.) Kohn's method works as well when
the Levi form $L$ is degenerate.
\end{rem}


\section{The characteristic equation}

In this section, we consider the characteristic equation for $\pr_t+\Box^{(q)}_b$.

Let $p_0(x, \xi)$ be the principal symbol of $\Box^{(q)}_b$.
We work with some real local coordinates $x=(x_1,x_2,\ldots,x_{2n-1})$
defined on an open set $\Omega\subset X$. We identify $\Omega$ with an open
set in $\Real^{2n-1}$. Let $\Omega^\Complex$ be an almost complexification of $\Omega$.
That is, $\Omega^\Complex$ is an open set in
$\Complex^{2n-1}$ with $\Omega^\Complex\bigcap\Real^{2n-1}=\Omega$. We identify $T^*(\Omega)$ with $\Omega\times\Real^{2n-1}$.
Similarly, let $T^*(\Omega)_\Complex$ be an open set in $\Complex^{2n-1}\times\Complex^{2n-1}$ with
$T^*(\Omega)_\Complex\bigcap(\Real^{2n-1}\times\Real^{2n-1})=T^*(\Omega)$. In this section, for any function $f$, we also write
$f$ to denote an almost analytic extension. (See Definition~\ref{d:c-almost}.)
We look for solutions $\psi(t,x,\eta)\in C^\infty
(\ol\Real_+\times T^*(\Omega)\setminus0)$ of the problem
\begin{equation} \label{e:c-chara}
\left\{ \begin{split}
& \frac{\displaystyle\pr\psi}{\displaystyle\pr t}-ip_0(x,\psi'_x)=
   O(\abs{{\rm Im\,}\psi}^N),\ \forall N\geq0,   \\
& \psi|_{t=0}=\seq{x\ ,\eta} \end{split}\right.
\end{equation}
with ${\rm Im\,}\psi(t, x,\eta)\geq0$. More precisely, we look for solutions $\psi(t,x,\eta)\in C^\infty
(\ol\Real_+\times T^*(\Omega)\setminus0)$ with ${\rm Im\,}\psi(t, x,\eta)\geq0$ such that $\psi|_{t=0}=\seq{x\ ,\eta}$
and for every compact set $K\subset T^*(\Omega)\setminus0$, $N\geq0$, there exists $c_{K,N}\geq0$, such that
\[\abs{\frac{\pr\psi}{\pr t}-ip_0(x,\psi'_x)}\leq c_{K,N}\abs{{\rm Im\,}\psi}^N\ \mbox{on}\ \ol\Real_+\times K.\]

Let $f(x, \xi)$, $g(x, \xi)\in C^\infty(T^*(\Omega)_\Complex)$. We write
$f=g$ ${\rm mod\,}$ $\abs{{\rm Im\,}(x, \xi)}^\infty$
if, given any compact subset $K$ of $T^*(\Omega)_\Complex$ and any integer
$N>0$, there is a constant $c>0$ such that
$\abs{(f-g)(x, \xi)}\leq c\abs{{\rm Im\,}(x, \xi)}^N,\; \forall (x, \xi)\in K$.
Let $U$ and $V$ be $C^\infty$ complex vector fields on $T^*(\Omega)_\Complex$. We write
$U=V$ ${\rm mod\,}$ $\abs{{\rm Im\,}(x, \xi)}^\infty$
if $U(f)=V(f)\ {\rm mod\,}\ \abs{{\rm Im\,}(x, \xi)}^\infty$
and $U(\ol f)=V(\ol f)$ ${\rm mod\,}$ $\abs{{\rm Im\,}(x, \xi)}^\infty$,
for all almost analytic functions $f$ on $T^*(\Omega)_\Complex$.
In Appendix A, we discuss the notions of almost analytic vector fields and equivalence of almost analytic vector fields.

In the complex domain, the Hamiltonian field $H_{p_0}$ is given by
$H_{p_0}=\frac{\pr p_0}{\pr\xi}\frac{\pr}{\pr x}-\frac{\pr p_0}{\pr x}\frac{\pr}{\pr\xi}$.
We notice that $H_{p_0}$ depends on the choice of almost analytic extension of $p_0$ but we can give an
invariant meaning of the exponential map $\exp(-itH_{p_0})$, $t\geq0$.
Note that $H_{p_0}$ vanishes on $\Sigma$.
We consider the real vector field $-iH_{p_0}+\ol{-iH_{p_0}}$.
Let $\Phi(t, \rho)$ be the $-iH_{p_0}+\ol{-iH_{p_0}}$ flow. We notice that for every $T>0$ there is an open neighborhood $U$ of $\Sigma$
in $T^*(\Omega)_\Complex$ such that for all $0\leq t\leq T$, $\Phi(t, \rho)$ is well-defined if $\rho\in U$. Since we only need to consider
Taylor expansions at $\Sigma$, for the convenience, we assume that $\Phi(t, \rho)$ is well-defined for all $t\geq0$ and
$\rho\in T^*(\Omega)_\Complex$. The following follows from Proposition~\ref{p:0709161529}

\begin{prop} \label{p:0709191406}
Let\, $\Phi(t, \rho)$ be as above. Let $U$ be a real vector field on $T^*(\Omega)_\Complex$ such that
$U=-iH_{p_0}+\ol{-iH_{p_0}}$ ${\rm mod\,}$ $\abs{{\rm Im\,}(x, \xi)}^\infty$.
Let\, $\hat\Phi(t, \rho)$ be the $U$ flow. Then, for every compact set $K\subset T^*(\Omega)_\Complex$, $N\geq0$, there is
$c_{N,K}(t)>0$,
such that
\[\abs{\Phi(t, \rho)-\hat\Phi(t, \rho)}\leq c_{N,K}(t){\rm dist\,}(\rho, \Sigma)^N,\ \rho\in K.\]
\end{prop}

For $t\geq0$, let
\begin{equation} \label{e:0709191622}
G_t=\set{(\rho, \Phi(t, \rho));\, \rho\in T^*(\Omega)_\Complex},
\end{equation}
where $\Phi(t, \rho)$ is as in Proposistion~\ref{p:0709191406}. We call
$G_t$ the graph of ${\rm exp\,}(-itH_{p_0})$. Since $H_{p_0}$ vanishes on $\Sigma$, we have
$\Phi(t, \rho)=\rho$ if $\rho\in\Sigma$.
$G_t$ depends on the choice of almost analytic extension
of $p_0$. Let $\hat p_0$ be another almost analytic extension of $p_0$. Let $\hat G_t$ be the graph of
${\rm exp\,}(-itH_{\hat p_0})$. From Proposition~\ref{p:0709191406}, it
follows that $G_t$ coincides to infinite order with $\hat G_t$ on
${\rm diag\,}(\Sigma\times\Sigma)$, for all $t\geq0$.

In Menikoff-Sj\"{o}strand~\cite{MS78}, it was shown that there exist
\[g(t, x, \eta), h(t, x, \eta)\in C^\infty(\ol\Real_+\times T^*(\Omega)_\Complex)\]
such that
$G_t=\set{(x, g(t, x, \eta), h(t, x, \eta), \eta);\, (x, \eta)\in T^*(\Omega)_\Complex}$.
Moreover, there exists $\psi(t, x, \eta)\in C^\infty(\ol\Real_+\times T^*(\Omega)_\Complex)$ such that
\[g(t, x, \eta)-\psi'_x(t, x, \eta)\ \ \mbox{and}\ \  h(t, x, \eta)-\psi'_\eta(t, x, \eta)\]
vanish to infinite order on $\Sigma$, for all $t\geq0$.
Furthermore, when $(t,x,\eta)$ is real, $\psi(t,x,\eta)$ solves (\ref{e:c-chara}) and we have,
\begin{equation} \label{e:c-charanext}
{\rm Im\,}\psi(t, x,\eta)\asymp\frac{t}{1+t}({\rm dist\,}((x,\eta),\Sigma))^2,\ \ t\geq0,\ \abs{\eta}=1.
\end{equation}
For the precise meaning of $\asymp$, see the discussion after Proposition~\ref{p:i-leading1}.

Moreover, we have the following

\begin{prop} \label{p:c-basis}
There exists $\psi(t,x,\eta)\in C^\infty(\ol\Real_+\times T^*(\Omega)\setminus0)$ such
that ${\rm Im\,}\psi\geq 0$ with equality precisely on $(\set{0}\times T^*(\Omega)\setminus0)
\bigcup(\Real_+\times\Sigma)$ and such that {\rm (\ref{e:c-chara})} holds where the error term is
uniform on every set of the form $[0,T]\times K$ with $T>0$ and $K\subset T^*(\Omega)\setminus0$ compact.
Furthermore, $\psi$ is unique up to a term which is
$O(\abs{{\rm Im\,}\psi}^N)$ locally uniformly for every $N$ and
\[\psi(t,x,\eta)=\seq{x,\eta}{\;\rm on\;}\Sigma,\; d_{x,\eta}(\psi-\seq{x,\eta})=0{\;\rm on\;}
\Sigma.\]
Moreover, we have
\begin{equation} \label{e:c-sj1}
{\rm Im\,}\psi(t, x,\eta)\asymp\abs{\eta}\frac{t\abs{\eta}}{1+t\abs{\eta}}{\rm dist\,}
((x, \frac{\eta}{\abs{\eta}}),\Sigma))^2,
\ \ t\geq0,\ \abs{\eta}\geq1.
\end{equation}
\end{prop}

\begin{prop} \label{p:c-basislimit}
There exists a function $\psi(\infty,x,\eta)\in C^\infty(T^*(\Omega)\setminus0)$ with a
uniquely determined Taylor expansion at each point of\, $\Sigma$ such that
\[\begin{split}
&\mbox{For every compact set $K\subset T^*(\Omega)\setminus0$ there is a constant $c_K>0$
such that}   \\
&{\rm Im\,}\psi(\infty,x,\eta)\geq c_K\abs{\eta}({\rm dist\,}((x,\frac{\eta}{\abs{\eta}}),\Sigma))^2,  \\
&d_{x,\eta}(\psi(\infty, x, \eta)-\seq{x,\eta})=0\ {\rm on\,}\ \Sigma.
\end{split}\]
If$\ $ $\lambda\in C(T^*(\Omega)\setminus 0)$, $\lambda>0$ and $\lambda|_\Sigma<\min\abs{\lambda_j}$,
where $\pm i\abs{\lambda_j}$ are the non-vanishing eigenvalues of the fundamental matrix of\, $\Box^{(q)}_b$,
then the solution $\psi(t,x,\eta)$ of {\rm (\ref{e:c-chara})} can be chosen so that for every
compact set $K\subset T^*(\Omega)\setminus0$ and all indices $\alpha$, $\beta$, $\gamma$,
there is a constant $c_{\alpha,\beta,\gamma,K}$ such that
\begin{equation} \label{e:c-************}
\abs{\pr^\alpha_x\pr^\beta_\eta\pr^\gamma_t(\psi(t,x,\eta)-\psi(\infty,x,\eta))}
\leq c_{\alpha,\beta,\gamma,K}e^{-\lambda(x,\eta)t}{\;\rm on\;}\ol\Real_+\times K.
\end{equation}
\end{prop}

For the proofs of Proposition~\ref{p:c-basis} and Proposition~\ref{p:c-basislimit}, we refer the reader to
~\cite{MS78}. From the positively homogeneity of $p_0$, it follows that we can choose
$\psi(t,x,\eta)$ in Proposition~\ref{p:c-basis} to be quasi-homogeneous of degree $1$ in the sense that
$\psi(t,x,\lambda\eta)=\lambda\psi(\lambda t,x,\eta)$, $\lambda>0$.
(See Definition~\ref{d:hf-heatquasi}.) This makes
$\psi(\infty,x,\eta)$ positively homogeneous of degree $1$.

We recall that $p_0=q_1\ol q_1+\cdots+q_{n-1}\ol q_{n-1}$.
We can take an almost analytic extension of $p_0$ so that
\begin{equation} \label{e:0707311415}
p_0(x, \xi)=\ol p_0(\ol x, \ol\xi).
\end{equation}
From (\ref{e:0707311415}), we have
\[-\frac{\pr\ol\psi}{\pr t}(t, x, -\eta)-ip_0(x, \ol\psi'_x(t, x, -\eta))=O(\abs{{\rm Im\,}\psi}^N),\ t\geq0,\]
for all $N\geq0$, $(x, \eta)$ real. Since $p_0(x, -\xi)=p_0(x, \xi)$, we have
\begin{align} \label{e:s-important***************&&&}
-\frac{\pr\ol\psi}{\pr t}(t, x, -\eta)-ip_0(x, -\ol\psi'_x(t, x, -\eta))=O(\abs{{\rm Im\,}\psi}^N),\ t\geq0,
\end{align}
for all $N\geq0$, $(x, \eta)$ real. From Proposition~\ref{p:c-basis}, we can take $\psi(t, x, \eta)$ so that
$\psi(t, x, \eta)=-\ol\psi(t, x, -\eta)$, $(x, \eta)$ is real. Hence,
\begin{equation} \label{e:nevergiveup1}
\psi(\infty, x, \eta)=-\ol\psi(\infty, x, -\eta),\ \ (x, \eta)\ \mbox{is real}.
\end{equation}

Put $\Td G_t=\set{(\ol y, \ol\eta, \ol x, \ol\xi);\, (x, \xi, y, \eta)\in G_t}$,
where $G_t$ is defined by (\ref{e:0709191622}). From (\ref{e:0707311415}), it follows that
$\Phi(t, \ol\rho)=\ol\Phi(-t, \rho)$,
where $\Phi(t, \rho)$ is as in Proposition~\ref{p:0709191406}. Thus, for all $t\geq0$,
\begin{equation} \label{e:0402-1}
G_t=\Td G_t.
\end{equation}
Put
\begin{equation} \label{e:0709191636}
C_t=\set{(x, \psi'_x(t, x, \eta), \psi'_\eta(t, x, \eta), \eta);\, (x, \eta)\in T^*(\Omega^\Complex)}
\end{equation}
and $\Td C_t=\set{(\ol y, \ol\eta, \ol x, \ol\xi);\, (x, \xi, y, \eta)\in C_t}$.
Since $C_t$ coincides to infinite order with $G_t$ on
${\rm diag\,}(\Sigma\times\Sigma)$, for all $t\geq0$, from (\ref{e:0402-1}), it follows
that $C_t$ coincides to infinite order with $\Td C_t$ on
${\rm diag\,}(\Sigma\times\Sigma)$, for all $t\geq0$. Letting $t\To\infty$, we get the following

\begin{prop} \label{p:0708261503}
Let
\begin{equation} \label{e:0709191634}
C_\infty=\set{(x, \psi'_x(\infty, x, \eta), \psi'_\eta(\infty, x, \eta), \eta);\, (x, \eta)\in T^*(\Omega^\Complex)}
\end{equation}
and $\Td C_\infty=\set{(\ol y, \ol\eta, \ol x, \ol\xi);\, (x, \xi, y, \eta)\in C_\infty}$.
Then $\Td C_\infty$ coincides to infinite order with $C_\infty$ on ${\rm diag\,}(\Sigma\times\Sigma)$.
\end{prop}

From Proposition~\ref{p:0708261503} and the global theory of Fourier integral operators (see Proposition~\ref{p:c-mesjmore}),
we have the following

\begin{prop} \label{p:0709201320}
The two phases
\[\psi(\infty, x, \eta)-\seq{y, \eta},
-\ol\psi(\infty, y, \eta)+\seq{x, \eta}\in C^\infty(\Omega\times\Omega\times\dot\Real^{2n-1})\]
are equivalent in the sense of Melin-Sj\"{o}strand~\cite{MS74}. (See Definition~\ref{d:0709171814}.)
\end{prop}

We recall that
\[\Sigma=\set{(x, \xi)\in T^*(\Omega)\smallsetminus 0;\, q_j(x, \xi)=\ol q_j(x, \xi)=0,\ \ j=1,\ldots,n-1}.\]
For any function $f\in C^\infty(T^*(\Omega))$, we use $\Td f$ to denote an almost analytic extension with respect to
the weight function
${\rm dist}((x, \xi), \Sigma)$. (See Definition~\ref{d:c-almost}.) Set
\begin{equation} \label{e:070731}
\Td\Sigma=\set{(x, \xi)\in T^*(\Omega)_\Complex\smallsetminus 0;\, \Td q_j(x, \xi)=\Td{\ol q}_j(x, \xi)=0,\ \ j=1,\ldots,n-1}.
\end{equation}
We say that $\Td\Sigma$ is an almost analytic extension with respect to the weight function
${\rm dist}((x, \xi), \Sigma)$ of $\Sigma$. Let $f(x, \xi)$, $g(x, \xi)\in C^\infty(W)$, where $W$ is
an open set in $T^*(\Omega)_\Complex$. We write $f=g$ ${\rm mod\,}$ $d^\infty_\Sigma$
if, given any compact subset $K$ of $W$ and any integer
$N>0$, there is a constant $c>0$ such that
$\abs{(f-g)(x, \xi)}\leq c{\rm dist\,}((x, \xi), \Sigma)^N$, $\forall (x, \xi)\in K$.
From the global theory of Fourier integral operators (see Proposition~\ref{p:c-mesjmore}),
we only need to consider Taylor expansions at $\Sigma$. We may work with the
following coordinates (for the proof, see~\cite{MS78})

\begin{prop} \label{p:c-SjoBou}
Let $\rho\in\Sigma$. Then in some open neighborhood $\Gamma$ of
$\rho$ in $T^*(\Omega)_\Complex$,
there are $C^\infty$ functions
$\Td x_j\in C^\infty(\Gamma)$, $\Td\xi_j\in C^\infty(\Gamma)$, $j=1,\ldots,2n-1$,
such that
\begin{enumerate}
\item $\Td x_j$, $\Td\xi_j$, $j=1,\ldots,2n-1$, are almost analytic functions with respect to the weight function
${\rm dist}((x, \xi), \Sigma)$.
\item ${\rm det\,}\Bigl(\frac{\pr(x,\xi)}{\pr(\Td x,
\Td\xi)}\Bigr)\neq 0\ {\rm on\,}\ (\Gamma)_\Real$,
where $(\Gamma)_\Real=\Gamma\bigcap T^*(\Omega)$ and $\Td x=(\Td x_1,\ldots,\Td x_{2n-1})$, $\Td\xi=(\Td\xi_1,\ldots,\Td\xi_{2n-1})$.
\item $\Td x_j$, $\Td\xi_j$, $j=1,\ldots,2n-1$, form local coordinates of\, $\Gamma$.
\item $(\Td x, \Td\xi)$ is symplectic to infinite order on $\Sigma$. That is,
$\{\Td x_j, \Td x_k\}=0$ ${\rm mod\,}$ $d^\infty_\Sigma$, $\{\Td\xi_j, \Td\xi_k\}=0$ ${\rm mod\,}\ d^\infty_\Sigma$,
$\{\Td\xi_j, \Td x_k\}=\delta_{j, k}$ ${\rm mod\,}\ d^\infty_\Sigma$,
where $j$, $k=1,\ldots,2n-1$. Here $\set{f, g}=\frac{\pr f}{\pr\xi}\frac{\pr g}{\pr x}-\frac{\pr f}{\pr x}\frac{\pr g}{\pr\xi}$, $f$, $g\in C^\infty(\Gamma)$.
\item  We write $\Td x'$, $\Td x''$, $\Td\xi'$, $\Td\xi''$ to denote $(\Td x_1,\ldots,\Td x_{n})$,
$(\Td x_{n+1},\ldots,\Td x_{2n-1})$, $(\Td\xi_1,\ldots,\Td\xi_{n})$ and
$(\Td\xi_{n+1},\ldots,\Td\xi_{2n-1})$ respectively. Then,
$\Td\Sigma\bigcap\Gamma$ coincides to infinite order with
$\set{(\Td x, \Td\xi);\, \Td x''=0,\ \Td\xi''=0}$
on $\Sigma\bigcap(\Gamma)_\Real$ and
\[\Sigma\bigcap(\Gamma)_\Real=\set{(\Td x, \Td\xi);\, \Td x''=0,\ \Td\xi''=0,\ \Td x'\ {\rm and\,}\ \Td\xi'\ {\rm are\, real\,}}.\]
\end{enumerate}
Furthermore, there is a $(n-1)\times(n-1)$ matrix of almost analytic functions $A(\Td x, \Td\xi)$ such that
for every compact set $K\subset\Gamma$ and $N\geq0$, there is a $c_{K,N}>0$, such that
\[\abs{p_0(\Td x,\Td\xi)-i\seq{A(\Td x,\Td\xi)\Td x'',\Td\xi''}}\leq c_{K,N}\abs{(\Td x'',\Td\xi'')}^N\ \mbox{on}\ K,\]
and when $\Td x'$ and $\Td\xi'$ are real, $A(\Td x',0,\Td\xi',0)$ has positive eigenvalues
\[\abs{\lambda_1},\ldots,\abs{\lambda_{n-1}},\]
where $\pm i\lambda_{1},\ldots,\pm i\lambda_{n-1}$ are the non-vanishing eigenvalues
of $F(\Td x',0,\Td\xi',0)$, the fundamental matrix of\, $\Box^{(q)}_b$. In particular,
\[\frac{1}{2}{\rm tr\,}A(\Td x',0,\Td\xi',0)=\frac{1}{2}\Td{\rm tr\,}
F(\Td x',0,\Td\xi',0).\]
Formally, we write
\begin{equation} \label{e:c-SjoBou}
p_0(\Td x,\Td\xi)=i\seq{A(\Td x,\Td\xi)\Td x'',\Td\xi''}
+O(\abs{(\Td x'',\Td\xi'')}^N).
\end{equation}
\end{prop}

\begin{rem} \label{r:c-SjoMe}
Set
\[E=\set{(t,x,\xi,y,\eta)\in\ol\Real_+\times T^*(\Omega)_\Complex\times T^*(\Omega)_\Complex;\,
(x,\xi,y,\eta)\in C_t},\]
where $C_t$ is defined by (\ref{e:0709191636}).
Let $(\Td x,\Td\xi)$ be the coordinates of Proposition~\ref{p:c-SjoBou}.
In the work of Menikoff-Sj\"{o}strand~\cite{MS78}, it was shown that there exists
$\Td\psi(t,\Td x,\Td\eta)\in C^\infty(\ol\Real_+\times\Gamma)$, where $\Gamma$ is as in
Proposition~\ref{p:c-SjoBou}, such that
\[
\left\{ \begin{split}
& \frac{\displaystyle\pr\Td\psi}{\displaystyle\pr t}-ip_0(\Td x,\Td\psi'_{\Td x})=
   O(\abs{(\Td x'', \Td\eta'')}^N), \mbox{for all}\ N>0,   \\
& \Td\psi|_{t=0}=\seq{\Td x\ ,\Td\eta} \end{split}\right.
\]
and $\Td\psi(t,\Td x,\Td\eta)$ is of the form
\begin{equation} \label{e:c-SjoMe}
\Td\psi(t,\Td x,\Td\eta)=\seq{\Td x',\Td\eta'}+\seq{e^{-tA(\Td x', 0,  \Td\eta', 0)}
\Td x'',\Td\eta''}+\Td\psi_2(t, \Td x, \Td\eta)+\Td\psi_3(t, \Td x, \Td\eta)+\cdots,
\end{equation}
where $A$ is as in Proposition~\ref{p:c-SjoBou} and $\Td\psi_j(t, \Td x, \Td\eta)$ is a $C^\infty$ homogeneous polynomial of degree $j$ in
$(\Td x'', \Td\eta'')$. If$\ $ $\lambda\in C(T^*(\Omega)\setminus 0)$, $\lambda>0$ and $\lambda|_{\Sigma}<\min\abs{\lambda_j}$, where
$\pm i\abs{\lambda_j}$ are the non-vanishing eigenvalues of the fundamental matrix of\, $\Box^{(q)}_b$,
then for every
compact set $K\subset\Sigma\bigcap(\Gamma)_\Real$ and all indices $\alpha$, $\beta$, $\gamma$, $j$,
there is a constant $c_{\alpha,\beta,\gamma, j, K}$ such that
\begin{equation} \label{e:0711101650}
\abs{\pr^\alpha_{\Td x}\pr^\beta_{\Td \eta}\pr^\gamma_t(\Td\psi_j(t,\Td x,\Td \eta))}
\leq c_{\alpha,\beta,\gamma,K}e^{-\lambda(\Td x,\Td \eta)t}{\;\rm on\;}\ol\Real_+\times K.
\end{equation}

Put $\Td{E}=\set{(t, \Td x, \frac{\pr\Td{\psi}}{\pr\Td{x}}(t,\Td x, \Td\eta), \frac{\pr\Td{\psi}}{\pr\Td{\eta}}(t, \Td x, \Td\eta),
\Td\eta);\, t\in\ol\Real_+,\ \Td x, \Td\eta\in C^\infty(\Gamma)}$.
We notice that $\Td{E}$ coincides to infinite order with $E$ on
$\ol\Real_+\times{\rm diag\,}((\Sigma\bigcap(\Gamma)_\Real)\times(\Sigma\bigcap(\Gamma)_\Real))$.
(See~\cite{MS78}.)
\end{rem}


\section{The heat equation, formal construction}

We work with some real local coordinates $x=(x_1,\ldots,x_{2n-1})$ defined on an open set $\Omega\subset X$.
We identify $T^*(\Omega)$ with $\Omega\times\Real^{2n-1}$.

\begin{defn} \label{d:hf-heatquasi}
We will say that $a\in C^\infty(\ol\Real_+\times T^*(\Omega))$ is quasi-homogeneous of
degree $j$ if $a(t,x,\lambda\eta)=\lambda^ja(\lambda t,x,\eta)$ for all $\lambda>0$.
\end{defn}

We can check that if $a$ is quasi-homogeneous of degree $j$, then $\pr^\alpha_x\pr^\beta_\eta\pr^\gamma_ta$ is quasi-homogeneous of degree
$j-\abs{\beta}+\gamma$.

In this section, we consider the problem
\begin{equation} \label{e:hf-heat}
\left\{ \begin{array}{ll}
(\pr_t+\Box^{(q)}_b)u(t,x)=0  & {\rm in\;}\Real_+\times\Omega  \\
u(0,x)=v(x) \end{array}\right..
\end{equation}
We shall start by making only a formal construction. We look for an approximate solution of
(\ref{e:hf-heat}) of the form $u(t,x)=A(t)v(x)$,
\begin{equation} \label{e:hf-fourierheat}
A(t)v(x)=\frac{1}{(2\pi)^{2n-1}}\int\int e^{i(\psi(t,x,\eta)-\seq{y,\eta})}a(t,x,\eta)v(y)dyd\eta
\end{equation}
where formally
\[a(t,x,\eta)\sim\sum^\infty_{j=0}a_j(t,x,\eta),\]
$a_j(t, x, \eta)\in C^\infty(\ol\Real_+\times T^*(\Omega);\,
\mathscr L(\Lambda^{0,q}T^*(\Omega), \Lambda^{0,q}T^*(\Omega)))$,
$a_j(t,x,\eta)$ is a quasi-homogeneous function of degree $-j$.

We let the full symbol of $\Box^{(q)}_b$ be:
\[\mbox{full symbol of }\ \Box^{(q)}_b=\sum^2_{j=0}p_j(x,\xi),\]
where $p_j(x,\xi)$ is positively homogeneous of order $2-j$. We apply $\pr_t+\Box^{(q)}_b$ formally
under the integral in (\ref{e:hf-fourierheat}) and then introduce the asymptotic expansion of
$\Box^{(q)}_b(ae^{i\psi})$. (See Proposition~\ref{p:0709171623}.) Setting $(\pr_t+\Box^{(q)}_b)(ae^{i\psi})\sim 0$ and regrouping
the terms according to the degree of quasi-homogeneity. We obtain the transport equations
\begin{equation} \label{e:hf-heattransport}
\left\{ \begin{array}{ll}
 T(t,x,\eta,\pr_t,\pr_x)a_0=O(\abs{{\rm Im\,}\psi}^N),\; \forall N   \\
 T(t,x,\eta,\pr_t,\pr_x)a_j+l_j(t,x,\eta,a_0,\ldots,a_{j-1})= O(\abs{{\rm Im\,}\psi}^N),\; \forall N.
 \end{array}\right.
\end{equation}
Here
\[T(t,x,\eta,\pr_t,\pr_x)=\pr_t-i\sum^{2n-1}_{j=1}\frac{\pr p_0}{\pr\xi_j}(x,\psi'_x)\frac{\pr}{\pr x_j}+q(t,x,\eta)\]
where
\[q(t,x,\eta)=p_1(x,\psi'_x)+\frac{1}{2i}\sum^{2n-1}_{j,k=1}\frac{\pr^2p_0(x,\psi'_x)}
    {\pr\xi_j\pr\xi_k}\frac{\pr^2\psi(t,x,\eta)}{\pr x_j\pr x_k}\]
and $l_j$ is a linear differential operator acting on $a_0,a_1,\ldots,a_{j-1}$. We note that
$q(t,x,\eta)\To q(\infty,x,\eta)$ exponentially fast in the sense of (\ref{e:c-************})
and that the same is true for the coefficients of $l_j$.

Let $C_t$, $E$ be as in (\ref{e:0709191636}) and Remark~\ref{r:c-SjoMe}. We recall that for $t\geq0$,
\[C_t=\set{(x,\xi,y,\eta)\in T^*(\Omega)_\Complex\times T^*(\Omega)_\Complex;\,
\xi=\frac{\pr\psi}{\pr x}(t,x,\eta), y=\frac{\pr\psi}{\pr\eta}(t,x,\eta)},\]
\[E=\set{(t,x,\xi,y,\eta)\in\ol\Real_+\times T^*(\Omega)_\Complex\times T^*(\Omega)_\Complex;\,
(x,\xi,y,\eta)\in C_t}\]
and for $t>0$, $(C_t)_\Real={\rm diag\,}(\Sigma\times\Sigma)=\set{(x,\xi,x,\xi)\in T^*(\Omega)\times T^*(\Omega);\, (x,\xi)\in\Sigma}$.

If we consider $a_0,a_1,\ldots$ as functions on $E$, then the equations (\ref{e:hf-heattransport})
involve differentiations along the vector field $\nu=\frac{\pr}{\pr t}-iH_{p_0}$.
We can consider only Taylor expansions at $\Sigma$. Until further notice, our computations
will only be valid to infinite order on $\Sigma$.

Consider $\nu$ as a vector field on $E$. In the coordinates $(t,x,\eta)$ we can express $\nu$:
\[\nu=\frac{\pr}{\pr t}-i\sum^{2n-1}_{j=1}\frac{\pr p_0}{\pr\xi_j}(x,\psi'_x)\frac{\pr}{\pr x_j}.\]
We can compute
\begin{equation} \label{e:hf-div}
{\rm div\,}(\nu)=\frac{1}{i}\left(\sum^{2n-1}_{j=1}\frac{\pr^2p_0(x,\psi'_x)}{\pr x_j\pr\xi_j}+
\sum^{2n-1}_{j, k=1}\frac{\pr^2p_0}{\pr\xi_j\pr\xi_k}(x,\psi'_x)\frac{\pr^2\psi}
{\pr x_j\pr x_k}(t,x,\eta)\right).
\end{equation}
For a smooth function $a(t,x,\eta)$ we introduce the $\frac{1}{2}$ density on $E$
\[ \alpha=a(t,x,\eta)\sqrt{dtdxd\eta}\]
which is well-defined up to some factor $i^\mu$. (See H\"{o}rmander \cite{Hor71}.) The Lie derivative of $\alpha$ along
$\nu$ is $L_{\nu}(\alpha)=(\nu(a)+\frac{1}{2}{\rm div\,}(\nu)a)\sqrt{dtdxd\eta}$.
We see from the expression for $T$ that
\begin{equation} \label{e:hf-divmore}
(Ta)\sqrt{dtdxd\eta}=(L_{\nu}+p^s_0(x,\psi'_x(t,x,\eta)))(a\sqrt{dtdxd\eta}),
\end{equation}
where $p^s_0(x,\xi)=p_1(x,\xi)+\frac{i}{2}\sum^{2n-1}_{j=1}\frac{\pr^2p_0(x,\xi)}{\pr x_j\pr\xi_j}$
is the subprincipal symbol (see (\ref{e:0807171410})). Now let
$(\Td x,\Td\xi)$ be the coordinates of Proposition~\ref{p:c-SjoBou}, in which $p_0$
takes the form (\ref{e:c-SjoBou}). In these coordinates we have
\begin{align} \label{e:hf-transmore}
H_{p_0}(\Td x,\Td\xi)
&= i\seq{A(\Td x,\Td\xi)\Td x''\ ,\frac{\pr}{\pr\Td x''}}-
   i\seq{\mbox{}{^t\hskip-2pt}A(\Td x,\Td\xi)\Td\xi'',\frac{\pr}
   {\pr\Td\xi''}}  \nonumber  \\
&\quad +\sum_{\abs{\alpha}=1,\abs{\beta}=1}(\Td x'')^\alpha(\Td\xi'')^\beta
        B_{\alpha\beta}(\Td x,\Td\xi,\frac{\pr}{\pr\Td x},
        \frac{\pr}{\pr\Td\xi})
\end{align}
and
\begin{align} \label{e:hf-hami}
\nu &=\frac{\pr}{\pr t}+\seq{A(\Td x,\Td\psi'_{\Td  x})\Td x'',
      \frac{\pr}{\pr\Td x''}} \nonumber  \\
    &\quad +\sum_{\abs{\alpha}=1,\abs{\beta}=1}(\Td x'')^\alpha
      (\Td\psi'_{\Td x''})^\beta C_{\alpha\beta}(\Td x',\Td\psi'_{\Td x},
      \frac{\pr}{\pr\Td x}).
\end{align}
Here $\Td\psi(t, \Td x, \Td\eta)$ is as in Remark~\ref{r:c-SjoMe}.

Let $f(t, x, \eta)\in C^\infty(\ol\Real_+\times T^*(\Omega)_\Complex)$,
$f(\infty, x, \eta)\in C^\infty(T^*(\Omega)_\Complex)$.
We say that $f(t, x, \eta)$ converges exponentially
fast to $f(\infty, x, \eta)$ if $f(t, x, \eta)-f(\infty, x, \eta)$
satisfies the same kind of estimates as (\ref{e:c-************}).
Recalling the form of $\Td\psi$ we obtain
\begin{align} \label{e:hf-hamimore}
\nu=\tilde{\nu} &=\frac{\pr}{\pr t}+\seq{A(\Td x',0, \Td\eta',0)\Td x'',
      \frac{\pr}{\pr\Td x''}}  \nonumber  \\
    &\quad +\sum_{\abs{\alpha+\beta}=2,\alpha\neq 0}(\Td x'')^\alpha(\Td\eta'')^\beta
      D_{\alpha\beta}(t,\Td x,\Td\eta,\frac{\pr}{\pr\Td x})
\end{align}
where $D_{\alpha\beta}$ converges exponentially fast to some limit as $t\To +\infty$. We have on $\Sigma$,
\begin{equation} \label{e:hf-tr}
\frac{1}{2}{\rm div\,}(\tilde{\nu})=\frac{1}{2}{\rm tr\,}A(\Td x',0,\Td\eta',0)=\frac{1}{2}
\Td{\rm tr\,}F(\Td x',0,\Td\eta',0)
\end{equation}
where $F(\Td x',0,\Td\eta',0)$ is the fundamental matrix of $\Box^{(q)}_b$.
We define $\Td a(t,\Td x,\Td\eta)$ by
\begin{equation} \label{e:hf-solution}
\Td a(t,\Td x,\Td\eta)\sqrt{dtd\Td xd\Td\eta}=
           a(t,x,\eta)\sqrt{dtdxd\eta}.
\end{equation}
Note that the last equation only defines $\Td a$ up to $i^\mu$.  We have
\[(Ta)\sqrt{dtdxd\eta}=(\Td T\Td a)\sqrt{dtd\Td xd\Td\eta}\]
where
\begin{align} \label{e:hf-trsolution}
\Td T
   &= \frac{\pr}{\pr t}+\seq{A(\Td x',0,\Td\eta',0)\Td x'',
      \frac{\pr}{\pr\Td x''}} \nonumber \\
   &\quad +\frac{1}{2}\Td{\rm tr\,}F(\Td x',0,\Td\eta',0)+
      p^s_0(\Td x',0,\Td\eta',0)+Q(t,\Td x,\Td\eta,\frac{\pr}
      {\pr\Td x}).
\end{align}
Here
\begin{align*}
Q(t,\Td x,\Td\eta,\frac{\pr}{\pr\Td x})
   &= \sum_{\abs{\alpha+\beta}=2,\alpha\neq 0}(\Td x'')^\alpha(\Td\eta'')^\beta
      D_{\alpha\beta}(t,\Td x,\Td\eta,\frac{\pr}{\pr\Td x}) \\
   &\quad +\sum_{\abs{\alpha+\beta}=1}(\Td x'')^\alpha(\Td\eta'')^\beta
      E_{\alpha\beta}(t,\Td x,\Td\eta).
\end{align*}
It is easy to see that $E_{\alpha\beta}$ and $D_{\alpha\beta}$ converge exponentially fast to
some limits $E_{\alpha\beta}(\infty,\Td x,\Td\eta)$ and $D_{\alpha\beta}
(\infty,\Td x,\Td\eta)$ respectively. We need the following

\begin{lem} \label{l:hf-BerSjo}
Let $A$ be a $d\times d$ matrix having only positive eigenvalues and consider the map
$\mathscr A:u\mapsto \seq{ A
\left( \begin{array}{c} x_1 \\ \vdots \\ x_d  \end{array} \right)\ , \
\left( \begin{array}{c}
\frac{\pr u}{\pr x_1} \\ \vdots \\ \frac{\pr u}{\pr x_d} \end{array} \right)
}$
on the space $P^m(\Real^d)$ of homogeneous polynomials of degree $m$. Then $\exp(t\mathscr A)(u)=u\circ(\exp(tA))$
and the map $\mathscr A$ is a bijection except for $m=0$.
\end{lem}

\begin{proof}
We notice that $U(t):u\mapsto u\circ\exp(tA)$ form a group of operators and that
$(\frac{\pr U(t)}{\pr t})\Big|_{t=0}=\mathscr A$.
This shows that $U(t)=\exp(t\mathscr A)$. To prove the second statement, suppose that $u\in P^m$,
$m\geq 1$ and $\mathscr A(u)=0$. Then $\exp(t\mathscr A)(u)=u$ for all $t$, in other words
$u(\exp(tA)(x))=u(x)$, $t\in\Real$, $x\in\Real^d$. Since $\exp(tA)(x)\To 0$ when $t\To -\infty$,
we obtain $u(x)=u(0)=0$, which proves the lemma.
\end{proof}

\begin{prop} \label{p:hf-fundamental}
Let
\[c_j(x, \eta)\in C^\infty(T^*(\Omega);\,
\mathscr L(\Lambda^{0,q}T^*(\Omega), \Lambda^{0,q}T^*(\Omega))),\ \ j=0, 1,\ldots,\]
be positively homogeneous functions of degree $-j$. Then, we can find solutions
\[a_j(t, x, \eta)\in C^\infty(\ol\Real_+\times T^*(\Omega);\,
\mathscr L(\Lambda^{0,q}T^*(\Omega), \Lambda^{0,q}T^*(\Omega))),\ \ j=0, 1,\ldots,\]
of the system {\rm (\ref{e:hf-heattransport})} with
$a_j(0, x, \eta)=c_j(x, \eta)$, $j=0, 1,\ldots$,
where $a_j(t,x,\eta)$ is a
quasi-homogeneous function of degree $-j$ such that $a_j(t,x,\eta)$ has unique Taylor expansions
on $\Sigma$, for all $j$. Furthermore, let $\lambda(x,\eta)\in C(T^*(\Omega))$ and $\lambda|_\Sigma<\min\tau_j$,
where $\tau_j$ are the eigenvalues of $\frac{1}{2}\Td{\rm tr\,}F+p^s_0$ on $\Sigma$. Then
for all indices $\alpha,\beta,\gamma,j$ and every compact set $K\subset\Sigma$ there exists a
constant $c>0$ such that
\begin{equation} \label{e:hf-estimate}
\abs{\pr^\gamma_t\pr^\alpha_x\pr^\beta_\eta a_j(t,x,\eta)}\leq
ce^{-t\lambda(x,\eta)}{\;\rm on\;}\ol\Real_+\times K.
\end{equation}
\end{prop}

\begin{proof}
We only need to study Taylor expansions on $\Sigma$. Let $(\Td x,\Td\xi)$ be the coordinates of
Proposition~\ref{p:c-SjoBou}. We define
$\Td a_j(t,\Td x,\Td\eta)$ from $a_j(t,x,\eta)$ as in (\ref{e:hf-solution}).
In order to prove (\ref{e:hf-estimate}),
it is sufficient to prove the corresponding statement for $\Td a_j$ (see section $1$ of~\cite{MS78}).
We introduce the Taylor expansion of $\Td a_0$ with respect to $(\Td x'',\Td\eta'')$.
$\Td a_0(t,\Td x,\Td\eta)=\sum^\infty_0\Td a^j_0(t,\Td x,
\Td\eta)$,
where $\Td a^j_0$ is a homogeneous polynomial of degree $j$ in $(\Td x'',\Td\eta'')$.
Let $c_0(\Td x, \Td\eta)=\sum_{j=0}\Td c^j_0(\Td x,\Td\eta)$,
where $\Td c^j_0$ is a homogeneous polynomial of degree $j$ in $(\Td x'',\Td\eta'')$.
From $\Td T\Td a_0=0$, we get
$\Td a^0_0(t,\Td x',\Td\eta')=e^{-t(\frac{1}{2}\Td{\rm tr\,}F+p^s_0)}\Td c^0_0(\Td x,\Td\eta)$.
It is easy to see that for all indices $\alpha$, $\beta$, $\gamma$ and
every compact set $K\subset\Sigma$ there exists a
constant $c>0$ such that
$\abs{\pr^\gamma_t\pr^\alpha_{\Td x}\pr^\beta_{\Td\eta}\Td a^0_0}\leq
ce^{-t\lambda(\Td x,\Td \eta)}$ on $\ol\Real_+\times K$,
where $\lambda(\Td x,\Td \eta)\in C(T^*(\Omega))$, $\lambda|_\Sigma<\min\tau_j$.
Here $\tau_j$ are the eigenvalues of $\frac{1}{2}\Td{\rm tr\,}F+p^s_0$ on $\Sigma$.

Again, from $\Td T\Td a_0=0$, we get
$(\frac{\pr}{\pr t}+\mathscr A+\frac{1}{2}\Td{\rm tr\,}F+p^s_0)
\Td a^{j+1}_0(t,\Td x,\Td\eta)=\Td b^{j+1}_0(t,\Td x,
\Td\eta)$
where $\Td b^{j+1}_0(t,\Td x,\Td\eta)$ satisfies the same kind of estimate as
$\Td a^0_0$. By Lemma~\ref{l:hf-BerSjo}, we see that $\exp(-t\mathscr A)$ is bounded for
$t\geq 0$. We deduce a similar estimate for the function $\Td a^{j+1}_0(t,\Td x,\Td\eta)$.
Continuing in this way we get all the desired estimates for $\Td a_0$. The next transport
equation takes the form $\Td T\Td a_1=\Td b$ where $\Td b$ satisfies
the estimates (\ref{e:hf-estimate}). We can repeat the procedure above and conclude that
$\Td a_1$ satisfies the estimates (\ref{e:hf-estimate}). From above, we see that
$\Td a_0$, $\Td a_1$ have the unique Taylor expansions on $\Sigma$. Continuing in
this way we get the proposition.
\end{proof}

From Proposition~\ref{p:hf-fundamental}, we have the following

\begin{prop} \label{p:hf-solution-ynot(q)}
Let $(n_-, n_+)$, $n_-+n_+=n-1$, be the signature of the Levi form. Let
$c_j(x, \eta)\in C^\infty(T^*(\Omega);\,
\mathscr L(\Lambda^{0,q}T^*(\Omega), \Lambda^{0,q}T^*(\Omega)))$, $j=0, 1,\ldots$,
be positively homogeneous functions of degree $-j$. Then, we can find solutions
\[a_j(t, x, \eta)\in C^\infty(\ol\Real_+\times T^*(\Omega);\,
\mathscr L(\Lambda^{0,q}T^*(\Omega), \Lambda^{0,q}T^*(\Omega))),\ \ j=0, 1,\ldots,\]
of the system {\rm (\ref{e:hf-heattransport})} with
$a_j(0, x, \eta)=c_j(x, \eta)$, $j=0, 1,\ldots$,
where $a_j(t,x,\eta)$ is a
quasi-homogeneous function of degree $-j$ and such that for all indices $\alpha,\beta,\gamma,j$, every
$\varepsilon>0$ and compact set $K\subset\Omega$ there exists a constant $c>0$ such that
\begin{equation} \label{e:hf-estimate3}
\abs{\pr^\gamma_t\pr^\alpha_x\pr^\beta_\eta a_j(t,x,\eta)}\leq
ce^{\varepsilon t\abs{\eta}}(1+\abs{\eta})^{-j-\abs{\beta}+\gamma}{\;\rm on\;}\ol\Real_+\times\Bigr((K\times\Real^{2n-1})\bigcap\Sigma\Bigr).
\end{equation}
Furthermore, there exists $\varepsilon_0>0$ such that for all indices $\alpha,\beta,\gamma,j$ and
every compact set $K\subset\Omega$,
there exists a constant $c>0$ such that
\begin{align} \label{e:hf-***************}
&\abs{\pr^\gamma_t\pr^\alpha_x\pr^\beta_\eta a_j(t,x,\eta)}\leq
ce^{-\varepsilon_0 t\abs{\eta}}(1+\abs{\eta})^{-j-\abs{\beta}+\gamma}  \nonumber\\
&{\;\rm on\;}\ol\Real_+\times\Bigr((K\times\Real^{2n-1})\bigcap\Sigma^+\Bigr)\ \ \mbox{if}\ \ q\neq n_+
\end{align}
and
\begin{align} \label{e:hf-***************bis}
&\abs{\pr^\gamma_t\pr^\alpha_x\pr^\beta_\eta a_j(t,x,\eta)}\leq
ce^{-\varepsilon_0 t\abs{\eta}}(1+\abs{\eta})^{-j-\abs{\beta}+\gamma}  \nonumber\\
&{\;\rm on\;}\ol\Real_+\times\Bigr((K\times\Real^{2n-1})\bigcap\Sigma^-\Bigr)\ \ \mbox{if}\ \ q\neq n_-.
\end{align}
\end{prop}

We need the following formula

\begin{prop} \label{p:hf-asym}
Let $Q$ be a $C^\infty$ differential operator on $\Omega$ of order
$k>0$ with full symbol $q(x, \xi)\in C^\infty(T^*(\Omega))$. For $0\leq q, q_1\leq n-1$, $q$, $q_1\in\Pstint$, let
\[a(t, x, \eta)\in C^\infty(\ol\Real_+\times T^*(\Omega);\,
\mathscr L(\Lambda^{0,q_1}T^*(\Omega), \Lambda^{0,q}T^*(\Omega))).\]
Then,
\begin{align*}
Q(x,D_x)(e^{i\psi(t,x,\eta)}a(t,x,\eta))=
e^{i\psi(t,x,\eta)}\sum_{\abs{\alpha}\leq k}\frac{1}{\alpha !}
q^{(\alpha)}(x,\psi'_x(t, x, \eta))(R_\alpha(\psi,D_x)a),
\end{align*}
where $D_x=-i\pr_x$, $R_\alpha(\psi,D_x)a=D^\alpha_y\set{e^{i\phi_2(t,x,y,\eta)}a(t,y,\eta)}\Big|_{y=x}$,
$\phi_2(t,x,y,\eta)=(x-y)\psi'_x(t,x,\eta)-(\psi(t,x,\eta)-\psi(t,y,\eta))$.
\end{prop}

For $0\leq q, q_1\leq n-1$, $q$, $q_1\in\Pstint$, let
\[a_j(t, x, \eta)\in C^\infty(\ol\Real_+\times T^*(\Omega);\,
\mathscr L(\Lambda^{0,q_1}T^*(\Omega), \Lambda^{0,q}T^*(\Omega))),\ \ j=0,1,\ldots,\]
be quasi-homogeneous functions of degree $m-j$, $m\in\integer$. Assume that
$a_j(t, x, \eta)$, $j=0,1,\dots$,
are the solutions of the system (\ref{e:hf-heattransport}). From the proof of Proposition~\ref{p:hf-fundamental},
it follows that for all indices $\alpha,\beta,\gamma,j$, every
$\varepsilon>0$ and compact set $K\subset\Omega$ there exists a constant $c>0$ such that
\begin{equation} \label{e:allerhome2}
\abs{\pr^\gamma_t\pr^\alpha_x\pr^\beta_\eta a_j(t,x,\eta)}\leq
ce^{\varepsilon t\abs{\eta}}(1+\abs{\eta})^{m-j-\abs{\beta}+\gamma}{\;\rm on\;}\ol\Real_+\times\Bigr((K\times\Real^{2n-1})\bigcap\Sigma\Bigr).
\end{equation}
Let $a(t, x, \eta)\in C^\infty(\ol\Real_+\times T^*(\Omega);\,
\mathscr L(\Lambda^{0,q_1}T^*(\Omega), \Lambda^{0,q}T^*(\Omega)))$
be the asymptotic sum of $a_j(t, x, \eta)$. (See Definition~\ref{d:ss-heatgeneral} and Remark~\ref{r:ss-heatclass}
for a precise meaning.)
We formally write $a(t, x, \eta)\sim\sum^\infty_{j=0}a_j(t, x, \eta)$.
Let
\[(\pr_t+\Box^{(q)}_b)(e^{i\psi(t,x,\eta)}a(t,x,\eta))=e^{i\psi(t,x,\eta)}b(t,x,\eta),\]
where
\[b(t,x,\eta)\sim\sum^\infty_{j=0}b_j(t,x,\eta),\]
$b_j$ is a quasi-homogeneous function of degree $m+2-j$,
$b_j\in C^\infty(\ol\Real_+\times T^*(\Omega);\,
\mathscr L(\Lambda^{0,q_1}T^*(\Omega), \Lambda^{0,q}T^*(\Omega)))$, $j=0,1,\ldots$.

From Proposition~\ref{p:hf-asym},
we see that for all $N$, every compact
set $K\subset\Omega$, $\varepsilon>0$, there exists $c>0$ such that
\begin{equation} \label{e:hf-soltrans1}
\abs{b(t,x,\eta)}\leq ce^{\varepsilon t\abs{\eta}}(\abs{\eta}^{-N}+\abs{\eta}^{m+2-N}({\rm Im\,}\psi(t,x,\eta)^N)
\end{equation}
on $\ol\Real_+\times\Bigr((K\times\Real^{2n-1})\bigcap\Sigma\Bigr)$, $\abs{\eta}\geq 1$.

Conversely, if
$(\pr_t+\Box^{(q)}_b)(e^{i\psi(t,x,\eta)}a(t,x,\eta))=e^{i\psi(t,x,\eta)}b(t,x,\eta)$
and $b$ satisfies the same kind of estimates as (\ref{e:hf-soltrans1}),
then $a_j(t,x,\eta)$, $j=0,1,\ldots$, solve the system (\ref{e:hf-heattransport}) to infinite order at $\Sigma$.
From this and the particular structure of the problem, we will next show

\begin{prop} \label{p:hf-BerSjo}
Let $(n_-, n_+)$, $n_-+n_+=n-1$, be the signature of the Levi form.
Suppose condition $Y(q)$ fails. That is, $q=n_-$ or $n_+$. Let
\[a_j(t, x, \eta)\in C^\infty(\ol\Real_+\times T^*(\Omega);\,
\mathscr L(\Lambda^{0,q}T^*(\Omega), \Lambda^{0,q}T^*(\Omega))),\ \ j=0,1,\ldots,\]
be the solutions of the system (\ref{e:hf-heattransport}) with
$a_0(0, x, \eta)=I$, $a_j(0, x, \eta)=0$ when $j>0$,
where $a_j(t,x,\eta)$ is a
quasi-homogeneous function of degree $-j$. Then we can find
\[a_j(\infty,x,\eta)\in C^\infty(T^*(\Omega);\,
\mathscr L(\Lambda^{0,q}T^*(\Omega), \Lambda^{0,q}T^*(\Omega))),\ \ j=0,1,\ldots,\]
where $a_j(\infty,x,\eta)$ is a positively homogeneous function of degree $-j$, $\varepsilon_0>0$ such that for
all indices $\alpha$, $\beta$, $\gamma$, $j$, every compact set $K\subset\Omega$,  there
exists $c>0$, such that
\begin{equation} \label{e:hf-Sjoestimate}
\abs{\pr^\gamma_t\pr^\alpha_x\pr^\beta_\eta(a_j(t,x,\eta)-a_j(\infty,x,\eta))}\leq
ce^{-\varepsilon_0t\abs{\eta}}(1+\abs{\eta})^{-j-\abs{\beta}+\gamma}
\end{equation}
on $\ol\Real_+\times\Bigr((K\times\Real^{2n-1})\bigcap\Sigma\Bigr)$, $\abs{\eta}\geq 1$.

Furthermore, for all $j=0,1,\ldots$,
\begin{equation} \label{e:hf-******************}
\left\{ \begin{array}{ll}
\textup{all derivatives of}\ a_j(\infty, x, \eta)\ \textup{vanish at}\ \ \Sigma^+, & \textup{if}\ \ q=n_-,\ n_-\neq n_+    \\
\textup{all derivatives of}\ a_j(\infty, x, \eta)\ \textup{vanish at}\ \ \Sigma^-, & \textup{if}\ \ q=n_+,\ n_-\neq n_+.
\end{array}\right..
\end{equation}
\end{prop}

\begin{proof}
We assume that $q=n_-$. Put
\[a(t, x, \eta)\sim\sum_j a_j(t, x, \eta).\]
Since $a_j(t,x,\eta)$, $j=0,1,\ldots$, solve the system (\ref{e:hf-heattransport}), we have
\[(\pr_t+\Box^{(q)}_b)(e^{i\psi(t,x,\eta)}a(t,x,\eta))=e^{i\psi(t,x,\eta)}b(t,x,\eta),\]
where $b(t,x,\eta)$ satisfies (\ref{e:hf-soltrans1}). Note that we have the interwing properties
\begin{equation} \label{e:hf-interwing}
\left\{ \begin{array}{l}
 \dbar_b\Box^{(q)}_b=\Box^{(q+1)}_b\dbar_b    \\
 \ol{\pr_b}^*\Box^{(q)}_b=\Box^{(q-1)}_b\ol{\pr_b}^*
 \end{array}\right..
\end{equation}
Now,
\[
\left\{  \begin{array}{l}
 \ol{\pr_b}^*(e^{i\psi}a)=e^{i\psi}\Td a   \\
 \dbar_b(e^{i\psi}a)=e^{i\psi}\hat a,
 \end{array}\right.
\]
$\Td a\sim\sum_{j=-1}^\infty\Td a_j(t,x,\eta)$, $\hat a\sim\sum_{j=-1}^\infty
\hat a_j(t,x,\eta)$, where
\[\hat a_j\in C^\infty(\ol\Real_+\times T^*(\Omega);\,
\mathscr L(\Lambda^{0,q}T^*(\Omega),\Lambda^{0,q+1}T^*(\Omega))),\ \ j=0,1,\ldots,\]
\[\Td a_j\in C^\infty(\ol\Real_+\times T^*(\Omega);\,
\mathscr L(\Lambda^{0,q}T^*(\Omega),\Lambda^{0,q-1}T^*(\Omega))),\ \ j=0,1,\ldots,\]
and $\hat a_j$, $\Td a_j$ are quasi-homogeneous of degree $1-j$.
From (\ref{e:hf-interwing}), we have
\begin{align*}
& (\pr_t+\Box^{(q-1)}_b)(e^{i\psi}\Td a)=e^{i\psi}b_1, \\
& (\pr_t+\Box^{(q+1)}_b)(e^{i\psi}\hat a)=e^{i\psi}b_2,
\end{align*}
where $b_1$, $b_2$ satisfy (\ref{e:hf-soltrans1}). Since $b_1$, $b_2$ satisfy (\ref{e:hf-soltrans1}),
$\Td a_j$, $\hat a_j$, $j=0,1,\ldots$, solve the system (\ref{e:hf-heattransport}) to infinite order at $\Sigma$.
We notice that $q-1\neq n_-$, $q+1\neq n_-$.
In view of the proof of Proposition~\ref{p:hf-fundamental}, we can find $\varepsilon_0>0$,
such that for all indices $\alpha$, $\beta$, $\gamma$, $j$, every compact set $K\subset\Omega$,
there exists $c>0$ such that
\begin{equation} \label{e:hf-hsiaomore}
\left\{ \begin{array}{l}
\abs{\pr^\gamma_t\pr^\alpha_x\pr^\beta_\eta\Td a_j(t,x,\eta)}\leq
ce^{-\varepsilon_0t\abs{\eta}}(1+\abs{\eta})^{1-j-\abs{\beta}+\gamma}  \\
\abs{\pr^\gamma_t\pr^\alpha_x\pr^\beta_\eta\hat a_j(t,x,\eta)}\leq
ce^{-\varepsilon_0t\abs{\eta}}(1+\abs{\eta})^{1-j-\abs{\beta}+\gamma}
\end{array}\right.
\end{equation}
on $\ol\Real_+\times\Bigr((K\times\Real^{2n-1})\bigcap\Sigma^-\Bigr)$, $\abs{\eta}\geq 1$.

Now $\Box^{(q)}_b=\dbar_b\ol{\pr_b}^*+\ol{\pr_b}^*\dbar_b$, so $\Box^{(q)}_b(e^{i\psi}a)=e^{i\psi}c$,
where $c$ satisfies the same kind of estimates as (\ref{e:hf-hsiaomore}). From this we see that
$\pr_t(e^{i\psi}a)=e^{i\psi}d$, where $d$ has the same properties as $c$. Since
$d=i(\pr_t\psi)a+\pr_ta$ and $\pr_t\psi$ satisfy the same kind of estimates as
(\ref{e:hf-hsiaomore}), $\pr _ta$ satisfies the same kind of estimates as
(\ref{e:hf-hsiaomore}). From this we conclude that we can find
$a(\infty,x,\eta)\sim\sum^\infty_{j=0}a_j(\infty,x,\eta)$,
where $a_j(\infty,x,\eta)$ is a matrix-valued $C^\infty$ positively homogeneous function of degree $-j$,
$\varepsilon_0>0$, such that for all indices $\alpha$, $\beta$, $\gamma$, $j$ and every compact
set $K\subset\Omega$,  there exists $c>0$ such that
\[\abs{\pr^\gamma_t\pr^\alpha_x\pr^\beta_\eta(a_j(t,x,\eta)-a_j(\infty,x,\eta))}\leq
ce^{-\varepsilon_0t\abs{\eta}}(1+\abs{\eta})^{-j-\abs{\beta}+\gamma}\]
on $\ol\Real_+\times\Bigr((K\times\Real^{2n-1})\bigcap\Sigma^-\Bigr)$, $\abs{\eta}\geq 1$.

If $n_-=n_+$, then $q-1\neq n_+$, $q+1\neq n_+$.
We can repeat the method above to conclude that we can find
$a(\infty,x,\eta)\sim\sum^\infty_{j=0}a_j(\infty,x,\eta)$,
where $a_j(\infty,x,\eta)$ is a matrix-valued $C^\infty$ positively homogeneous function of degree $-j$,
$\varepsilon_0>0$, such that for all indices $\alpha$, $\beta$, $\gamma$, $j$ and every compact
set $K\subset\Omega$,  there exists $c>0$ such that
\[\abs{\pr^\gamma_t\pr^\alpha_x\pr^\beta_\eta(a_j(t,x,\eta)-a_j(\infty,x,\eta))}\leq
ce^{-\varepsilon_0t\abs{\eta}}(1+\abs{\eta})^{-j-\abs{\beta}+\gamma}\]
on $\ol\Real_+\times\Bigr((K\times\Real^{2n-1})\bigcap\Sigma^+\Bigr)$, $\abs{\eta}\geq 1$.

Now, we assume that $n_-\neq n_+$. From (\ref{e:hf-***************}),
we can find $\varepsilon_0>0$, such that for all indices $\alpha$, $\beta$, $\gamma$,
$j$ and every compact set $K\subset\Omega$,  there exists $c>0$ such that
\[\abs{\pr^\gamma_t\pr^\alpha_x\pr^\beta_\eta a_j(t,x,\eta)}\leq
ce^{-\varepsilon_0t\abs{\eta}}(1+\abs{\eta})^{-j-\abs{\beta}+\gamma}\]
on $\ol\Real_+\times\Bigr((K\times\Real^{2n-1})\bigcap\Sigma^+\Bigr)$, $\abs{\eta}\geq 1$.
The proposition follows.
\end{proof}


\section{Some symbol classes}

We continue to work with some local coordinates
$x=(x_1,\ldots,x_{2n-1})$
defined on an open set $\Omega\subset X$. We identify $T^*(\Omega)$ with $\Omega\times\Real^{2n-1}$.

\begin{defn} \label{d:ss-heatgeneral}
Let $r(x,\eta)$ be a non-negative real continuous function on $T^*(\Omega)$. We assume that
$r(x,\eta)$ is positively homogeneous of degree 1, that is, $r(x,\lambda\eta)=\lambda r(x,\eta)$, for
$\lambda\geq1$, $\abs{\eta}\geq 1$. For $0\leq q_1, q_2\leq n-1$, $q_1, q_2\in\Pstint$ and $k\in\Real$, we say that
\[a\in\hat S^k_r(\ol\Real_+\times T^*(\Omega);\, \mathscr L(\Lambda^{0,q_1}T^*(\Omega), \Lambda^{0,q_2}T^*(\Omega)))\]
if $a\in C^\infty(\ol\Real_+\times T^*(\Omega);\, \mathscr L(\Lambda^{0,q_1}T^*(\Omega), \Lambda^{0,q_2}T^*(\Omega)))$
and for all indices $\alpha$, $\beta$, $\gamma$, every compact set $K\subset\Omega$ and every
$\varepsilon>0$, there exists a constant $c>0$ such that
\[\abs{\pr^\gamma_t\pr^\alpha_x\pr^\beta_\eta a(t,x,\eta)}\leq
 ce^{t(-r(x,\eta)+\varepsilon\abs{\eta})}(1+\abs{\eta})^{k+\gamma-\abs{\beta}},\;x\in K,\;\abs{\eta}\geq 1.\]
\end{defn}

\begin{rem} \label{r:ss-heatclass}
It is easy to see that we have the following properties:
 \begin{enumerate}
  \item If $a\in\hat S^k_{r_1}$, $b\in\hat S^l_{r_2}$ then $ab\in\hat S^{k+l}_{r_1+r_2}$,
        $a+b\in\hat S^{\max(k,l)}_{\min(r_1,r_2)}$.
  \item If $a\in\hat S^k_r$ then $\pr^\gamma_t\pr^\alpha_x\pr^\beta_\eta a\in
        \hat S^{k-\abs{\beta}+\gamma}_r$.
  \item If $a_j\in\hat S^{k_j}_r$, $j=0,1,2,\ldots$ and $k_j\searrow -\infty$ as $j\To\infty$,
        then there exists $a\in\hat S^{k_0}_r$ such that $a-\sum^{v-1}_0a_j\in
        \hat S^{k_v}_r$, for all $v=1,2,\ldots$. Moreover, if $\hat S^{-\infty}_r$ denotes $\bigcap_{k\in\Real}\hat S^k_r$
        then $a$ is unique modulo $\hat S^{-\infty}_r$.
 \end{enumerate}

If $a$ and $a_j$ have the properties of $(c)$, we write
\[a\sim\sum^\infty_0a_j\ \ \textup{in the symbol space}\ \hat S^{k_0}_r.\]
\end{rem}

From Proposition~\ref{p:hf-solution-ynot(q)} and the standard Borel construction,
we get the following

\begin{prop} \label{p:ss-solution}
Let
\[c_j(x, \eta)\in C^\infty(T^*(\Omega);\, \mathscr L(\Lambda^{0,q}T^*(\Omega),\Lambda^{0,q}T^*(\Omega))),\ \ j=0, 1,\ldots\]
be positively homogeneous functions of degree $-j$. We can find solutions
$a_j(t,x,\eta)\in
C^\infty(\ol\Real_+\times T^*(\Omega);\,
\mathscr L(\Lambda^{0,q}T^*(\Omega), \Lambda^{0,q}T^*(\Omega)))$, $j=0, 1,\ldots$
of the system (\ref{e:hf-heattransport}) with the conditions
$a_j(0, x, \eta)=c_j(x, \eta)$, $j=0, 1,\ldots$, where $a_j$ is a quasi-homogeneous function of degree $-j$ such that
\[a_j\in\hat S^{-j}_r(\ol\Real_+\times T^*(\Omega);\, \mathscr L(\Lambda^{0,q}T^*(\Omega), \Lambda^{0,q}T^*(\Omega))),\ \ j=0,1,\ldots,\]
for some $r$ with $r>0$ if\, $Y(q)$ holds and $r=0$ if\, $Y(q)$ fails.

If the Levi form has signature $(n_-, n_+)$, $n_-+n_+=n-1$, then we can take $r>0$,
\[ \left\{ \begin{array}{ll}
\textup{near}\quad\Sigma^+, & {\rm if\ } q=n_-, n_-\neq n_+,  \\
\textup{near}\quad\Sigma^-, & {\rm if\ } q=n_+, n_-\neq n_+.
\end{array}\right.\]
\end{prop}

Again, from Proposition~\ref{p:hf-BerSjo} and the standard Borel construction, we get the following

\begin{prop} \label{p:ss-SBmore}
Let $(n_-, n_+)$, $n_-+n_+=n-1$, be the signature of the Levi form.
Suppose condition $Y(q)$ fails. That is, $q=n_-$ or $n_+$.
We can find solutions
\[a_j(t,x,\eta)\in
C^\infty(\ol\Real_+\times T^*(\Omega);\,
\mathscr L(\Lambda^{0,q}T^*(\Omega), \Lambda^{0,q}T^*(\Omega))),\ \ j=0, 1,\ldots\]
of the system {\rm (\ref{e:hf-heattransport})} with
$a_0(0, x, \eta)=I$, $a_j(0, x, \eta)=0$ when $j>0$,
where $a_j(t,x,\eta)$ is a quasi-homogeneous function of degree $-j$, such that for some $r>0$ as in
Definition~\ref{d:ss-heatgeneral},
\[a_j(t,x,\eta)-a_j(\infty,x,\eta)\in\hat S^{-j}_r(\ol\Real_+\times T^*(\Omega);\,
\mathscr L(\Lambda^{0,q}T^*(\Omega), \Lambda^{0,q}T^*(\Omega))),\]
$j=0, 1,\ldots$, where
\[a_j(\infty,x,\eta)\in C^\infty(T^*(\Omega);\,
\mathscr L(\Lambda^{0,q}T^*(\Omega), \Lambda^{0,q}T^*(\Omega))),\ \ j=0,1,\ldots,\]
and $a_j(\infty, x, \eta)$ is a positively homogeneous function of degree $-j$.

Furthermore, for all $j=0,1,\ldots$,
\[ \left\{ \begin{array}{ll}
a_j(\infty, x, \eta)=0\quad\mbox{in a conic neighborhood of}\ \ \Sigma^+, & {\rm if\ } q=n_-, n_-\neq n_+,  \\
a_j(\infty ,x,\eta)=0\quad\mbox{in a conic neighborhood of}\ \ \Sigma^-, & {\rm if\ } q=n_+, n_-\neq n_+.
\end{array}\right.\]
\end{prop}

Let $b(t, x, \eta)\in\hat S^k_r$, $r>0$. Our next goal is to define the operator
\[B(t, x, y)=\int e^{i(\psi(t,x,\eta)-\seq{y,\eta})}b(t,x,\eta)d\eta \]
as an oscillatory integral (see the proof of Proposition~\ref{p:ss-oscill}, for the precise meaning of the integral $B(t, x, y)$).
We have the following

\begin{prop} \label{p:ss-oscill}
Let
\[b(t,x,\eta)\in\hat S^k_r(\ol\Real_+\times T^*(\Omega);\, \mathscr L(\Lambda^{0,q_1}T^*(\Omega), \Lambda^{0,q_2}T^*(\Omega)))\]
with $r>0$. Then we can define
\[B(t): C^\infty_0(\Omega;\, \Lambda^{0,q_1}T^*(\Omega))\To C^\infty(\ol\Real_+;\, C^\infty
(\Omega;\, \Lambda^{0,q_2}T^*(\Omega)))\]
with distribution kernel $B(t, x, y)=\int e^{i(\psi(t,x,\eta)-\seq{y,\eta})}b(t,x,\eta)d\eta$
and $B(t)$ has a unique continuous extension
\[B(t): \mathcal E'(\Omega;\, \Lambda^{0,q_1}T^*(\Omega))\To C^\infty(\ol\Real_+;\, \mathcal D'
(\Omega;\, \Lambda^{0,q_2}T^*(\Omega))).\]

We have
\[B(t, x, y)\in C^\infty(\ol\Real_+;\, C^\infty(\Omega\times\Omega\smallsetminus
{\rm diag\,}(\Omega\times\Omega);\, \mathscr L(\Lambda^{0,q_1}T^*(\Omega),\Lambda^{0,q_2}T^*(\Omega))))\]
and
$B(t, x, y)|_{t>0}\in
C^\infty(\Real_+\times\Omega\times\Omega;\, \mathscr L(\Lambda^{0,q_1}T^*(\Omega),\Lambda^{0,q_2}T^*(\Omega)))$.
\end{prop}

\begin{proof}
Let
\[S^*\Omega=\set{(x,\eta)\in\Omega\times\dot{\Real}^{2n-1};\, \abs{\eta}=1}.\]
Set $S^*\Sigma=\Sigma\bigcap S^*\Omega$. Let
$V\subset\ol\Real_+\times S^*\Omega$ be a neighborhood of $(\Real_+\times S^*\Sigma)\bigcup(\set{0}\times S^*\Omega)$
such that $V_t=\set{(x,\eta);\, (t,x,\eta)\in V}$
is independent of $t$ for large $t$. Set
$W=\set{(t,x,\eta)\in\ol\Real_+\times\Omega\times\dot{\Real}^{2n-1};\, (\abs{\eta}t,x,
\frac{\eta}{\abs{\eta}})\in V}$.
Let $\chi_V\in C^\infty(\ol\Real_+\times S^*\Omega)$ have its support in $V$, be
equal to 1 in a neighborhood of $(\Real_+\times S^*\Sigma)\bigcup(\set{0}\times S^*\Omega)$, and be
independent of $t$, for large $t$. Set
$\chi_W(t,x,\eta)=\chi_V(\abs{\eta}t,x,\frac{\eta}{\abs{\eta}})\in C^\infty(\ol\Real_+\times\Omega\times\dot\Real^{2n-1})$.
We have $\chi_W(t,x,\lambda\eta)=\chi_W(\lambda t,x,\eta)$, $\lambda>0$.
We can choose $V$ sufficiently small so that
\begin{equation} \label{e:he-MenSjo}
\abs{\psi'_x(t,x,\eta)-\eta}\leq\frac{\abs{\eta}}{2}{\;\rm in\;}W.
\end{equation}
We formally set
\begin{align*}
B(t,x,y)  &= \int e^{i(\psi(t,x,\eta)-\seq{y,\eta})}(1-\chi_W(t,x,\eta))b(t,x,\eta)d\eta    \\
          &\quad +\int e^{i(\psi(t,x,\eta)-\seq{y,\eta})}\chi_W(t,x,\eta)b(t,x,\eta)d\eta   \\
          &= B_1(t,x,y)+B_2(t,x,y)
\end{align*}
where in $B_1(t,x,y)$ and $B_2(t,x,y)$ we have introduced the cut-off functions $(1-\chi_W)$
and $\chi_W$ respectively. Choose $\chi\in C^\infty_0(\Real^{2n-1})$ so that $\chi(\eta)=1$
when $\abs{\eta}<1$ and $\chi(\eta)=0$ when $\abs{\eta}>2$. Since ${\rm Im\,}\psi>0$
outside $(\Real_+\times\Sigma)\bigcup(\set{0}\times\dot{\Real}^{2n-1})$, we have
${\rm Im\,}\psi(t,x,\eta)\geq c\abs{\eta}$ outside $W$, where $c>0$. The kernel
$B_{1,\eps}(t,x,y)=\int e^{i(\psi(t,x,\eta)-\seq{y,\eta})}(1-\chi_W(t,x,\eta))b(t,x,\eta)
\chi(\eps\eta)d\eta$
converges in the space
$C^\infty(\ol\Real_+\times\Omega\times\Omega;\, \mathscr L(\Lambda^{0,q_1}T^*(\Omega),\Lambda^{0,q_2}T^*(\Omega)))$ as
$\eps\To 0$. This means that
\[B_1(t,x,y)=\lim_{\eps\To 0}B_{1,\eps}(t,x,y)\in C^\infty(\ol\Real_+\times\Omega
\times\Omega;\, \mathscr L(\Lambda^{0,q_1}T^*(\Omega),\Lambda^{0,q_2}T^*(\Omega))).\]
To study $B_2(t,x,y)$ take a $u(y)\in C^\infty_0(K;\, \Lambda^{0,q_1}T^*(\Omega))$, $K\subset\subset\Omega$ and set
$\chi_\nu(\eta)=\chi(2^{-\nu}\eta)-\chi(2^{1-\nu}\eta),\;\nu>0,\;\chi_0(\eta)=\chi(\eta)$.
Then we have $\sum^\infty_{\nu=0}\chi_\nu=1$ and $2^{\nu-1}\leq\abs{\eta}\leq 2^{\nu+1}$
when $\eta\in{\rm supp\,}\chi_\nu$, $\nu\neq 0$.
We assume that $b(t, x, \eta)=0$ if $\abs{\eta}\leq1$.
If $x\in K$, we obtain for all indices $\alpha$, $\beta$ and every $\varepsilon>0$, there exists
$c_{\varepsilon,\alpha,\beta,K}>0$, such that
\begin{equation} \label{e:he-Hor}
\abs{D^\alpha_xD^\beta_\eta(\chi_\nu(\eta)\chi_W(t,x,\eta)b(t,x,\eta))}\leq c_{\varepsilon,\alpha,\beta,K}
e^{t(-r(x,\eta)+\varepsilon\abs{\eta})}(1+\abs{\eta})^{k-\abs{\beta}}.
\end{equation}
Note that $\abs{D^\alpha\chi_\nu(\eta)}\leq c_\alpha(1+\abs{\eta})^{-\abs{\alpha}}$ with a constant
independent of $\nu$. We have
\begin{align*}
B_{2,\nu+1}  &=\int\int e^{i(\psi(t,x,\eta)-\seq{y,\eta})}\chi_{\nu+1}(\eta)\chi_W(t,x,\eta)
               b(t,x,\eta)u(y)dyd\eta    \\
             &=2^{(2n-1)\nu}\int e^{i\lambda(\psi(\lambda t,x,\eta)-\seq{y,\eta})}\chi_1(\eta)
               \chi_W(t,x,\lambda\eta)b(t,x,\lambda\eta)u(y)dyd\eta,
\end{align*}
where $\lambda=2^\nu$. Since (\ref{e:he-Hor}) holds, we have
$\abs{D^\alpha_\eta(\chi_W(t,x,2^\nu\eta)b(t,x,2^\nu\eta))}\leq c2^{k\nu}$
if $x\in K$, $1<\abs{\eta}<4$, where $c>0$. Since $d_y(\psi(\lambda t,x,\eta)-\seq{y\ ,\eta})\neq 0$,
if $\eta\neq 0$, we can integrate by parts and obtain
\[\abs{B_{2,\nu+1}}\leq c2^{\nu(2n-1+k-m)}\sum_{\abs{\alpha}\leq m}\sup\abs{D^\alpha u}.\]
Since $m$ can be chosen arbitrary large, we conclude that $\sum_\nu\abs{B_{2,\nu}}:$ converges and that $B(t)$
defines an operator
\[B(t): C^\infty_0(\Omega_y;\, \Lambda^{0,q_1}T^*(\Omega))\To
C^\infty(\ol\Real_+\times\Omega_x;\, \Lambda^{0,q_2}T^*(\Omega)).\]
Let $B^*(t)$ be the formal adjoint of $B(t)$ with respect to $(\ |\ )$. From (\ref{e:he-MenSjo}), we see that $\psi'_x(t,x,\eta)\neq 0$ on $W$.
We can repeat the procedure above and conclude that $B^*(t)$ defines an operator
\[B^*(t):C^\infty_0(\Omega_x;\, \Lambda^{0,q_2}T^*(\Omega))\To C^\infty(\ol\Real_+\times\Omega_y;\, \Lambda^{0,q_1}T^*(\Omega)).\]
Hence, we can extend $B(t)$ to
$\mathscr E'(\Omega;\, \Lambda^{0,q_1}T^*(\Omega))\To C^\infty(\ol\Real_+;\, \mathscr D'(\Omega;\, \Lambda^{0,q_2}T^*(\Omega)))$
by the formula
\[(B(t)u(y)\ |\ v(x))=(u(y)\ |\ B^*(t)v(x)),\]
where $u\in\mathscr E'(\Omega;\, \Lambda^{0,q_1}T^*(\Omega))$, $v\in C^\infty_0(\Omega;\, \Lambda^{0,q_2}T^*(\Omega))$.

When $x\neq y$ and $(x,y)\in\Sigma\times\Sigma$, we have
$d_\eta(\psi(t,x,\eta)-\seq{y\ ,\eta})\neq 0$, we can repeat the procedure above and conclude that
$B(t,x,y)\in C^\infty(\ol\Real_+;\, C^\infty(\Omega\times\Omega\smallsetminus
{\rm diag\,}(\Omega\times\Omega);\, \mathscr L(\Lambda^{0,q_1}T^*(\Omega),\Lambda^{0,q_2}T^*(\Omega))))$.

Finally, in view of the exponential decrease as $t\To\infty$ of the symbol $b(t, x, \eta)$, we see that
the kernel $B(t)|_{t>0}$ is smoothing.
\end{proof}

Let $b(t,x,\eta)\in\hat S^k_r$ with $r>0$. Our next step is to show that we can also define the operator
$B(x,y)=\int\Bigr(\int^{\infty}_0 e^{i(\psi(t,x,\eta)-\seq{y,\eta})}b(t,x,\eta)dt\Bigr)d\eta$
as an oscillatory integral (see the proof of Proposition~\ref{p:ss-oscillmore}, for the precise meaning of the integral $B(x, y)$).
We have the following

\begin{prop} \label{p:ss-oscillmore}
Let
\[b(t,x,\eta)\in\hat S^k_r(\ol\Real_+\times T^*(\Omega);\, \mathscr L(\Lambda^{0,q_1}T^*(\Omega),\Lambda^{0,q_2}T^*(\Omega)))\]
with $r>0$. Assume that $b(t,x,\eta)=0$ when
$\abs{\eta}\leq 1$. We can define
\[B: C^\infty_0(\Omega;\, \Lambda^{0,q_1}T^*(\Omega))\To C^\infty(\Omega;\, \Lambda^{0,q_2}T^*(\Omega))\]
with distribution kernel
\[B(x, y)=\frac{1}{(2\pi)^{2n-1}}\int\Bigr(\int^{\infty}_0 e^{i(\psi(t,x,\eta)-\seq{y,\eta})}b(t,x,\eta)dt\Bigr)d\eta\]
and $B$ has a unique continuous extension
\[B: \mathcal E'(\Omega;\, \Lambda^{0,q_1}T^*(\Omega))\To\mathcal D'(\Omega;\, \Lambda^{0,q_2}T^*(\Omega)).\]
Moreover,
$B(x,y)\in C^\infty(\Omega\times\Omega\smallsetminus{\rm diag\,}(\Omega\times\Omega);\,
\mathscr L(\Lambda^{0,q_1}T^*(\Omega),\Lambda^{0,q_2}T^*(\Omega)))$.
\end{prop}

\begin{proof}
Let $W$ and $\chi_W(t,x,\eta)$ be as in Proposition~\ref{p:ss-oscill}. We formally set
\begin{align*}
B(x,y)  &=\frac{1}{(2\pi)^{2n-1}}\int\int^\infty_0 e^{i(\psi(t,x,\eta)-\seq{y,\eta})}(1-\chi_W(t,x,\eta))b(t,x,\eta)dtd\eta    \\
        &\quad +\frac{1}{(2\pi)^{2n-1}}\int\int^\infty_0 e^{i(\psi(t,x,\eta)-\seq{y,\eta})}\chi_W(t,x,\eta)b(t,x,\eta)dtd\eta   \\
        &= B_1(x,y)+B_2(x,y)
\end{align*}
where in $B_1(x,y)$ and $B_2(x,y)$ we have introduced the cut-off functions $(1-\chi_W)$ and
$\chi_W$ respectively. Since ${\rm Im\,}\psi(t,x,\eta)\geq c'\abs{\eta}$ outside $W$, where $c'>0$, we have
\[ \abs{e^{i(\psi(t,x,\eta)-\seq{y,\eta})}(1-\chi_W(t,x,\eta))b(t,x,\eta)}\leq ce^{-c'\abs{\eta}}
e^{-\varepsilon_0t\abs{\eta}}(1+\abs{\eta})^k,\ \varepsilon_0>0\]
and similar estimates for the derivatives.
From this, we see that $B_1(x,y)\in C^\infty(\Omega\times\Omega;\,
\mathscr L(\Lambda^{0,q_1}T^*(\Omega),\Lambda^{0,q_2}T^*(\Omega)))$.

Choose $\chi\in C^\infty_0(\Real^{2n-1})$ so that $\chi(\eta)=1$ when $\abs{\eta}<1$ and
$\chi(\eta)=0$ when $\abs{\eta}>2$. To study $B_2(x,y)$ take a $u(y)\in C^\infty_0(K;\, \Lambda^{0,q_1}T^*(\Omega))$,
$K\subset\subset\Omega$ and set
\begin{align*}
&B_{2,\lambda}(x) \\
&=\frac{1}{(2\pi)^{2n-1}}\int^\infty_0\!\!\Bigl(\int\!\!\int\!\!
e^{i(\psi(t,x,\eta)-\seq{y,\eta})}b(t,x,\eta)\chi_W(t,x,\eta)\chi(\frac{\eta}{\lambda})u(y)dyd\eta\Bigr)dt.
\end{align*}
We have
\begin{align*}
B_{2,2\lambda}(x)-B_{2,\lambda}(x)
&=\frac{\lambda^{2n-1}}{(2\pi)^{2n-1}}\int^\infty_0\!\!\Bigl(\int\!\!\int\!\! e^{i\lambda(\psi(\lambda t,x,\eta)-\seq{y,\eta})}\chi_W
(t,x,\lambda\eta)b(t,x,\lambda\eta)  \\
&\quad (\chi(\frac{\eta}{2})-\chi(\eta))u(y)dyd\eta\Bigr)dt.
\end{align*}
Since $d_y(\psi(\lambda t,x,\eta)-\seq{y,\eta})\neq 0$, $\eta\neq 0$, we obtain
\begin{align*}
\lefteqn{\abs{\int\!\int e^{i\lambda(\psi(\lambda t,x,\eta)-\seq{y,\eta})}\chi_W(t,x,\lambda\eta)b(t,x,\lambda\eta)
  (\chi(\frac{\eta}{2})-\chi(\eta))u(y)dyd\eta}}   \\
&\quad\leq c\lambda^{-N}\sum_{\abs{\alpha}\leq N}\sup\abs{D^\alpha_{y,\eta}
  \chi_W(t,x,\lambda\eta)b(t,x,\lambda\eta)(\chi(\frac{\eta}{2})-\chi(\eta))u(y)}  \\
&\quad\leq c'\lambda^{-N}e^{-\varepsilon_0t\abs{\eta}}(1+\abs{\lambda})^{k},
\end{align*}
where $c$, $c'$, $\varepsilon_0>0$. Hence $B_2(x)=\lim_{\lambda\To\infty}B_{2,\lambda}(x)$ exists. Thus, $B(x,y)$
defines an operator
$B:C^\infty_0(\Omega_y;\, \Lambda^{0,q_1}T^*(\Omega))\To C^\infty(\Omega_x;\, \Lambda^{0,q_2}T^*(\Omega))$.
Let $B^*$ be the formal adjoint of $B$ with respect to $(\ |\ )$.
Since $\psi'_x(t,x,\eta)\neq 0$ on $W$, we can repeat the procedure above and conclude that
$B^*$ defines an operator
$B^*:C^\infty_0(\Omega_x;\, \Lambda^{0,q_2}T^*(\Omega))\To C^\infty(\Omega_y;\, \Lambda^{0,q_1}T^*(\Omega))$.
We can extend $B$ to
$\mathscr E'(\Omega;\, \Lambda^{0,q_1}T^*(\Omega))\To\mathscr D'(\Omega;\, \Lambda^{0,q_2}T^*(\Omega))$
by the formula
\[(Bu(y)\ |\ v(x))=(u(y)\ |\ B^*v(x)),\]
where $u\in\mathscr E'(\Omega;\, \Lambda^{0,q_1}T^*(\Omega))$, $v\in C^\infty_0(\Omega;\, \Lambda^{0,q_2}T^*(\Omega))$.

Finally, when $x\neq y$ and $(x,y)\in\Sigma\times\Sigma$, we have
$d_\eta(\psi(t,x,\eta)-\seq{y\ ,\eta})\neq 0$,
we can repeat the procedure above and conclude that
$B(x,y)\in C^\infty(\Omega\times\Omega\smallsetminus{\rm diag\,}(\Omega\times\Omega);\,
\mathscr L(\Lambda^{0,q_1}T^*(\Omega), \Lambda^{0,q_2}T^*(\Omega)))$.
\end{proof}

\begin{rem} \label{r:ss-oscillmore}
Let $a(t, x, \eta)\in \hat S^k_0(\ol\Real_+\times T^*(\Omega);\,
\mathscr L(\Lambda^{0,q_1}T^*(\Omega),\Lambda^{0,q_2}T^*(\Omega)))$.
We assume $a(t, x, \eta)=0$ if $\abs{\eta}\leq 1$ and
\[a(t, x, \eta)-a(\infty, x, \eta)\in \hat S^{k}_r(\ol\Real_+\times T^*(\Omega);\,
\mathscr L(\Lambda^{0,q_1}T^*(\Omega),\Lambda^{0,q_2}T^*(\Omega)))\]
with $r>0$, where
$a(\infty, x, \eta)\in C^\infty(T^*(\Omega);\,
\mathscr L(\Lambda^{0,q_1}T^*(\Omega),\Lambda^{0,q_2}T^*(\Omega)))$.
Then we can also define
\[A(x, y)=\int\Bigr(\int^{\infty}_0\Bigr(e^{i(\psi(t,x,\eta)-\seq{y,\eta})}a(t,x,\eta)-e^{i(\psi(\infty,x,\eta)-\seq{y,\eta})}
a(\infty,x,\eta)\Bigr)dt\Bigr)d\eta\]
as an oscillatory integral by the following formula:
\[A(x, y)=\int\!\!\Bigr(\int^{\infty}_0\!\! e^{i(\psi(t,x,\eta)-
\seq{y,\eta})}(-t)(i\psi^{'}_t(t,x,\eta)a(t,x,\eta)+a'_t(t,x,\eta))dt\Bigr)d\eta.\]
We notice that $(-t)(i\psi'_t(t,x,\eta)a(t,x,\eta)+a'_t(t,x,\eta))\in\hat S^{k+1}_r$, $r>0$.
\end{rem}

Let $B$ be as in the proposition~\ref{p:ss-oscillmore}. We can show that $B$ is a matrix
of pseudodifferential operators of order $k-1$ type $(\frac{1}{2},\frac{1}{2})$. We review some facts
about pseudodifferential operators of type $(\frac{1}{2},\frac{1}{2})$.

\begin{defn} \label{d:ss-symbol}
Let $k\in\Real$ and $q\in\Pstint$.
$S^k_{\frac{1}{2},\frac{1}{2}}(T^*(\Omega);\, \mathscr L(\Lambda^{0,q}T^*(\Omega),\Lambda^{0,q}T^*(\Omega)))$
is the space
of all $a\in C^\infty(T^*(\Omega);\, \mathscr L(\Lambda^{0,q}T^*(\Omega),\Lambda^{0,q}T^*(\Omega)))$
such that for every
compact set $K\subset\Omega$ and all $\alpha\in\Pstint^{2n-1}$, $\beta\in\Pstint^{2n-1}$, there is a constant
$c_{\alpha,\beta,K}>0$ such that
$\abs{\pr^\alpha_x\pr^\beta_\xi a(x,\xi)}\leq c_{\alpha,\beta,K}(1+\abs{\xi})^{k-\frac{\abs{\beta}}{2}+
\frac{\abs{\alpha}}{2}},\;(x,\xi)\in T^*(\Omega)$, $x\in K$.
$S^k_{\frac{1}{2},\frac{1}{2}}$ is called the space of symbols of order $k$ type
$(\frac{1}{2},\frac{1}{2})$. We write
$S^{-\infty}_{\frac{1}{2},\frac{1}{2}}=\bigcap_{m\in\Real} S^m_{\frac{1}{2},\frac{1}{2}}$,
$S^\infty_{\frac{1}{2},\frac{1}{2}}=\bigcup_{m\in\Real} S^m_{\frac{1}{2},\frac{1}{2}}$.
\end{defn}

Let
$a(x, \xi)\in S^k_{\frac{1}{2},\frac{1}{2}}(T^*(\Omega);\,
\mathscr L(\Lambda^{0,q}T^*(\Omega),\Lambda^{0,q}T^*(\Omega)))$. We can also define
\[A(x, y)=\frac{1}{(2\pi)^{2n-1}}\int e^{i\seq{x-y,\xi}}a(x,\xi)d\xi\]
as an oscillatory integral and we can show that
\[A:C^\infty_0(\Omega;\, \Lambda^{0,q}T^*(\Omega))\To C^\infty(\Omega;\, \Lambda^{0,q}T^*(\Omega))\]
is continuous
and has unique continuous extension:
$A:\mathscr E'(\Omega;\, \Lambda^{0,q}T^*(\Omega))\To\mathscr D'(\Omega;\, \Lambda^{0,q}T^*(\Omega))$.

\begin{defn} \label{d:ss-pseudo}
Let $k\in\Real$ and let $0\leq q\leq n-1$, $q\in\Pstint$. A pseudodifferential operator of order $k$ type
$(\frac{1}{2},\frac{1}{2})$ from sections of $\Lambda^{0,q}T^*(\Omega)$ to sections of $\Lambda^{0,q}T^*(\Omega)$
is a continuous linear map
$A:C^\infty_0(\Omega;\, \Lambda^{0,q}T^*(\Omega))\To\mathscr D'(\Omega;\, \Lambda^{0,q}T^*(\Omega))$
such that the distribution kernel of $A$ is
\[K_A=A(x, y)=\frac{1}{(2\pi)^{2n-1}}\int e^{i\seq{x-y,\xi}}a(x, \xi)d\xi\]
with $a\in S^k_{\frac{1}{2},\frac{1}{2}}(T^*(\Omega);\, \mathscr L(\Lambda^{0,q}T^*(\Omega),\Lambda^{0,q}T^*(\Omega)))$.
We call $a(x, \xi)$ the symbol of $A$. We shall write
$L^k_{\frac{1}{2},\frac{1}{2}}(\Omega;\, \Lambda^{0,q}T^*(\Omega),\Lambda^{0,q}T^*(\Omega))$
to denote the space of
pseudodifferential operators of order $k$ type $(\frac{1}{2},\frac{1}{2})$ from sections of
$\Lambda^{0,q}T^*(\Omega)$ to sections of $\Lambda^{0,q}T^*(\Omega)$. We write
$L^{-\infty}_{\frac{1}{2},\frac{1}{2}}=\bigcap_{m\in\Real}L^m_{\frac{1}{2},\frac{1}{2}}$,
$L^\infty_{\frac{1}{2},\frac{1}{2}}=\bigcup_{m\in\Real}L^m_{\frac{1}{2},\frac{1}{2}}$.
\end{defn}

We recall the following classical proposition of Calderon-Vaillancourt (see chapter $XVIII$ of H\"{o}rmander~\cite{Hor85})

\begin{prop} \label{p:he-calderon}
If $A\in L^k_{\frac{1}{2},\frac{1}{2}}(\Omega;\, \Lambda^{0,q}T^*(\Omega),\Lambda^{0,q}T^*(\Omega))$.
Then,
\[A:H^s_{\rm comp}(\Omega;\, \Lambda^{0,q}T^*(\Omega))\To H^{s-k}_{\rm loc}(\Omega;\, \Lambda^{0,q}T^*(\Omega))\]
is continuous, for all $s\in\Real$. Moreover,
if $A$ is properly supported (for the precise meaning of properly supported operators, see page $28$ of~\cite{GS94}), then
\[A:H^s_{\rm loc}(\Omega;\, \Lambda^{0,q}T^*(\Omega))\To H^{s-k}_{\rm loc}(\Omega;\, \Lambda^{0,q}T^*(\Omega))\]
is continuous, for all $s\in\Real$.
\end{prop}

We need the following properties of the phase $\psi(t, x, \eta)$.

\begin{lem} \label{l:ss-sjostrand}
For every compact set $K\subset\Omega$ and all $\alpha\in\Pstint^{2n-1}$, $\beta\in\Pstint^{2n-1}$,
$\abs{\alpha}+\abs{\beta}\geq1$,
there exists a constant $c_{\alpha,\beta,K}>0$, such that
\[\abs{\pr^\alpha_{x}\pr^\beta_\eta(\psi(t, x, \eta)-\seq{x, \eta})}\leq c_{\alpha,\beta,K}(1+\abs{\eta})^
{\frac{\abs{\alpha}-\abs{\beta}}{2}}
(1+{\rm Im\,}\psi(t, x,\eta))^{\frac{\abs{\alpha}+\abs{\beta}}{2}},\]
if\, $\abs{\alpha}+\abs{\beta}=1$ and
\[\abs{\pr^\alpha_{x}\pr^\beta_\eta(\psi(t, x, \eta)-\seq{x, \eta})}\leq c_{\alpha,\beta,K}(1+\abs{\eta})^
{1-\abs{\beta}},\ \ \mbox{if}\ \ \abs{\alpha}+\abs{\beta}\geq2,\]
where $x\in K$, $t\in\ol\Real_+$, $\abs{\eta}\geq1$.
\end{lem}

\begin{proof}
For $\abs{\eta}=1$,
we consider Taylor expansions of $\pr_{x_j}(\psi(t, x, \eta)-\seq{x, \eta})$, $j=1,\ldots,2n-1$,
at $(x_0, \eta_0)\in\Sigma$,
\begin{align*}
\pr_{x_j}(\psi(t, x, \eta)-\seq{x, \eta})&=\sum_k\frac{\pr^2\psi}{\pr x_k\pr x_j}(t, x_0, \eta_0)(x_k-x^{(k)}_0) \\
&\quad+\sum_k\frac{\pr^2\psi}{\pr \eta_k\pr x_j}(t, x_0, \eta_0)(\eta_k-\eta^{(k)}_0) \\
&\quad+O(\abs{(x-x_0)}^2+\abs{(\eta-\eta_0)}^2),
\end{align*}
where $x_0=(x^{(1)}_0,\ldots,x^{(2n-1)}_0)$, $\eta_0=(\eta^{(1)}_0,\ldots,\eta^{(2n-1)}_0)$.
Thus, for every compact set $K\subset\Omega$ there exists a constant $c>0$, such that
\[\abs{\pr_x(\psi(t, x, \eta)-\seq{x, \eta})}\leq c\frac{t}{1+t}{\rm dist\,}((x,\eta),\Sigma)),\]
where $x\in K$, $t\in\ol\Real_+$ and $\abs{\eta}=1$. In view of (\ref{e:c-sj1}), we see that
${\rm Im\,}\psi(t, x,\eta)\asymp(\frac{t}{1+t}){\rm dist\,}((x,\eta),\Sigma))^2$, $\abs{\eta}=1$.
Hence,
$(\frac{t}{1+t})^{\frac{1}{2}}{\rm dist\,}((x,\eta),\Sigma))\asymp ({\rm Im\,}\psi(t, x,\eta))^{\frac{1}{2}}$, $\abs{\eta}=1$.
Thus, for every compact set $K\subset\Omega$ there exists a constant $c>0$, such that
$\abs{\pr_x(\psi(t, x, \eta)-\seq{x, \eta})}\leq c(\frac{t}{1+t})^{\frac{1}{2}}({\rm Im\,}
\psi(t, x,\eta))^{\frac{1}{2}}$, $\abs{\eta}=1$, $x\in K$.
From above, we get for $\abs{\eta}\geq1$,
\begin{align*}
\abs{\pr_{x}(\psi(t, x, \eta)-\seq{x, \eta})}
    &=\abs{\eta}\abs{\pr_{x}(\psi(t\abs{\eta}, x,\frac{\eta}{\abs{\eta}})-\seq{x, \frac{\eta}{\abs{\eta}}})}  \\
    &\leq c\abs{\eta}^{\frac{1}{2}}(\frac{t\abs{\eta}}{1+t\abs{\eta}})^
{\frac{1}{2}}({\rm Im\,}\psi(t, x,\eta))^{\frac{1}{2}} \\
&\leq c'(1+\abs{\eta})^{\frac{1}{2}}(1+{\rm Im\,}\psi(t, x,\eta))^{\frac{1}{2}},
\end{align*}
where $c$, $c'>0$, $x\in K$, $t\in\ol\Real_+$. Here $K$ is as above. Similarly,
for every compact set $K\subset\Omega$ there exists a constant $c>0$, such that
\begin{align*}
\abs{\pr_{\eta}(\psi(t, x, \eta)-\seq{x, \eta})}
    &\leq c(1+\abs{\eta})^{-\frac{1}{2}}({\rm Im\,}\psi(t, x,\eta))^{\frac{1}{2}},
\end{align*}
where $x\in K$, $t\in\ol\Real_+$ and $\abs{\eta}\geq1$.

Note that $\abs{\pr^\alpha_{x}\pr^\beta_\eta(\psi(t, x, \eta)-\seq{x, \eta})}$ is quasi-homogeneous of degree $1-\abs{\beta}$.
For $\abs{\alpha}+\abs{\beta}\geq2$, we have
$\abs{\pr^\alpha_{x}\pr^\beta_\eta(\psi(t, x, \eta)-\seq{x, \eta})}\leq c(1+\abs{\eta})^{1-\abs{\beta}}$,
where $c>0$, $x\in K$, $t\in\ol\Real_+$ and $\abs{\eta}\geq1$. Here $K$ is as above.
The lemma follows.
\end{proof}

We also need the following

\begin{lem} \label{l:allergohome1071024}
For every compact set $K\subset\Omega$ and all $\alpha\in\Pstint^{2n-1}$, $\beta\in\Pstint^{2n-1}$,
there exist a constant $c_{\alpha,\beta,K}>0$ and $\varepsilon>0$, such that
\begin{align*}
&\abs{\pr^\alpha_{x}\pr^\beta_\eta(t\psi'_t(t, x, \eta))}\leq c_{\alpha,\beta,K}(1+\abs{\eta})^
{\frac{\abs{\alpha}-\abs{\beta}}{2}}e^{-t\varepsilon\abs{\eta}}
(1+{\rm Im\,}\psi(t, x,\eta))^{1+\frac{\abs{\alpha}+\abs{\beta}}{2}}
\end{align*}
if\, $\abs{\alpha}+\abs{\beta}\leq1$ and
\[\abs{\pr^\alpha_{x}\pr^\beta_\eta(t\psi'_t(t, x, \eta))}\leq c_{\alpha,\beta,K}(1+\abs{\eta})^{1-\abs{\beta}}e^{-t\varepsilon\abs{\eta}}\]
if\, $\abs{\alpha}+\abs{\beta}\geq2$, where $x\in K$, $t\in\ol\Real_+$, $\abs{\eta}\geq1$.
\end{lem}

\begin{proof}
The proof is essentially the same as the proof of Lemma~\ref{l:ss-sjostrand}.
\end{proof}

\begin{lem} \label{l:allergohome1}
For every compact set $K\subset\Omega$ and all $\alpha\in\Pstint^{2n-1}$, $\beta\in\Pstint^{2n-1}$,
there exist a constant $c_{\alpha,\beta,K}>0$ and $\varepsilon>0$, such that
\begin{equation} \label{e:allergohome1}
\abs{\pr^\alpha_{x}\pr^\beta_\eta(e^{i(\psi(t, x, \eta)-\seq{x, \eta}})}\leq c_{\alpha,\beta,K}(1+\abs{\eta})^
{\frac{\abs{\alpha}-\abs{\beta}}{2}}e^{-{\rm Im\,}\psi(t, x, \eta)}
(1+{\rm Im\,}\psi(t, x,\eta))^{\frac{\abs{\alpha}+\abs{\beta}}{2}}
\end{equation}
and
\begin{align} \label{e:allergohome2}
&\abs{\pr^\alpha_x\pr^\beta_\eta
(e^{i(\psi(t, x, \eta)-\seq{x, \eta}}t\psi'_t(t, x, \eta))} \nonumber \\
&\quad\leq c_{\alpha,\beta,K}(1+\abs{\eta})^
{\frac{\abs{\alpha}-\abs{\beta}}{2}}e^{-t\varepsilon\abs{\eta}}e^{-{\rm Im\,}\psi(t, x, \eta)}
(1+{\rm Im\,}\psi(t, x,\eta))^{1+\frac{\abs{\alpha}+\abs{\beta}}{2}},
\end{align}
where $x\in K$, $t\in\ol\Real_+$, $\abs{\eta}\geq1$.
\end{lem}

\begin{proof}
First, we prove (\ref{e:allergohome1}).
We proceed by induction over $\abs{\alpha}+\abs{\beta}$. For $\abs{\alpha}+\abs{\beta}\leq1$,
from Lemma~\ref{l:ss-sjostrand}, we get (\ref{e:allergohome1}). Let $\abs{\alpha}+\abs{\beta}\geq2$.
Then
\begin{align*}
&\abs{\pr^\alpha_{x}\pr^\beta_\eta(e^{i(\psi(t, x, \eta)-\seq{x, \eta}})} \\
&\leq c\!\!\!\!\!\sum_{\alpha'+\alpha''=\alpha, \beta'+\beta''=\beta, (\alpha'', \beta'')\neq0}\!\!
\abs{\pr^{\alpha'}_{x}\pr^{\beta'}_\eta(e^{i(\psi(t, x, \eta)-\seq{x, \eta}})
\pr^{\alpha''}_{x}\pr^{\beta''}_\eta(i\psi(t, x, \eta)-i\seq{x, \eta})},
\end{align*}
$c>0$. By the induction assumption, we have for every compact set $K\subset\Omega$,
there exists a constant $c>0$, such that
\begin{equation} \label{e:allergohome3}
\abs{\pr^{\alpha'}_{x}\pr^{\beta'}_\eta(e^{i(\psi(t, x, \eta)-\seq{x, \eta}})}\leq c(1+\abs{\eta})^
{\frac{\abs{\alpha'}-\abs{\beta'}}{2}}e^{-{\rm Im\,}\psi(t, x, \eta)}
(1+{\rm Im\,}\psi(t, x,\eta))^{\frac{\abs{\alpha'}+\abs{\beta'}}{2}},
\end{equation}
where $x\in K$, $t\in\ol\Real_+$, $\abs{\eta}\geq1$.
From Lemma~\ref{l:ss-sjostrand}, we have
\begin{equation} \label{e:allergohome4}
\abs{\pr^{\alpha''}_{x}\pr^{\beta''}_\eta(i\psi(t, x, \eta)-i\seq{x, \eta})}\leq
c(1+\abs{\eta})^
{\frac{\abs{\alpha''}-\abs{\beta''}}{2}}
(1+{\rm Im\,}\psi(t, x,\eta))^{\frac{\abs{\alpha''}+\abs{\beta''}}{2}},
\end{equation}
where $x\in K$, $t\in\ol\Real_+$, $\abs{\eta}\geq1$. Combining (\ref{e:allergohome3})
with (\ref{e:allergohome4}), we get (\ref{e:allergohome1}).

From Leibniz's formula, Lemma~\ref{l:allergohome1071024} and (\ref{e:allergohome1}), we get
(\ref{e:allergohome2}).
\end{proof}

\begin{lem} \label{l:ss-symbol}
Let $b(t,x,\eta)\in\hat S^k_r(\ol\Real_+\times T^*(\Omega);\, \mathscr L(\Lambda^{0,q}T^*(\Omega),\Lambda^{0,q}T^*(\Omega)))$
with $r>0$. We assume that $b(t,x,\eta)=0$ when
$\abs{\eta}\leq 1$. Then
\[q(x,\eta)=\int^{\infty}_{0} e^{i(\psi(t,x,\eta)-\seq{x,\eta})}b(t,x,\eta)dt\]
$\in
S^{k-1}_{\frac{1}{2},\frac{1}{2}}(T^*(\Omega);\,
\mathscr L(\Lambda^{0,q}T^*(\Omega),\Lambda^{0,q}T^*(\Omega)))$.
\end{lem}

\begin{proof}
From Leibniz's formula, we have
\begin{align*}
&\abs{\pr^{\alpha}_{x}\pr^{\beta}_{\eta}(e^{i(\psi(t,x,\eta)-\seq{x,\eta})}b(t,x,\eta))} \\
&\leq c\sum_{\alpha'+\alpha''=\alpha, \beta'+\beta''=\beta}
\abs{(\pr^{\alpha'}_{x}\pr^{\beta'}_\eta e^{i(\psi(t, x, \eta)-\seq{x, \eta}})(\pr^{\alpha''}_x\pr^{\eta''}_\eta b(t, x, \eta))},
\ \ c>0.
\end{align*}
From (\ref{e:allergohome1}) and the definition of $\hat S^k_r$, we have for every compact set $K\subset\Omega$,
there exist a constant $c>0$ and $\varepsilon>0$, such that
\begin{align*}
&\abs{\pr^\alpha_x\pr^\beta_\eta q(x, \eta)} \\
&\leq c\int^\infty_0\!\!e^{{-\rm Im\,}\psi(t, x, \eta)}(1+\abs{\eta})^{k+\frac{\abs{\alpha}-\abs{\beta}}{2}}
(1+{\rm Im\,}\psi(t, x, \eta))^{\frac{\abs{\alpha}+\abs{\beta}}{2}}e^{-\varepsilon t\abs{\eta}}dt \\
&\leq c'(1+\abs{\eta})^{k-1+\frac{\abs{\alpha}-\abs{\beta}}{2}},
\end{align*}
where $c'>0$, $x\in K$. The lemma follows.
\end{proof}

We will next show

\begin{prop} \label{p:ss-pseudo}
Let
\[b(t,x,\eta)\in\hat S^k_{r}(\ol\Real_+\times T^*(\Omega);\, \mathscr L(\Lambda^{0,q}T^*(\Omega),\Lambda^{0,q}T^*(\Omega)))\]
with $r>0$. We assume that $b(t,x,\eta)=0$ when $\abs{\eta}\leq 1$. Let $B$
be as in Proposition~\ref{p:ss-oscillmore}. Then
$B\in L^{k-1}_{\frac{1}{2},\frac{1}{2}}
(\Omega;\, \Lambda^{0,q}T^*(\Omega),\Lambda^{0,q}T^*(\Omega))$
with symbol
\[q(x,\eta)=\int^\infty_0e^{i(\psi(t,x,\eta)-\seq{x,\eta})}b(t,x,\eta)dt\]
$\in S^{k-1}_{\frac{1}{2},\frac{1}{2}}(T^*(\Omega);\, \mathscr L(\Lambda^{0,q}T^*(\Omega),\Lambda^{0,q}T^*(\Omega)))$.
\end{prop}

\begin{proof}
Choose $\chi\in C^\infty_0(\Real^{2n-1})$ so that $\chi(\eta)=1$ when $\abs{\eta}<1$ and
$\chi(\eta)=0$ when $\abs{\eta}>2$. Take a $u(y)\in C^\infty_0(\Omega;\, \Lambda^{0,q}T^*(\Omega))$, then
\begin{align*}
Bu &=\lim_{\eps\To 0}\frac{1}{(2\pi)^{2n-1}}\int^\infty_0\!\Bigr(\int e^{i(\psi(t,x,\eta)-\seq{y,\eta})}
         b(t,x,\eta)u(y)\chi(\eps\eta)d\eta\Bigr)dt  \\
        &=\lim_{\eps\To 0}\frac{1}{(2\pi)^{2n-1}}\int^\infty_0 e^{i\seq{x-y,\eta}}
\Bigl(\int e^{i(\psi(t,x,\eta)-\seq{x,\eta})}b(t,x,\eta)u(y)\chi(\eps\eta)\Bigr)d\eta dt  \\
        &=\lim_{\eps\To 0}\frac{1}{(2\pi)^{2n-1}}\int e^{i\seq{x-y,\eta}}q(x,\eta)u(y)
         \chi(\eps\eta)d\eta.
\end{align*}
From Lemma~\ref{l:ss-symbol}, we know that $q(x,\eta)\in S^{k-1}_{\frac{1}{2},\frac{1}{2}}$.
Thus
\[\lim_{\eps\To 0}\frac{1}{(2\pi)^{2n-1}}\int e^{i\seq{x-y,\eta}}q(x,\eta)u(y)\chi(\eps\eta)d\eta\in
L^{k-1}_{\frac{1}{2},\frac{1}{2}}
(\Omega;\, \Lambda^{0,q}T^*(\Omega),\Lambda^{0,q}T^*(\Omega)).\]
\end{proof}

We need the following

\begin{lem} \label{l:ss-symbol2}
Let $a(\infty,x,\eta)\in C^\infty(T^*(\Omega);\, \mathscr L(\Lambda^{0,q}T^*(\Omega),\Lambda^{0,q}T^*(\Omega)))$
be a classical symbol of order $k$, that is
\[a(\infty,x,\eta)\sim\sum^{\infty}_{j=0}a_j(\infty,x,\eta)\]
in the H\"{o}rmander symbol space $S^k_{1,0}(T^*(\Omega);\, \mathscr L(\Lambda^{0,q}T^*(\Omega),\Lambda^{0,q}T^*(\Omega)))$, where
\[a_j(\infty,x,\eta)\in C^\infty(T^*(\Omega);\, \mathscr L(\Lambda^{0,q}T^*(\Omega),\Lambda^{0,q}T^*(\Omega))),\ \ j=0,1,\ldots,\]
$a_j(\infty, x, \lambda\eta)=\lambda^{k-j}a_j(\infty, x, \eta)$, $\lambda\geq1$, $\abs{\eta}\geq1$, $j=0,1,\ldots$. Assume that
$a(\infty,x,\eta)=0$ when $\abs{\eta}\leq 1$. Then
\[p(x, \eta)=\int^{\infty}_0\Bigr(e^{i(\psi(t,x,\eta)-\seq{x,\eta})}-e^{i(\psi(\infty,x,\eta)-\seq{x,\eta})}\Bigr)
a(\infty,x,\eta)dt\]
$\in S^{k-1}_{\frac{1}{2},\frac{1}{2}}(T^*(\Omega);\,
\mathscr L(\Lambda^{0,q}T^*(\Omega),\Lambda^{0,q}T^*(\Omega)))$.
\end{lem}

\begin{proof}
Note that
\[p(x,\eta)=\int^{\infty}_0\!\!\!e^{i(\psi(t,x,\eta)-\seq{x,\eta})}(-t)i\psi'_{t}(t, x, \eta)a(\infty, x,\eta)dt.\]
From (\ref{e:allergohome2}),
we can repeat the procedure in the proof of Lemma~\ref{l:ss-symbol} to get the lemma.
\end{proof}

\begin{rem} \label{r:0809042013}
Let $a(t, x, \eta)\in\hat S^k_0(\ol\Real_+\times T^*(\Omega);\, \mathscr L(\Lambda^{0,q}T^*(\Omega),\Lambda^{0,q}T^*(\Omega)))$.
We assume $a(t, x, \eta)=0$, if $\abs{\eta}\leq 1$ and
\[a(t, x, \eta)-a(\infty, x, \eta)
\in\hat S^{k}_r(\ol\Real_+\times T^*(\Omega);\, \mathscr L(\Lambda^{0,q}T^*(\Omega),\Lambda^{0,q}T^*(\Omega)))\]
with $r>0$, where $a(\infty, x, \eta)$ is as in Lemma~\ref{l:ss-symbol2}. By Lemma~\ref{l:ss-symbol} and Lemma~\ref{l:ss-symbol2},
we have
\begin{align*}
&\int^{\infty}_0\Bigr(e^{i(\psi(t,x,\eta)-\seq{x,\eta})}a(t,x,\eta)-e^{i(\psi(\infty,x,\eta)-\seq{x,\eta})}
a(\infty,x,\eta)\Bigr)dt  \\
&=\int^{\infty}_0e^{i(\psi(t,x,\eta)-\seq{x,\eta})}(a(t, x, \eta)-a(\infty, x, \eta))dt
+\int^{\infty}_0\Bigr(e^{i(\psi(t,x,\eta)-\seq{x,\eta})} \\
&\quad-e^{i(\psi(\infty,x,\eta)
-\seq{x,\eta})}\Bigr)a(\infty, x, \eta))dt
\in S^{k-1}_{\frac{1}{2},\frac{1}{2}}(T^*(\Omega);\, \mathscr L(\Lambda^{0,q}T^*(\Omega),\Lambda^{0,q}T^*(\Omega))).
\end{align*}
Let
\begin{align*}
&A(x, y)=\frac{1}{(2\pi)^{2n-1}}\int\!\!\int^{\infty}_0\Bigl(e^{i(\psi(t,x,\eta)-\seq{y,\eta})}a(t,x,\eta) \\
&\quad-e^{i(\psi(\infty,x,\eta)-\seq{y,\eta})}a(\infty,x,\eta)\Bigl)dtd\eta
\end{align*}
be as in the Remark~\ref{r:ss-oscillmore}. Then as in Proposition~\ref{p:ss-pseudo}, we can show that
$A\in L^{k-1}_{\frac{1}{2},
\frac{1}{2}}(\Omega;\, \Lambda^{0,q}T^*(\Omega),\Lambda^{0,q}T^*(\Omega))$
with symbol
\[q(x,\eta)=\int^{\infty}_0\Bigr(e^{i(\psi(t,x,\eta)-\seq{x,\eta})}a(t,x,\eta)-
e^{i(\psi(\infty,x,\eta)-\seq{x,\eta})}a(\infty,x,\eta)\Bigr)dt\]
$\in S^{k-1}_{\frac{1}{2},\frac{1}{2}}(T^*(\Omega);\, \mathscr L(\Lambda^{0,q}T^*(\Omega),\Lambda^{0,q}T^*(\Omega)))$.
\end{rem}

We have the following

\begin{prop} \label{p:ss-symbol}
Let $a(\infty,x,\eta)\in C^\infty(T^*(\Omega);\, \mathscr L(\Lambda^{0,q}T^*(\Omega),\Lambda^{0,q}T^*(\Omega)))$
be a classical symbol of order $k$. Then
\[a(x,\eta)=e^{i(\psi(\infty,x,\eta)-\seq{x,\eta})}a(\infty,x,\eta)
\in S^k_{\frac{1}{2},\frac{1}{2}}(T^*(\Omega);\, \mathscr L(\Lambda^{0,q}T^*(\Omega),\Lambda^{0,q}T^*(\Omega))).\]
\end{prop}

\begin{proof}
In view of Lemma~\ref{l:allergohome1}, we have
for every compact set $K\subset\Omega$ and all $\alpha\in\Pstint^{2n-1}$, $\beta\in\Pstint^{2n-1}$,
there exists a constant $c_{\alpha,\beta,K}>0$, such that
\[\abs{\pr^\alpha_{x}\pr^\beta_\eta(e^{i(\psi(\infty, x, \eta)-\seq{x, \eta}})}\leq c_{\alpha,\beta,K}(1+\abs{\eta})^
{\frac{\abs{\alpha}-\abs{\beta}}{2}}e^{-{\rm Im\,}\psi(\infty, x, \eta)}
(1+{\rm Im\,}\psi(\infty, x,\eta))^{\frac{\abs{\alpha}+\abs{\beta}}{2}},\]
where $x\in K$, $\abs{\eta}\geq1$. From this and Leibniz's formula, we get the proposition.
\end{proof}


\section{The heat equation}

Until further notice, we work with some local coordinates
$x=(x_1,\ldots,x_{2n-1})$
defined on an open set $\Omega\subset X$. Let
\[b(t,x,\eta)\in\hat S^k_{r}(\ol\Real_+\times T^*(\Omega);\, \mathscr L(\Lambda^{0,q}T^*(\Omega),\Lambda^{0,q}T^*(\Omega)))\]
with $r>0$. We assume that $b(t,x,\eta)=0$ when
$\abs{\eta}\leq 1$. From now on, we write
$\frac{1}{(2\pi)^{2n-1}}\int\Bigr(\int^\infty_0e^{i(\psi(t,x,\eta)-\seq{y,\eta})}b(t,x,\eta)dt\Bigr)d\eta$
to denote the kernel of pseudodifferential operator of order $k-1$ type
$(\frac{1}{2},\frac{1}{2})$ from sections of $\Lambda^{0,q}T^*(\Omega)$ to sections of
$\Lambda^{0,q}T^*(\Omega)$ with symbol
\[\int^\infty_0e^{i(\psi(t,x,\eta)-\seq{x,\eta})}b(t,x,\eta)dt\in
S^{k-1}_{\frac{1}{2},\frac{1}{2}}(T^*(\Omega);\, \mathscr L(\Lambda^{0,q}T^*(\Omega),\Lambda^{0,q}T^*(\Omega))).\]
(See Proposition~\ref{p:ss-pseudo}.)

Let
$a(t,x,\eta)\in\hat S^k_{0}(\ol\Real_+\times T^*(\Omega);\, \mathscr L(\Lambda^{0,q}T^*(\Omega),\Lambda^{0,q}T^*(\Omega)))$.
Assume that
\[a(t, x, \eta)=0\]
when $\abs{\eta}\leq 1$ and that
\[a(t, x, \eta)-a(\infty, x, \eta)\in\hat S^{k}_r
(\ol\Real_+\times T^*(\Omega);\, \mathscr L(\Lambda^{0,q}T^*(\Omega),\Lambda^{0,q}T^*(\Omega)))\]
with $r>0$, where
$a(\infty, x, \eta)\in C^\infty(T^*(\Omega);\, \mathscr L(\Lambda^{0,q}T^*(\Omega),\Lambda^{0,q}T^*(\Omega)))$
is a classical symbol of order $k$.
From now on, we write
\[\frac{1}{(2\pi)^{2n-1}}\int\Bigl(\int^{\infty}_0\Bigl(e^{i(\psi(t,x,\eta)-\seq{y,\eta})}a(t,x,\eta)-e^{i(\psi(\infty,x,\eta)-\seq{y,\eta})}
a(\infty,x,\eta)\Bigl)dt\Bigl)d\eta\]
to denote the kernel of pseudodifferential operator of order $k-1$ type
$(\frac{1}{2},\frac{1}{2})$ from sections of $\Lambda^{0,q}T^*(\Omega)$ to sections of
$\Lambda^{0,q}T^*(\Omega)$ with symbol
\[\int^{\infty}_0\Bigl(e^{i(\psi(t,x,\eta)-\seq{x,\eta})}a(t,x,\eta)-
e^{i(\psi(\infty,x,\eta)-\seq{x,\eta})}a(\infty,x,\eta)\Bigl)dt\]
in $S^{k-1}_{\frac{1}{2},\frac{1}{2}}((T^*(\Omega);\, \mathscr L(\Lambda^{0,q}T^*(\Omega),\Lambda^{0,q}T^*(\Omega)))$.
(See Remark~\ref{r:0809042013}.) From
Proposition~\ref{p:hf-asym}, we have the following

\begin{prop} \label{p:he-basis}
Let $Q$ be a $C^\infty$ differential operator on $\Omega$ of order $m$. Let
\[b(t,x,\eta)\in\hat S^k_{r}(\ol\Real_+\times T^*(\Omega);\, \mathscr L(\Lambda^{0,q}T^*(\Omega),\Lambda^{0,q}T^*(\Omega)))\]
with $r>0$. We assume that $b(t,x,\eta)=0$ when
$\abs{\eta}\leq 1$. Set
\[Q(e^{i(\psi(t,x,\eta)-\seq{y,\eta})}b(t,x,\eta))=e^{i(\psi(t,x,\eta)-\seq{y,\eta})}
c(t,x,\eta),\]
$c(t, x, \eta)\in\hat S^{k+m}_r(\ol\Real_+\times T^*(\Omega);\, \mathscr L(\Lambda^{0,q}T^*(\Omega),\Lambda^{0,q}T^*(\Omega)))$, $r>0$.
Put
\begin{align*}
&B(x, y)=\frac{1}{(2\pi)^{2n-1}}\int\Bigl(\int^{\infty}_0 e^{i(\psi(t,x,\eta)-\seq{y,\eta})}b(t,x,\eta)dt\Bigl)d\eta, \\
&C(x, y)=\frac{1}{(2\pi)^{2n-1}}\int\Bigl(\int^{\infty}_0 e^{i(\psi(t,x,\eta)-\seq{y,\eta})}c(t,x,\eta)dt\Bigl)d\eta.
\end{align*}
We have $Q\circ B\equiv C$.
\end{prop}

\begin{prop} \label{p:he-basismore}
Let $Q$ be a $C^\infty$ differential operator on $\Omega$ of order $m$. Let
\[b(t,x,\eta)\in\hat S^k_{0}(\ol\Real_+\times T^*(\Omega);\, \mathscr L(\Lambda^{0,q}T^*(\Omega),\Lambda^{0,q}T^*(\Omega))).\]
We assume that $b(t, x, \eta)=0$ when $\abs{\eta}\leq 1$ and that
\[b(t, x, \eta)-b(\infty, x, \eta)\in
\hat S^{k}_r(\ol\Real_+\times T^*(\Omega);\, \mathscr L(\Lambda^{0,q}T^*(\Omega),\Lambda^{0,q}T^*(\Omega)))\]
with $r>0$, where
$b(\infty, x, \eta)\in C^\infty(T^*(\Omega);\, \mathscr L(\Lambda^{0,q}T^*(\Omega),\Lambda^{0,q}T^*(\Omega)))$
is a classical symbol of order $k$. Set
\begin{align*}
&Q\Bigl(e^{i(\psi(t,x,\eta)-\seq{y,\eta})}b(t,x,\eta)-e^{i(\psi(\infty,x,\eta)-\seq{y,\eta})}b(\infty,x,\eta)\Bigl) \\
&=e^{i(\psi(t,x,\eta)-\seq{y,\eta})}c(t, x, \eta)-e^{i(\psi(\infty,x,\eta)-\seq{y,\eta})}c(\infty,x,\eta),
\end{align*}
where
$c(t, x, \eta)\!\in\hat S^{k+m}_0\!(\ol\Real_+\times T^*(\Omega);\, \mathscr L(\Lambda^{0,q}T^*(\Omega),\Lambda^{0,q}T^*(\Omega)))$,
\[c(\infty, x, \eta)\in C^\infty(T^*(\Omega);\,
\mathscr L(\Lambda^{0,q}T^*(\Omega),\Lambda^{0,q}T^*(\Omega)))\]
is a classical symbol of order $k+m$. Then
\[c(t, x, \eta)-c(\infty, x, \eta)\in
\hat S^{k+m}_r(\ol\Real_+\times T^*(\Omega);\, \mathscr L(\Lambda^{0,q}T^*(\Omega),\Lambda^{0,q}T^*(\Omega)))\]
with $r>0$. Put
\begin{align*}
B(x, y)=\frac{1}{(2\pi)^{2n-1}}\int\Bigl(\int^{\infty}_0 &\Bigl(e^{i(\psi(t,x,\eta)-\seq{y,\eta})}b(t,x,\eta)-  \\
                                                &e^{i(\psi(\infty,x,\eta)-\seq{y,\eta})}b(\infty,x,\eta)\Bigl)dt\Bigl)d\eta, \\
C(x, y)=\frac{1}{(2\pi)^{2n-1}}\int\Bigl(\int^{\infty}_0 &\Bigl(e^{i(\psi(t,x,\eta)-\seq{y,\eta})}c(t,x,\eta)-   \\
                                                &e^{i(\psi(\infty,x,\eta)-\seq{y,\eta})}c(\infty,x,\eta)\Bigl)dt\Bigl)d\eta.
\end{align*}
We have $Q\circ B\equiv C$.
\end{prop}

We return to our problem. From now on, we assume that our operators are properly
supported. We assume that $Y(q)$ holds. Let
\[a_j(t,x,\eta)\in\hat S^{-j}_r(\ol\Real_+\times T^*(\Omega);\, \mathscr L(\Lambda^{0,q}T^*(\Omega),\Lambda^{0,q}T^*(\Omega))), \ \
j=0,1,\ldots, r>0,\]
be as in Proposition~\ref{p:ss-solution} with $a_0(0,x,\eta)=I$, $a_j(0,x,\eta)=0$ when $j>0$. Let
$a(t,x,\eta)\sim\sum^\infty_{j=0}a_j(t, x, \eta)$ in $\hat S^{0}_r(\ol\Real_+\times T^*(\Omega);\,
\mathscr L(\Lambda^{0,q}T^*(\Omega),\Lambda^{0,q}T^*(\Omega)))$,
where $r>0$. Let
\begin{equation} \label{e:070727home1}
(\pr_t+\Box^{(q)}_b)\Bigl(e^{i(\psi(t,x,\eta)-\seq{y,\eta})}a(t,x,\eta)\Bigl)=
e^{i(\psi(t,x,\eta)-\seq{y,\eta})}b(t,x,\eta).
\end{equation}
From Proposition~\ref{p:hf-asym}, we see that for
every compact set $K\subset\Omega$, $\varepsilon>0$ and all indices $\alpha$, $\beta$ and $N\in\Pstint$,
there exists $c_{\alpha,\beta,N,\eps,K}>0$ such that
\begin{equation} \label{e:he-chinyu}
\abs{\pr^\alpha_x\pr^\beta_\eta b(t,x,\eta)}
\leq c_{\alpha,\beta,N,\eps,K}e^{t(-r(x,\eta)+\varepsilon\abs{\eta})}(\abs{\eta}^{-N}+\abs{\eta}^{2-N}
({\rm Im\,}\psi(t,x,\eta))^N),
\end{equation}
where $t\in\ol\Real_+$, $x\in K$, $\abs{\eta}\geq 1$. Choose $\chi\in C^\infty_0
(\Real^{2n-1})$ so that $\chi(\eta)=1$ when $\abs{\eta}<1$ and $\chi(\eta)=0$ when $\abs{\eta}>2$.
Set
\begin{equation} \label{e:allerallerallerhome1}
A(x,y)=\frac{1}{(2\pi)^{2n-1}}\int\Bigl(\int^{\infty}_0 e^{i(\psi(t,x,\eta)-\seq{y,\eta})}a(t,x,\eta)
(1-\chi(\eta))dt\Bigl)d\eta.
\end{equation}
We have the following proposition

\begin{prop} \label{p:he-yofq}
Suppose $Y(q)$ holds. Let $A$ be as in (\ref{e:allerallerallerhome1}). We have\,
$\Box^{(q)}_bA\equiv I$.
\end{prop}

\begin{proof}
We have
\begin{align*}
&\Box^{(q)}_b\Bigl(\frac{1}{(2\pi)^{2n-1}}e^{i(\psi(t,x,\eta)-\seq{y,\eta})}
a(t,x,\eta)(1-\chi(\eta))\Bigl) \\
&\quad=\frac{1}{(2\pi)^{2n-1}}e^{i(\psi(t,x,\eta)-\seq{y,\eta})}b(t,x,\eta)(1-\chi(\eta)) \\
&\quad\quad -\frac{1}{(2\pi)^{2n-1}}\frac{\pr}{\pr t}\Bigl(e^{i(\psi(t,x,\eta)-\seq{y,\eta})}a(t, x, \eta)\Bigl)
(1-\chi(\eta)),
\end{align*}
where $b(t, x, \eta)$ is as in (\ref{e:070727home1}), (\ref{e:he-chinyu}).
From Proposition~\ref{p:he-basis}, we have
\begin{align*}
\Box^{(q)}_b\circ A
&\equiv\frac{1}{(2\pi)^{2n-1}}\int\Bigl(\int^{\infty}_0 e^{i(\psi(t,x,\eta)-\seq{y,\eta})}b(t,x,\eta)(1-\chi(\eta))
  dt\Bigl)d\eta  \\
&\quad-\frac{1}{(2\pi)^{2n-1}}\int\Bigl(\int^{\infty}_0\frac{\pr}{\pr t}\Bigl(e^{i(\psi(t,x,\eta)-\seq{y,\eta})}
 a(t, x, \eta)\Bigl)(1-\chi(\eta))dt\Bigl)d\eta.
\end{align*}
From (\ref{e:he-chinyu}), it follows that
$\int\Bigl(\int^{\infty}_0 e^{i(\psi(t,x,\eta)-\seq{y,\eta})}b(t,x,\eta)(1-\chi(\eta))
  dt\Bigl)d\eta$
is smoothing. Choose a cut-off function $\chi_1(\eta)\in C^\infty_0(\Real^{2n-1})$ so that
$\chi_1(\eta)=1$ when $\abs{\eta}<1$ and $\chi_1(\eta)=0$ when $\abs{\eta}>2$. Take a
$u(y)\in C^\infty_0(\Omega;\, \Lambda^{0,q}T^*(\Omega))$,
then
\begin{align*}
& \lim_{\eps\To 0}\int\!\!\!\int\!\!\!\Bigl(\int^{\infty}_0\!\!\!\frac{\pr}{\pr t}
  \Bigl(e^{i(\psi(t,x,\eta)-\seq{y,\eta}}a(t, x, \eta)\Bigl)(1-\chi(\eta))\chi_1(\eps\eta)u(y)dt\Bigl)d\eta dy  \\
&\quad=-\lim_{\eps\To 0}\int\!\!\!\int\!\!\!e^{i\seq{x-y,\eta}}(1-\chi(\eta))
  \chi_1(\eps\eta)u(y)d\eta dy.
\end{align*}
Hence
$\frac{1}{(2\pi)^{2n-1}}\int\!\Bigl(\int^{\infty}_0
\frac{\pr}{\pr t}\Bigl(e^{i(\psi(t,x,\eta)-\seq{y,\eta})}a(t, x, \eta)\Bigl)(1-\chi(\eta))dt\Bigl)d\eta\equiv-I$.
Thus $\Box^{(q)}_b\circ A\equiv I$.
\end{proof}

\begin{rem} \label{r:he-para}
We assume that $Y(q)$ holds. From Proposition~\ref{p:he-yofq}, we know that, for every local
coordinate patch $X_{j}$, there exists a properly supported operator
$A_j: \mathscr D'(X_j;\, \Lambda^{0,q}T^*(X_j))\To\mathscr D'(X_j;\, \Lambda^{0,q}T^*(X_j))$
such that
\[A_{j}: H^s_{\rm loc}(X_j;\, \Lambda^{0,q}T^*(X_j))\To
H^{s+1}_{\rm loc}(X_j;\, \Lambda^{0,q}T^*(X_j))\]
and $\Box^{(q)}_b\circ A_j-I: H^s_{\rm loc}(X_j;\, \Lambda^{0,q}T^*(X_j))\To
H^{s+m}_{\rm loc}(X_j;\, \Lambda^{0,q}T^*(X_j))$
for all $s\in\Real$ and $m\geq 0$. We assume that $X=\bigcup^k_{j=1}X_{j}$. Let
$\set{\chi_j}$ be a $C^\infty$ partition of unity subordinate to $\set{X_{j}}$ and set
$Au= \sum_jA_j(\chi_ju)$, $u\in\mathscr D'(X;\, \Lambda^{0,q}T^*(X))$.
$A$ is well-defined as a continuous operator:
\[A: H^s(X;\, \Lambda^{0,q}T^*(X))\To H^{s+1}(X;\, \Lambda^{0,q}T^*(X))\]
for all $s\in\Real$. We notice that $A$ is properly supported. We have
$\Box^{(q)}_b\circ A-I: H^s(X;\, \Lambda^{0,q}T^*(X))\To H^{s+m}(X;\, \Lambda^{0,q}T^*(X))$
for all $s\in\Real$ and $m\geq 0$.
\end{rem}

Assume that $Y(q)$ fails. Let
\[a_j(t, x, \eta)\in\hat S^{-j}_0(\ol\Real_+\times T^*(\Omega);\, \mathscr L(\Lambda^{0,q}T^*(\Omega),\Lambda^{0,q}T^*(\Omega))),\]
$j=0,1,\ldots$,
and
$a_j(\infty, x, \eta)\in C^\infty(T^*(\Omega);\, \mathscr L(\Lambda^{0,q}T^*(\Omega),\Lambda^{0,q}T^*(\Omega)))$,
$j=0,1,\ldots$,
be as in Proposition~\ref{p:ss-SBmore}. We recall that for some $r>0$,
\[a_j(t,x,\eta)-a_j(\infty,x,\eta)\in\hat S^{-j}_r(\ol\Real_+\times T^*(\Omega);\,
\mathscr L(\Lambda^{0,q}T^*(\Omega),\Lambda^{0,q}T^*(\Omega))),\]
$j=0,1,\ldots$. Let
\begin{equation} \label{e:0707292108}
a(\infty, x, \eta)\sim\sum^\infty_{j=0}a_j(\infty, x, \eta)
\end{equation}
in the H\"{o}rmander symbol space $S^{0}_{1,0}(T^*(\Omega);\,
\mathscr L(\Lambda^{0,q}T^*(\Omega),\Lambda^{0,q}T^*(\Omega)))$. Let
\begin{equation} \label{e:allerallerallergo1}
a(t,x,\eta)\sim\sum^\infty_{j=0}a_j(t, x, \eta)
\end{equation}
in $\hat S^{0}_0(\ol\Real_+\times T^*(\Omega);\,
\mathscr L(\Lambda^{0,q}T^*(\Omega),\Lambda^{0,q}T^*(\Omega)))$.
We take $a(t, x, \eta)$ so that for every compact set $K\subset\Omega$ and all indices $\alpha$, $\beta$, $\gamma$, $k$,
there exists $c>0$, $c$ is independent of $t$, such that
\begin{equation} \label{e:allerallerallergo2}
\abs{\pr^\gamma_t\pr^\alpha_x\pr^\beta_\eta(a(t, x, \eta)-\sum^k_{j=0}a_j(t, x, \eta))}
\leq c(1+\abs{\eta})^{-k-1+\gamma-\abs{\beta}},
\end{equation}
where $t\in\ol\Real_+$, $x\in K$, $\abs{\eta}\geq1$, and
\[a(t, x, \eta)-a(\infty, x, \eta)\in\hat S^0_r(\ol\Real_+\times T^*(\Omega);\,
\mathscr L(\Lambda^{0,q}T^*(\Omega),\\ \Lambda^{0,q}T^*(\Omega)))\]
with $r>0$. Let
\begin{equation} \label{e:allerimportanthome3}
(\pr_t+\Box^{(q)}_b)\Bigl(e^{i(\psi(t,x,\eta)-\seq{y,\eta})}a(t,x,\eta)\Bigl)=
 e^{i(\psi(t,x,\eta)-\seq{y,\eta})}b(t,x,\eta).
\end{equation}
Then
$b(t, x, \eta)\in\hat S^2_0(\ol\Real_+\times T^*(\Omega);\, \mathscr L(\Lambda^{0,q}T^*(\Omega),\Lambda^{0,q}T^*(\Omega)))$
and
\begin{equation} \label{e:allerallergogogohome1}
b(t, x, \eta)-b(\infty, x, \eta)\in
\hat S^2_r(\ol\Real_+\times T^*(\Omega);\, \mathscr L(\Lambda^{0,q}T^*(\Omega),\Lambda^{0,q}T^*(\Omega)))
\end{equation}
with $r>0$, where $b(\infty, x, \eta)$ is a classical symbol of order $2$. Moreover, we have
\begin{align} \label{e:0707281025}
&(\pr_t+\Box^{(q)}_b)\Bigl(\frac{1}{(2\pi)^{2n-1}}\Bigl(e^{i(\psi(t,x,\eta)-\seq{y,\eta})}
a(t,x,\eta)-e^{i(\psi(\infty,x,\eta)-\seq{y,\eta})}a(\infty,x,\eta)\Bigr)\Bigr) \nonumber  \\
&=\frac{1}{(2\pi)^{2n-1}}\Bigl(e^{i(\psi(t,x,\eta)-\seq{y,\eta})}b(t,x,\eta)
-e^{i(\psi(\infty,x,\eta)-\seq{y,\eta})}b(\infty,x,\eta)\Bigl).
\end{align}

From Proposition~\ref{p:hf-asym}, we see that for every compact set $K\subset\Omega$ and all indices $\alpha$, $\beta$ and
$N\in\Pstint$, there exists $c_{\alpha,\beta,N,K}>0$ such that
\begin{equation} \label{e:he-chinyumore1}
\abs{\pr^\alpha_x\pr^\beta_\eta b(t,x,\eta)}
\leq c_{\alpha,\beta,N,K}(\abs{\eta}^{-N}+\abs{\eta}^{2-N}({\rm Im\,}\psi(t,x,\eta))^N),
\end{equation}
where $t\in\ol\Real_+$, $x\in K$, $\abs{\eta}\geq1$. Thus,
\begin{equation} \label{e:he-chinyumore2}
\abs{\pr^\alpha_x\pr^\beta_\eta b(\infty,x,\eta)}
\leq c_{\alpha,\beta,N,K}(\abs{\eta}^{-N}+\abs{\eta}^{2-N}({\rm Im\,}\psi(\infty,x,\eta))^N).
\end{equation}
From (\ref{e:allerallergogogohome1}), (\ref{e:he-chinyumore1}) and (\ref{e:he-chinyumore2}), it follows that
for every compact set $K\subset\Omega$, $\eps>0$ and all indices $\alpha$, $\beta$ and
$N\in\Pstint$, there exists $c_{\alpha,\beta,N,\eps,K}>0$ such that
\begin{align} \label{e:he-chinyumore3}
\lefteqn{ \abs{\pr^\alpha_x\pr^\beta_\eta\Bigl(b(t,x,\eta)-b(\infty, x, \eta)\Bigl)}}\nonumber  \\
&\quad \leq c_{\alpha,\beta,N,\eps,K}\Bigl
(e^{t(-r(x, \eta)+\varepsilon\abs{\eta})}(\abs{\eta}^{-N}+\abs{\eta}^{2-N}({\rm Im\,}\psi(t,x,\eta))^N)\Bigl)^
{\frac{1}{2}},
\end{align}
where $t\in\ol\Real_+$, $x\in K$, $\abs{\eta}\geq1$, $r>0$.

Choose $\chi\in C^\infty_0
(\Real^{2n-1})$ so that $\chi(\eta)=1$ when $\abs{\eta}<1$ and $\chi(\eta)=0$ when $\abs{\eta}>2$.
Set
\begin{align} \label{e:allerimportanthome1}
G(x, y) &= \frac{1}{(2\pi)^{2n-1}}\int\Bigl(\int^{\infty}_0\Bigl(e^{i(\psi(t,x,\eta)-\seq{y,\eta})}a(t,x,\eta)-\nonumber  \\
  &\quad\quad\quad e^{i(\psi(\infty,x,\eta)-\seq{y,\eta})}a(\infty,x,\eta)\Bigr)(1-\chi(\eta))dt\Bigl)d\eta.
\end{align}
Put
\begin{equation} \label{e:allerimportanthome2}
S(x, y)=\frac{1}{(2\pi)^{2n-1}}\int e^{i(\psi(\infty,x,\eta)-\seq{y,\eta})}a(\infty,x,\eta)d\eta.
\end{equation}
We have the following

\begin{prop} \label{p:he-yofnotq}
We assume that $Y(q)$ fails. Let $G$ and $S$
be as in (\ref{e:allerimportanthome1}) and (\ref{e:allerimportanthome2}) respectively. Then
$S+\Box^{(q)}_b\circ G\equiv I$
and
$\Box^{(q)}_b\circ S\equiv 0$.
\end{prop}

\begin{proof}
We have
\begin{align*}
\lefteqn{\Box^{(q)}_b\Bigl(\frac{1}{(2\pi)^{2n-1}}e^{i(\psi(t,x,\eta)-\seq{y,\eta})}a(t,x,\eta)\Bigl)}  \\
&\quad=\frac{1}{(2\pi)^{2n-1}}e^{i(\psi(t,x,\eta)-\seq{y,\eta})}\Bigl(b(t,x,\eta)-i\frac{\pr\psi}
 {\pr t}a-\frac{\pr a}{\pr t}\Bigl),
\end{align*}
where $b(t, x, \eta)$ is as in (\ref{e:allerimportanthome3}).
Letting $t\To\infty$, we get
\begin{align*}
\lefteqn{\Box^{(q)}_b\Bigl(\frac{1}{(2\pi)^{2n-1}}e^{i(\psi(\infty,x,\eta)-\seq{y,\eta})}
a(\infty,x,\eta)\Bigl)}   \\
&\quad=\frac{1}{(2\pi)^{2n-1}}e^{i(\psi(\infty,x,\eta)-\seq{y,\eta})}b(\infty,x,\eta),
\end{align*}
where $b(\infty, x, \eta)$ is as in (\ref{e:allerallergogogohome1}) and (\ref{e:he-chinyumore2}).
From (\ref{e:he-chinyumore2}), we have
\[\frac{1}{(2\pi)^{2n-1}}\int e^{i(\psi(\infty,x,\eta)-\seq{y,\eta})}b(\infty,x,\eta)d\eta\]
is smoothing. Thus
$\Box^{(q)}_b\circ S\equiv 0$.

In view of (\ref{e:0707281025}), we have
\begin{align*}
\lefteqn{\Box^{(q)}_b\Bigl(\frac{1}{(2\pi)^{2n-1}}\Bigl(e^{i(\psi(t,x,\eta)-\seq{y,\eta})}
a(t,x,\eta)(1-\chi(\eta))}  \\
&\hskip50pt-e^{i(\psi(\infty,x,\eta)-\seq{y,\eta})}a(\infty,x,\eta)(1-\chi(\eta)\Bigr)\Bigr)  \\
&=\frac{1}{(2\pi)^{2n-1}}\Bigl(e^{i(\psi(t,x,\eta)-\seq{y,\eta})}b(t,x,\eta) \\
&\hskip30pt-e^{i(\psi(\infty,x,\eta)-\seq{y,\eta})}b(\infty,x,\eta)\Bigl)(1-\chi(\eta))  \\
&\hskip50pt-\frac{1}{(2\pi)^{2n-1}}\frac{\pr}{\pr t}\Bigl(e^{i(\psi(t,x,\eta)-\seq{y,\eta})}a(t, x, \eta)\Bigl)
(1-\chi(\eta)).
\end{align*}
From Proposition~\ref{p:he-basismore}, we have
\begin{align*}
\Box^{(q)}_b\circ G
&=\Box^{(q)}_b\Bigl(\frac{1}{(2\pi)^{2n-1}}\int\Bigl(\int^{\infty}_0\Bigl(e^{i(\psi(t,x,\eta)-\seq{y,\eta})}a(t,x,\eta)  \\
&\quad-e^{i(\psi(\infty,x,\eta)-\seq{y,\eta})}a(\infty,x,\eta)\Bigl)(1-\chi(\eta))dt\Bigl)d\eta\Bigl) \\
&\quad\equiv\frac{1}{(2\pi)^{2n-1}}\Bigl(\int\Bigl(\int^{\infty}_0\Bigl(e^{i(\psi(t,x,\eta)-\seq{y,\eta})}b(t,x,\eta) \\
&\quad-e^{i(\psi(\infty,x,\eta)-\seq{y,\eta})}b(\infty,x,\eta)\Bigl)(1-\chi(\eta))dt\Bigl)d\eta  \\
&\quad-\int\Bigl(\int^{\infty}_0 \frac{\pr}{\pr t}\Bigl(e^{i(\psi(t,x,\eta)-\seq{y,\eta})}a(t, x, \eta)\Bigl)
(1-\chi(\eta))dt\Bigl)d\eta\Bigl).
\end{align*}
In view of (\ref{e:he-chinyumore2}) and (\ref{e:he-chinyumore3}), we see that
\begin{align*}
&\int\Bigl(\int^{\infty}_0\Bigl(e^{i(\psi(t,x,\eta)-\seq{y,\eta})}b(t,x,\eta) \\
&\quad-e^{i(\psi(\infty,x,\eta)-\seq{y,\eta})}b(\infty,x,\eta)\Bigl)(1-\chi(\eta))dt\Bigl)d\eta \\
&=\int\Bigl(\int^{\infty}_0\Bigl(e^{i(\psi(t,x,\eta)-\seq{y,\eta})}-
e^{i(\psi(\infty,x,\eta)-\seq{y,\eta})}\Bigl)
b(\infty, x, \eta)(1-\chi(\eta))dt\Bigl)d\eta \\
&\quad+\int\Bigl(\int^{\infty}_0e^{i(\psi(t,x,\eta)-\seq{y,\eta})}
\Bigl(b(t, x, \eta)-b(\infty, x, \eta)\Bigl)(1-\chi(\eta))dt\Bigl)d\eta
\end{align*}
is smoothing.

Choose a cut-off function $\chi_1(\eta)\in C^\infty_0(\Real^{2n-1})$ so that
$\chi_1(\eta)=1$ when $\abs{\eta}<1$ and $\chi_1(\eta)=0$ when $\abs{\eta}>2$. Take a
$u\in C^\infty_0(\Omega;\, \Lambda^{0,q}T^*(\Omega))$, then
\begin{align*}
&\lim_{\eps\To 0}\int\Bigl(\int^{\infty}_0\frac{\pr}{\pr t}
\Bigl(e^{i(\psi(t,x,\eta)-\seq{y,\eta})}a(t, x, \eta)\Bigl)
(1-\chi(\eta))\chi_1(\eps\eta)u(y)dt\Bigl)d\eta dy  \\
&=\lim_{\eps\To 0}\int e^{i(\psi(\infty,x,\eta)-\seq{y,\eta})}a(\infty,x,\eta)(1-\chi(\eta))
\chi_1(\eps\eta)u(y)d\eta dy \\
&\quad-\lim_{\eps\To 0}\int e^{i\seq{x-y,\eta}}(1-\chi(\eta))\chi_1(\eps\eta)u(y)d\eta dy.
\end{align*}
Hence
\[\frac{1}{(2\pi)^{2n-1}}\int\Bigl(\int^{\infty}_0\frac{\pr}{\pr t}
\Bigl(e^{i(\psi(t,x,\eta)-\seq{y,\eta})}a(t, x, \eta)\Bigl)(1-\chi(\eta))dt\Bigl)d\eta\equiv S-I.\]
Thus $S+\Box_{b,q}\circ G\equiv I$.
\end{proof}

In the rest of this section, we recall some facts about Hilbert space theory for $\Box^{(q)}_b$. Our basic reference for these matters is~\cite{BG88}.
Let $A$ be as in Remark~\ref{r:he-para}.
$A$ has a formal adjoint
$A^*:\mathscr D'(X;\, \Lambda^{0,q}T^*(X))\To\mathscr D'(X;\, \Lambda^{0,q}T^*(X))$,
$(A^*u\ |\ v)=(u\ |\ Av)$, $u\in C^\infty(X;\, \Lambda^{0,q}T^*(X))$, $v\in C^\infty(X;\, \Lambda^{0,q}T^*(X))$.

\begin{lem} \label{l:he-para}
$A^*$ is well-defined as a continuous operator:
\[A^*: H^s(X;\, \Lambda^{0,q}T^*(X))\To H^{s+1}(X;\, \Lambda^{0,q}T^*(X))\]
for all $s\in\Real$. Moreover, we have $A^*\equiv A$.
\end{lem}

\begin{proof}
The first statement is a consequence of the theorem of Calderon and Vaillancourt.
(See Proposition~\ref{p:he-calderon}.) In view of
Remark~\ref{r:he-para}, we see that $\Box^{(q)}_b\circ A\equiv I$. Thus
$A^*\circ\Box^{(q)}_b\equiv I$.
We have
\begin{align*}
A^*-A &\equiv A^*\circ(\Box^{(q)}_b\circ A)-A   \\
      &\equiv(A^*\circ\Box^{(q)}_b)\circ A-A      \\
      &\equiv A-A              \\
      &\equiv 0.
\end{align*}
The lemma follows.
\end{proof}

From this, we get a two-sided parametrix for $\Box^{(q)}_b$.

\begin{prop} \label{p:he-para}
We assume that $Y(q)$ holds. Let $A$ be as in Remark~\ref{r:he-para}. Then
$\Box^{(q)}_b\circ A\equiv A\circ\Box^{(q)}_b\equiv I$.
\end{prop}

\begin{proof}
In view of Lemma~\ref{l:he-para}, we have $A^*\equiv A$. Thus
$I\equiv\Box^{(q)}_b\circ A\equiv A^*\circ\Box^{(q)}_b\equiv A\circ\Box^{(q)}_b$.
\end{proof}

\begin{rem} \label{r:he-beals}
The existence of a two-sided parametrix for $\Box^{(q)}_b$ under condition $Y(q)$ is a classical result.
See~\cite{BG88}.
\end{rem}

\begin{defn} \label{d:he-global}
Suppose $Q$ is a closed densely defined operator
\[Q: H\supset{\rm Dom\,}Q\To{\rm Ran\,}Q\subset H,\]
where $H$ is a Hilbert space. Suppose that $Q$ has closed range. By the partial inverse of $Q$, we
mean the bounded operator $N: H\To H$ such that
$Q\circ N=\pi_2$, $N\circ Q=\pi_1$ on ${\rm Dom\,}Q$,
where $\pi_1$, $\pi_2$ are the orthogonal projections in $H$ such that
${\rm Ran\,}\pi_1=({\rm Ker\,}Q)^\bot$, ${\rm Ran\,}\pi_2={\rm Ran\,}Q$.
In other words, for $u\in H$, let
$\pi_2u=Qv$, $v\in({\rm Ker\,}Q)^\bot\bigcap{\rm Dom\,}Q$.
Then, $Nu=v$.
\end{defn}

Set ${\rm Dom\,}\Box^{(q)}_b=\set{u\in L^2(X;\, \Lambda^{0,q}T^*(X));\,
\Box^{(q)}_bu\in L^2(X;\, \Lambda^{0,q}T^*(X))}$.

\begin{lem} \label{l:he-range}
We consider $\Box^{(q)}_b$ as an operator
\[\Box^{(q)}_b: L^2(X;\, \Lambda^{0,q}T^*(X))\supset{\rm Dom\,}\Box^{(q)}_b\To L^2(X;\, \Lambda^{0,q}T^*(X)).\]
If $Y(q)$ holds, then $\Box^{(q)}_b$ has closed range.
\end{lem}

\begin{proof}
Suppose $u_j\in{\rm Dom\,}\Box^{(q)}_b$ and
$\Box^{(q)}_bu_j=v_j\To v$ in $L^2(X;\, \Lambda^{0,q}T^*(X))$.
We have to show that there exists $u\in{\rm Dom\,}\Box^{(q)}_b$ such that
$\Box^{(q)}_bu=v$.
From Proposition~\ref{p:he-para}, we have
$\Box^{(q)}_bA=I-F_1$, $A\Box^{(q)}_b=I-F_2$,
where $F_j$, $j=1,2$, are smoothing operators. Now,
$A\Box^{(q)}_bu_j=(I-F_2)u_j\To Av$ in $L^2(X;\, \Lambda^{0,q}T^*(X))$.
Since $F_2$ is compact, there exists a subsequence
$u_{j_k}\To u$ in $L^2(X;\, \Lambda^{0,q}T^*(X))$, $k\To\infty$.
We have $(I-F_2)u=Av$ and
$\Box^{(q)}_bu_{j_k}\To\Box^{(q)}_bu$ in $H^{-2}(X;\, \Lambda^{0,q}T^*(X))$, $k\To\infty$.
Thus $\Box^{(q)}_bu=v$. Now $v\in L^2(X;\, \Lambda^{0,q}T^*(X))$, so $u\in{\rm Dom\,}\Box^{(q)}_b$.
We have proved the lemma.
\end{proof}

It follows that ${\rm Ran\,}\Box^{(q)}_b=({\rm Ker\,}\Box^{(q)}_b)^\bot$. Notice also that
$\Box^{(q)}_b$ is self-adjoint.
Now, we can prove the following classical result

\begin{prop} \label{p:BG-***}
We assume that $Y(q)$ holds. Then $\dim{\rm Ker\,}\Box^{(q)}_b<\infty$ and $\pi^{(q)}$ is a
smoothing operator . Let $N$ be the partial inverse. Then $N=A+F$ where $A$ is as in Proposition~\ref{p:he-para} and $F$ is a
smoothing operator. Let $N^*$ be the formal adjoint of $N$,
\[(N^*u\ |\ v)=(u\ |\ Nv),\ u\in C^\infty(X;\, \Lambda^{0,q}T^*(X)),\ v\in C^\infty(X;\, \Lambda^{0,q}T^*(X)).\]
Then, $N^*=N$ on $L^2(X;\, \Lambda^{0,q}T^*(X))$.
\end{prop}

\begin{proof}
From Proposition~\ref{p:he-para}, we have
$A\Box^{(q)}_b=I-F_1$, $\Box^{(q)}_bA=I-F_2$,
where $F_1$, $F_2$ are smoothing operators. Thus ${\rm Ker\,}\Box^{(q)}_b\subset{\rm Ker\,}(I-F_1)$.
Since $F_1$ is compact,  ${\rm Ker\,}(I-F_1)$ is finite dimensional and contained in
$C^\infty(X;\, \Lambda^{0,q}T^*(X))$. Thus
$\dim{\rm Ker\,}\Box^{(q)}_b<\infty$ and ${\rm Ker\,}\Box^{(q)}_b\subset C^\infty(X;\, \Lambda^{0,q}T^*(X))$.

Let $\set{\phi_1,\phi_2,\ldots,\phi_m}$ be an orthonormal basis for ${\rm Ker\,}\Box^{(q)}_b$.
The projection $\pi^{(q)}$ is given by
$\pi^{(q)} u=\sum^m_{j=1}(u\ |\ \phi_j)\phi_j$.
Thus $\pi^{(q)}$ is a smoothing operator. Notice that $I-\pi^{(q)}$ is the orthogonal projection onto ${\rm Ran\,}\Box^{(q)}_b$ since
$\Box^{(q)}_b$ is formally self-adjoint with closed range.

For $u\in C^\infty(X;\, \Lambda^{0,q}T^*(X))$, we have
\begin{align*}
(N-A)u &= (A\Box^{(q)}_b+F_1)Nu-Au  \\
       &= A(I-\pi^{(q)})u+F_1Nu-Au      \\
       &= -A\pi^{(q)} u+F_1N(\Box^{(q)}_b A+F_2)u    \\
       &= -A\pi^{(q)} u+F_1(I-\pi^{(q)})Au+F_1NF_2u.
\end{align*}
Here
$-A\pi^{(q)}$, $F_1(1-\pi^{(q)})A: H^s(X;\, \Lambda^{0,q}T^*(X))\To H^{s+m}(X;\, \Lambda^{0,q}T^*(X))$
$\forall s\in\Real$ and $m\geq 0$, so $-A\pi^{(q)}$, $F_1(1-\pi^{(q)})A$ are smoothing operators. Since
\begin{align*}
F_1NF_2&:\mathscr E'(X;\, \Lambda^{0,q}T^*(X))\To C^\infty(X;\, \Lambda^{0,q}T^*(X)) \\
&\To L^2(X;\, \Lambda^{0,q}T^*(X))\To C^\infty(X;\, \Lambda^{0,q}T^*(X)),
\end{align*}
we have that $F_1NF_2$ is a smoothing operator. Thus $N=A+F$, $F$ is a smoothing operator.

Since $N\pi^{(q)}=\pi^{(q)} N=0=N^*\pi^{(q)}=\pi^{(q)} N^*=0$,
we have $N^*=(N\Box^{(q)}_b+\pi^{(q)})N^*=N\Box^{(q)}_bN^*=N$.
The proposition follows.
\end{proof}

Now, we assume that $Y(q)$ fails but that $Y(q-1)$, $Y(q+1)$ hold.
In view of Lemma~\ref{l:he-range}, we see that $\Box^{(q-1)}_b$ and $\Box^{(q+1)}_b$ have closed range.
We write $\dbar^{(q)}_b$ to denote the map:
$\dbar_b: C^\infty(X;\, \Lambda^{0,q}T^*(X))\To C^\infty(X;\, \Lambda^{0,q+1}T^*(X))$.
Let $\dbar^{(q),*}_b$ denote the formal adjoint of $\dbar_b$. We have
$\dbar^{(q),*}_b: C^\infty(X;\, \Lambda^{0,q+1}T^*(X))\To C^\infty(X;\, \Lambda^{0,q}T^*(X))$.
Let $N^{(q+1)}_b$ and $N^{(q-1)}_b$ be the partial inverses of $\Box^{(q+1)}_b$ and $\Box^{(q-1)}_b$
respectively. From Proposition~\ref{p:BG-***}, we have
\[(N^{(q+1)}_b)^*=N^{(q+1)}_b,\ \ (N^{(q-1)}_b)^*=N^{(q-1)}_b,\]
where $(N^{(q+1)}_b)^*$ and $(N^{(q-1)}_b)^*$ are the formal adjoints of $N^{(q+1)}_b$ and $N^{(q-1)}_b$ respectively.
Put
\begin{equation} \label{e:he-n}
M=\dbar^{(q),*}_b(N^{(q+1)}_b)^2\dbar^{(q)}_b+\dbar^{(q-1)}_b(N^{(q-1)}_b)^2\dbar^{(q-1),*}_b
\end{equation}
and
\begin{equation} \label{e:he-pi}
\pi=I-(\dbar^{(q),*}_b N^{(q+1)}_b\dbar^{(q)}_b+\dbar^{(q-1)}_b N^{(q-1)}_b\dbar^{(q-1),*}_b).
\end{equation}
In view of Proposition~\ref{p:BG-***}, we see that $M$ is well-defined as a continuous operator:
$M: H^s(X;\, \Lambda^{0,q}T^*(X))\To H^{s}(X;\, \Lambda^{0,q}T^*(X))$
and $\pi$ is well-defined as a continuous operator:
$\pi: H^s(X;\, \Lambda^{0,q}T^*(X))\To H^{s-1}(X;\, \Lambda^{0,q}T^*(X))$,
for all $s\in\Real$.

Let $\pi^*$ and $M^*$ be the formal adjoints of $\pi$ and $M$ respectively. We have the following

\begin{lem} \label{l:he-szego}
If we consider $\pi$ and $M$ as operators:
\[\pi, M:\mathscr D'(X;\, \Lambda^{0,q}T^{*}(X))\To\mathscr D'(X;\, \Lambda^{0,q}T^{*}(X)),\]
then
\begin{align}
&\pi^*=\pi,\ M^*=M, \label{e:he-sj1} \\
&\Box^{(q)}_b\pi=0=\pi\Box^{(q)}_b, \label{e:he-sj2}\\
&\pi+\Box^{(q)}_bM=I=\pi+M\Box^{(q)}_b, \label{e:he-sj3}\\
&\pi M=0=M\pi, \label{e:he-sj4} \\
&\pi^2=\pi. \label{e:he-sj1morehome}
\end{align}

\end{lem}

\begin{proof}
From (\ref{e:he-n}) and (\ref{e:he-pi}), we get (\ref{e:he-sj1}).

For $u\in C^\infty(X;\, \Lambda^{0,q+1}T^*(X))$, we have
\begin{align*}
0&=(\Box^{(q+1)}_b\pi^{(q+1)}_bu\ |\ \pi^{(q+1)}_bu) \\
 &=(\dbar^{(q+1)}_b\pi^{(q+1)}_bu\ |\ \dbar^{(q+1)}_b\pi^{(q+1)}_bu)+(\dbar^{(q),*}_b\pi^{(q+1)}_bu\ |\ \dbar^{(q),*}_b\pi^{(q+1)}_bu).
\end{align*}
Thus,
\begin{equation} \label{e:07072717}
\dbar^{(q+1)}_b\pi^{(q+1)}_b=0,\ \dbar^{(q),*}_b\pi^{(q+1)}_b=0.
\end{equation}
Hence, by taking the formal adjoints
\begin{equation} \label{e:0707271702}
\pi^{(q+1)}_b\dbar^{(q+1), *}_b=0,\ \pi^{(q+1)}_b\dbar^{(q)}_b=0.
\end{equation}
Similarly,
\begin{equation} \label{e:0707271701}
\dbar^{(q-1)}_b\pi^{(q-1)}_b=0,\ \pi^{(q-1)}_b\dbar^{(q-1),*}_b=0.
\end{equation}
Note that
\begin{equation} \label{e:0707271703}
\dbar^{(q),*}_b\Box^{(q+1)}_b=\Box^{(q)}_b\dbar^{(q),*}_b,\ \dbar^{(q-1)}_b\Box^{(q-1)}_b=\Box^{(q)}_b\dbar^{(q-1)}_b.
\end{equation}
Now,
\begin{align*}
\lefteqn{\Box^{(q)}_b(\dbar^{(q),*}_b N^{(q+1)}_b\dbar^{(q)}_b+\dbar^{(q-1)}_b N^{(q-1)}_b\dbar^{(q-1),*}_b)}  \\
&\quad= \dbar^{(q),*}_b\Box^{(q+1)}_b N^{(q+1)}_b\dbar^{(q)}_b+\dbar^{(q-1)}_b\Box^{(q-1)}_b N^{(q-1)}_b
        \dbar^{(q-1),*}_b   \\
&\quad= \dbar^{(q),*}_b(I-\pi^{(q+1)}_b)\dbar^{(q)}_b+\dbar^{(q-1)}_b(I-\pi^{(q-1)}_b)\dbar^{(q-1),*}_b   \\
&\quad= \Box^{(q)}_b.
\end{align*}
Here we used (\ref{e:07072717}), (\ref{e:0707271701}) and (\ref{e:0707271703}).
Hence, $\Box^{(q)}_b\pi=0$.
We have that $\pi\Box^{(q)}_b \\ =(\Box^{(q)}_b\pi)^*=0$,
where $(\Box^{(q)}_b\pi)^*$ is the formal adjoint of $\Box^{(q)}_b\pi$.
We get (\ref{e:he-sj2}).

Now,
\begin{align}
\Box^{(q)}_b M  &=\dbar^{(q),*}_b\Box^{(q+1)}_b(N^{(q+1)}_b)^2\dbar^{(q)}_b+
                  \dbar^{(q-1)}_b\Box^{(q-1)}_b(N^{(q-1)}_b)^2\dbar^{(q-1),*}_b  \nonumber  \\
              &=  \dbar^{(q),*}_b(I-\pi^{(q+1)}_b)N^{(q+1)}_b\dbar^{(q)}_b+
                  \dbar^{(q-1)}_b(I-\pi^{(q-1)}_b)N^{(q-1)}_b\dbar^{(q-1),*}_b  \nonumber  \\
              &=  \dbar^{(q),*}_b N^{(q+1)}_b\dbar^{(q)}_b+
                  \dbar^{(q-1)}_b N^{(q-1)}_b\dbar^{(q-1),*}_b           \nonumber  \\
              &=  I-\pi.
\end{align}
Here we used (\ref{e:07072717}), (\ref{e:0707271701}) and (\ref{e:0707271703}). Thus,
$\Box^{(q)}_b M+\pi=I$.
We have $\pi+M\Box^{(q)}_b=(\Box^{(q)}_b M+\pi)^*=I$,
where $(\Box^{(q)}_b M+\pi)^*$ is the formal adjoint of $\Box^{(q)}_b M+\pi$.
We get (\ref{e:he-sj3}).

Now,
\begin{align*}
M(I-\pi)
   &=  M(\dbar^{(q),*}_b N^{(q+1)}_b\dbar^{(q)}_b+\dbar^{(q-1)}_b N^{(q-1)}_b\dbar^{(q-1),*}_b)    \\
   &=  \dbar^{(q),*}_b(N^{(q+1)}_b)^2\dbar^{(q)}_b\dbar^{(q),*}_b N^{(q+1)}_b\dbar^{(q)}_b      \\
   &\quad +\dbar^{(q-1)}_b(N^{(q-1)}_b)^2\dbar^{(q-1),*}_b\dbar^{(q-1)}_b N^{(q-1)}_b\dbar^{(q-1),*}_b.
\end{align*}
From (\ref{e:07072717}), (\ref{e:0707271702}) and (\ref{e:0707271703}), we have
\begin{align*}
\dbar^{(q)}_b\dbar^{(q),*}_b N^{(q+1)}_b
   &=  (I-\pi^{(q+1)}_b)\dbar^{(q)}_b\dbar^{(q),*}_b N^{(q+1)}_b      \\
   &=  N^{(q+1)}_b\Box^{(q+1)}_b\dbar^{(q)}_b\dbar^{(q),*}_b N^{(q+1)}_b    \\
   &=  N^{(q+1)}_b\dbar^{(q)}_b\dbar^{(q),*}_b\Box^{(q+1)}_b N^{(q+1)}_b    \\
   &=  N^{(q+1)}_b\dbar^{(q)}_b\dbar^{(q),*}_b(I-\pi^{(q+1)}_b)      \\
   &=  N^{(q+1)}_b\dbar^{(q)}_b\dbar^{(q),*}_b.
\end{align*}
Similarly, we have
$\dbar^{(q-1),*}_b\dbar^{(q-1)}_b N^{(q-1)}_b=N^{(q-1)}_b\dbar^{(q-1),*}_b\dbar^{(q-1)}_b$.
Hence,
\begin{align*}
M(I-\pi)
   &=\dbar^{(q),*}_b(N^{(q+1)}_b)^2 N^{(q+1)}_b\dbar^{(q)}_b\dbar^{(q),*}_b\dbar^{(q)}_b      \\
   &\quad+\dbar^{(q-1)}_b(N^{(q-1)}_b)^2 N^{(q-1)}_b\dbar^{(q-1),*}_b\dbar^{(q-1)}_b\dbar^{(q-1),*}_b  \\
   &=\dbar^{(q),*}_b(N^{(q+1)}_b)^2N^{(q+1)}_b\Box^{(q+1)}_b\dbar^{(q)}_b           \\
   &\quad +\dbar^{(q-1)}_b(N^{(q-1)}_b)^2N^{(q-1)}_b\Box^{(q-1)}_b\dbar^{(q-1),*}_b  \\
   &=\dbar^{(q),*}_b(N^{(q+1)}_b)^2(I-\pi^{(q+1)}_b)\dbar^{(q)}_b        \\
   &\quad +\dbar^{(q-1)}_b(N^{(q-1)}_b)^2(I-\pi^{(q-1)}_b)\dbar^{(q-1),*}_b      \\
   &=\dbar^{(q),*}_b(N^{(q+1)}_b)^2\dbar^{(q)}_b+
      \dbar^{(q-1)}_b(N^{(q-1)}_b)^2\dbar^{(q-1),*}_b       \\
   &=M.
\end{align*}
Here we used (\ref{e:0707271702}) and (\ref{e:0707271701}).
Thus, $M\pi=0$. We have
$\pi M=(M\pi)^*=0$,
where $(M\pi)^*$ is the formal adjoint of $M\pi$.
We get (\ref{e:he-sj4}).

Finally,
$\pi=(\Box^{(q)}_b M+\pi)\pi=\pi^2$.
We get (\ref{e:he-sj1morehome}). The lemma follows.
\end{proof}

\begin{lem} \label{l:he-pi}
If we restrict $\pi$ to $L^2(X;\, \Lambda^{0,q}T^*(X))$, then $\pi=\pi^{(q)}$. That is, $\pi$ is the orthogonal projection onto
${\rm Ker\,}\Box^{(q)}_b$.
Thus, $\pi$ is well-defined as a continuous operator:
$\pi :L^2(X;\, \Lambda^{0,q}T^*(X))\To L^2(X;\, \Lambda^{0,q}T^*(X))$.
\end{lem}

\begin{proof}
From (\ref{e:he-sj2}), we get
${\rm Ran\, }\pi\subset{\rm Ker\, }\Box^{(q)}_b$ in the space of distributions. From (\ref{e:he-sj3}), we get
$\pi u=u$, when $u\in{\rm Ker\, }\Box^{(q)}_b$, so ${\rm Ran\, }\pi={\rm Ker\, }\Box^{(q)}_b$ and
$\pi^2=\pi=\pi^*\pi=\pi^*$.
For $\varphi$, $\phi\in C^\infty(X;\, \Lambda^{0,q}T^*(X))$, we get
$((1-\pi)\varphi\ |\ \pi\phi)=0$
so ${\rm Ran\, }(I-\pi)\perp{\rm Ran\, }\pi$ and $\varphi=(I-\pi)\varphi+\pi\varphi$ is the orthogonal
decomposition. It follows that $\pi$ restricted to $L^2(X;\, \Lambda^{0,q}T^*(X))$ is the orthogonal projection
onto ${\rm Ker\,}\Box^{(q)}_b$.
\end{proof}

\begin{lem} \label{l:he-szegomore}
If we consider $\Box^{(q)}_b$ as an unbounded operator
\[\Box^{(q)}_b: L^2(X;\, \Lambda^{0,q}T^*(X))\supset{\rm Dom\,}\Box^{(q)}_b\To L^2(X;\, \Lambda^{0,q}T^*(X)),\]
then $\Box^{(q)}_b$ has closed range and $M: L^2(X;\, \Lambda^{0,q}T^*(X))\To{\rm Dom\,}\Box^{(q)}_b$
is the partial inverse.
\end{lem}

\begin{proof}
From (\ref{e:he-sj3}) and Lemma~\ref{l:he-pi}, we see that
$M: L^2(X;\, \Lambda^{0,q}T^*(X))\To{\rm Dom\,}\Box^{(q)}_b$
and ${\rm Ran\,}\Box^{(q)}_b\supset{\rm Ran\, }(I-\pi)$. If
$\Box^{(q)}_b u=v$, $u,v\in L^2(X;\, \Lambda^{0,q}T^*(X))$,
then $(I-\pi)v=(I-\pi)\Box^{(q)}_b u=v$ since $\pi\Box^{(q)}_b=\Box^{(q)}_b\pi=0$.
Hence ${\rm Ran\,}\Box^{(q)}_b\subset{\rm Ran\,}(I-\pi)$
so $\Box^{(q)}_b$ has closed range.

From (\ref{e:he-sj4}), we know that $M\pi=\pi M=0$. Thus, $M$ is the partial inverse.
\end{proof}

From Lemma~\ref{l:he-pi} and Lemma~\ref{l:he-szegomore} we get the following classical
result

\begin{prop} \label{p:he-bealsgreiner}
We assume that $Y(q)$ fails but that $Y(q-1)$ and $Y(q+1)$ hold. Then $\Box^{(q)}_b$ has closed range. Let
$M$ and $\pi$ be as in (\ref{e:he-n}) and (\ref{e:he-pi}) respectively. Then $M$ is
the partial inverse of\, $\Box^{(q)}_b$ and $\pi=\pi^{(q)}$.
\end{prop}


\section{The Szeg\"{o} Projection}

In this section, we assume that $Y(q)$ fails. From
Proposition~\ref{p:he-yofnotq}, we know that, for every local coordinate patch $X_{j}$,
there exist
\[G_j\in L^{-1}_{\frac{1}{2},\frac{1}{2}}(X_j;\, \Lambda^{0,q}T^*(X_j),\Lambda^{0,q}T^*(X_j)),\ \
S_j\in L^0_{\frac{1}{2},\frac{1}{2}}(X_j;\, \Lambda^{0,q}T^*(X_j),\Lambda^{0,q}T^*(X_j))\]
such that
\begin{equation} \label{e:s-1}
\left\{\begin{array}{l}
S_j+\Box^{(q)}_bG_j\equiv I   \\  \Box^{(q)}_b S_j\equiv0
\end{array}\right.
\end{equation}
in the space $\mathscr D'(X_j\times X_j;\, \mathscr L(\Lambda^{0,q}T^*(X_j),\Lambda^{0,q}T^*(X_j)))$. Furthermore,
the distribution kernel $K_{S_j}$ of $S_j$ is of the form
\begin{equation} \label{e:0707292058}
K_{S_j}(x, y)=\frac{1}{(2\pi)^{2n-1}}\int e^{i(\psi(\infty,x,\eta)-\seq{y,\eta})}a(\infty,x,\eta)d\eta,
\end{equation}
where
$\psi(\infty, x, \eta)$, $a(\infty,x,\eta)$
are as in Proposition~\ref{p:c-basislimit} and (\ref{e:0707292108}). From now on, we assume that $S_j$ and $G_j$ are
properly supported operators.

We assume that $X=\bigcup^k_{j=1}X_{j}$. Let $\chi_j$ be a
$C^\infty$ partition of unity subordinate to $\set{X_{j}}$. From (\ref{e:s-1}), we have
\begin{equation} \label{e:0709231506}
\left\{\begin{array}{l}
S_j\chi_j+\Box^{(q)}_b G_j\chi_j\equiv\chi_j  \\ \Box^{(q)}_b S_j\chi_j\equiv0
\end{array}\right.
\end{equation}
in the space $\mathscr D'(X_j\times X_j;\, \mathscr L(\Lambda^{0,q}T^*(X_j),\Lambda^{0,q}T^*(X_j)))$. Thus,
\begin{equation} \label{e:0709231507}
\left\{\begin{array}{l}
S+\Box^{(q)}_b G\equiv I  \\ \Box^{(q)}_b S\equiv0
\end{array}\right.
\end{equation}
in $\mathscr D'(X\times X;\, \mathscr L(\Lambda^{0,q}T^*(X),\Lambda^{0,q}T^*(X)))$, where
\[S, G:\mathscr D'(X;\, \Lambda^{0,q}T^*(X))\To\mathscr D'(X;\, \Lambda^{0,q}T^*(X)),\]
\begin{equation} \label{e:0707302028}
\left\{\begin{split}
& Su=\sum^k_{j=1}S_{j}(\chi_ju), u\in\mathscr D'(X;\, \Lambda^{0,q}T^*(X))   \\
& Gu=\sum^k_{j=1}G_{j}(\chi_ju), u\in\mathscr D'(X;\, \Lambda^{0,q}T^*(X))
\end{split}\right..
\end{equation}
From (\ref{e:0709231507}), we can repeat the method of Beals and Greiner (see page $173$-page $176$ of~\cite{BG88}) with minor change to show that
$S$ is the Szeg\"{o} projection (up to some smoothing operators) if $\Box^{(q)}_b$ has closed range.

Let $S^*$, $G^*:\mathscr D'(X;\, \Lambda^{0,q}T^*(X))\To\mathscr D'(X;\, \Lambda^{0,q}T^*(X))$
be the formal adjoints of $S$ and $G$ respectively. As in Lemma~\ref{l:he-para},
we see that $S^*$ and $G^*$ are well-defined as continuous operators
\begin{equation} \label{e:s-3}\left\{\begin{split}
& S^*:H^s(X;\, \Lambda^{0,q}T^*(X))\To H^s(X;\, \Lambda^{0,q}T^*(X))   \\
& G^*: H^s(X;\, \Lambda^{0,q}T^*(X))\To H^{s+1}(X;\, \Lambda^{0,q}T^*(X))
\end{split}\right.,
\end{equation}
for all $s\in\Real$. We have the following

\begin{lem} \label{l:s-1}
Let $S$ be as in (\ref{e:0709231507}), (\ref{e:0707302028}). We have $S\equiv S^*S$.
It follows that $S\equiv S^*$ and $S^2\equiv S$.
\end{lem}

\begin{proof}
From (\ref{e:0709231507}), it follows that $S^*+G^*\Box^{(q)}_b\equiv I$.
We have  $S\equiv(S^*+G^*\Box^{(q)}_b)\circ S\equiv S^*S+G^*\Box^{(q)}_b S\equiv S^*S$.
The lemma follows.
\end{proof}

Let
\begin{equation} \label{e:0707292122}
H=(I-S)\circ G.
\end{equation}
$H$ is well-defined as a continuous operator
\[H: H^s(X;\, \Lambda^{0,q}T^*(X))\To H^{s+1}(X;\, \Lambda^{0,q}T^*(X))\]
for all $s\in\Real$. The formal adjoint $H^*$ is well-defined as a continuous operator:
$H^*: H^s(X;\, \Lambda^{0,q}T^*(X))\To H^{s+1}(X;\, \Lambda^{0,q}T^*(X))$, for all $s\in\Real$.

\begin{lem} \label{l:s-2}
Let $S$ and $H$ be as in (\ref{e:0709231507}), (\ref{e:0707302028}) and (\ref{e:0707292122}). Then
\begin{align}
&SH\equiv 0, \label{e:0708252235} \\
&S+\Box^{(q)}_b H\equiv I. \label{e:0708252236}
\end{align}
\end{lem}

\begin{proof}
We have $SH\equiv S(I-S)G\equiv (S-S^2)G\equiv0$
since $S^2\equiv S$, where $G$ is as in (\ref{e:0709231507}).
From (\ref{e:0709231507}), it follows that $S+\Box^{(q)}_b H=S+\Box^{(q)}_b (I-S)G\equiv  I-\Box^{(q)}_b SG\equiv I$.
The lemma follows.
\end{proof}

\begin{lem} \label{l:s-3}
Let $H$ be as in (\ref{e:0707292122}). Then $H\equiv H^*$.
\end{lem}

\begin{proof}
Taking the adjoint in (\ref{e:0708252236}), we get
$S^*+H^*\Box^{(q)}_b\equiv I$.
Hence
\[H\equiv(S^*+H^*\Box^{(q)}_b)H\equiv S^*H+H^*\Box^{(q)}_b H.\]
From Lemma~\ref{l:s-1}, Lemma~\ref{l:s-2}, we have
$S^*H\equiv SH\equiv 0$.
Hence $H\equiv H^*\Box^{(q)}_b H\equiv H^*$.
\end{proof}

Summing up, we get the following

\begin{prop} \label{p:s-szegoglobal}
We assume that $Y(q)$ fails. Let
\[
\left\{\begin{split}
& S: \mathscr D'(X;\, \Lambda^{0,q}T^*(X))\To
  \mathscr D'(X;\, \Lambda^{0,q}T^*(X))   \\
&  H: \mathscr D'(X;\, \Lambda^{0,q}T^*(X))\To
  \mathscr D'(X;\, \Lambda^{0,q}T^*(X))
\end{split}\right.
\]
be as in (\ref{e:0707302028}) and (\ref{e:0707292122}). Then, $S$ and $H$ are well-defined as continuous operators
\begin{align}
&S:H^s(X;\, \Lambda^{0,q}T^*(X))\To H^{s}(X;\, \Lambda^{0,q}T^*(X)), \label{e:0709231529} \\
&H: H^s(X;\, \Lambda^{0,q}T^*(X))\To H^{s+1}(X;\, \Lambda^{0,q}T^*(X)), \label{e:0709231530}
\end{align}
for all $s\in\Real$. Moreover, we have
\begin{align}
& H\Box^{(q)}_b+S\equiv S+\Box^{(q)}_b H\equiv I,                 \label{e:s-1moreg}  \\
& \Box^{(q)}_b S\equiv S\Box^{(q)}_b\equiv0,                     \label{e:s-2moreg}  \\
& S\equiv S^*\equiv S^2,                                     \label{e:s-3moreg}  \\
& SH\equiv HS\equiv0,                                       \label{e:s-4moreg}  \\
& H\equiv H^*.                                            \label{e:s-5moreg}
\end{align}
\end{prop}

\begin{rem} \label{r:s-szego}
If $S'$, $H':\mathscr D'(X;\, \Lambda^{0,q}T^*(X))\To\mathscr D'(X;\, \Lambda^{0,q}T^*(X))$
satisfy (\ref{e:0709231529})-(\ref{e:s-5moreg}), then
$S'\equiv(H\Box^{(q)}_b+S)S'\equiv SS'\equiv S(\Box^{(q)}_b H'+S')\equiv S$
and
\[H'\equiv(H\Box^{(q)}_b+S)H'\equiv(H\Box^{(q)}_b+S')H'\equiv H\Box^{(q)}_b H'\equiv H(\Box^{(q)}_b H'+S')\equiv H.\]
Thus,(\ref{e:0709231529})-(\ref{e:s-5moreg}) determine $S$ and $H$ uniquely up to smoothing operators.
\end{rem}

Now we can prove the following

\begin{prop} \label{p:s-sz}
We assume that $Y(q)$ fails. Suppose
$\Box^{(q)}_b$ has closed range. Let $N$ be the partial inverse of\, $\Box^{(q)}_b$. Then
$N=H+F$ and $\pi^{(q)}=S+K$, where $H$, $S$ are as in Proposition~\ref{p:s-szegoglobal}, $F$, $K$ are smoothing
operators.
\end{prop}

\begin{proof}
We may replace $S$ by $I-\Box^{(q)}_b H$ and we have $\Box^{(q)}_b H+S=I=H^*\Box^{(q)}_b+S^*$.
Now,
\begin{equation} \label{e:s-in}
\pi^{(q)}=\pi^{(q)}(\Box^{(q)}_b H+S)=\pi^{(q)} S,
\end{equation}
hence
\begin{equation} \label{e:0707302126}
(\pi^{(q)})^*=S^*(\pi^{(q)})^*=\pi^{(q)}=S^*\pi^{(q)}.
\end{equation}
Similarly,
\begin{equation} \label{e:s-ve}
S=(N\Box^{(q)}_b+\pi^{(q)})S=\pi^{(q)} S+NF_1,
\end{equation}
where $F_1$ is a smoothing operator. From (\ref{e:s-in}) and (\ref{e:s-ve}), we have
\begin{equation} \label{e:s-inor}
S-\pi^{(q)}=S-\pi^{(q)} S=NF_1.
\end{equation}
Hence $(S^*-\pi^{(q)})(S-\pi^{(q)})=F^*_1N^2F_1$. On the other hand,
\begin{align*}
(S^*-\pi^{(q)})(S-\pi^{(q)}) &= S^*S-S^*\pi^{(q)}-\pi^{(q)} S+(\pi^{(q)})^2   \\
                 &= S^*S-\pi^{(q)}     \\
                 &= S-\pi^{(q)}+F_2,
\end{align*}
where $F_2$ is a smoothing operator. Here we used (\ref{e:s-in}) and (\ref{e:0707302126}).
Now,
\begin{align*}
F^*_1N^2F_1: &\ \mathscr D'(X;\, \Lambda^{0,q}T^*(X))\To C^\infty(X;\, \Lambda^{0,q}T^*(X))   \\
             &\To L^2(X;\, \Lambda^{0,q}T^*(X))\To C^\infty(X;\, \Lambda^{0,q}T^*(X)).
\end{align*}
Hence $F^*_1N^2F_1$ is smoothing. Thus $S-\pi^{(q)}$ is smoothing.

We have,
\begin{align*}
N-H  &=  N(\Box^{(q)}_b H+S)-H    \\
     &=  (I-\pi^{(q)})H+NS-H         \\
     &=  NS-\pi^{(q)} H              \\
     &=  N(S-\pi^{(q)})+F_3          \\
     &=  NF_4+F_3
\end{align*}
where $F_4$ and $F_3$ are smoothing operators. Now,
\begin{align*}
N-H^*  &= N^*-H^*          \\
         &= F^*_4N+F^*_3           \\
         &= F^*_4(NF_4+F_3+H)+F^*_3.
\end{align*}
Note that
\begin{align*}
F^*_4NF_4: &\ \mathscr D'(X;\, \Lambda^{0,q}T^*(X))\To C^\infty(X;\, \Lambda^{0,q}T^*(X))   \\
           &\To L^2(X;\, \Lambda^{0,q}T^*(X))\To C^\infty(X;\, \Lambda^{0,q}T^*(X)).
\end{align*}
and $F^*_4H: H^s(X\ ;\Lambda^{0,q}T^*(X))\To H^{s+m}(X\ ;\Lambda^{0,q}T^*(X))$
for all $s\in\Real$ and $m\geq 0$. Hence $N-H^*$ is smoothing and so is $(N-H^*)^*=N-H$.
\end{proof}

From Proposition~\ref{p:s-szegoglobal} and Proposition~\ref{p:s-sz},
we obtain the following

\begin{thm} \label{t:s-thm1}
We recall that we work with the assumption that $Y(q)$ fails. Let $(n_-, n_+)$, $n_-+n_+=n-1$, be the signature of the Levi form $L$.
Suppose $\Box^{(q)}_b$ has closed range. Then for every local
coordinate patch $U\subset X$, the distribution kernel of\, $\pi^{(q)}$ on
$U\times U$ is of the form
\begin{equation} \label{e:0707311443}
K_{\pi^{(q)}}(x,y)\equiv\frac{1}{(2\pi)^{2n-1}}\int e^{i(\psi(\infty,x,\eta)-\seq{y,\eta})}
  a(\infty,x,\eta)d\eta,
\end{equation}
$a(\infty, x, \eta)\in S^0_{1,0}(T^*(U)\ ;\mathscr L(\Lambda^{0,q}T^*(U),\Lambda^{0,q}T^*(U)))$,
\[a(\infty, x, \eta)\sim\sum^{\infty}_0a_j(\infty, x, \eta)\]
in the H\"{o}rmander symbol space $S^0_{1,0}(T^*(U);\, \mathscr L(\Lambda^{0,q}T^*(U),\Lambda^{0,q}T^*(U)))$,
where
$a_j(\infty, x, \eta)\in C^\infty(T^*(U);\, \mathscr L(\Lambda^{0,q}T^*(U),\Lambda^{0,q}T^*(U)))$, $j=0,1,\ldots$,
$a_j(\infty, x, \lambda\eta)=\lambda^{-j}a_j(\infty, x, \eta)$, $\lambda\geq1$, $\abs{\eta}\geq1$, $j=0,1,\ldots$. Here
$\psi(\infty, x, \eta)$ is as in Proposition~\ref{p:c-basislimit} and (\ref{e:nevergiveup1}). We recall that
$\psi(\infty,x,\eta)\in C^{\infty}(T^*(U))$,
$\psi(\infty,x,\lambda\eta)=\lambda\psi(\infty,x,\eta)$, $\lambda>0$,
${\rm Im\,}\psi(\infty,x,\eta)\asymp\abs{\eta}({\rm dist\,}((x,\frac{\eta}{\abs{\eta}}),\Sigma))^2$ and
\begin{equation} \label{e:0709192125}
\psi(\infty, x, \eta)=-\ol\psi(\infty, x, -\eta).
\end{equation}
Moreover, for all $j=0,1,\ldots$,
\begin{equation} \label{e:s-important}
\left\{ \begin{array}{ll}
a_j(\infty, x, \eta)=0\ \ \mbox{in a conic neighborhood of}\ \ \Sigma^+, & \textup{if}\ \ q=n_-,\ n_-\neq n_+,    \\
a_j(\infty, x, \eta)=0\ \ \mbox{in a conic neighborhood of}\ \ \Sigma^-, & \textup{if}\ \ q=n_+,\ n_-\neq n_+.
\end{array}\right.
\end{equation}
\end{thm}

In the rest of this section, we will study the singularities of the distribution kernel of the Szeg\"{o} projection. We need

\begin{defn} \label{d:0709181445}
Let $M$ be a real paracompact $C^\infty$ manifold and let $\Lambda$ be a $C^\infty$ closed submanifold of $M$.
Let $U$ be an open set in $M$, $U\bigcap\Lambda\neq\emptyset$.
We let $C^\infty_\Lambda(U)$ denote the set of equivalence classes of $f\in C^\infty(U)$ under the equivalence
relation
\[f\equiv g\ \ \mbox{in the space}\ C^\infty_\Lambda(U)\]
if for every $z_0\in\Lambda\bigcap U$, there exists a neighborhood $W\subset U$ of $z_0$ such that
$f=g+h$ on $W$,
where $h\in C^\infty(W)$ and $h$ vanishes to infinite order on $\Lambda\bigcap W$.
\end{defn}

In view of Proposition~\ref{p:c-basislimit}, we see that $\psi(\infty,x,\eta)$ has a
uniquely determined Taylor expansion at each point of $\Sigma$. Thus, we can define $\psi(\infty,x,\eta)$
as an element in $C^\infty_\Sigma(T^*(X))$. We also write $\psi(\infty,x,\eta)$ for the
equivalence class of $\psi(\infty, x, \eta)$ in the space $C^\infty_\Sigma(T^*(X))$.

Let $M$ be a real paracompact $C^\infty$ manifold and let $\Lambda$ be a $C^\infty$ closed submanifold of $M$. If
$x_0\in\Lambda$, we let
$A(\Lambda,n,x_0)$ be the set
\begin{align} \label{e:07091181757}
A(\Lambda,n,x_0)&=\{(U, f_1,\ldots,f_n);\, U\ \mbox{is an open neighborhood of}\ x_0,\nonumber \\
&\quad f_1,\ldots,f_n\in C^\infty_\Lambda(U), f_j|_{\Lambda\bigcap U}=0,\ j=1,\ldots,n,
df_1,\ldots,df_n\nonumber \\
&\quad\mbox{are linearly independent over $\Complex$ at each point of}\ U\}.
\end{align}

\begin{defn} \label{d:0709181523}
If $x_0\in\Lambda$, we let $A_{x_0}(\Lambda, n, x_0)$ denote the set of equivalence classes of $A(\Lambda,n, x_0)$
under the equivalence relation
\[\Gamma_1=(U, f_1,\ldots,f_n)\sim\Gamma_2=(V, g_1,\ldots,g_n),\ \ \Gamma_1, \Gamma_2\in A(\Lambda,n, x_0),\]
if there exists an open set
$W\subset U\bigcap V$ of $x_0$ such that
$g_j\equiv\sum^n_{k=1}a_{j,k}f_k$ in the space $C^\infty_\Lambda(W)$, $j=1,\ldots,n$,
where $a_{j,k}\in C^\infty_\Lambda(W)$, $j$, $k=1,\ldots,n$, and
$\left(a_{j,k}\right)^n_{j,k=1}$ is invertible.

If $(U, f_1,\ldots,f_n)\in A(\Lambda,n, x_0)$, we write
$(U, f_1,\ldots,f_n)_{x_0}$ for the equivalence class of $(U, f_1,\ldots,f_n)$ in the set $A_{x_0}(\Lambda, n, x_0)$,
which is called the germ of $(U, f_1,\ldots,f_n)$ at $x_0$.
\end{defn}

\begin{defn} \label{d:0709181850}
Let $M$ be a real paracompact $C^\infty$ manifold and let $\Lambda$ be a $C^\infty$ closed submanifold of $M$.
A formal manifold $\Omega$ of codimensin $k$ at $\Lambda$ associated to $M$ is given by:
\[\begin{split}
&\mbox{For each point of $x\in\Lambda$, we assign a germ $\Gamma_x\in A_x(\Lambda,k,x)$ in such a way that}   \\
&\mbox{for every point $x_0\in\Lambda$ has an open neighborhood $U$ such that there exist} \\
&\mbox{$f_1,\ldots,f_k\in C^\infty_\Lambda(U)$, $f_j|_{\Lambda\bigcap U}=0$, $j=1,\ldots,k$,
$df_1,\ldots,df_k$ are linearly} \\
&\mbox{independent over $\Complex$ at each point of $U$, having the
following property:} \\
&\mbox{whatever $x\in U$, the germ $(U,f_1,\ldots,f_k)_x$ is equal to $\Gamma_x$}.
\end{split}\]
Formally, we write $\Omega=\set{\Gamma_x;\, x\in\Lambda}$.
If the codimension of $\Omega$ is $1$, we call $\Omega$ a formal hypersurface at $\Lambda$.

Let $\Omega=\set{\Gamma_x;\, x\in\Lambda}$ and $\Omega_1=\set{\Td\Gamma_x;\, x\in\Lambda}$ be two formal manifolds
at $\Lambda$. If $\Gamma_x=\Td\Gamma_x$, for all $x\in\Lambda$, we write $\Omega=\Omega_1$ at $\Lambda$.
\end{defn}

\begin{defn} \label{d:0709182210}
Let $\Omega=\set{\Gamma_x;\, x\in\Lambda}$ be a formal manifold of codimensin $k$ at $\Lambda$ associated to $M$,
where $\Lambda$ and $M$ are as above.
The tangent space of $\Omega$ at $x_0\in\Lambda$ is given by:
\[\mbox{the tangent space of}\ \Omega\ \mbox{at}\ x_0=\set{u\in\Complex T_{x_0}(M);\, \seq{df_j(x_0), u}=0,\ j=1,\ldots,k},\]
where $\Complex T_{x_0}(M)$ is the complexified tangent space of $M$ at $x_0$, $(U,f_1,\ldots,f_k)$ is a
representative of $\Gamma_{x_0}$. We write $T_{x_0}(\Omega)$ to denote the
tangent space of $\Omega$ at $x_0$.
\end{defn}

Let $(x, y)$ be some coordinates of $X\times X$. From now on, we use the notations $\xi$ and $\eta$ for
the dual variables of $x$ and $y$ respectively.

\begin{rem} \label{r:s-important}
For each point $(x_0, \eta_0, x_0, \eta_0)\in{\rm diag\,}(\Sigma\times\Sigma)$, we assign a germ
\begin{equation} \label{e:0710251217}
\Gamma_{(x_0, \eta_0, x_0, \eta_0)}
=(T^*(X)\times T^*(X), \xi-\psi'_x(\infty, x, \eta), y-\psi'_\eta(\infty, x, \eta))_{(x_0, \eta_0, x_0, \eta_0)}.
\end{equation}
Let $C_\infty$ be the formal manifold at ${\rm diag\,}(\Sigma\times\Sigma)$:
\[C_\infty=\set{\Gamma_{(x, \eta, x, \eta)};\, (x, \eta, x, \eta)\in{\rm diag\,}(\Sigma\times\Sigma)}.\]
$C_\infty$ is strictly positive in the sense that
\[\frac{1}{i}\sigma(v, \ol{v})>0,\ \
\forall v\in T_\rho(C_\infty)\setminus\Complex T_\rho({\rm diag\,}(\Sigma\times\Sigma)),\]
where $\rho\in{\rm diag\,}(\Sigma\times\Sigma)$. Here $\sigma$ is the canonical two form on
$\Complex T^*_\rho(X)\times\Complex T^*_\rho(X)$ (see (\ref{e:0807172212})).
\end{rem}

The following is well-known (see section $1$ of~\cite{MS78}, for the proof)

\begin{prop} \label{p:0708261128}
There exists a formal manifold $J_+=\set{J_{(x, \eta)};\, (x, \eta)\in\Sigma}$ at $\Sigma$ associated to $T^*(X)$
such that ${\rm codim}J_+=n-1$
and $\forall (x_0, \eta_0)\in\Sigma$, if\, $(U, f_1,\ldots,f_{n-1})$ is a representative of $J_{(x_0, \eta_0)}$, then
\begin{equation} \label{e:0709201340}
\set{f_j, f_k}\equiv0\ \ \mbox{in the space}\ C^\infty_\Sigma(U),\ j, k=1,\ldots,n-1,
\end{equation}
\begin{equation} \label{e:0708281851}
p_0\equiv\sum^{n-1}_{j=1}g_jf_j\ \ \mbox{in the space}\ C^\infty_\Sigma(U),
\end{equation}
where $g_j\in C^\infty_\Sigma(U)$, $j=1,\ldots,n-1$, and
\begin{equation} \label{e:0708281853}
\frac{1}{i}\sigma(H_{f_j}, H_{\ol f_j})>0\ \ \mbox{at}\ \ (x_0, \eta_0)\in\Sigma,\ j=1,\ldots,n-1.
\end{equation}
We also write $f_j$ to denote an almost analytic extension of $f_j$. Then,
\begin{equation} \label{e:0710271152}
f_j(x, \psi'_x(\infty, x, \eta))\ \mbox{vanishes to infinite order on}\ \Sigma,\ j=1,\ldots,n-1.
\end{equation}
Moreover, we have
\begin{align} \label{e:0709201346}
T_\rho(C_\infty)=&
\{(v+\sum^{n-1}_{j=1}t_jH_{f_j}(x_0, \eta_0), v+\sum^{n-1}_{j=1}s_jH_{\ol f_j}(x_0, \eta_0));\, \nonumber \\
&\quad v\in T_{(x_0, \eta_0)}(\Sigma),\ t_j, s_j\in\Complex, j=1,\ldots,n-1\},
\end{align}
where
$\rho=(x_0, \eta_0, x_0, \eta_0)\in{\rm diag\,}(\Sigma\times\Sigma)$ and $C_\infty$ is as in Remark~\ref{r:s-important}.
\end{prop}

We return to our problem. We need the following

\begin{lem} \label{l:0708262029}
We have
\begin{equation} \label{e:0709211945}
\psi''_{\eta\eta}(\infty, p, \omega_0(p))\omega_0(p)=0
\end{equation}
and
\begin{align} \label{e:mad4}
{\rm Rank}\left(\psi''_{\eta\eta}(\infty, p, \omega_0(p))\right)=2n-2,
\end{align}
for all $p\in X$.
\end{lem}

\begin{proof}
Since $\psi'_\eta(\infty, x, \eta)$ is positively homogeneous of degree $0$, it follows that
$\psi''_{\eta\eta}(\infty, p, \omega_0(p))\omega_0(p)=0$.
Thus,
${\rm Rank}\left(\psi''_{\eta\eta}(\infty, p, \omega_0(p))\right)\leq2n-2$.
From
${\rm Im\,}\psi(\infty, x, \eta)\asymp\abs{\eta}{\rm dist\,}((x, \frac{\eta}{\abs{\eta}});\ \Sigma)^2$,
we have
${\rm Im\,}\psi''_{\eta\eta}(\infty, p, \omega_0(p))V\neq0$, if
$V\notin\set{\lambda\omega_0(p);\, \lambda\in\Complex}$.
Thus, for all $V\notin\set{\lambda\omega_0(p);\, \lambda\in\Complex}$, we have
\begin{align*}
\seq{\psi''_{\eta\eta}(\infty, p, \omega_0(p))V, \ol V}&=\seq{{\rm Re\,}\psi''_{\eta\eta}(\infty, p, \omega_0(p))V, \ol V} \\
&\quad+i\seq{{\rm Im\,}\psi''_{\eta\eta}(\infty, p, \omega_0(p))V, \ol V}\neq0.
\end{align*}
We get (\ref{e:mad4}).
\end{proof}

Until further notice, we assume that $q=n_+$. For $p\in X$, we take local coordinates
$x=(x_1, x_2, \ldots, x_{2n-1})$
defined on some neighborhood $\Omega$ of $p\in X$ such that
\begin{equation} \label{e:0709252119}
\omega_0(p)=dx_{2n-1},\ x(p)=0
\end{equation}
and $\Lambda^{0,1}T_p(X)\oplus\Lambda^{1,0}T_p(X)=\set{\sum^{2n-2}_{j=1}a_j\frac{\pr}{\pr x_j};\, a_j\in\Complex,\ \ j=1,\ldots,2n-2}$.
We take $\Omega$ so that if $x_0\in\Omega$ then $\eta_{0,2n-1}>0$ where $\omega_0(x_0)=(\eta_{0,1},\ldots,\eta_{0,2n-1})$.

Until further notice, we work in $\Omega$ and we work with the local coordinates $x$.
Choose $\chi(x, \eta)\in C^\infty(T^*(X))$ so that $\chi(x, \eta)=1$ in a conic neighborhood of $(p, \omega_0(p))$,
$\chi(x, \eta)=0$ outside $T^*(\Omega)$,
$\chi(x, \eta)=0$ in a conic neighborhood of $\Sigma^-$ and $\chi(x, \lambda\eta)=\chi(x, \eta)$ when $\lambda>0$.
We introduce the cut-off functions $\chi(x, \eta)$  and $(1-\chi(x, \eta))$ in the integral (\ref{e:0707311443}):
$K_{\pi^{(q)}}(x, y)\equiv K_{\pi^{(q)}_+}(x, y)+K_{\pi^{(q)}_-}(x, y)$,
\begin{align} \label{e:0709230113}
&K_{\pi^{(q)}_+}(x,y)\equiv\frac{1}{(2\pi)^{2n-1}}\int e^{i(\psi(\infty,x,\eta)-\seq{y,\eta})}
  \chi(x, \eta)a(\infty,x,\eta)d\eta,\nonumber \\
&K_{\pi^{(q)}_-}(x,y)\equiv\frac{1}{(2\pi)^{2n-1}}\int e^{i(\psi(\infty,x,\eta)-\seq{y,\eta})}
(1-\chi(x, \eta))a(\infty,x,\eta)d\eta.
\end{align}
Now, we study $K_{\pi^{(q)}_+}$. We write $t$ to denote $\eta_{2n-1}$. Put $\eta'=(\eta_1,\ldots,\eta_{2n-2})$. We have
\begin{align} \label{e:0709211700}
&K_{\pi^{(q)}_+}\equiv\frac{1}{(2\pi)^{2n-1}}\int\!\!\! e^{i(\psi(\infty, x, (\eta', t))-\seq{y, (\eta', t)})}
\chi(x, (\eta', t))a(\infty, x, (\eta', t))d\eta'dt\nonumber \\
&=\frac{1}{(2\pi)^{2n-1}}\int^\infty_0\!\!\!\int\!\!\! e^{it(\psi(\infty, x, (w, 1))-\seq{y, (w, 1)})}
t^{2n-2}\chi(x, (tw, t))a(\infty, x, (tw, t))dwdt
\end{align}
where $\eta'=tw$, $w\in\Real^{2n-2}$.
The stationary phase method of Melin and Sj\"{o}strand (see Proposition~\ref{p:0709171439})
then permits us to carry out
the $w$ integration in (\ref{e:0709211700}), to get
\begin{equation} \label{e:0709221719}
K_{\pi^{(q)}_+}(x, y)\equiv\int^{\infty}_0 e^{it\phi_+(x, y)}s_+(x, y, t)dt
\end{equation}
with
\begin{equation} \label{e:0710260900}
s_+(x, y, t)\sim\sum^\infty_{j=0}s^j_+(x, y)t^{n-1-j}
\end{equation}
in $S^{n-1}_{1,0}(\Omega\times\Omega\times]0, \infty[;\ \mathscr L(\Lambda^{0, q}T^*_y(X), \Lambda^{0, q}T^*_x(X)))$,
where
\[s^j_+(x, y)\in C^\infty(\Omega\times\Omega;\ \mathscr L(\Lambda^{0, q}T^*_y(X), \Lambda^{0, q}T^*_x(X))),\ \ j=0, 1,\ldots,\]
and $\phi_+(x, y)\in C^\infty(\Omega\times\Omega)$ is the corresponding critical value.
(See Proposition~\ref{p:0709171438} for a review.) For $x\in\Omega$, let $\sigma(x)\in\Real^{2n-2}$ be the vector:
\begin{equation} \label{e:0710251651}
(x, (\sigma(x), 1))\in\Sigma^+.
\end{equation}
Since $d_{w}(\psi(\infty, x, (w, 1))-\seq{y, (w, 1)})=0$ at $x=y$, $w=\sigma(x)$,
it follows that when $x=y$, the corresponding critical point is $w=\sigma(x)$ and consequently
\begin{align}
&\phi_+(x, x)=0,\label{e:0710260943}\\
&(\phi_+)'_x(x, x)=\psi'_x(\infty, x, (\sigma(x), 1))=(\sigma(x), 1),\ (\phi_+)'_y(x, x)=-(\sigma(x), 1). \label{e:0709212016}
\end{align}

The following is well-known (see Proposition~\ref{p:0709171438})

\begin{prop} \label{p:0709212022}
In some open neighborhood $Q$ of $p$ in $\Omega$, we have
\begin{align} \label{e:0709212032}
&{\rm Im\,}\phi_+(x, y)\geq c\inf_{w\in W}\Bigr({\rm Im\,}\psi(\infty, x, (w, 1))+
\abs{d_w(\psi(\infty, x, (w, 1))-\seq{y, (w, 1)})}^2\Bigr),\nonumber\\
&\quad(x, y)\in Q\times Q,
\end{align}
where $c$ is a positive constant and $W$ is some open set of the origin in $\Real^{2n-2}$.
\end{prop}

We have the following

\begin{prop} \label{p:0709212100}
In some open neighborhood $Q$ of $p$ in $\Omega$, there is a constant $c>0$ such that
\begin{equation} \label{e:0709212105}
{\rm Im\,}\phi_+(x, y)\geq c\abs{x'-y'}^2,\ \ (x, y)\in Q\times Q,
\end{equation}
where $x'=(x_1,\ldots,x_{2n-2})$, $y'=(y_1,\ldots,y_{2n-2})$ and
$\abs{x'-y'}^2=(x_1-y_1)^2+\cdots+(x_{2n-2}-y_{2n-2})^2$.
\end{prop}

\begin{proof}
From $\psi(\infty, x, (w, 1))-\seq{y, (w, 1)}=\seq{x-y, (w, 1)}+O(\abs{w-\sigma(x)}^2)$
we can check that
\[d_w(\psi(\infty, x, (w, 1))-\seq{y, (w, 1)})=\seq{x'-y', dw}+O(\abs{w-\sigma(x)}),\]
where $\sigma(x)$ is as in (\ref{e:0710251651}) and $x'=(x_1,\ldots,x_{2n-2})$, $y'=(y_1,\ldots,y_{2n-2})$. Thus, there are
constants $c_1$, $c_2>0$ such that
\[\abs{d_w(\psi(\infty, x, (w, 1))-\seq{y, (w, 1)})}^2\geq c_1\abs{x'-y'}^2-c_2\abs{w-\sigma(x)}^2\]
for $(x, w)$ in some compact set of $\Omega\times\dot{\Real}^{2n-2}$.
If $\frac{c_1}{2}\abs{(x'-y')}^2\geq c_2\abs{w-\sigma(x)}^2$, then
\begin{equation} \label{e:0709231244}
\abs{d_\omega(\psi(\infty, x, \omega)-\seq{y, \omega})}^2\geq\frac{c_1}{2}\abs{(x'-y')}^2.
\end{equation}
Now, we assume that $\abs{(x'-y')}^2\leq\frac{2c_2}{c_1}\abs{w-\sigma(x)}^2$.
We have
\begin{equation} \label{e:0709231245}
{\rm Im\,}\psi(\infty, x, (w, 1))\geq c_3\abs{w-\sigma(x)}^2\geq\frac{c_1c_3}{2c_2}\abs{(x'-y')}^2,
\end{equation}
for $(x, w)$ in some compact set of $\Omega\times\dot{\Real}^{2n-2}$, where $c_3$ is a positive constant.
From (\ref{e:0709231244}), (\ref{e:0709231245}) and Proposition~\ref{p:0709212022}, we have
${\rm Im\,}\phi_+(x, y)\geq c\abs{(x'-y')}^2$
for $x$, $y$ in some neighborhood of $p$, where $c$ is a positive constant. We get the proposition.
\end{proof}

\begin{rem} \label{r:0709252155}
For each point $(x_0, x_0)\in{\rm diag\,}(\Omega\times\Omega)$, we assign a germ
\[H_{+,(x_0, x_0)}=(\Omega\times\Omega, \phi_+(x, y))_{(x_0, x_0)}.\]
Let $H_+$ be the formal hypersurface at ${\rm diag\,}(\Omega\times\Omega)$:
\begin{equation} \label{e:0708281710}
H_+=\set{H_{+,(x, x)};\, (x, x)\in{\rm diag\,}(\Omega\times\Omega)}.
\end{equation}
The formal conic conormal bundle $\Lambda_{\phi_+t}$ of $H_+$ is given by:
For each point
\[(x_0, \eta_0, x_0, \eta_0)\in{\rm diag\,}((\Sigma^+\bigcap T^*(\Omega))\times(\Sigma^+\bigcap T^*(\Omega))),\]
we assign a germ
\begin{align*}
\Lambda_{(x_0, \eta_0, x_0, \eta_0)}
&=(T^*(U)\times T^*(U), \xi_j-(\phi_+)'_{x_j}t,\ j=1,\ldots,2n-1,\\
&\quad\eta_k-(\phi_+)'_{y_k}t,\ k=1,\ldots,2n-2,\ \phi_+(x, y))_{(x_0, \eta_0, x_0, \eta_0)},
\end{align*}
where $t=\frac{\eta_{2n-1}}{(\phi_+)'_{y_{2n-1}}}$ and $U\subset\Omega$ is an open set of $x_0$ such that
$(\phi_+)'_{y_{2n-1}}\neq0$ on $U\times U$.
Then,
\begin{equation} \label{e:0709221642}
\Lambda_{\phi_+t}=\set{\Lambda_{(x, \eta, x, \eta)};\, (x, \eta, x, \eta)
\in{\rm diag\,}((\Sigma^+\bigcap T^*(\Omega))\times(\Sigma^+\bigcap T^*(\Omega)))}.
\end{equation}
$\Lambda_{\phi_+t}$ is a formal manifold at ${\rm diag\,}((\Sigma^+\bigcap T^*(\Omega))\times(\Sigma^+\bigcap T^*(\Omega)))$.
In fact, $\Lambda_{\phi_+t}$ is the positive Lagrangean manifold associated to $\phi_+t$ in the sense of Melin and Sj\"{o}strand.
(See~\cite{MS74} and Appendix A.)

Let
\[(W, f_1(x, \xi, y, \eta),\ldots,f_{4n-2}(x, \xi, y, \eta))\]
be a representative of
$\Gamma_{(x_0, \eta_0, x_0, \eta_0)}$,  where $\Gamma_{(x_0, \eta_0, x_0, \eta_0)}$ is as in (\ref{e:0710251217}). Put
\[\Gamma'_{(x_0, \eta_0, x_0, \eta_0)}=(W, f_1(x, \xi, y, -\eta),\ldots,f_{4n-2}(x, \xi, y, -\eta))_{(x_0, \eta_0, x_0, \eta_0)}.\]
Let $C'_\infty$ be the formal manifold at ${\rm diag\,}(\Sigma^+\times\Sigma^+)$:
\begin{equation} \label{e:0710261031}
C'_\infty=\set{\Gamma'_{(x, \eta, x, \eta)};\, (x, \eta, x, \eta)
\in{\rm diag\,}(\Sigma^+\times\Sigma^+)}.
\end{equation}
We notice that $\psi(\infty, x, \eta)-\seq{y, \eta}$ and $\phi_+(x, y)t$ are
equivalent at each point of
\[{\rm diag\,}((\Sigma^+\bigcap T^*(\Omega))\times(\Sigma^+\bigcap T^*(\Omega)))\]
in the sense of Melin-Sj\"{o}strand (see
Definition~\ref{d:0709171814}). From the global theory of Fourier integral operators
(see Proposition~\ref{p:c-mesjmore}), we get
\begin{equation} \label{e:0710261020}
\Lambda_{\phi_+t}=C'_\infty\ \ \mbox{at}\ {\rm diag\,}((\Sigma^+\bigcap T^*(\Omega))\times(\Sigma^+\bigcap T^*(\Omega))).
\end{equation}
See Proposition~\ref{p:0709201613} and Proposition~\ref{p:c-mesjmore}, for the details. Formally,
\[C_\infty=\set{(x, \xi, y, \eta);\, (x, \xi, y, -\eta)\in\Lambda_{\phi_+t}}.\]
\end{rem}

Put $\hat\phi_+(x, y)=-\ol\phi_+(y, x)$.
We claim that
\begin{equation} \label{e:0709222215}
\Lambda_{\phi_+t}=\Lambda_{\hat\phi_+t}\ \ \mbox{at}\ {\rm diag\,}((\Sigma^+\bigcap T^*(\Omega))\times(\Sigma^+\bigcap T^*(\Omega))),
\end{equation}
where $\Lambda_{\hat\phi_+t}$ is defined as in (\ref{e:0709221642}).
From Proposition~\ref{p:0709201320}, it follows that $\phi_+(x, y)t$ and $-\ol\phi_+(y, x)t$ are
equivalent at each point of ${\rm diag\,}((\Sigma^+\bigcap T^*(\Omega))\times(\Sigma^+\bigcap T^*(\Omega)))$
in the sense of Melin-Sj\"{o}strand. Again from
the global theory of Fourier integral operators we get (\ref{e:0709222215}).

From (\ref{e:0709222215}), we get the following

\begin{prop} \label{p:0709221724}
There is a function $f\in C^\infty(\Omega\times\Omega)$, $f(x, x)\neq0$, such that
\begin{equation} \label{e:0708281715}
\phi_+(x, y)+f(x, y)\ol\phi_+(y, x)
\end{equation}
vanishes to infinite order on $x=y$.
\end{prop}

From (\ref{e:0708281715}), we can replace $\phi_+(x, y)$ by
$\frac{1}{2}(\phi_+(x, y)-\ol\phi_+(y, x))$.
Thus, we have
\begin{equation} \label{e:0710260935}
\phi_+(x, y)=-\ol\phi_+(y, x).
\end{equation}
From (\ref{e:0709212016}), we see that
$(x, d_x\phi_+(x, x))\in\Sigma^+$, $d_y\phi_+(x, x)=-d_x\phi_+(x, x)$.
We can replace $\phi_+(x, y)$ by
$\frac{2\phi_+(x, y)}{\norm{d_x\phi_+(x, x)}+\norm{d_x\phi_+(y, y)}}$.
Thus,
\begin{equation} \label{e:0710260940}
d_x\phi_+(x, x)=\omega_0(x),\ d_y\phi_+(x, x)=-\omega_0(x).
\end{equation}

Similarly,
\[K_{\pi^{(q)}_-}(x, y)\equiv\int^{\infty}_{0} e^{i\phi_-(x, y)t}s_-(x, y, t)dt,\]
where $K_{\pi^{(q)}_-}(x, y)$ is as in (\ref{e:0709230113}). From (\ref{e:0709192125}), it follows that when $q=n_-=n_+$, we can take
$\phi_-(x, y)$ so that $\phi_+(x, y)=-\ol\phi_-(x, y)$.

Our method above only works locally. From above, we know that there exist open sets
$X_j$, $j=1,2,\dots,k$, $X=\bigcup^k_{j=1}X_j$, such that
\[K_{\pi^{(q)}_+}(x, y)\equiv\int^{\infty}_{0} e^{i\phi_{+,j}(x, y)t}s_{+,j}(x, y, t)dt\]
on $X_j\times X_j$, where $\phi_{+,j}$ satisfies (\ref{e:0710260943}), (\ref{e:0709212032}), (\ref{e:0709212105}), (\ref{e:0710261020}),
(\ref{e:0710260935}), (\ref{e:0710260940}) and $s_{+, j}(x, y, t)$, $j=0,1,\ldots,$ are as in (\ref{e:0710260900}).
From the global theory of Fourier integral operators, we have
\[\Lambda_{\phi_{+,j}t}=C'_\infty=\Lambda_{\phi_{+,k}t}\ \ \mbox{at}\ {\rm diag\,}
((\Sigma^+\bigcap T^*(X_j\bigcap X_k))\times(\Sigma^+\bigcap T^*(X_j\bigcap X_k))),\]
for all $j$, $k$, where $\Lambda_{\phi_{+,j}t}$, $\Lambda_{\phi_{+,k}t}$ are defined
as in (\ref{e:0709221642}) and $C'_\infty$ is as in (\ref{e:0710261031}). Thus,
there is a function
$f_{j, k}\in C^\infty((X_j\bigcap X_k)\times(X_j\bigcap X_k))$, such that
\begin{equation} \label{e:0708291523}
\phi_{+,j}(x, y)-f_{j, k}(x, y)\phi_{+,k}(x, y)
\end{equation}
vanishes to infinite order on $x=y$, for all $j$, $k$. Let $\chi_j(x, y)$ be a
$C^\infty$ partition of unity subordinate to $\set{X_{j}\times X_j}$ with
$\chi_j(x, y)=\chi_j(y, x)$
and set
\[\phi_+(x, y)=\sum\chi_j(x, y)\phi_{+,j}(x, y).\]
From (\ref{e:0708291523}) and
the global theory of Fourier integral operators, it follows that $\phi_{+, j}(x, y)t$ and $\phi_+(x, y)t$ are
equivalent at each point of
\[{\rm diag\,}((\Sigma^+\bigcap T^*(X_j))\times(\Sigma^+\bigcap T^*(X_j)))\]
in the sense of Melin-Sj\"{o}strand, for all $j$.
Again, from the global theory of Fourier integral operators, we get the main result of this work

\begin{thm} \label{t:szegomain}
We assume that the Levi form has signature $(n_-, n_+)$, $n_-+n_+=n-1$. Let $q=n_-$ or $n_+$.
Suppose $\Box^{(q)}_b$ has closed range. Then,
\begin{align*}
&K_{\pi^{(q)}}=K_{\pi^{(q)}_+}\ \ {\rm if\,}\ \ n_+=q\neq n_-,  \\
&K_{\pi^{(q)}}=K_{\pi^{(q)}_-}\ \ {\rm if\,}\ \ n_-=q\neq n_+,  \\
&K_{\pi^{(q)}}=K_{\pi^{(q)}_+}+K_{\pi^{(q)}_-}\ \ {\rm if}\ \ n_+=q=n_-,
\end{align*}
where $K_{\pi^{(q)}_+}(x, y)$ satisfies
\begin{align*}
K_{\pi^{(q)}_+}(x, y)\equiv\int^{\infty}_{0} e^{i\phi_+(x, y)t}s_+(x, y, t)dt
\end{align*}
with
\[s_+(x, y, t)\in S^{n-1}_{1, 0}(X\times X\times]0, \infty[;\, \mathscr L(\Lambda^{0,q}T^*_y(X), \Lambda^{0,q}T^*_x(X))),\]
\begin{equation} \label{e:0709232340}
s_+(x, y, t)\sim\sum^\infty_{j=0}s^j_+(x, y)t^{n-1-j}
\end{equation}
in $S^{n-1}_{1,0}(X\times X\times]0, \infty[;\ \mathscr L(\Lambda^{0, q}T^*_y(X), \Lambda^{0, q}T^*_x(X)))$, where
\[s^j_+(x, y)\in C^\infty(X\times X;\, \mathscr L(\Lambda^{0, q}T^*_y(X), \Lambda^{0, q}T^*_x(X))),\ j=0,1,\ldots,\]
\[\phi_+(x, y)\in C^\infty(X\times X),\ \ {\rm Im\,}\phi_+(x, y)\geq0,\]
\[\phi_+(x, x)=0,\ \ \phi_+(x, y)\neq0\ \ \mbox{if}\ \ x\neq y,\]
\[d_x\phi_+\neq0,\ \ d_y\phi_+\neq0\ \ \mbox{where}\ \ {\rm Im\,}\phi_+=0,\]
\begin{equation} \label{e:nevergiveup3}
d_x\phi_+(x, y)|_{x=y}=\omega_0(x),\ \
d_y\phi_+(x, y)|_{x=y}=-\omega_0(x),
\end{equation}
\[\phi_+(x, y)=-\ol\phi_+(y, x).\]
Moreover, $\phi_+(x, y)$ satisfies (\ref{e:0710261020}). Similarly,
\begin{align*}
K_{\pi^{(q)}_-}(x, y)\equiv\int^{\infty}_{0} e^{i\phi_-(x, y)t}s_-(x, y, t)dt
\end{align*}
with
\[s_-(x, y, t)\sim\sum^\infty_{j=0}s^j_-(x, y)t^{n-1-j}\]
in $S^{n-1}_{1,0}(X\times X\times]0, \infty[;\ \mathscr L(\Lambda^{0, q}T^*_y(X), \Lambda^{0, q}T^*_x(X)))$, where
\[s^j_-(x, y)\in C^\infty(X\times X;\, \mathscr L(\Lambda^{0, q}T^*_y(X), \Lambda^{0, q}T^*_x(X))),\ j=0,1,\ldots,\]
and when $q=n_-=n_+$, $\phi_-(x, y)=-\ol\phi_+(x, y)$.
\end{thm}


\section{The leading term of the Szeg\"o Projection}

To compute the leading term of the Szeg\"{o} projection, we have to know the tangential Hessian of $\phi_+$ at each point
of ${\rm diag\,}(X\times X)$ (see (\ref{e:0710261800})), where $\phi_+$ is as in Theorem~\ref{t:szegomain}.
We work with local coordinates $x=(x_1,\ldots,x_{2n-1})$ defined on an open set $\Omega\subset X$.
The tangential Hessian of $\phi_+(x, y)$ at $(p, p)\in{\rm diag\,}(X\times X)$ is the bilinear map:
\begin{align} \label{e:0710261800}
T_{(p, p)}H_+\times T_{(p, p)}H_+&\To\Complex,\nonumber \\
(u, v)&\To\seq{(\phi_+'')(p, p)u, v},\ \ u, v\in T_{(p, p)}H_+,
\end{align}
where $H_+$ is as in (\ref{e:0708281710}) and
$(\phi_+)''=\left[
\begin{array}[c]{cc}
  (\phi_+)''_{xx} & (\phi_+)''_{xy} \\
  (\phi_+)''_{yx} & (\phi_+)''_{yy}
\end{array}\right]$.

From (\ref{e:nevergiveup3}), we can check that $T_{(p,p)}H_+$ is spanned by
\begin{equation} \label{e:0708291117}
(u, v),\ \ (Y(p), Y(p)),\ \ u, v\in\Lambda^{1,0}T_p(X)\oplus\Lambda^{0,1}T_p(X).
\end{equation}

Now, we compute the the tangential Hessian  of $\phi_+$ at $(p, p)\in{\rm diag\,}(X\times X)$.
We need to understand the tangent space of the formal manifold
\[J_+=\set{J_{(x, \eta)};\, (x, \eta)\in\Sigma}\]
at $\rho=(p, \lambda\omega_0(p))\in\Sigma^+$, $\lambda>0$, where $J_+$ is as in Proposition~\ref{p:0708261128}.

Let  $\lambda_j$, $j=1,\ldots,n-1$, be the eigenvalues of the Levi form $L_p$. We recall that
$2i\abs{\lambda_j}\abs{\sigma_{iY}(\rho)}$, $j=1,\ldots,n-1$ and
$-2i\abs{\lambda_j}\abs{\sigma_{iY}(\rho)}$, $j=1,\ldots,n-1$, are the non-vanishing eigenvalues of
the fundamental matrix $F_\rho$. (See (\ref{e:h-eigenvalues}).)
Let $\Lambda^+_\rho\subset\Complex T_\rho(T^*(X))$ be the span of the eigenspaces of $F_\rho$ corresponding
to $2i\abs{\lambda_j}\abs{\sigma_{iY}(\rho)}$, $j=1,\ldots,n-1$.
It is well known (see~\cite{Sjo74},~\cite{MS78} and Boutet de Monvel-Guillemin~\cite{BG81}) that
\[T_\rho(J_+)=\Complex T_\rho(\Sigma)\oplus\Lambda^+_\rho,\ \ \Lambda^+_\rho=T_\rho(J_+)^\bot,\]
where $T_\rho(J_+)^\bot$ is the
orthogonal to $T_\rho(J_+)$ in $\Complex T_\rho(T^*(X))$ with respect to the canonical two form $\sigma$.
We need the following

\begin{lem} \label{l:s-final1}
Let $\rho=(p, \lambda\omega_0(p))\in\Sigma^+$, $\lambda>0$. Let
$\ol Z_1(x),\ldots,\ol Z_{n-1}(x)$
be an orthonormal frame of\, $\Lambda^{1,0}T_x(X)$ varying smoothly with $x$ in a neighborhood of\, $p$, for which
the Levi form is diagonalized at $p$. Let $q_j(x, \xi)$, $j=1,\ldots,n-1$, be the principal symbols of\, $Z_j(x)$, $j=1,\ldots,n-1$.
Then, $\Lambda^+_\rho$ is spanned by
\begin{equation} \label{e:final1}
\left\{ \begin{array}{ll}
H_{q_j}(\rho),\ {\rm if\,}\ \ \frac{1}{i}\set{q_j, \ol q_j}(\rho)>0  \\
H_{\ol q_j}(\rho),\ {\rm if\,}\ \ \frac{1}{i}\set{q_j, \ol q_j}(\rho)<0
\end{array}\right..
\end{equation}
We recall that (see (\ref{e:h-Levi})) $\frac{1}{i}\set{q_j, \ol q_j}(\rho)=-2\lambda L_p(\ol Z_j, Z_j)$.
\end{lem}

\begin{proof}
In view of (\ref{e:h-fundamental}), we see that $H_{q_j}(\rho)$ and $H_{\ol q_j}(\rho)$ are the eigenvectors of
the fundamental matrix $F_\rho$ corresponding to $\set{q_j, \ol q_j}(\rho)$ and $\set{\ol q_j, q_j}(\rho)$, for all $j$.
Since $\Lambda^+_\rho$ is the span of the eigenspaces of the fundamental matrix $F_\rho$
corresponding to
$2i\lambda\abs{\lambda_j}$, $j=1,\ldots,n-1$, where $\lambda_j$, $j=1,\ldots,n-1$, are the eigenvalues
of the Levi form $L_p$. Thus,
$\Lambda^+_\rho$ is spanned by
\[
\left\{ \begin{array}{ll}
H_{q_j}(\rho),\ {\rm if\,}\ \ \frac{1}{i}\set{q_j, \ol q_j}(\rho)>0  \\
H_{\ol q_j}(\rho),\ {\rm if\,}\ \ \frac{1}{i}\set{q_j, \ol q_j}(\rho)<0
\end{array}\right..\]
\end{proof}

We assume that $(U, f_1,\ldots,f_{n-1})$ is a representative of $J_{\rho}$. We also write $f_j$ to denote an almost analytic extension of $f_j$,
for all $j$. It is well known that (see~\cite{MS78} and (\ref{e:0710271152})) there exist $h_j(x, y)\in C^\infty(X\times X)$, $j=1,\ldots,n-1$, such that
\begin{equation} \label{e:0710261822}
f_j(x, (\phi_+)'_x)-h_j(x, y)\phi_+(x, y)
\end{equation}
vanishes to infinite order on $x=y$, $j=1,\ldots,n-1$.
From Lemma~\ref{l:s-final1}, we may assume that
\begin{equation} \label{e:final1bis}
\left\{ \begin{array}{ll}
H_{f_j}(\rho)=H_{q_j}(\rho),\ {\rm if\,}\ \ \frac{1}{i}\set{q_j, \ol q_j}(\rho)>0 \\
H_{f_j}(\rho)=H_{\ol q_j}(\rho),\ {\rm if\,}\ \ \frac{1}{i}\set{q_j, \ol q_j}(\rho)<0
\end{array}\right..
\end{equation}
Here $q_j$, $j=1,\ldots,n-1$, are as in Lemma~\ref{l:s-final1}.

We take local coordinates
\[x=(x_1,\ldots,x_{2n-1}),\ \ z_j=x_{2j-1}+ix_{2j},\ \ j=1,\ldots,n-1,\]
defined on some neighborhood of $p$ such that
$\omega_0(p)=\sqrt{2}dx_{2n-1}$, $x(p)=0$,
$(\frac{\pr}{\pr x_j}(p)\ |\ \frac{\pr}{\pr x_k}(p))=2\delta_{j,k}$, $j, k=1,\ldots,2n-1$
and
\[\ol Z_j=\frac{\pr}{\pr z_j}-\frac{1}{\sqrt{2}}i\lambda_j\ol z_j\frac{\pr}{\pr x_{2n-1}}-
\frac{1}{\sqrt{2}}c_jx_{2n-1}\frac{\pr}{\pr x_{2n-1}}+O(\abs{x}^2),\ \ j=1,\ldots,n-1,\]
where $\ol Z_j$, $j=1,\ldots,n-1$, are as in Lemma~\ref{l:s-final1},
$c_j\in\Complex$, $\frac{\pr}{\pr z_j}=\frac{1}{2}(\frac{\pr}{\pr x_{2j-1}}-i\frac{\pr}{\pr x_{2j}})$, $j=1,\ldots,n-1$
and $\lambda_j$, $j=1,\ldots,n-1$, are the eigenvalues of $L_p$.
(This is always possible. See page $157$-page $160$ of~\cite{BG88}.)

We may assume that
$\lambda_j>0$, $j=1,\ldots,q$, $\lambda_j<0$, $j=q+1,\ldots,n-1$.
Let $\xi=(\xi_1,\ldots,\xi_{2n-1})$ denote the dual variables of $x$.
From (\ref{e:final1bis}), we can check that
\begin{align} \label{e:0804250813}
&f_j(x, \xi)=-\frac{i}{2}(\xi_{2j-1}-i\xi_{2j})-\frac{1}{\sqrt{2}}\lambda_j\ol z_j\xi_{2n-1}
+i\frac{1}{\sqrt{2}}c_jx_{2n-1}\xi_{2n-1}+O(\abs{(x,\xi')}^2),\nonumber\\
&\quad j=1,\ldots,q,\ \ \xi'=(\xi_1,\ldots,\xi_{2n-2}), \nonumber \\
&f_j(x, \xi)=\frac{i}{2}(\xi_{2j-1}+i\xi_{2j})-\frac{1}{\sqrt{2}}\lambda_jz_j\xi_{2n-1}
-i\frac{1}{\sqrt{2}}\ol c_jx_{2n-1}\xi_{2n-1}+O(\abs{(x,\xi')}^2),\nonumber
\\
&\quad j=q+1,\ldots,n-1,\ \ \xi'=(\xi_1,\ldots,\xi_{2n-2}).
\end{align}
We write $y=(y_1,\ldots,y_{2n-1})$, $w_j=y_{2j-1}+iy_{2j}$, $j=1,\ldots,n-1$,
$\frac{\pr}{\pr w_j}=\frac{1}{2}(\frac{\pr}{\pr y_{2j-1}}-i\frac{\pr}{\pr y_{2j}})$,
$\frac{\pr}{\pr\ol w_j}=\frac{1}{2}(\frac{\pr}{\pr y_{2j-1}}+i\frac{\pr}{\pr y_{2j}})$, $j=1,\ldots,n-1$
and
$\frac{\pr}{\pr\ol z_j}=\frac{1}{2}(\frac{\pr}{\pr x_{2j-1}}+i\frac{\pr}{\pr x_{2j}})$, $j=1,\ldots,n-1$.
From (\ref{e:0710261822}) and (\ref{e:0804250813}), we have
\begin{align} \label{e:0804242211}
&-i\frac{\pr\phi_+}{\pr z_j}-\frac{1}{\sqrt{2}}(\lambda_j\ol z_j-ic_jx_{2n-1})\frac{\pr\phi_+}{\pr x_{2n-1}}
=h_j\phi_++O(\abs{(x, y)}^2),\ \ j=1,\ldots,q, \nonumber \\
&i\frac{\pr\phi_+}{\pr\ol z_j}-\frac{1}{\sqrt{2}}(\lambda_jz_j+i\ol c_jx_{2n-1})\frac{\pr\phi_+}{\pr x_{2n-1}}
=h_j\phi_++O(\abs{(x, y)}^2),\nonumber \\
&\quad j=q+1,\ldots,n-1.
\end{align}
From (\ref{e:0804242211}), it is straight forward to see that
\begin{align} \label{e:0804242212}
\frac{\pr^2\phi_+}{\pr z_j\pr z_k}(0,0)&=\frac{\pr^2\phi_+}{\pr z_j\pr\ol z_k}(0,0)-i\lambda_j\delta_{j,k}=
\frac{\pr^2\phi_+}{\pr z_j\pr w_k}(0,0)=\frac{\pr^2\phi_+}{\pr z_j\pr\ol w_k}(0,0) \nonumber \\
&=\frac{\pr^2\phi_+}{\pr z_j\pr x_{2n-1}}(0,0)+\frac{\pr^2\phi_+}{\pr z_j\pr y_{2n-1}}(0,0)-c_j=0,\nonumber \\
&\quad 1\leq j\leq q,\ \ 1\leq k\leq n-1,
\end{align}
and
\begin{align} \label{e:0807150845}
\frac{\pr^2\phi_+}{\pr\ol z_j\pr\ol z_k}(0,0)&=\frac{\pr^2\phi_+}{\pr\ol z_j\pr z_k}(0,0)+i\lambda_j\delta_{j,k}=
\frac{\pr^2\phi_+}{\pr\ol z_j\pr w_k}(0,0)=\frac{\pr^2\phi_+}{\pr\ol z_j\pr\ol w_k}(0,0) \nonumber \\
&=\frac{\pr^2\phi_+}{\pr\ol z_j\pr x_{2n-1}}(0,0)+\frac{\pr^2\phi_+}{\pr\ol z_j\pr y_{2n-1}}(0,0)-\ol c_j=0,\nonumber \\
&\quad q+1\leq j\leq n-1,\ \ 1\leq k\leq n-1.
\end{align}

Since $d_x\phi_+|_{x=y}=\omega_0(x)$, we have
$\ol f_j(x, (\phi_+)'_x(x, x))=0$, $j=1,\ldots,n-1$.
Thus,
\begin{align} \label{e:0804242311}
&i\frac{\pr\phi_+}{\pr\ol z_j}(x,x)-\frac{1}{\sqrt{2}}(\lambda_jz_j+i\ol c_jx_{2n-1})\frac{\pr\phi_+}{\pr x_{2n-1}}(x,x)=O(\abs{x}^2),\ \
j=1,\ldots,q, \nonumber \\
&-i\frac{\pr\phi_+}{\pr z_j}(x,x)-\frac{1}{\sqrt{2}}(\lambda_j\ol z_j-ic_jx_{2n-1})\frac{\pr\phi_+}{\pr x_{2n-1}}(x,x)=O(\abs{x}^2),\nonumber \\
&\quad j=q+1,\ldots,n-1.
\end{align}
From (\ref{e:0804242311}), it is straight forward to see that
\begin{align} \label{e:0804242312}
\frac{\pr^2\phi_+}{\pr\ol z_j\pr\ol z_k}(0,0)+\frac{\pr^2\phi_+}{\pr\ol z_j\pr\ol w_k}(0,0)&
=\frac{\pr^2\phi_+}{\pr\ol z_j\pr z_k}(0,0)+\frac{\pr^2\phi_+}{\pr\ol z_j\pr w_k}(0,0)+i\lambda_j\delta_{j,k} \nonumber \\
&=\frac{\pr^2\phi_+}{\pr\ol z_j\pr x_{2n-1}}(0,0)+\frac{\pr^2\phi_+}{\pr\ol z_j\pr y_{2n-1}}(0,0)-\ol c_j=0,\nonumber \\
&\quad 1\leq j\leq q,\ \ 1\leq k\leq n-1,
\end{align}
and
\begin{align} \label{e:0807150911}
\frac{\pr^2\phi_+}{\pr z_j\pr z_k}(0,0)+\frac{\pr^2\phi_+}{\pr z_j\pr w_k}(0,0)&
=\frac{\pr^2\phi_+}{\pr z_j\pr\ol z_k}(0,0)+\frac{\pr^2\phi_+}{\pr z_j\pr\ol w_k}(0,0)-i\lambda_j\delta_{j,k} \nonumber \\
&=\frac{\pr^2\phi_+}{\pr z_j\pr x_{2n-1}}(0,0)+\frac{\pr^2\phi_+}{\pr z_j\pr y_{2n-1}}(0,0)-c_j=0,\nonumber \\
&\quad q+1\leq j\leq n-1,\ \ 1\leq k\leq n-1.
\end{align}
Since $\phi_+(x, y)=-\ol\phi_+(y, x)$, from (\ref{e:0807150845}), we have
\begin{align*}
\frac{\pr^2\phi_+}{\pr\ol z_j\pr w_k}(0,0)&=-\frac{\pr^2\ol\phi_+}{\pr\ol w_j\pr z_k}(0,0) \\
&=-\ol{\frac{\pr^2\phi_+}{\pr w_j\pr\ol z_k}(0,0)}=0,\ \ q+1\leq k\leq n-1.
\end{align*}
Combining this with (\ref{e:0804242312}), we get
\begin{equation} \label{e:0804242330}
\frac{\pr^2\phi_+}{\pr\ol z_j\pr z_k}(0,0)=0,\ \ 1\leq j\leq q,\ \ q+1\leq k\leq n-1.
\end{equation}
Similarly,
\begin{align} \label{e:0804242345}
&\frac{\pr^2\phi_+}{\pr\ol z_j\pr\ol z_k}(0,0)=0,\ \ 1\leq j\leq q,\ \ 1\leq k\leq q,\nonumber \\
&\frac{\pr^2\phi_+}{\pr z_j\pr z_k}(0,0)=0,\ \ q+1\leq j\leq n-1,\ \ q+1\leq k\leq n-1.
\end{align}
From (\ref{e:0804242212}) and (\ref{e:0804242312}), we have
\begin{equation} \label{e:0804250017}
\frac{\pr^2\phi_+}{\pr\ol z_j\pr w_k}(0,0)=-i\lambda_j\delta_{j,k}-\frac{\pr^2\phi_+}{\pr\ol z_j\pr z_k}(0,0)
=-2i\lambda_j\delta_{j,k},\ \ 1\leq j, k\leq q.
\end{equation}
Similarly,
\begin{equation} \label{e:0804250018}
\frac{\pr^2\phi_+}{\pr z_j\pr\ol w_k}(0,0)=2i\lambda_j\delta_{j,k},\ \ q+1\leq j, k\leq n-1.
\end{equation}
Since $\phi_+(x, x)=0$, we have
\begin{equation} \label{e:finalequation1}
\frac{\pr^2\phi_+}{\pr x_{2n-1}\pr x_{2n-1}}(0,0)+2\frac{\pr^2\phi_+}{\pr x_{2n-1}\pr y_{2n-1}}(0,0)+
\frac{\pr^2\phi_+}{\pr y_{2n-1}\pr y_{2n-1}}(0,0)=0.
\end{equation}
Combining (\ref{e:0804242212}), (\ref{e:0807150845}), (\ref{e:0804242312}), (\ref{e:0807150911}),
(\ref{e:0804242330}), (\ref{e:0804242345}), (\ref{e:0804250017}),
(\ref{e:0804250018}) and (\ref{e:finalequation1}), we completely determine the tangential Hessian of $\phi_+(x, y)$ at $(p, p)$.

\begin{thm} \label{t:phasemainplus}
With the notations used before, in some small neighborhood of $(p, p)\in X\times X$, we have
\begin{align} \label{e:t-final*****}
&\phi_+(x, y)=\sqrt{2}(x_{2n-1}-y_{2n-1})+i\sum^{n-1}_{j=1}\abs{\lambda_j}\abs{z_j-w_j}^2
+\sum^{n-1}_{j=1}\Bigr(i\lambda_j(z_j\ol w_j-\ol z_jw_j)\nonumber \\
&\quad+c_j(z_jx_{2n-1}-w_jy_{2n-1})+\ol c_j(\ol z_jx_{2n-1}-\ol w_jy_{2n-1})\Bigr) \nonumber \\
&\quad+\sqrt{2}(x_{2n-1}-y_{2n-1})f(x, y)+O(\abs{(x, y)}^3),\nonumber \\
&\quad f\in C^\infty,\ \ f(0,0)=0,\ \ f(x, y)=\ol f(y, x).
\end{align}
\end{thm}

We have the classical formulas
\begin{equation} \label{e:0}
\int^\infty_0e^{-tx}t^mdt= \left\{ \begin{array}{ll}
m!x^{-m-1}, &\mbox{if } m\in\integer, m\geq0  \\
\frac{(-1)^m}{(-m-1)!}x^{-m-1}(\log x+c-\sum^{-m-1}_1\frac{1}{j}), &\mbox{if } m\in\integer, m<0
\end{array}\right..
\end{equation}
Here $x\neq0$, ${\rm Re\,}x\geq0$ and $c$ is the Euler constant, i.e.
$c=\lim_{m\To\infty}(\sum^m_1\frac{1}{j}-\log m)$.
Note that
\begin{align} \label{e:0809061310}
&\int^{\infty}_0e^{i\phi_+(x, y)t}\sum^\infty_{j=0}s^j_+(x, y)t^{n-1-j}dt\nonumber  \\
&=\lim_{\varepsilon\to0+}
\int^{\infty}_0e^{-(-i(\phi_+(x, y)+i\varepsilon)t)}\sum^\infty_{j=0}s^j_+(x, y)t^{n-1-j}dt.
\end{align}
We have the following corollary of Theorem~\ref{t:szegomain}

\begin{cor} \label{c:1}
There exist
\[F_+, G_+, F_-, G_-\in C^\infty(X\times X;\ \mathscr L(\Lambda^{0, q}T^*_y(X), \Lambda^{0, q}T^*_x(X)))\]
such that
\[K_{\pi^{(q)}_+}=F_+(-i(\phi_+(x, y)+i0))^{-n}+G_+\log(-i(\phi_+(x, y)+i0)),\]
\[K_{\pi^{(q)}_-}=F_-(-i(\phi_-(x, y)+i0))^{-n}+G_-\log(-i(\phi_-(x, y)+i0)).\]

Moreover, we have
\begin{align} \label{e:000}
&F_+=\sum^{n-1}_0(n-1-k)!s^k_+(x, y)(-i\phi_+(x, y))^k+f_+(x, y)(\phi_+(x, y))^n,\nonumber \\
&F_-=\sum^{n-1}_0(n-1-k)!s^k_-(x, y)(-i\phi_-(x, y))^k+f_-(x, y)(\phi_-(x, y))^n,\nonumber \\
&G_+\equiv\sum^\infty_0\frac{(-1)^{k+1}}{k!}s^{n+k}_+(x, y)(-i\phi_+(x, y))^k,\nonumber \\
&G_-\equiv\sum^\infty_0\frac{(-1)^{k+1}}{k!}s^{n+k}_-(x, y)(-i\phi_-(x, y))^k,
\end{align}
where $s^k_+$, $k=0,1,\ldots$, are as in (\ref{e:0709232340}) and
\[f_+(x, y), f_-(x, y)\in C^\infty(X\times X;\, \mathscr L(\Lambda^{0,q}T^*_y(X), \Lambda^{0,q}T^*_x(X))).\]
\end{cor}

In the rest of this section, we assume that $q=n_+$. We will compute the leading term of $K_{\pi^{(q)}_+}$.
For a given point $p\in X$, let $x=(x_1, x_2, \ldots, x_{2n-1})$ be the local coordinates
as in Theorem~\ref{t:phasemainplus}. We recall that $w_0(p)=\sqrt{2}dx_{2n-1}$, $x(p)=0$,
$\Lambda^{1, 0}T_p(X)\oplus\Lambda^{0, 1}T_p(X)=\set{\sum_j^{2n-2}a_j\frac{\pr}{\pr x_j};\ a_j\in\Complex}$,
$(\frac{\pr}{\pr x_j}(p)\ |\ \frac{\pr}{\pr x_k}(p))=2\delta_{j,k}$, $j, k=1,\ldots,2n-1$.
We identify $\Omega$ with some open set in $\Real^{2n-1}$.
We represent the Hermitian inner product $(\ |\ )$ on $\Complex T(X)$ by
$(u\ |\ v)=\seq{Hu, \ol v}$,
where $u, v\in\Complex T(X)$, $H$ is a positive definite Hermitian matrix and $\seq{, }$ is given by
$\seq{s, t}=\sum^{2n-1}_{j=1}s_jt_j$, $s=(s_1,\ldots,s_{2n-1})\in\Complex^{2n-1}$,
$t=(t_1,\ldots,t_{2n-1})\in\Complex^{2n-1}$. Here we identify $\Complex T(X)$ with $\Complex^{2n-1}$.
Let $h(x)$ denote the determinant of $H$. The induced volume form on $X$ is given by $\sqrt{h(x)}dx$. We have
$h(p)=2^{2n-1}$.
Now,
\begin{align*}
&(K_{\pi^{(q)}_+}\circ K_{\pi^{(q)}_+})(x, y)\equiv \\
&\quad\int^\infty_0\!\int^\infty_0\!\Bigl(\!\int e^{it\phi_+(x, w)+is\phi_+(w, y)}s_+(x, w, t)s_+(w, y, s)\sqrt{h(w)}dw
\Bigl)dtds.
\end{align*}
Let $s=t\sigma$, we get
\begin{align*}
&(K_{\pi^{(q)}_+}\circ K_{\pi^{(q)}_+})(x, y) \\
&\quad\equiv\int^\infty_0\!\!\!\int^\infty_0\!\!\!\Bigl(\!\int\!\!\!
e^{it\phi(x, y, w, \sigma)}s_+(x, w, t)s_+(w, y, t\sigma)t\sqrt{h(w)}dw\Bigl)d\sigma dt,
\end{align*}
where $\phi(x, y, w, \sigma)=\phi_+(x, w)+\sigma\phi_+(w, y)$.
It is easy to see that
\[{\rm Im\,}\phi(x, y, w, \sigma)\geq0,\ \ d_w\phi(x, y, w, \sigma)|_{x=y=w}=(\sigma-1)\omega_0(x).\]
Thus, $x=y=w$, $\sigma=1$, $x$ is real, are real critical points.

Now, we will compute the Hessian of $\phi$ at $x=y=w=p$, $p$ is real, $\sigma=1$. We write $H_\phi(p)$ to denote the
Hessian of $\phi$ at $x=y=w=p$, $p$ is real, $\sigma=1$. $H_\phi(p)$ has the following form:
$H_\phi(p)=\left[
\begin{array}[c]{cc}
  0 & {^t\hskip-1pt}(\phi_+)'_x \\
  (\phi_+)'_x & (\phi_+)''_{xx}+(\phi_+)''_{yy}
\end{array}\right]$.
Since $(\phi_+)'_x(p)=\omega_0(p)=\sqrt{2}dx_{2n-1}$,
we have
$H_\phi(p)=\left[
\begin{array}[c]{ccc}
  0, & 0,\ldots,0,&\sqrt{2} \\
   \vdots & A,&* \\
  \sqrt{2}&*&*
\end{array}\right]$,
where $A$ is the linear map
\begin{align*}
&A:\Lambda^{1,0}T_p(X)\oplus\Lambda^{0,1}T_p(X)\To\Lambda^{1,0}T_p(X)\oplus\Lambda^{0,1}T_p(X),\\
&\seq{Au, v}=\seq{((\phi_+)''_{xx}+(\phi_+)''_{yy})u, v},\ \forall\
u, v\in\Lambda^{1,0}T_p(X)\oplus\Lambda^{0,1}T_p(X).
\end{align*}
From (\ref{e:t-final*****}), it follows that $A$ has the eigenvalues:
\begin{equation}
4i\abs{\lambda_1(p)}, 4i\abs{\lambda_1(p)},\ldots,4i\abs{\lambda_{n-1}(p)}, 4i\abs{\lambda_{n-1}(p)},
\end{equation}
where $\lambda_j(p)$, $j=1,\cdots,(n-1)$, are the eigenvalues of the Levi form $L_p$. We have,
\begin{align} \label{e:070404-1}
{\rm det}(\frac{H_\phi(p)}{i})=2^{4n-3}\abs{\lambda_1(p)}^2\cdots\abs{\lambda_{n-1}(p)}^2.
\end{align}

From the stationary phase formula (see Proposition~\ref{p:0709171439}), we get
\[(K_{\pi^{(q)}_+}\circ K_{\pi^{(q)}_+})(x, y)\equiv\int^\infty_0e^{it\phi_1(x, y)}a(x, y, t)dt,\]
where \[a(x, y, t)\sim\sum^\infty_{j=0}a_j(x, y)t^{n-1-j}\]
in the symbol space $S^{n-1}_{1, 0}(\Omega\times\Omega\times[0, \infty[;\, \mathscr L(\Lambda^{0, q}T^*_y(X), \Lambda^{0, q}T^*_x(X)))$,
$a_j(x, y)\in C^\infty(\Omega\times\Omega;\, \mathscr L(\Lambda^{0,q}T^*(X),\Lambda^{0,q}T^*(X)))$, $j=0,1,\ldots$,
and $\phi_1(x, y)$ is the corresponding critical value. Moreover, we have
\begin{align} \label{e:070404-2}
a_0(p, p)&=\left({\rm det}\frac{H_\phi(p)}{2\pi i}\right)^{-\frac{1}{2}}s^0_+(p, p)\circ s^0_+(p, p)\sqrt{h(p)}\nonumber \\
         &=2\abs{\lambda_1(p)}^{-1}\cdots\abs{\lambda_{n-1}(p)}^{-1}\pi^n s^0_+(p, p)\circ s^0_+(p, p),
\end{align}
where $s^0_+$ is as in (\ref{e:0709232340}). We notice that
\begin{equation} \label{e:0709232333}
\phi_1(x, x)=0,\ \ (\phi_1)'_x(x, x)=(\phi_+)'_x(x, x),\ \ (\phi_1)'_y(x, x)=(\phi_+)'_y(x, x).
\end{equation}

From (\ref{e:0}) and (\ref{e:0809061310}), it follows that
\begin{align} \label{e:0710272245}
(K_{\pi^{(q)}_+}\circ K_{\pi^{(q)}_+})(x, y)&\equiv F_1(-i(\phi_1(x, y)+i0))^{-n}+G_1\log(-i(\phi_1(x, y)+i0))\nonumber \\
                                &\equiv F_+(-i(\phi_+(x, y)+i0))^{-n}+G_+\log(-i(\phi_+(x, y)+i0)),
\end{align}
where
\[F_1=\sum^{n-1}_0(n-1-k)!a_j(-i\phi_1)^k+f_1\phi_1^n,\]
$f_1$, $G_1\in C^\infty(\Omega\times\Omega;\, \mathscr L(\Lambda^{0, q}T^*_y(X), \Lambda^{0, q}T^*_x(X)))$, $F_+$ and $G_+$ are as in
Corollary~\ref{c:1}.
From (\ref{e:0709232333}) and (\ref{e:0710272245}), we see that
$s^0_+(p, p)=a_0(p, p)$.
From this and (\ref{e:070404-2}), we get
\begin{equation} \label{e:0710272223}
2\abs{\lambda_1(p)}^{-1}\cdots\abs{\lambda_{n-1}(p)}^{-1}\pi^n s^0_+(p, p)\circ s^0_+(p, p)=s^0_+(p, p).
\end{equation}

Let $\mathscr N_x(p^s_0+\frac{1}{2}\Td{\rm tr\,}F)=\set{u\in\Lambda^{0, q}T^*_x(X);\,
(p^s_0+\frac{1}{2}\Td{\rm tr\,}F)(x, \omega_0(x))u=0}$,
where $p^s_0$ is the subprincipal symbol of $\Box^{(q)}_b$ and $F$ is the fundamental matrix of $\Box^{(q)}_b$.
From the asymptotic expansion of $\Box^{(q)}_b(e^{i\phi_+}s_+)$, we see that
$s^0_+(p, p)u\in\mathscr N_p(p^s_0+\frac{1}{2}\Td{\rm tr\,}F)$
for all $u\in\Lambda^{0, q}T^*_p(X)$. (See section $5$.) Let
\[I_1=(\frac{1}{2}\abs{\lambda_1}\cdots\abs{\lambda_{n-1}})^{-1}\pi^ns^0_+(p, p).\]
From (\ref{e:0710272223}), we see that
\begin{equation} \label{e:0709061552}
I^2_1=I_1.
\end{equation}
Since $K_{(\pi^{(q)}_+)^*}\equiv K_{\pi^{(q)}_+}$
and
$\phi_+(x, y)=-\ol\phi_+(y, x)$,
we have $(s^0_+)^*(p, p)=s^0_+(p, p)$
and hence
\begin{equation} \label{e:0709061603}
I_1^*=I_1,
\end{equation}
where $(\pi^{(q)}_+)^*$ is the adjoint of $\pi^{(q)}_+$, $(s^0_+)^*(p, p)$, $I^*_1$ are the adjoints of $s^0_+(p, p)$ and $I_1$ in the space
$\mathscr L(\Lambda^{0, q}T^*_p(X), \Lambda^{0, q}T^*_p(X)))$
with respect to $(\ |\ )$ respectively. Note that
${\rm dim\,}\mathscr N_p(p^s_0+\frac{1}{2}\Td{\rm tr\,}F)=1$.
(See section $3$.) Combining this with (\ref{e:0709061552}), (\ref{e:0709061603}) and $s^0_+(p,p)\neq0$,
it follows that
$I_1:\Lambda^{0, q}T^*_p(X)\To\Lambda^{0, q}T^*_p(X)$
is the orthogonal projection onto $\mathscr N_p(p^s_0+\frac{1}{2}\Td{\rm tr\,}F)$.

For a given point $p\in X$, let
$\ol Z_1(x),\ldots,\ol Z_{n-1}(x)$
be an orthonormal frame of $\Lambda^{1,0}T_x(X)$, for which
the Levi form is diagonalized at $p$. Let $e_j(x)$, $j=1,\ldots,n-1$
denote the basis of $\Lambda^{0,1}T^*_x(X)$, which is dual to $Z_j(x)$, $j=1,\ldots,n-1$. Let
$\lambda_j(x)$, $j=1,\ldots,n-1$ be the eigenvalues of the Levi form $L_x$. We assume that
$\lambda_j(p)>0$ if $1\leq j\leq n_+$.
Then $I_1=\prod_{j=1}^{j=n_+}e_j(p)^\wedge e_j^{\wedge, *}(p)$ at $p$.
(See section $3$.) Summing up, we have proved

\begin{prop} \label{p:leading1}
For a given point $p\in X$, let
$\ol Z_1(x),\ldots,\ol Z_{n-1}(x)$
be an orthonormal frame of $\Lambda^{1,0}T_x(X)$, for which
the Levi form is diagonalized at $p$. Let $e_j(x)$, $j=1,\ldots,n-1$
denote the basis of $\Lambda^{0,1}T^*_x(X)$, which is dual to $Z_j(x)$, $j=1,\ldots,n-1$. Let
$\lambda_j(x)$, $j=1,\ldots,n-1$ be the eigenvalues of the Levi form $L_x$. We assume that $q=n_+$ and that
$\lambda_j(p)>0$ if\, $1\leq j\leq n_+$. Then
\[F_+(p, p)=(n-1)!\frac{1}{2}\abs{\lambda_1(p)}\cdots
\abs{\lambda_{n-1}(p)}\pi^{-n}\prod_{j=1}^{j=n_+}e_j(p)^\wedge e_j(p)^{\wedge, *}.\]
\end{prop}


\section{The Szeg\"o projection on non-orientable CR manifolds}

In this section, $(X, \Lambda^{1, 0}T(X))$ is a compact connected not necessarily orientable
CR manifold of dimension $2n-1$, $n\geq2$. We will use the same notations as before.
The definition of the Levi form (see Definition~\ref{d:d-Leviform}) depends on the choices of $\omega_0$. However,
the number of non-zero eigenvalues is independent of the choices of $\omega_0$. Thus,
it makes sense to say that the Levi-form is non-degenerate.
As before, we assume that the Levi form $L$ is non-degenerate at each point of $X$. We have the following

\begin{lem} \label{l:no1}
Let $(n_-, n_+)$, $n_-+n_+=n-1$,
be the signature of the Levi-form $L$. (The signature of the Levi-form $L$
depends on the choices of $\omega_0$.) If $n_-\neq n_+$ at a point of\, $X$, then $X$ is orientable.
\end{lem}

\begin{proof}
Since $X$ is connected, $n_-\neq n_+$ at a point of $X$ implies
$n_-\neq n_+$ at each point of $X$. Let $X=\bigcup U_j$, where $U_j$ is a
local coordinate patch of $X$. On $U_j$, we can choose an orthonormal frame
$\omega_{1,j}(x),\ldots,\omega_{n-1,j}(x)$
for $\Lambda^{1,0}T^*_x(U_j)$, then
$\ol\omega_{1,j}(x),\ldots,\ol\omega_{n-1,j}(x)$
is an orthonormal frame for $\Lambda^{0,1}T^*_x(U_j)$. The $(2n-2)$-form
\[\omega_j=i^{n-1}\omega_{1,j}\wedge\ol\omega_{1,j}\wedge\cdots\wedge\omega_{n-1,j}\wedge\ol\omega_{n-1,j}\]
is real and is independent of the choice of the orthonormal frame.
There is a real $1$-form $\omega_{0,j}(x)$ of length one which is orthogonal to
$\Lambda^{1,0}T^*_x(U_j)\oplus\Lambda^{0,1}T^*_x(U_j)$. We take $\omega_{0,j}$ so that $n_-<n_+$ on $U_j$. Since
$\omega_{0,j}$ is unique up to sign and $n_-<n_+$ on $U_j$, for all $j$, we have
$\omega_{0,j}(x)=\omega_{0,k}(x)$ on $U_j\bigcap U_k$,
so $\omega_0$ is globally defined. The lemma follows.
\end{proof}

We only need to consider the case $n_-=n_+$. We recall that if $n_-=n_+$ then $\Box^{(q)}_b$ has closed range. In view of the
proof of Theorem~\ref{t:szegomain}, we have the following

\begin{thm} \label{t:szegomain-nonorientable}
Let $(n_-, n_+)$ be the signature of the Levi form. We assume that $q=n_-=n_+$. Put
$\hat\Sigma=\set{(x, \xi)\in T^*(X)\setminus0;\, \xi=\lambda\omega_0(x), \lambda\neq0}$,
where $\omega_0$ is the locally unique real $1$ form determined up to sign by
$\norm{\omega_0}=1$,
\[\omega_0\bot(\Lambda^{0,1}T^*(X)\oplus\Lambda^{1, 0}T^*(X)).\]
Then $\pi^{(q)}:H^s_{{\rm loc\,}}(X;\, \Lambda^{0, q}T^*(X))\To H^s_{{\rm loc\,}}(X;\, \Lambda^{0, q}T^*(X))$ is continuous,
for all $s\in\Real$ and ${\rm WF}'(K_{\pi^{(q)}})={\rm diag\,}(\hat{\Sigma}\times\hat{\Sigma})$,
where
\[{\rm WF}'(K_{\pi^{(q)}})=\set{(x, \xi, y, \eta)\in T^*(X)\times T^*(X);\, (x, \xi, y, -\eta)\in{\rm WF}(K_{\pi^{(q)}})}.\]
Here ${\rm WF}(K_{\pi^{(q)}})$ is the wave front set of $K_{\pi^{(q)}}$ in the sense of H\"{o}rmander~\cite{Hor71}.

For every local coordinate patch $U$,
we fix a $\omega_0$ on $U$. We
define
\begin{align*}
&\Sigma^+=\set{(x, \xi)\in T^*(U)\setminus0;\, \xi=\lambda\omega_0(x), \lambda>0},  \\
&\Sigma^-=\set{(x, \xi)\in T^*(U)\setminus0;\, \xi=\lambda\omega_0(x), \lambda<0}.
\end{align*}
We have $K_{\pi^{(q)}}=K_{\pi^{(q)}_+}+K_{\pi^{(q)}_-}$ on $U\times U$,
where $K_{\pi^+}(x, y)$ satisfies
\begin{align*}
K_{\pi^{(q)}_+}(x, y)\equiv\int^{\infty}_{0} e^{i\phi_+(x, y)t}s_+(x, y, t)dt\ \ \mbox{on}\ U\times U
\end{align*}
with
\[s_+(x, y, t)\in S^{n-1}_{1, 0}(U\times U\times]0, \infty[;\, \mathscr L(\Lambda^{0,q}T^*_y(X), \Lambda^{0,q}T^*_x(X))),\]
\[s_+(x, y, t)\sim\sum^\infty_{j=0}s^j_+(x, y)t^{n-1-j}\]
in the symbol space $S^{n-1}_{1, 0}(U\times U\times]0, \infty[;\, \mathscr L(\Lambda^{0,q}T^*_y(X), \Lambda^{0,q}T^*_x(X)))$,
where
\[s^j_+(x, y)\in C^\infty(U\times U;\, \mathscr L(\Lambda^{0, q}T^*_y(X), \Lambda^{0, q}T^*_x(X))),\ j=0,1,\ldots,\]
\[\phi_+(x, y)\in C^\infty(U\times U),\ \ {\rm Im\,}\phi_+(x, y)\geq0,\]
\[\phi_+(x, x)=0,\ \ \phi_+(x, y)\neq0\ \ \mbox{if}\ \ x\neq y,\]
\[d_x\phi_+\neq0,\ \ d_y\phi_+\neq0\ \ \mbox{where}\ \ {\rm Im\,}\phi_+=0,\]
\[d_x\phi_+(x, y)|_{x=y}=\omega_0(x),\ \ d_y\phi_+(x, y)|_{x=y}=-\omega_0(x),\]
\[\phi_+(x, y)=-\ol\phi_+(y, x).\]
Moreover, $\phi_+(x, y)$ satisfies (\ref{e:t-final*****}). Similarly,
\begin{align*}
K_{\pi^{(q)}_-}(x, y)\equiv\int^{\infty}_{0} e^{i\phi_-(x, y)t}s_-(x, y, t)dt
\end{align*}
with
\[s_-(x, y, t)\sim\sum^\infty_{j=0}s^j_-(x, y)t^{n-1-j}\]
in the symbol space $S^{n-1}_{1, 0}(U\times U\times]0, \infty[;\, \mathscr L(\Lambda^{0,q}T^*_y(X), \Lambda^{0,q}T^*_x(X)))$,
where
\[s^j_-(x, y)\in C^\infty(U\times U;\, \mathscr L(\Lambda^{0, q}T^*_y(X), \Lambda^{0, q}T^*_x(X))),\ j=0,1,\ldots,\]
\[\phi_-(x, y)=-\ol\phi_+(x, y).\]
\end{thm}

\newpage
\addcontentsline{toc}{part}{Part II:\ \ On the singularities of the Bergman projection for $(0, q)$ forms}
\part*{\mbox{} \\ \mbox{} \\ \mbox{} \mbox{} \\ \mbox{} \\ \mbox{} \mbox{} \\ \mbox{} \\ \mbox{}
\quad\quad\quad\quad\quad\quad\quad\quad Part {\rm II\,} \\
\mbox{} \\ \mbox{} \\ \mbox{} \\
\quad\quad On the singularities of the Bergman\\ \quad\quad projection for $(0, q)$ forms}
\mbox{}\newpage


\setcounter{section}{0}
\section{Introduction and statement of the main results}

In this paper, we assume that all manifolds are paracompact. (For the precise definition, see page $156$ of Kelley~\cite{K75}.)
Let $M$ be a relatively compact open subset with $C^\infty$ boundary $\Gamma$ of a
complex manifold $M'$ of dimension $n$ with a smooth Hermitian metric $(\ |\ )$ on its holomorphic tangent bundle. (See (\ref{Be:0712261513}).)
The Hermitian metric induces a Hermitian metric on the bundle of $(0, q)$ forms of $M'$
(see the discussion after (\ref{Be:0712261513}) and section $2$)
and a positive density $(dM')$(see (\ref{Be:0803031608})).
Let $\Box$ be the $\dbar$-Neumann Laplacian on $M$ (see Folland-Kohn~\cite{FK72} or (\ref{Be:0710301300})) and let $\Box^{(q)}$
denote the restriction to $(0, q)$ forms. For $p\in\Gamma$, let $L_p$ be the Levi form of $\Gamma$ at $p$
(see (\ref{Be:0710301253}) or Definition~\ref{Bd:BK-Leviform}). Given $q$, $0\leq q\leq n-1$, the Levi form
is said to satisfy condition $Z(q)$ at $p\in \Gamma$ if it has at least $n-q$ positive or at least $q+1$
negative eigenvalues. When condition $Z(q)$ holds at each point of $\Gamma$, Kohn's $L^2$ estimates give the
hypoellipicity with loss of one dervative for the solutions
of $\Box^{(q)}u=f$. (See \cite{FK72} or Theorem~\ref{Bt:0710301305}.)
The Bergman projection is the
orthogonal projection onto the kernel of $\Box^{(q)}$ in the $L^2$ space. When condition $Z(q)$ fails at some point of $\Gamma$, one is
interested in the Bergman projection on the level of $(0, q)$ forms. When $q=0$ and the Levi form is positive definite,
the existence of the complete asymptotic expansion of the singularities of the
Bergman projection was
obtained by Fefferman~\cite{Fer74} on the diagonal and subsequently by Boutet de Monvel-Sj\"{o}strand~(see \cite{BS76}) in complete generality.
If $q=n-1$ and the Levi form is
negative definite, H\"{o}rmander~\cite{Hor04} obtained the corresponding asymptotics for the Bergman projection in the distribution sense. We have
been influenced by these works.

We now start to formulate the main results. First, we introduce some standard notations. Let $\Omega$ be a
$C^\infty$ manifold equipped with a smooth density of integration.
We let $T(\Omega)$ and $T^*(\Omega)$ denote the tangent bundle of $\Omega$ and the cotangent bundle of $\Omega$ respectively.
The complexified tangent bundle of $\Omega$ and the complexified cotangent bundle of $\Omega$ will be denoted by $\Complex T(\Omega)$
and $\Complex T^*(\Omega)$ respectively. We write $\seq{\ ,}$ to denote the pointwise duality between $T(\Omega)$ and $T^*(\Omega)$.
We extend $\seq{\ ,}$ bilinearly to $\Complex T(\Omega)\times\Complex T^*(\Omega)$. Let $E$ be a $C^\infty$ vector
bundle over $\Omega$. The fiber of $E$ at $x\in\Omega$ will be denoted by $E_x$.
Let $Y\subset\subset\Omega$ be an open set. The spaces of
smooth sections of $E$ over $Y$ and distribution sections of $E$ over $Y$ will be denoted by $C^\infty(Y;\, E)$ and $\mathscr D'(Y;\, E)$ respectively.
Let $\mathscr E'(Y;\, E)$ be the subspace of $\mathscr D'(Y;\, E)$ of sections with compact support in $Y$.
For $s\in\Real$, we let $H^s(Y;\, E)$ denote the Sobolev space
of order $s$ of sections of $E$ over $Y$. Put
\[H^s_{\rm loc\,}(Y;\, E)=\set{u\in\mathscr D'(Y;\, E);\, \varphi u\in H^s(Y;\, E),\;
      \forall\varphi\in C^\infty_0(Y)}\]
and $H^s_{\rm comp\,}(Y;\, E)=H^s_{\rm loc}(Y;\, E)\bigcap\mathscr E'(Y;\, E)$.

Let $F$ be a $C^\infty$ vector bundle over $M'$. Let $C^\infty(\ol M;\, F)$, $\mathscr D'(\ol M;\, F)$, $H^s(\ol M;\, F)$ denote the spaces of
restrictions to $M$ of elements in $C^\infty(M';\, F)$, $\mathscr D'(M';\, F)$ and $H^s(M';\, F)$ respectively. Let
$C^\infty_0(M;\, F)$ be the subspace of $C^\infty(\ol M;\, F)$ of sections with compact support in $M$.

Let $\Lambda^{1,0}T(M')$ and $\Lambda^{0,1}T(M')$ be the holomorphic tangent bundle of $M'$ and the
anti-holomorphic tangent boundle of $M'$ respectively. (See (\ref{Be:0803312107}).)
In local coordinates $z=(z_1,\ldots,z_n)$, we represent the Hermitian metric on $\Lambda^{1,0}T(M')$ by
\begin{align} \label{Be:0712261513}
&(u\ |\ v)=g(u, \ol v),\ \ u, v\in\Lambda^{1,0}T(M'), \nonumber \\
&g=\sum^{n}_{j, k=1}g_{j, k}(z)dz_j\otimes d\ol z_k,
\end{align}
where $g_{j, k}(z)=\ol g_{k, j}(z)\in C^\infty,\ j, k=1,\ldots,n,$ and $\left(g_{j, k}(z)\right)^n_{j, k=1}$ is positive definite at each point.
We extend the Hermitian metric $(\ |\ )$ to $\Complex T(M')$ in a natural way by requiring $\Lambda^{1,0}T(M')$ to be orthogonal to
$\Lambda^{0,1}T(M')$ and satisfy $\ol{(u\ |\ v)}=(\ol u\ |\ \ol v)$, $u, v\in\Lambda^{0, 1}T(M')$.

The Hermitian metric $(\ |\ )$ on $\Complex T(M)$ induces, by duality,
a Hermitian metric on $\Complex T^*(M)$ that we shall also denote by $(\ |\ )$. (See (\ref{Be:CR-Gamma2}).) For $q\in\Pstint$,
let $\Lambda^{0, q}T^*(M')$ be the bundle of $(0, q)$ forms of $M'$. (See (\ref{Be:0710301340}).)
The Hermitian metric $(\ |\ )$ on $\Complex T^*(M')$ induces a Hermitian metric on
$\Lambda^{0, q}T^*(M')$ also denoted by $(\ |\ )$. (See (\ref{Be:0710301345}).)

Let $r\in C^\infty(M')$ be a defining
function of $\Gamma$ such that $r$ is real, $r=0$ on $\Gamma$, $r<0$ on $M$ and $dr\neq0$ near $\Gamma$.
From now on, we take a defining function $r$ so that $\norm{dr}=1$ on $\Gamma$.

The Hermitian metric $(\ |\ )$ on $\Complex T(M')$ induces a Hermitian metric $(\ |\ )$ on $\Complex T(\Gamma)$.
For $z\in\Gamma$, we identify $\Complex T^*_z(\Gamma)$ with the space
\begin{equation} \label{Be:0803111619}
\set{u\in\Complex T^*_z(M');\, (u\ |\ dr)=0}.
\end{equation}

We associate to the Hermitian metric $\sum^{n}_{j, k=1}g_{j, k}(z)dz_j\otimes d\ol z_k$ a real $(1, 1)$ form ( see page $144$ of Kodaira~\cite{K86})
$\omega=i\sum^n_{j, k=1}g_{j, k}dz_j\wedge d\ol z_k$.
Let
\begin{equation} \label{Be:0803031608}
dM'=\frac{\omega^n}{n!}
\end{equation}
be the volume element and let $(\ |\ )_M$ be the inner product on the space $C^\infty(\ol M;\, \Lambda^{0,q}T^*(M'))$
defined by
\begin{equation} \label{Be:0803032229}
(f\ |\ h)_M=\int_M(f\ |\ h)(dM')=\int_M(f\ |\ h)\frac{\omega^n}{n!},\ \ f, h\in C^\infty(\ol M;\, \Lambda^{0,q}T^*(M')).
\end{equation}
Similarly, we take $(d\Gamma)$ as the induced volume form on $\Gamma$
and let $(\ |\ )_\Gamma$ be the inner product on $C^\infty(\Gamma;\, \Lambda^{0,q}T^*(M'))$
defined by
\begin{equation} \label{Be:0710312307}
(f\ |\ g)_\Gamma=\int_\Gamma(f\ |\ g)d\Gamma,\ \ f,g\in C^\infty(\Gamma;\, \Lambda^{0,q}T^*(M')).
\end{equation}

Let $\dbar:C^\infty(M';\, \Lambda^{0,q}T^*(M'))\To C^\infty(M';\, \Lambda^{0,q+1}T^*(M'))$
be the part of the exterior differential operator which maps forms of type $(0,q)$ to forms of
type $(0,q+1)$ and we denote by
$\ol{\pr_f}^*: C^\infty(M';\, \Lambda^{0,q+1}T^*(M'))\To C^\infty(M';\, \Lambda^{0,q}T^*(M'))$
the formal adjoint of $\dbar$. That is
\[(\dbar f\ |\ h)_{M'}=(f\ |\ \ol{\pr_f}^*h)_{M'},\]
$f\in C^\infty_0(M';\, \Lambda^{0,q}T^*(M'))$, $h\in C^\infty(M';\, \Lambda^{0,q+1}T^*(M'))$,
where $(\ |\ )_{M'}$ is defined by
$(g\ |\ k)_{M'}=\int_{M'}(g\ |\ k)(dM')$, $g, k\in C^\infty_0(M';\, \Lambda^{0,q}T^*(M'))$.
We shall also use the notation $\dbar$ for the closure in $L^2$ of the $\dbar$ operator, initially defined on
$C^\infty(\ol M;\, \Lambda^{0,q}T^*(M'))$ and $\dbar^*$  for the Hilbert space adjoint of $\dbar$.

The $\dbar$-Neumann Laplacian on $(0, q)$ forms is then the self-adjoint operator in the space $L^2(M;\, \Lambda^{0,q}T^*(M'))$
(see chapter ${\rm I\,}$ of~\cite{FK72})
\begin{equation} \label{Be:0710301300}
\Box^{(q)}=\ol{\pr}\ \ol{\pr}^*+\dbar^*\dbar.
\end{equation}
We notice that
\begin{align} \label{Be:0710301635}
{\rm Dom\,}\Box^{(q)}&=\{u\in L^2(M;\, \Lambda^{0,q}T^*(M'));\, u\in{\rm Dom\,}\dbar^*\bigcap{\rm Dom\,}\dbar,\nonumber \\
&\quad\dbar^*u\in{\rm Dom\,}\dbar,\ \dbar u\in{\rm Dom\,}\dbar^*\}
\end{align}
and $C^\infty(\ol M;\, \Lambda^{0,q}T^*(M'))\bigcap{\rm Dom\,}\Box^{(q)}$ is dense in ${\rm Dom\,}\Box^{(q)}$ for the norm
\[u\in{\rm Dom\,}\Box^{(q)}\To \norm{u}+\norm{\dbar u}+\norm{\ol{\pr}^*u}.\]
(See also page $14$ of~\cite{FK72}.)

Let
$\Box^{(q)}_f=\dbar\ \ol{\pr_f}^*+\ol{\pr_f}^*\dbar: C^\infty(M';\, \Lambda^{0,q}T^*(M'))\To C^\infty(M';\, \Lambda^{0,q}T^*(M'))$
denote the complex Laplace-Beltrami operator on $(0, q)$ forms and denote by $\sigma_{\Box^{(q)}_f}$ the principal symbol of $\Box^{(q)}_f$.

Let $\frac{\pr}{\pr r}$ be the dual vector of $dr$. That is
$(u\ |\ \frac{\pr}{\pr r})=\seq{u, dr}$,
for all $u\in\Complex T(M')$. Put
\begin{equation} \label{Be:0710301720}
\omega_0=J^t(dr),
\end{equation}
where $J^t$ is the complex structure map for the cotangent bundle.
(See (\ref{Be:0710301715}).)

Let $\Lambda^{1,0}T(\Gamma)$ and $\Lambda^{0,1}T(\Gamma)$ be the holomorphic tangent bundle of $\Gamma$ and
the anti-holomorphic tangent bundle of $\Gamma$ respectively. (See (\ref{Be:0803312225}).)
The Levi form $L_p(Z, \ol W)$, $p\in\Gamma$, $Z$, $W\in\Lambda^{1,0}T_p(\Gamma)$,
is the Hermitian quadratic form on $\Lambda^{1,0}T_p(\Gamma)$ defined as follows:
\begin{equation} \label{Be:0710301253} \begin{split}
&\mbox{For any $Z$, $W\in \Lambda^{1,0}T_p(\Gamma)$, pick $\Td Z$, $\Td W\in
C^\infty(\Gamma;\, \Lambda^{1,0}T(\Gamma))$ that satisfy}  \\
&\mbox{$\Td Z(p)=Z$, $\Td W(p)=W$. Then }
L_p(Z,\ol W)=\frac{1}{2i}\seq{[\Td Z\ ,\ol{\Td W}](p)\ ,
\omega_0(p)}.
\end{split}
\end{equation}
The eigenvalues of the Levi form at $p\in \Gamma$ are the eigenvalues of the Hermitian form $L_p$ with
respect to the inner product $(\ |\ )$ on $\Lambda^{1, 0}T_p(\Gamma)$. If the Levi form is non-degenerate at $p\in \Gamma$,
let $(n_-,n_+)$, $n_-+n_+=n-1$, be the signature of $L_p$. Then $Z(q)$ holds at $p$ if and only if $q\neq n_-$.

We recall the H\"{o}rmander symbol spaces

\begin{defn} \label{Bd:0712101500}
Let $m\in\Real$. Let $U$ be an open set in $M'\times M'$.
\[S^{m}_{1, 0}(U\times]0, \infty[;\, \mathscr L(\Lambda^{0,q}T^*_y(M'), \Lambda^{0,q}T^*_x(M')))\]
is the space of all $a(x, y, t)\in C^\infty(U\times]0, \infty[;\, \mathscr L(\Lambda^{0,q}T^*_y(M'), \Lambda^{0,q}T^*_x(M')))$
such that for all
compact sets $K\subset U$ and all $\alpha\in\Pstint^{2n}$, $\beta\in\Pstint^{2n}$, $\gamma\in\Pstint$, there is a
constant $c>0$ such that
$\abs{\pr^\alpha_x\pr^\beta_y\pr^\gamma_t a(x, y, t)}\leq c(1+\abs{t})^{m-\abs{\gamma}}$,
$(x, y, t)\in K\times]0, \infty[$.
$S^m_{1, 0}$ is called the space of symbols of order $m$ type $(1, 0)$. We write $S^{-\infty}_{1, 0}=\bigcap S^m_{1, 0}$.

Let $S^{m}_{1, 0}(U\bigcap(\ol M\times\ol M)\times]0, \infty[;\, \mathscr L(\Lambda^{0,q}T^*_w(M'), \Lambda^{0,q}T^*_z(M')))$ denote the space of
restrictions to $U\bigcap(M\times M)\times]0, \infty[$ of elements in
\[S^{m}_{1, 0}(U\times]0, \infty[;\, \mathscr L(\Lambda^{0,q}T^*_w(M'), \Lambda^{0,q}T^*_z(M'))).\]
\end{defn}

Let
\[a_j\in S^{m_j}_{1, 0}(U\bigcap(\ol M\times\ol M)\times]0, \infty[;\, \mathscr L(\Lambda^{0,q}T^*_w(M'), \Lambda^{0,q}T^*_z(M'))),\ \ j=0,1,2,\ldots,\]
with $m_j\searrow -\infty$, $j\To \infty$.
Then there exists
\[a\in S^{m_0}_{1, 0}(U\bigcap(\ol M\times\ol M)\times]0, \infty[;\, \mathscr L(\Lambda^{0,q}T^*_w(M'), \Lambda^{0,q}T^*_z(M')))\]
such that
\[a-\sum_{0\leq j<k}a_j\in S^{m_k}_{1, 0}(U\bigcap(\ol M\times\ol M)\times]0, \infty[;\, \mathscr L(\Lambda^{0,q}T^*_w(M'), \Lambda^{0,q}T^*_z(M'))),\]
for every $k\in\Pstint$. (See Proposition~$1.8$ of Grigis-Sj\"{o}strand \cite{GS94} or H\"{o}rmander \cite{Hor71}.)

If $a$ and $a_j$ have the properties above, we write
\[a\sim\sum^\infty_{j=0}a_j\ \ {\rm in\,}\ \
S^{m_0}_{1, 0}(U\bigcap(\ol M\times\ol M)\times[0, \infty[;\, \mathscr L(\Lambda^{0,q}T^*_w(M'), \Lambda^{0,q}T^*_z(M'))).\]

Let
\[\Pi^{(q)}: L^2(M;\, \Lambda^{0,q}T^*(M'))\To {\rm Ker\,}\Box^{(q)}\]
be the Bergman projection, i.e.
the orthogonal projection onto the kernel of $\Box^{(q)}$. Let
$K_{\Pi^{(q)}}(z, w)\in\mathscr D'(M\times M;\, \mathscr L(\Lambda^{0,q}T^*_w(M'),\Lambda^{0,q}T^*_z(M')))$
be the distribution kernel of $\Pi^{(q)}$.

Let $X$ and $Y$ be $C^\infty$ vector bundles over $M'$. Let
\[C, D: C^\infty_0(M;\, X)\To \mathscr D'(M;\, Y)\]
with distribution kernels
$K_C(z, w), K_D(z, w)\in\mathscr D'(M\times M;\, \mathscr L(X_w, Y_z))$.
We write
$C\equiv D$ mod $C^\infty(U\bigcap(\ol M\times\ol M))$
if $K_C(z, w)=K_D(z, w)+F(z, w)$,
where $F(z, w)|_U\in C^\infty(U\bigcap(\ol M\times\ol M);\, \mathscr L(X_w, Y_z))$ and $U$ is an open set in $M'\times M'$.

Given $q$, $0\leq q\leq n-1$. Put
\begin{equation} \label{Be:0803111559}
\Gamma_q=\set{z\in\Gamma;\, Z(q)\ \mbox{fails at}\ z}.
\end{equation}
If the Levi form is non-degenerate at each point of $\Gamma$, then $\Gamma_q$ is a union of connected components of $\Gamma$.

The main result of this work is the following

\begin{thm} \label{Bt:i-BP-Bergmanmain}
Let $M$ be a relatively compact open subset with $C^\infty$ boundary $\Gamma$ of a complex analytic
manifold $M'$ of dimension $n$.
We assume that the Levi form is non-degenerate at each point of\, $\Gamma$. Let $q$, $0\leq q\leq n-1$.
Suppose that $Z(q)$ fails at some point of\, $\Gamma$ and that $Z(q-1)$ and $Z(q+1)$ hold at each point of\, $\Gamma$.
Then
\[K_{\Pi^{(q)}}(z, w)\in
C^\infty(\ol M\times\ol M\setminus{\rm diag\,}(\Gamma_q\times\Gamma_q);\, \mathscr L(\Lambda^{0,q}T^*_w(M'),\Lambda^{0,q}T^*_z(M'))).\]
Moreover, in a neighborhood $U$ of ${\rm diag\,}(\Gamma_q\times\Gamma_q)$, $K_{\Pi^{(q)}}(z, w)$ satisfies
\begin{equation}
K_{\Pi^{(q)}}(z, w)\equiv\int^\infty_0e^{i\phi(z, w)t}b(z, w, t)dt\ \ {\rm mod\,}\ \ C^\infty(U\bigcap(\ol M\times\ol M))
\end{equation}
(for the precise meaning of the oscillatory integral $\int^\infty_0e^{i\phi(z, w)t}b(z, w, t)dt$, see Remark~\ref{Br:0712111922} below)
with
\[b(z, w, t)\in S^{n}_{1, 0}(U\bigcap(\ol M\times\ol M)\times]0, \infty[;\, \mathscr L(\Lambda^{0,q}T^*_w(M'), \Lambda^{0,q}T^*_z(M'))),\]
\[b(z, w, t)\sim\sum^\infty_{j=0}b_j(z, w)t^{n-j}\]
in the space $S^{n}_{1, 0}(U\bigcap(\ol M\times\ol M)\times]0, \infty[;\, \mathscr L(\Lambda^{0,q}T^*_w(M'), \Lambda^{0,q}T^*_z(M')))$,
\[b_0(z, z)\neq0,\ z\in\Gamma_q,\]
where
$b_j(z, w)\in C^\infty(U\bigcap(\ol M\times\ol M);\, \mathscr L(\Lambda^{0,q}T^*_w(M'), \Lambda^{0,q}T^*_z(M')))$, $j=0,1,\ldots$,
and
\begin{align}
&\phi(z, w)\in C^\infty(U\bigcap(\ol M\times\ol M)),\ \ {\rm Im\,}\phi\geq0, \label{Be:0711180902} \\
&\phi(z, z)=0,\ \ z\in \Gamma_q,\ \ \phi(z, w)\neq0\ \ \mbox{if}\ \ (z, w)\notin{\rm diag\,}(\Gamma_q\times\Gamma_q), \label{Be:0711180903} \\
&{\rm Im\,}\phi(z, w)>0\ \ \mbox{if}\ \ (z, w)\notin\Gamma\times\Gamma, \label{Be:0711180905} \\
&\phi(z, w)=-\ol\phi(w, z) \label{Be:0711180906}.
\end{align}
For $p\in\Gamma_q$, we have
\begin{align} \label{Be:i0711142325}
&\sigma_{\Box^{(q)}_f}(z, d_z\phi(z, w)) \ \ \mbox{vanishes to infinite order at}\ \ z=p, \nonumber \\
&\quad\mbox{$(z, w)$ is in some neighborhood of\, $(p, p)$ in $M'$}.
\end{align}

For $z=w$, $z\in\Gamma_q$, we have
$d_z\phi=-\omega_0-idr$, $d_w\phi=\omega_0-idr$.

Moreover, we have $\phi(z, w)=\phi_-(z, w)$ if $z, w\in\Gamma_q$,
where $\phi_-(z, w)\in C^\infty(\Gamma_q\times\Gamma_q)$ is the phase appearing in the
description of the Szeg\"{o} projection in part ${\rm I\,}$ (see also Theorem~\ref{Bt:0801090850} below).
More properties of the phase $\phi(z, w)$ will be given in Theorem~\ref{Bt:0711041300}.
\end{thm}

\begin{rem} \label{Br:0712111922}
Let $\phi$ and $b(z, w, t)$ be as in Theorem~\ref{Bt:i-BP-Bergmanmain}. Let
$y=(y_1,\ldots,y_{2n-1})$
be local coordinates on $\Gamma$ and extend $y_1,\ldots,y_{2n-1}$ to real smooth functions in some neighborhood of $\Gamma$.
We work with local coordinates
\[w=(y_1,\ldots,y_{2n-1},r)\]
defined on some neighborhood $U$ of $p\in\Gamma_q$.
Let $u\in C^\infty_0(U;\, \Lambda^{0,q}T^*(M'))$. Choose a cut-off function $\chi(t)\in C^\infty(\Real)$
so that $\chi(t)=1$ when $\abs{t}<1$ and $\chi(t)=0$ when $\abs{t}>2$. Set
\[(B_\eps u)(z)=\int\int^\infty_0e^{i\phi(z, w)t}b(z, w, t)\chi(\eps t)u(w)dtdw.\]
Since $d_y\phi\neq0$ where ${\rm Im\,}\phi=0$ (see (\ref{Be:0803271451})),
we can integrate by parts in $y$ and $t$ and obtain
$\lim_{\eps\To0}(B_\eps u)(z)\in C^\infty(\ol M;\, \Lambda^{0, q}T^*(M'))$.
This means that
$B=\lim_{\eps\To0}B_\eps: C^\infty(\ol M;\, \Lambda^{0,q}T^*(M'))\To C^\infty(\ol M;\, \Lambda^{0, q}T^*(M'))$
is continuous. We write $B(z, w)$ to denote the distribution kernel of $B$. Formally,
$B(z, w)=\int^\infty_0e^{i\phi(z, w)t}b(z, w, t)dt$.
\end{rem}

From (\ref{Be:i0711142325}) and Remark $1.6$ of part ${\rm I\,}$ it follows that

\begin{thm} \label{Bt:0711041300}
Under the assumptions of\, Theorem~\ref{Bt:i-BP-Bergmanmain}, let $p\in\Gamma_q$.
We choose local complex analytic coordinates
$z=(z_1,\ldots,z_n)$, $z_j=x_{2j-1}+ix_{2j}$, $j=1,\ldots,n$,
vanishing at $p$ such that the metric on $\Lambda^{1,0}T(M')$ is
$\sum^n_{j=1}dz_j\otimes d\ol z_j$ at $p$
and $r(z)=\sqrt{2}{\rm Im\,}z_n+\sum^{n-1}_{j=1}\lambda_j\abs{z_j}^2+O(\abs{z}^3)$,
where $\lambda_j$, $j=1,\ldots,n-1$, are the eigenvalues of $L_p$. (This is always possible. See Lemma $3.2$ of~\cite{Hor04}.)
We also write $w=(w_1,\ldots,w_n)$, $w_j=y_{2j-1}+iy_{2j}$, $j=1,\ldots,n$.
Then, we can take $\phi(z, w)$ so that
\begin{align} \label{Be:0711180900}
\phi(z, w)&=-\sqrt{2}x_{2n-1}+\sqrt{2}y_{2n-1}-ir(z)\Bigr(1+\sum^{2n-1}_{j=1}a_jx_j+\frac{1}{2}a_{2n}x_{2n}\Bigr)\nonumber \\
&\quad-ir(w)\Bigr(1+\sum^{2n-1}_{j=1}\ol a_jy_j+\frac{1}{2}\ol a_{2n}y_{2n}\Bigr)+i\sum^{n-1}_{j=1}\abs{\lambda_j}\abs{z_j-w_j}^2 \nonumber \\
&\quad+\sum^{n-1}_{j=1}i\lambda_j(\ol z_jw_j-z_j\ol w_j)+O\Bigr(\abs{(z, w)}^3\Bigr)
\end{align}
in some neighborhood of\, $(p, p)$ in $M'\times M'$, where
$a_j=\frac{1}{2}\frac{\pr\sigma_{\Box^{(q)}_f}}{\pr x_j}(p, -\omega_0(p)-idr(p))$, $j=1,\ldots,2n$.
\end{thm}

We have the following corollary of Theorem~\ref{Bt:i-BP-Bergmanmain}

\begin{cor} \label{Bi-c:co1}
Under the assumptions of\, Theorem~\ref{Bt:i-BP-Bergmanmain} and let $U$ be a small neighborhood of\, ${\rm diag\,}(\Gamma_q\times\Gamma_q)$.
Then there exist smooth functions
$F, G\in C^\infty(U\bigcap(\ol M\times\ol M));\,
\mathscr L(\Lambda^{0,q}T^*_w(M'), \Lambda^{0,q}T^*_z(M')))$
such that
\[K_{\Pi^{(q)}}=F(-i(\phi(z, w)+i0))^{-n-1}+G\log(-i(\phi(z, w)+i0)).\]
Moreover, we have
\begin{align} \label{Bi-e:000}
&F=\sum^{n}_{j=0}(n-j)!b_j(z, w)(-i\phi(z, w))^j+f(z, w)(\phi(z, w))^{n+1},\nonumber \\
&G\equiv\sum^\infty_{j=0}\frac{(-1)^{j+1}}{j!}b_{n+j+1}(z, w)(-i\phi(z, w))^j\ \ {\rm mod\,}\ \ C^\infty(U\bigcap(\ol M\times\ol M))
\end{align}
where
$f(z, w)\in C^\infty(U\bigcap(\ol M\times\ol M);\, \mathscr L(\Lambda^{0,q}T^*_w(M'), \Lambda^{0,q}T^*_z(M')))$.
\end{cor}

If $w\in\Lambda^{0,1}T^*_z(M')$, let $w^{\wedge, *}: \Lambda^{0,q+1}T^*_z(M')\To \Lambda^{0,q}T^*_z(M')$
be the adjoint of left exterior multiplication
$w^\wedge: \Lambda^{0,q}T^*_z(M')\To \Lambda^{0,q+1}T^*_z(M')$.
That is,
\begin{equation} \label{Be:0710312317}
(w^\wedge u\ |\ v)=(u\ |\ w^{\wedge, *}v),
\end{equation}
for all $u\in\Lambda^{0,q}T^*_z(M')$, $v\in\Lambda^{0,q+1}T^*_z(M')$.
Notice that $w^{\wedge, *}$ depends anti-linearly on $w$.

Let $\Lambda^{0,1}T^*(\Gamma)$ be the bundle of boundary $(0, 1)$ forms. (See (\ref{Be:0803312226}) and (\ref{Be:0801241459}).)

\begin{prop} \label{Bp:i-BP-leading}
Under the assumptions of\, Theorem~\ref{Bt:i-BP-Bergmanmain}, let $p\in\Gamma_q$, $q=n_-$. Let
$U_1(z),\ldots,U_{n-1}(z)$
be an orthonormal frame of $\Lambda^{1,0}T_z(\Gamma)$, $z\in\Gamma$, for which
the Levi form is diagonalized at $p$. Let $e_j(z)$, $j=1,\ldots,n-1$
denote the basis of $\Lambda^{0,1}T^*_z(\Gamma)$, $z\in\Gamma$, which is dual to $\ol U_j(z)$, $j=1,\ldots,n-1$. Let
$\lambda_j(z)$, $j=1,\ldots,n-1$ be the eigenvalues of the Levi form $L_z$, $z\in\Gamma$. We assume that
$\lambda_j(p)<0$ if $1\leq j\leq n_-$. Then
\begin{equation} \label{Be:0711171545}
F(p, p)=n!\abs{\lambda_1(p)}\cdots
\abs{\lambda_{n-1}(p)}\pi^{-n}
2\Bigr(\prod_{j=1}^{j=n_-}e_j(p)^\wedge e_j^{\wedge, *}(p)\Bigr)\circ(\dbar r(p))^{\wedge, *}(\dbar r(p))^\wedge,
\end{equation}
where $F$ is as in Corollary~\ref{Bi-c:co1}.
\end{prop}

In the rest of this section, we outline the proof of Theorem~\ref{Bt:i-BP-Bergmanmain}. We assume that the Levi form is non-degenerate at each point of
$\Gamma$. We pause and recall a general fact of distribution theory.
(See H\"{o}rmander \cite{Hor03}.) Let $E$, $F$ be $C^\infty$ vector
bundles over $C^\infty$ manifolds $G$ and $H$ respectively. We take smooth densities of integration on $G$ and $H$ respectively.
If $A: C^\infty_0(G;\, E)\To \mathscr D'(H;\, F)$ is continuous,  we write $K_A(x, y)$ or $A(x, y)$ to denote the distribution kernel of $A$.
Then the following two statements are equivalent
\begin{enumerate}
\item $A$ is continuous: $\mathscr E'(G;\, E)\To C^\infty(H;\, F)$,
\item $K_A\in C^\infty(H\times G;\, \mathscr L(E_y, F_x))$.
\end{enumerate}
If $A$ satisfies (a) or (b), we say that $A$ is smoothing. Let
$B: C^\infty_0(G;\, E)\To \mathscr D'(H;\, F)$. We write $A\equiv B$ if $A-B$ is a smoothing operator.

Let $\gamma$ denote the operator of restriction to the boundary $\Gamma$. Let us consider the map
\begin{align} \label{Be:08090715350}
F^{(q)}:H^2(\ol M;\, \Lambda^{0,q}T^*(M' ))&\To H^0(\ol M;\, \Lambda^{0,q}T^*(M' ))\oplus H^{\frac{3}{2}}(\Gamma;\, \Lambda^{0,q}T^*(M')), \nonumber \\
&u\To (\Box^{(q)}_fu, \gamma u).
\end{align}
It is well-known that ${\rm dim\,}{\rm Ker\,} F^{(q)}<\infty$ and
${\rm Ker\,}F^{(q)}\subset C^\infty(\ol M;\,\Lambda^{0,q}T^*(M' ))$. Let
\begin{equation} \label{Be:0809071557}
K^{(q)}:H^2(\ol M;\, \Lambda^{0,q}T^*(M' ))\To{\rm Ker\,}F^{(q)}
\end{equation}
be the orthogonal projection with respect to $(\ |\ )_M$. Then,
\begin{equation} \label{Be:0809071607}
K^{(q)}\in C^\infty(\ol M\times\ol M;\, \mathscr L(\Lambda^{0,q}T^*(M'),\Lambda^{0,q}T^*(M'))).
\end{equation}
Put
\begin{equation} \label{Be:0809071603}
\Td\Box^{(q)}_f=\Box^{(q)}_f+K^{(q)}
\end{equation}
and consider the map
\begin{align} \label{Be:0809071609}
\Td F^{(q)}:H^2(\ol M;\, \Lambda^{0,q}T^*(M' ))&\To H^0(\ol M;\, \Lambda^{0,q}T^*(M' ))\oplus H^{\frac{3}{2}}(\Gamma;\, \Lambda^{0,q}T^*(M')), \nonumber \\
&u\To (\Td\Box^{(q)}_fu, \gamma u).
\end{align}
We can check that $\Td F^{(q)}$ is injective (see section $4$). Let
\begin{equation} \label{Be:0711021200}
\Td P:C^\infty(\Gamma;\, \Lambda^{0,q}T^*(M'))\To C^\infty(\ol M;\, \Lambda^{0,q}T^*(M'))
\end{equation}
be the Poisson operator for $\Td\Box^{(q)}_f$ which is well-defined since (\ref{Be:0809071609}) is injective.
It is well-known that $\Td P$ extends continuously
\[\Td P:H^{s}(\Gamma;\, \Lambda^{0,q}T^*(M'))\To H^{s+\frac{1}{2}}(\ol M;\, \Lambda^{0,q}T^*(M')),\ \ \forall\ s\in\Real.\]
(See page $29$ of Boutet de Monvel~\cite{Bou71}.) Let
\[\Td P^*:\mathscr E'(\ol M;\, \Lambda^{0,q}T^*(M'))\To\mathscr D'(\Gamma;\, \Lambda^{0,q}T^*(M'))\]
be the operator defined by
$(\Td P^*u\ |\ v)_\Gamma=(u\ |\ \Td Pv)_M$, $u\in\mathscr E'(\ol M;\, \Lambda^{0,q}T^*(M'))$, $v\in C^\infty(\Gamma;\, \Lambda^{0,q}T^*(M'))$.
It is well-known (see page $30$ of~\cite{Bou71}) that $\Td P^*$ is continuous:
$\Td P^*:L^2(M;\, \Lambda^{0,q}T^*(M'))\To H^{\frac{1}{2}}(\Gamma;\, \Lambda^{0,q}T^*(M'))$
and
\[\Td P^*:C^\infty(\ol M;\, \Lambda^{0,q}T^*(M'))\To C^\infty(\Gamma;\, \Lambda^{0,q}T^*(M')).\]

We use the inner product $[\ |\ ]$ on $H^{-\frac{1}{2}}(\Gamma;\, \Lambda^{0,q}T^*(M'))$ defined as follows:
\[[u\ |\ v]=(\Td Pu\ |\ \Td Pv)_M,\]
where $u$, $v\in H^{-\frac{1}{2}}(\Gamma;\, \Lambda^{0,q}T^*(M'))$. We consider $(\dbar r)^{\wedge, *}$ as an operator
\[(\dbar r)^{\wedge, *}:H^{-\frac{1}{2}}(\Gamma;\, \Lambda^{0,q}T^*(M'))\To H^{-\frac{1}{2}}(\Gamma;\, \Lambda^{0,q-1}T^*(M')).\]
Note that $(\dbar r)^{\wedge, *}$ is the pointwise adjoint of $\dbar r$ with respect to $(\ |\ )$. Let
\begin{equation} \label{Be:i0712302117}
T: H^{-\frac{1}{2}}(\Gamma;\, \Lambda^{0,q}T^*(M'))\To {\rm Ker\,}(\dbar r)^{\wedge, *}
\end{equation}
be the orthogonal projection onto ${\rm Ker\,}(\dbar r)^{\wedge, *}$ with respect to $[\ |\ ]$. That is,
if $u\in H^{-\frac{1}{2}}(\Gamma;\, \Lambda^{0,q}T^*(M'))$, then
$(\dbar r)^{\wedge, *}Tu=0$ and
$[(I-T)u\ |\ g]=0$, for all $g\in{\rm Ker\,}(\dbar r)^{\wedge, *}$.
In section $4$, we will show that $T$ is a classical pseudodifferential operator of order $0$ with principal symbol
$2(\dbar r)^{\wedge, *}(\dbar r)^\wedge$.
For $q\in\Pstint$, let $\Lambda^{0, q}T^*(\Gamma)$ be the bundle of boundary $(0, q)$ forms. (See (\ref{Be:0801241459}).)
If $u\in C^\infty(\Gamma;\, \Lambda^{0,q}T^*(M'))$, then $u\in{\rm Ker\,}(\dbar r)^{\wedge, *}$ if and only if
$u\in C^\infty(\Gamma;\, \Lambda^{0,q}T^*(\Gamma))$. Put
\begin{equation} \label{Be:i0712302232}
\ol{\pr_\beta}=T\gamma\dbar\Td P: C^\infty(\Gamma;\, \Lambda^{0,q}T^*(\Gamma))\To C^\infty(\Gamma;\, \Lambda^{0,q+1}T^*(\Gamma)).
\end{equation}
$\ol{\pr_\beta}$ is a classical
pseudo\-differential operator of order one from boundary $(0, q)$ forms to boundary $(0, q+1)$ forms. It is easy to see that
$\ol{\pr_\beta}=\ol{\pr_b}$+lower order terms,
where $\ol{\pr_b}$ is the tangential Cauchy-Riemann operator. (See~\cite{FK72} or section $6$.) In section $6$, we will show that
$(\ol{\pr_\beta})^2=0$.
Let
\[\ol{\pr_\beta}^\dagger: C^\infty(\Gamma;\, \Lambda^{0,q+1}T^*(\Gamma))\To C^\infty(\Gamma;\, \Lambda^{0,q}T^*(\Gamma))\]
be the formal adjoint of $\ol{\pr_\beta}$ with respect to $[\ |\ ]$.
$\ol{\pr_\beta}^\dagger$ is a classical
pseudodifferential operator of order one from boundary $(0, q+1)$ forms to boundary $(0, q)$ forms. In section $6$, we will show that
$\ol{\pr_\beta}^\dagger=\gamma\ol{\pr_f}^*\Td P$.

Put
\[\Box^{(q)}_\beta=\ol{\pr_\beta}\ \ol{\pr_\beta}^\dagger+\ol{\pr_\beta}^\dagger\ol{\pr_\beta}:
C^\infty(\Gamma;\, \Lambda^{0,q}T^*(\Gamma))\To C^\infty(\Gamma;\, \Lambda^{0,q}T^*(\Gamma)).\]
For simplicity, we assume that $\Gamma=\Gamma_q$, $\Gamma_q\neq\Gamma_{n-1-q}$. ($\Gamma_q$ is given by (\ref{Be:0803111559}).)
We can repeat the method of part ${\rm I\,}$ (see section $7$) to construct
\[A\in L^{-1}_{\frac{1}{2},
\frac{1}{2}}(\Gamma;\, \Lambda^{0,q}T^*(\Gamma),\Lambda^{0,q}T^*(\Gamma)),\ \  B\in L^{0}_{\frac{1}{2},
\frac{1}{2}}(\Gamma;\, \Lambda^{0,q}T^*(\Gamma),\Lambda^{0,q}T^*(\Gamma))\]
such that $A\Box^{(q)}_\beta+B\equiv B+\Box^{(q)}_\beta A\equiv I$,
$\ol{\pr_\beta} B\equiv0$, $\ol{\pr_\beta}^\dagger B\equiv0$,
and $B\equiv B^\dagger\equiv B^2$,
where $L^m_{\frac{1}{2}, \frac{1}{2}}$ is the space of pseudodifferential operators of order $m$ type $(\frac{1}{2}, \frac{1}{2})$
(see Definition~\ref{Bd:ss-pseudomore}) and $B^\dagger$ is the formal adjoint of $B$ with respect to $[\ |\ ]$. Moreover, $K_{B}(x, y)$ satisfies
$K_{B}(x, y)\equiv\int^{\infty}_{0} e^{i\phi_-(x, y)t}b(x, y, t)dt$,
where $\phi_-(x, y)$ and $b(x, y, t)$ are as in Theorem~\ref{Bt:0801090850}.
In section $8$, we will show that
\[\Pi^{(q)}\equiv\Td PBT(\Td P^*\Td P)^{-1}\Td P^*\ \ {\rm mod\,} C^\infty(\ol M\times\ol M)\]
and
$\Td PBT(\Td P^*\Td P)^{-1}\Td P^*(z, w)\equiv\int^\infty_0e^{i\phi(z, w)t}b(z, w, t)dt$ mod $C^\infty(\ol M\times\ol M)$,
where $\phi(z, w)$ and $b(z, w, t)$ are as in Theorem~\ref{Bt:i-BP-Bergmanmain}.


\section{Terminology and notations, a review}

In this section, we will review some standard terminology in complex geometry.
For more details on the subject, see Kodaira~\cite{K86}.

Let $E$ be a finite dimensional vector space with a complex structure $J$. By definition, a complex structure $J$ is a $\Real$-linear map
$J:E\To E$
that satisfies $J^2=-I$.
Let $\Complex E$ be the complexification of $E$. That is,
$\Complex E=\set{u+iv;\, u, v\in E}$.
Any vector in $\Complex E$ can be written $f=u+iv$
where $u$, $v\in E$ and any $\Real$-linear map between real vector spaces can be extended to a $\Complex$-linear map between the complexifications,
simply by putting $Tf=Tu+iTv$.
In particular, we can extend $J$ to a $\Complex$-linear map
$J:\Complex E\To\Complex E$.
Clearly, it still holds that $J^2=-I$. This implies that we have a decomposition as a direct sum
$\Complex E=\Lambda^{1,0}E\oplus\Lambda^{0,1}E$
where $Ju=iu$ if $u\in\Lambda^{1,0}E$ and
$Ju=-iu$ if $u\in\Lambda^{0,1}E$.

Let us now return to our original situation where $E=T_p(M')$, $p\in M'$. Given holomorphic coordinates
$z_j=x_{2j-1}+ix_{2j}$, $j=1,\ldots,n$,
we get a basis for $T_p(M')$
$\frac{\pr}{\pr x_1},\frac{\pr}{\pr x_2},\ldots,\frac{\pr}{\pr x_{2n-1}},\frac{\pr}{\pr x_{2n}}$.
The complex structure $J$ on $T_p(M')$ is defined by
\begin{equation} \label{Be:0710301713}
 \left\{ \begin{split}
& J(\frac{\displaystyle\partial}{\displaystyle\partial x_{2j-1}})=
     \frac{\displaystyle\partial}{\displaystyle\partial x_{2j}},\ j=1,\ldots,n,   \\
& J(\frac{\displaystyle\partial}{\displaystyle\partial x_{2j}})=
      -\frac{\displaystyle\partial}{\displaystyle\partial x_{2j-1}},\ j=1,\ldots,n. \end{split} \right.
\end{equation}
$J$ does not depend on the choice of holomorphic coordinates.

The complex structure map $J^t: T^*_p(M')\To T^*_p(M')$,
for the cotangent space
is defined as the adjoint of $J$, that is $\seq{Ju\ ,\nu}=\seq{u\ ,J^t\nu}$, $u\in T_p(M')$,
$v\in T^*_p(M')$. We have
\begin{equation} \label{Be:0710301715}
 \left\{ \begin{array}{l}
J^t(dx_{2j-1})=-dx_{2j},\ j=1,\ldots,n,  \\
J^t(dx_{2j})=dx_{2j-1},\ j=1,\ldots,n.
\end{array}\right.
\end{equation}
We can now apply our previous discussion of complex structures on real vector spaces to $T_p(M')$ and $T^*_p(M')$. We then get
decompositions
\begin{align} \label{Be:0803312107}
&\Complex T_p(M')=\Lambda^{1,0}T_p(M')\oplus\Lambda^{0,1}T_p(M'), \nonumber \\
&\Complex T^*_p(M')=\Lambda^{1,0}T^*_p(M')\oplus\Lambda^{0,1}T^*_p(M').
\end{align}
For $u\in\Lambda^{1,0}T_p(M')$, $\nu\in\Lambda^{0,1}T^*_p(M')$,
$-i\seq{u, \nu}=\seq{u, J^t\nu}=\seq{Ju, \nu}=i\seq{u, \nu}$.
Thus, $\seq{u, \nu}=0$.

For $p$, $q\in\Pstint$, the bundle of $(p, q)$ forms of $M'$ is given by
\begin{equation} \label{Be:0710301340}
\Lambda^{p, q}T^*(M')=\Lambda^p(\Lambda^{1,0}T^*(M'))\wedge\Lambda^q(\Lambda^{0,1}T^*(M')).
\end{equation}
That is, the fiber of $\Lambda^{p, q}T^*(M')$ at $z\in M'$ is the
vector space $\Lambda^p(\Lambda^{1,0}T^*_z(M'))\wedge\Lambda^q(\Lambda^{0,1}T^*_z(M'))$ of all finite sums of
$W_1\wedge\cdots\wedge W_p\wedge V_1\wedge\cdots\wedge V_q$,
where
$W_k\in\Lambda^{1, 0}T^*_z(M')$, $k=1,\ldots,p$, $V_j\in\Lambda^{0, 1}T^*_z(M')$, $j=1,\ldots,q$.
Here $\wedge$ denotes the wedge product.

We recall that if $\left(g_{j, k}\right)_{1\leq j,k\leq n}$ is a smooth positive definite Hermitian matrix then
the $(1, 1)$ tensor form $g=\sum^n_{j, k=1}g_{j, k}dz_j\otimes d\ol z_k$ can be viewed as a Hermitian metric $(\ |\ )$ on $\Complex T(M')$
in the following way:
\begin{align*}
&\frac{1}{2}(g(u, \ol v)+\ol{g(v, \ol u)})=(u\ |\ v)=g(u, \ol v),\ \ u, v\in\Lambda^{1,0}T(M'), \\
&(u\ |\ w)=0,\ \ u\in\Lambda^{1, 0}T(M'),\ w\in\Lambda^{0,1}T(M'), \\
&\ol{(u\ |\ v)}=(\ol u\ |\ \ol v),\ u, v\in\Lambda^{0, 1}T(M').
\end{align*}
We can check that $(Ju\ |\ Jv)=(u\ |\ v)$, $u, v\in\Complex T(M')$.
For $t$, $s\in T(M')$, we write
$t=\frac{1}{2}(u+\ol u)$, $s=\frac{1}{2}(v+\ol v)$, $u, v\in\Lambda^{1,0}T(M')$.
Then, $(t\ |\ s)=\frac{1}{4}(u\ |\ v)+\frac{1}{4}\ol{(u\ |\ v)}=\frac{1}{2}{\rm Re\,}(u\ |\ v)$
is real. Thus, the Hermitian metric $g$ induces a $J$-invariant Riemannian metric $(\ |\ )$ on $T(M')$.

The Hermitian metric $(\ |\ )$ on $\Complex T(M)$ induces, by duality,
a Hermitian metric on $\Complex T^*(M)$ that we
shall also denote by $(\ |\ )$ in the following way. For a given point $z\in M'$, let $A$ be the anti-linear map
$A:\Complex T_z(M')\To\Complex T^*_z(M')$
defined by
\begin{equation} \label{Be:CR-Gamma1}
(u\ |\ v)=\seq{u, Av},\ \  u, v\in\Complex T_z(M').
\end{equation}
Since $(\ |\ )$ and $\seq{, }$ are real, $A$ maps $T_z(M')$ to $T^*_z(M')$. A simple computation shows that
$J^tAJ=A$, $JA^{-1}J^t=A^{-1}$.
In particular (since $A$ is anti-linear),
\begin{equation} \label{Be:0712131135}
A\Lambda^{1, 0}T_z(M')=\Lambda^{1, 0}T^*_z(M'),\ A\Lambda^{0, 1}T_z(M')=\Lambda^{0, 1}T^*_z(M').
\end{equation}
For $\omega$, $\mu\in\Complex T^*_z(M')$, we put
\begin{equation} \label{Be:CR-Gamma2}
(\omega\ |\ \mu)=(A^{-1}\mu\ |\ A^{-1}\omega).
\end{equation}
We have
$(\omega\ |\ \mu)=0$ if $\omega\in\Lambda^{1, 0}T^*_z(M')$, $\mu\in\Lambda^{0, 1}T^*_z(M')$.

The Hermitian metric $(\ |\ )$ on $\Lambda^{p,q}T^*(M')$
is defined by
\begin{align} \label{Be:0710301345}
&(w_1\wedge\cdots\wedge w_p\wedge u_{1}\wedge\cdots\wedge u_q\ |\ t_1\wedge\cdots\wedge t_p\wedge
v_{1}\wedge\cdots\wedge v_q) \nonumber \\
&\quad=\det\left((w_{j}\ |\ t_{k})\right)_
{1\leq j,k\leq p}\det\left((u_{j}\ |\ v_{k})\right)_
{1\leq j,k\leq q},\nonumber \\
&\quad u_j, v_k\in \Lambda^{0,1}T^*(M'),\ j, k=1,\ldots,q,\ w_j, t_k\in \Lambda^{1,0}T^*(M'),\ j, k=1,\ldots,p,
\end{align}
and we extend the definition to arbitrary $(p, q)$ forms by sesqui-linearity.

The Hermitian metric $(\ |\ )$ on $\Complex T(M')$ induces a Hermitian metric $(\ |\ )$ on $\Complex T(\Gamma)$.
For $p\in\Gamma$, we have $T_p(\Gamma)=\set{u\in T_p(M');\, \seq{u, dr}=(u\ |\ \frac{\pr}{\pr r})=0}$,
where $\frac{\pr}{\pr r}=A^{-1}dr$. Here $A$ is as in (\ref{Be:CR-Gamma1}).

Put $\mathscr C_p=T_p(\Gamma)\bigcap JT_p(\Gamma)$, $\mathscr C^*_p=T^*_p(\Gamma)\bigcap J^tT^*_p(\Gamma)$
and
\begin{equation} \label{Be:0710301722}
Y=J(\frac{\pr}{\pr r}).
\end{equation}
Here we identify $\Complex T^*_p(\Gamma)$ with the space
\[\set{u\in\Complex T^*_p(M');\, \seq{u, \frac{\pr}{\pr r}}=(u\ |\ dr)=0}.\]
We have
$\mathscr C_p=\set{u\in T_p(\Gamma);\, \seq{u, \omega_0(p)}=0}$, $\mathscr C^*_p=\set{u\in T^*_p(\Gamma);\, \seq{u, Y(p)}=0}$.
($\omega_0$ is given by (\ref{Be:0710301720}).)
Note that ${\rm dim\,}_{\Real}\mathscr C_p={\rm dim\,}_{\Real}\mathscr C^*_p=2n-2$.
As before, we have
$\Complex\mathscr C_p=\Lambda^{1,0}T_p(\Gamma)\oplus\Lambda^{0,1}T_p(\Gamma)$
and
$\Complex\mathscr C^*_p=\Lambda^{1,0}T^*_p(\Gamma)\oplus\Lambda^{0,1}T^*_p(\Gamma)$,
where
\begin{align} \label{Be:0803312225}
&Ju=iu\ \ \mbox{if}\ \ u\in\Lambda^{1,0}T_p(\Gamma),\nonumber \\
&Ju=-iu\ \ \mbox{if}\ \ u\in\Lambda^{0,1}T_p(\Gamma)
\end{align}
and
\begin{align} \label{Be:0803312226}
&J^t\mu=i\mu\ \ \mbox{if}\ \ \mu\in\Lambda^{1,0}T^*_p(\Gamma),\nonumber \\
&J^t\mu=-i\mu\ \ \mbox{if}\ \ \mu\in\Lambda^{0,1}T^*_p(\Gamma).
\end{align}
We have the orthogonal decompositions with respect to $(\ |\ )$
\[\Complex T_p(\Gamma)=\Lambda^{1,0}T_p(\Gamma)\oplus\Lambda^{0,1}T_p(\Gamma)\oplus\set{\lambda Y(p);\, \lambda\in\Complex},\]
\[\Complex T^*_p(\Gamma)=\Lambda^{1,0}T^*_p(\Gamma)\oplus\Lambda^{0,1}T^*_p(\Gamma)\oplus\set{\lambda\omega_0(p);\, \lambda\in\Complex}.\]

We notice that
$J(iY+\frac{\pr}{\pr r})=J\Bigr(iJ(\frac{\pr}{\pr r})+\frac{\pr}{\pr r}\Bigr)=-i(iY+\frac{\pr}{\pr r})$.
Thus, $iY+\frac{\pr}{\pr r}\in\Lambda^{0,1}T(M')$. Near $\Gamma$, put
\begin{equation} \label{Be:0711012320}
T^{*,0,1}_z=\set{u\in\Lambda^{0,1}T^*_z(M');\, (u\ |\ \dbar r(z))=0}
\end{equation}
and
\begin{equation} \label{Be:0804010844}
T^{0,1}_z=\set{u\in\Lambda^{0,1}T_z(M');\, (u\ |\ (iY+\frac{\pr}{\pr r})(z))=0}.
\end{equation}
We have the orthogonal decompositions with respect to $(\ |\ )$
\begin{equation} \label{Be:0804010847}
\Lambda^{0,1}T^*_z(M')=T^{*,0,1}_z\oplus\set{\lambda(\dbar r)(z);\, \lambda\in\Complex},
\end{equation}
\begin{equation} \label{Be:0804010848}
\Lambda^{0,1}T_z(M')=T^{0,1}_z\oplus\set{\lambda(iY+\frac{\pr}{\pr r})(z);\, \lambda\in\Complex}.
\end{equation}
Note that $T^{*,0,1}_z=\Lambda^{0,1}T^*_z(\Gamma)$, $T^{0,1}_z=\Lambda^{0,1}T_z(\Gamma)$, $z\in\Gamma$.

For $q\in\Pstint$, the bundle of boundary $(0, q)$ forms is given by
\begin{equation} \label{Be:0804010812}
\Lambda^{0, q}T^*(\Gamma)=\Lambda^q(\Lambda^{0,1}T^*(\Gamma)).
\end{equation}
Note that
\begin{equation} \label{Be:0801241459}
\Lambda^{0,q}T^*_z(\Gamma)=\set{u\in\Lambda^{0,q}T^*_z(M');\, (u\ |\ \dbar r(z)\wedge g)=0,\ \
                  \forall g\in\Lambda^{0,q-1}T^*_z(M')}.
\end{equation}

We recall the following

\begin{defn} \label{Bd:BK-Leviform}
For $p\in\Gamma$,
the Levi form $L_p(Z, \ol W)$, $Z$, $W\in\Lambda^{1,0}T_p(\Gamma)$, is the Hermitian quadratic form on $\Lambda^{1,0}T_p(\Gamma)$ defined as follows:
\begin{equation} \label{Be:070502-*} \begin{split}
&\mbox{For any $Z$, $W\in \Lambda^{1,0}T_p(\Gamma)$, pick $\Td Z$, $\Td W\in
C^\infty(\Gamma;\, \Lambda^{1,0}T(\Gamma))$ that satisfy}  \\
&\mbox{$\Td Z(p)=Z$, $\Td W(p)=W$. Then }
L_p(Z,\ol W)=\frac{1}{2i}\seq{[\Td Z\ ,\ol{\Td W}](p)\ ,
\omega_0(p)}.
\end{split}
\end{equation}
Here $[\Td Z\ ,\ol{\Td W}]=\Td Z\ol{\Td W}-
\ol{\Td W}\Td Z$
denotes the commutator of $\Td Z$ and $\ol{\Td W}$.
\end{defn}

It is easy to see that the definition of the Levi form $L_p$ is independent
of the choices of $\Td Z$ and $\Td W$. We give it in detail for the convenience of
the reader (see the discussion before Lemma~$2.4$ of part ${\rm I\,}$)

\begin{lem} \label{Bl:d-Leviform}
Let $\Td Z, \Td W\in C^\infty(\Gamma;\, \Lambda^{1,0}T(\Gamma))$. We have
\begin{equation} \label{Be:*}
\frac{1}{2i}\seq{[\Td Z\ ,\ol{\Td W}](p)\ ,\omega_0(p)}=-\frac{1}{2i}\seq{\Td Z(p)\wedge\ol{\Td W}(p), d\omega_0(p)}.
\end{equation}
\end{lem}

\begin{defn} \label{Bd:0710301255}
The eigenvalues of the Levi form at $p\in \Gamma$ are the eigenvalues of the Hermitian form $L_p$ with
respect to the inner product $(\ |\ )$ on $\Lambda^{1, 0}T_p(\Gamma)$.
\end{defn}


\section{The $\dbar$-Neumann problem, a review}

In this section, we will
give a brief discussion of the $\dbar$-Neumann problem in a setting appropriate for our purpose. General
references for this section are the books by H\"{o}rmander~\cite{Hor90},~\cite{FK72} and
Chen-Shaw~\cite{CS01}.

As in section $1$, let $M$ be a relatively compact open subset with smooth boundary $\Gamma$ of a
complex manifold $M'$ of dimension $n$ with a smooth Hermitian metric
on its holomorphic tangent bundle. We will use the same notations as before.
We have the following (see page $13$ of~\cite{FK72}, for the proof)

\begin{lem} \label{Bl:BP-integral}
For all $f\in C^\infty(\ol M;\, \Lambda^{0,q}T^*(M'))$, $g\in C^\infty(\ol M;\, \Lambda^{0,q+1}T^*(M'))$,
\begin{equation} \label{Be:0710312315}
(\dbar f\ |\ g)_M=(f\ |\ \ol{\pr_f}^*g)_M+(\gamma f\ |\ \gamma(\dbar r)^{\wedge, *}g)_\Gamma,
\end{equation}
where $(\dbar r)^{\wedge, *}$ is defined by (\ref{Be:0710312317}).
We recall that $\ol{\pr_f}^*$ is the formal adjoint of $\dbar$ and $\gamma$ is the operator of restriction to the boundary $\Gamma$.
\end{lem}

We recall that
\begin{align} \label{Be:0710301602}
{\rm Dom\,}\dbar&=\{u\in L^2(M;\, \Lambda^{0,q}T^*(M'));\, \mbox{there exist}\ u_j\in C^\infty(\ol M;\, \Lambda^{0,q}T^*(M')),\nonumber \\
&\quad j=1,2,\ldots,
\mbox{and}\ v\in L^2(M;\, \Lambda^{0, q+1}T^*(M')),\ \mbox{such that} \nonumber \\
&\quad u_j\To u\ \mbox{in}\ L^2(M;\, \Lambda^{0,q}T^*(M')),
j\To\infty,\ \mbox{and}\nonumber \\
&\quad \dbar u_j\To v\ \mbox{in}\ L^2(M;\, \Lambda^{0,q+1}T^*(M')),\ j\To\infty\}.
\end{align}
We write $\dbar u=v$.

The Hilbert space adjoint $\dbar^*$ of $\dbar$ is defined on the domain of $\dbar^*$ consisting of all
$f\in L^2(M;\, \Lambda^{0,q+1}T^*(M'))$ such that for some constant $c>0$,
$\abs{(f\ |\ \dbar g)_M}\leq c\norm{g}$, for all $g\in C^\infty(\ol M;\, \Lambda^{0,q}T^*(M'))$.
For such a $f$,
\[g\To(f\ |\ \dbar g)_M\]
extends to a bounded anti-linear functional on  $L^2(M;\, \Lambda^{0,q}T^*(M'))$ so
\[(f\ |\ \dbar g)_M=(\Td f\ |\ g)_M\]
for some $\Td f\in L^2(M;\, \Lambda^{0, q}T^*(M'))$. We have $\dbar^*f=\Td f$.

From Lemma~\ref{Bl:BP-integral}, it follows that
\begin{equation} \label{Be:0710312325}
{\rm Dom\,}\dbar^*\bigcap C^\infty(\ol M;\, \Lambda^{0,q}T^*(M'))=
\set{u\in C^\infty(\ol M;\, \Lambda^{0,q}T^*(M'));\, \gamma(\dbar r)^{\wedge, *}u=0}
\end{equation}
and
\begin{equation} \label{Be:0710312326}
\dbar^*=\ol{\pr_f}^*\ \ \mbox{on}\ {\rm Dom\,}\dbar^*\bigcap C^\infty(\ol M;\, \Lambda^{0,q}T^*(M')).
\end{equation}

The $\dbar$-Neumann Laplacian on $(0, q)$ forms is then the operator in the space $L^2(M;\, \Lambda^{0,q}T^*(M'))$
\[\Box^{(q)}=\dbar\ \dbar^*+\dbar^*\dbar.\]
We notice that $\Box^{(q)}$ is self-adjoint. (See chapter ${\rm I\,}$ of~\cite{FK72}.)
We have
\begin{align*}
{\rm Dom\,}\Box^{(q)}&=\{u\in L^2(M;\, \Lambda^{0,q}T^*(M'));\, u\in{\rm Dom\,}\dbar^*\bigcap{\rm Dom\,}\dbar,\\
&\quad\dbar^*u\in{\rm Dom\,}\dbar, \dbar u\in{\rm Dom\,}\dbar^*\}.
\end{align*}
Put $D^{(q)}={\rm Dom\,}\Box^{(q)}\bigcap C^\infty(\ol M\ ;\Lambda^{0,q}T^*(M'))$.
From (\ref{Be:0710312325}), we have
\begin{equation} \label{Be:010312335}
D^{(q)}=\set{u\in C^\infty(\ol M;\, \Lambda^{0,q+1}T^*(M'));\, \gamma(\dbar r)^{\wedge, *}u=0,\ \gamma(\dbar r)^{\wedge, *}\dbar u=0}.
\end{equation}
In view of (\ref{Be:0801241459}), we see that $u\in D^{(q)}$ if and only if
$\gamma u\in C^\infty(\Gamma;\, \Lambda^{0,q}T^*(\Gamma))$ and $\gamma\dbar u\in C^\infty(\Gamma;\, \Lambda^{0,q+1}T^*(\Gamma))$.
We have the following

\begin{lem} \label{Bl:0710312345}
Let $q\geq1$. For every $u\in{\rm Dom\,}\dbar^*\bigcap C^\infty(\ol M;\, \Lambda^{0,q+1}T^*(M'))$, we have
\[\dbar^*u\in{\rm Dom\,}\dbar^*\bigcap C^\infty(\ol M;\, \Lambda^{0,q}T^*(M')).\]
\end{lem}

\begin{proof}
Let
\[u\in{\rm Dom\,}\dbar^*\bigcap C^\infty(\ol M;\, \Lambda^{0,q+1}T^*(M')).\]
For $g\in C^\infty(\ol M;\, \Lambda^{0,q-1}T^*(M'))$, we have
\begin{align*}
0=(\ol{\pr_f}^*\dbar^*u\ |\ g)_M&=(\dbar^*u\ |\ \dbar g)_M-(\gamma(\dbar r)^{\wedge, *}\dbar^*u\ |\ \gamma g)_\Gamma \\
&=(u\ |\ \dbar\dbar g)_M-(\gamma(\dbar r)^{\wedge, *}\dbar^*u\ |\ \gamma g)_\Gamma \\
&=-(\gamma(\dbar r)^{\wedge, *}\dbar^*u\ |\ \gamma g)_\Gamma.
\end{align*}
Here we used (\ref{Be:0710312315}). Thus,
$\gamma(\dbar r)^{\wedge, *}\dbar^*u=0$.
The lemma follows.
\end{proof}

\begin{defn} \label{Bd:BP-Neumann}
The boundary conditions
\[\gamma(\dbar r)^{\wedge, *}u=0, \ \ \gamma(\dbar r)^{\wedge, *}\dbar u=0,\ \  u\in C^\infty(\ol M\ ,\Lambda^{0,q}T^*(M'))\]
are called $\dbar$-Neumann boundary conditions.
\end{defn}

\begin{defn} \label{Bd:BK-problem}
The $\dbar$-Neumann problem in $M$ is the problem of finding, given a form
$f\in C^\infty(\ol M;\, \Lambda^{0,q}T^*(M'))$, another form $u\in D^{(q)}$ verifying
$\Box^{(q)}u=f$.
\end{defn}

\begin{defn} \label{BK-Zofqmore}
Given $q$, $0\leq q\leq n-1$. The Levi form
is said to satisfy condition $Z(q)$ at $p\in \Gamma$ if it has at least $n-q$ positive or at least $q+1$
negative eigenvalues. If the Levi form is non-degenerate at $p\in \Gamma$,
let $(n_-,n_+)$, $n_-+n_+=n-1$, be the signature of $L_p$. Then $Z(q)$ holds at $p$ if and only if $q\neq n_-$.
\end{defn}

The following classical results are due to Kohn. For the proofs, see~\cite{FK72}.

\begin{thm} \label{Bt:0710301305}
We assume that $Z(q)$ holds at each point of\, $\Gamma$. Then ${\rm Ker\,}\Box^{(q)}$ is a finite dimensional
subspace of $C^\infty(\ol M;\, \Lambda^{0,q}T^*(M'))$, $\Box^{(q)}$ has closed range and $\Pi^{(q)}$ is a smoothing operator. That is, the
distribution kernel
\[K_{\Pi^{(q)}}(z, w)\in C^\infty(\ol M\times\ol M;\, \mathscr L(\Lambda^{0,q}T^*_w(M'),\Lambda^{0,q}T^*_z(M'))).\]

Moreover, there exists an continuous operator
\[N^{(q)}:L^2(M;\, \Lambda^{0,q}T^*(M'))\To{\rm Dom\,}\Box^{(q)}\]
such that $N^{(q)}\Box^{(q)}+\Pi^{(q)}=I$ on ${\rm Dom\,}\Box^{(q)}$,
\[\Box^{(q)}N^{(q)}+\Pi^{(q)}=I\]
on $L^2(M;\, \Lambda^{0,q}T^*(M'))$.
Furthermore,
\[N^{(q)}\Bigr (C^\infty(\ol M;\, \Lambda^{0,q}T^*(M'))\Bigr)\subset C^\infty(\ol M;\, \Lambda^{0,q}T^*(M'))\]
and for each $s\in\Real$ and
all $f\in C^\infty(\ol M;\, \Lambda^{0,q}T^*(M'))$, there is a constant $c>0$, such that
$\norm{N^{(q)}f}_{s+1}\leq c\norm{f}_s$
where $\norm{\ }_s$ denotes any of the equivalent norms defining $H^s(\ol M;\, \Lambda^{0, q}T^*(M'))$.
\end{thm}

\begin{thm} \label{Bt:0710301306}
Suppose that $Z(q)$ fails at some point of\, $\Gamma$ and that $Z(q-1)$ and $Z(q+1)$ hold at each point of\, $\Gamma$. Then,
\begin{equation} \label{Be:0809102247}
\Pi^{(q)}u=(I-\dbar N^{(q-1)}\dbar^*-\dbar^* N^{(q+1)}\dbar)u,\ u\in{\rm Dom\,}\dbar^*\bigcap C^\infty(\ol M;\, \Lambda^{0,q}T^*(M')),
\end{equation}
where $N^{(q+1)}$ and $N^{(q-1)}$ are as in Theorem~\ref{Bt:0710301305}. In particular,
\[\Pi^{(q)}:{\rm Dom\,}\dbar^*\bigcap C^\infty(\ol M;\, \Lambda^{0,q}T^*(M'))\To D^{(q)}.\]
\end{thm}


\section{The operator $T$}

First, we claim that $\Td F^{(q)}$ is injective, where $\Td F^{(q)}$ is given by (\ref{Be:0809071609}).
If $u\in{\rm Ker\,}\Td F^{(q)}$, then $u\in C^\infty(\ol M;\, \Lambda^{0,q}T^*(M' ))$, $\Td\Box^{(q)}_fu=0$
and $\gamma u=0$.
We can check that
\begin{align*}
(\Td\Box^{(q)}_fu\ |\ u)_M&=(\Box^{(q)}_fu\ |\ u)_M+(K^{(q)}u\ |\ u)_M \\
&=(\dbar u\ |\ \dbar u)_M+(\ol{\pr}^*_fu\ |\ \ol{\pr}^*_fu)_M+(K^{(q)}u\ |\ K^{(q)}u)_M=0.
\end{align*}
Thus, $u\in {\rm Ker\,}F^{(q)}\bigcap {\rm Ker\,}K^{(q)}$ ($F^{(q)}$ is given by (\ref{Be:08090715350})). We get $u=0$. Hence,
$\Td F^{(q)}$ is injective.
The Poisson operator
\[\Td P: C^\infty(\Gamma;\, \Lambda^{0,q}T^*(M'))\To C^\infty(\ol M;\, \Lambda^{0,q}T^*(M'))\]
of $\Td\Box^{(q)}_f$ is well-defined. That is, if
$u\in C^\infty(\Gamma;\, \Lambda^{0,q}T^*(M'))$,
then
\[\Td Pu\in C^\infty(\ol M;\, \Lambda^{0,q}T^*(M')),\ \ \Td\Box^{(q)}_f\Td Pu=0,\ \
\gamma\Td Pu=u.\]
Moreover, if $v\in C^\infty(\ol M;\, \Lambda^{0,q}T^*(M'))$ and $\Td\Box^{(q)}_fv=0$, then
$v=\Td P\gamma v$. Furthermore, it is straight forward to see that
\begin{equation} \label{Be:0809081416}
\dbar\Td P u=\Td P\gamma\dbar\Td Pu,\ \ \ol{\pr}^*_f\Td Pu=\Td P\gamma\ol{\pr}^*_f\Td Pu,\ \ u\in C^\infty(\Gamma;\, \Lambda^{0,q}T^*(M')).
\end{equation}

We recall that (see section $1$) $\Td P$ extends continuously
\[\Td P:H^{s}(\Gamma;\, \Lambda^{0,q}T^*(M'))\To H^{s+\frac{1}{2}}(\ol M;\, \Lambda^{0,q}T^*(M')),\ \ \forall\ s\in\Real.\]
As in section $1$, let $\Td P^*:\mathscr E'(\ol M;\, \Lambda^{0,q}T^*(M'))\To\mathscr D'(\Gamma;\, \Lambda^{0,q}T^*(M'))$
be the operator defined by
\[(\Td P^*u\ |\ v)_\Gamma=(u\ |\ \Td Pv)_M,\]
$u\in\mathscr E'(\ol M;\, \Lambda^{0,q}T^*(M'))$, $v\in C^\infty(\Gamma;\, \Lambda^{0,q}T^*(M'))$.
We recall that (see section $1$) $\Td P^*$ is continuous:
\[\Td P^*:L^2(M;\, \Lambda^{0,q}T^*(M'))\To H^{\frac{1}{2}}(\Gamma;\, \Lambda^{0,q}T^*(M'))\]
and $\Td P^*:C^\infty(\ol M;\, \Lambda^{0,q}T^*(M'))\To C^\infty(\Gamma;\, \Lambda^{0,q}T^*(M'))$.

Let $L$ be a classical pseudodifferential operator on a $C^\infty$ manifold. From now on, we let $\sigma_L$ denote the principal symbol of $L$.
The operator
\[\Td P^*\Td P:C^\infty(\Gamma;\, \Lambda^{0,q}T^*(M'))\To C^\infty(\Gamma;\, \Lambda^{0,q}T^*(M'))\]
is a classical elliptic pseudodifferential operator of order $-1$ and invertible (since $\Td P$ is injective). (See Boutet de Monvel~\cite{Bou66}.)
Let $\La_\Gamma$ be the real Laplacian on $\Gamma$ and
let $\sqrt{-\La_\Gamma}$ be the square root of $-\La_\Gamma$. It is well-known (see~\cite{Bou66}) that
\begin{equation} \label{Be:0801041932}
\sigma_{\Td P^*\Td P}=\sigma_{(2\sqrt{-\La_\Gamma})^{-1}}.
\end{equation}
Let $(\Td P^*\Td P)^{-1}:C^\infty(\Gamma;\, \Lambda^{0,q}T^*(M'))\To C^\infty(\Gamma;\, \Lambda^{0,q}T^*(M'))$
be the inverse of $\Td P^*\Td P$. $(\Td P^*\Td P)^{-1}$ is a classical elliptic pseudodifferential operator of order $1$ with scalar principal symbol. We have
\begin{equation} \label{Be:0711021310}
\sigma_{(\Td P^*\Td P)^{-1}}=\sigma_{2\sqrt{-\La_\Gamma}}.
\end{equation}

\begin{defn} \label{Bd:0711021315}
The Neumann operator $\mathscr N^{(q)}$ is the operator on the space $C^\infty(\Gamma;\, \Lambda^{0,q}T^*(M'))$ defined as follows:
\[\mathscr N^{(q)}f=\gamma\frac{\pr}{\pr r}\Td Pf,\ \ f\in C^\infty(\Gamma;\, \Lambda^{0,q}T^*(M')).\]
\end{defn}

The following is well-known (see page $95$ of Greiner-Stein~\cite{GS77})

\begin{lem} \label{Bl:0711021315}
$\mathscr N^{(q)}:C^\infty(\Gamma;\, \Lambda^{0,q}T^*(M'))\To C^\infty(\Gamma;\, \Lambda^{0,q}T^*(M'))$
is a classical elliptic pseudodifferential operator of order $1$ with scalar principal symbol and we have
\begin{equation} \label{Be:0711021317}
\sigma_{\mathscr N^{(q)}}=\sigma_{\sqrt{-\La_\Gamma}}.
\end{equation}
\end{lem}

We use the inner product $[\ |\ ]$ on $H^{-\frac{1}{2}}(\Gamma;\, \Lambda^{0,q}T^*(M'))$ defined as follows:
\begin{equation} \label{Be:0801042209}
[u\ |\ v]=(\Td Pu\ |\ \Td Pv)_M=(\Td P^*\Td Pu\ |\ v)_\Gamma,
\end{equation}
where $u$, $v\in H^{-\frac{1}{2}}(\Gamma;\, \Lambda^{0,q}T^*(M'))$. We consider $(\dbar r)^{\wedge, *}$ as an operator
\[(\dbar r)^{\wedge, *}:H^{-\frac{1}{2}}(\Gamma;\, \Lambda^{0,q}T^*(M'))\To H^{-\frac{1}{2}}(\Gamma;\, \Lambda^{0,q-1}T^*(M')).\]
Let
\begin{equation} \label{Be:0712302117}
T: H^{-\frac{1}{2}}(\Gamma;\, \Lambda^{0,q}T^*(M'))\To {\rm Ker\,}(\dbar r)^{\wedge, *}=H^{-\frac{1}{2}}(\Gamma;\, \Lambda^{0,q}T^*(\Gamma))
\end{equation}
be the orthogonal projection onto ${\rm Ker\,}(\dbar r)^{\wedge, *}$ with respect to $[\ |\ ]$. That is,
if $u\in H^{-\frac{1}{2}}(\Gamma;\, \Lambda^{0,q}T^*(M'))$, then
$(\dbar r)^{\wedge, *}Tu=0$
and $[(I-T)u\ |\ g]=0,\ \ \forall\ g\in{\rm Ker\,}(\dbar r)^{\wedge, *}$.

\begin{lem} \label{Bl:0712302115}
$T$ is a classical pseudodifferential operator of order $0$ with principal symbol
$2(\dbar r)^{\wedge, *}(\dbar r)^\wedge$.
Moreover,
\begin{equation} \label{Be:0801201632}
I-T=(\Td P^*\Td P)^{-1}(\dbar r)^\wedge R,
\end{equation}
where $R:C^\infty(\Gamma;\, \Lambda^{0,q}T^*(M'))\To C^\infty(\Gamma;\, \Lambda^{0,q-1}T^*(M'))$
is a classical pseudodifferential operator of order $-1$.
\end{lem}

\begin{proof}
Let $E=2(\dbar r)^{\wedge, *}\bigl((\dbar r)^{\wedge, *}\bigl)^\dagger+2\bigl((\dbar r)^{\wedge, *}\bigl)^\dagger(\dbar r)^{\wedge, *}$,
\[E: H^{-\frac{1}{2}}(\Gamma;\, \Lambda^{0,q}T^*(M'))\To H^{-\frac{1}{2}}(\Gamma;\, \Lambda^{0,q}T^*(M')),\]
where $\bigl((\dbar r)^{\wedge, *}\bigl)^\dagger$ is the formal adjoint of $(\dbar r)^{\wedge, *}$ with respect to $[\ |\ ]$. That is,
\[[(\dbar r)^{\wedge, *}u\ |\ v]=[u\ |\ \bigl((\dbar r)^{\wedge, *}\bigl)^\dagger v],\]
$u\in H^{-\frac{1}{2}}(\Gamma;\, \Lambda^{0,q}T^*(M'))$, $v\in H^{-\frac{1}{2}}(\Gamma;\, \Lambda^{0,q-1}T^*(M'))$.
We can check that
\begin{equation} \label{Be:0712302130}
\bigl((\dbar r)^{\wedge, *}\bigl)^\dagger=(\Td P^*\Td P)^{-1}(\dbar r)^\wedge(\Td P^*\Td P).
\end{equation}
Thus, the principal symbol of $E$ is
$2(\dbar r)^{\wedge, *}(\dbar r)^\wedge+2(\dbar r)^\wedge(\dbar r)^{\wedge, *}$.
Since
\[\norm{dr}=1=(\norm{\bar\pr r}^2+\norm{\pr r}^2)^{\frac{1}{2}}\]
on $\Gamma$, we have
\begin{equation} \label{Be:0801082257}
\norm{\bar\pr r}^2=\norm{\pr r}^2=\frac{1}{2}\ \ \mbox{on}\ \ \Gamma.
\end{equation}
From this, we can check that
\[2(\dbar r)^{\wedge, *}(\dbar r)^\wedge+2(\dbar r)^\wedge(\dbar r)^{\wedge, *}=
I:H^{-\frac{1}{2}}(\Gamma;\, \Lambda^{0,q}T^*(M'))\To H^{-\frac{1}{2}}(\Gamma;\, \Lambda^{0,q}T^*(M')),\]
where $I$ is the identity map.
$E$ is a classical elliptic pseudodifferential operator with principal symbol $I$. Then $\dim{\rm Ker\,}E<\infty$. Let $G$ be the orthogonal
projection onto ${\rm Ker\,}E$ and $N$ be the partial inverse. Then $G$ is a
smoothing operator and $N$ is a classical elliptic pseudodifferential operator of order $0$
with principal symbol $I$ (up to some smoothing operator). We have
\begin{equation} \label{Be:0712302144}
EN+G=2\Bigr((\dbar r)^{\wedge, *}\bigl((\dbar r)^{\wedge, *}\bigl)^\dagger+
2\bigl((\dbar r)^{\wedge, *}\bigl)^\dagger(\dbar r)^{\wedge, *}\Bigr)N+G=I
\end{equation}
on $H^{-\frac{1}{2}}(\Gamma;\, \Lambda^{0,q}T^*(M'))$. Put
$\Td T=2(\dbar r)^{\wedge, *}\bigl((\dbar r)^{\wedge, *}\bigl)^\dagger N+G$.
Note that
\[{\rm Ker\,}E=\set{u\in
H^{-\frac{1}{2}}(\Gamma;\, \Lambda^{0,q}T^*(M'));\, (\dbar r)^{\wedge, *}u=0,\ \bigl((\dbar r)^{\wedge, *}\bigl)^\dagger u=0}.\]
From this and $(\dbar r)^{\wedge, *}\circ(\dbar r)^{\wedge, *}=0$, we see that
\[\Td Tg\in{\rm Ker\,}(\dbar r)^{\wedge, *},\ \ g\in H^{-\frac{1}{2}}(\Gamma;\, \Lambda^{0,q}T^*(M')).\]
From (\ref{Be:0712302144}), we have
$I-\Td T=2\bigl((\dbar r)^{\wedge, *}\bigl)^\dagger(\dbar r)^{\wedge, *}N$
and
\begin{align*}
[(I-\Td T)g\ |\ u]&=[2\bigl((\dbar r)^{\wedge, *}\bigl)^\dagger(\dbar r)^{\wedge, *}Ng\ |\ u] \\
&=[2(\dbar r)^{\wedge, *}Ng\ |\ (\dbar r)^{\wedge, *}u] \\
&=0,\ \ u\in{\rm Ker\,}(\dbar r)^{\wedge, *}, g\in H^{-\frac{1}{2}}(\Gamma;\, \Lambda^{0,q}T^*(M')).
\end{align*}
Thus, $g=\Td Tg+(I-\Td T)g$
is the orthogonal decomposition with respect to $[\ |\ ]$. Hence,
$\Td T=T$.
The lemma follows.
\end{proof}

Now, we assume that $Z(q)$ fails at some point of\, $\Gamma$ and that $Z(q-1)$ and $Z(q+1)$ hold at each point of\, $\Gamma$. Put
\begin{equation} \label{Be:0809092110}
\Td\Pi^{(q)}=\Pi^{(q)}(I-K^{(q)}),
\end{equation}
where $K^{(q)}$ is as in (\ref{Be:0809071557}). It is straight forward to see that
\begin{equation} \label{Be:0809071558}
\Td\Pi^{(q)}=\Pi^{(q)}-K^{(q)}=(I-K^{(q)})\Pi^{(q)}=\Td P\gamma\Td\Pi^{(q)}
\end{equation}
and $(\Td\Pi^{(q)})^2=\Td\Pi^{(q)},\ \ (\Td\Pi^{(q)})^*=\Td\Pi^{(q)}$,
where $(\Td\Pi^{(q)})^*$ is the formal adjoint of $\Td\Pi^{(q)}$ with respect to $(\ |\ )_M$.

\begin{prop} \label{Bp:0801042202}
We assume that $Z(q)$ fails at some point of\, $\Gamma$ and that $Z(q-1)$ and $Z(q+1)$ hold at each point of\, $\Gamma$. Then,
\begin{align} \label{Be:0801042303}
\Td\Pi^{(q)}u&=\Td\Pi^{(q)}\Td PT(\Td P^*\Td P)^{-1}\Td P^*u \nonumber \\
&=(I-K^{(q)})(I-\dbar N^{(q-1)}\dbar^*-\dbar^* N^{(q+1)}\dbar)\Td PT(\Td P^*\Td P)^{-1}\Td P^*u,\nonumber \\
&u\in C^\infty(\ol M;\, \Lambda^{0,q}T^*(M')),
\end{align}
where $N^{(q+1)}$, $N^{(q-1)}$ are as in Theorem~\ref{Bt:0710301305}, $T$ is as in (\ref{Be:0712302117}).
In particular,
\[\Td\Pi^{(q)}:C^\infty(\ol M;\, \Lambda^{0,q}T^*(M'))\To D^{(q)}.\]
\end{prop}

\begin{proof}
Let $v\in{\rm Dom\,}\dbar^*\bigcap C^\infty(\ol M;\, \Lambda^{0,q}T^*(M'))$. From Theorem~\ref{Bt:0710301306} and (\ref{Be:0809071558}),
we see that $\Td\Pi^{(q)}v\in D^{(q)}$ and $\Td\Pi^{(q)}v=\Td P\gamma\Td\Pi^{(q)}v$.
Note that
\[(f-\Td P(\Td P^*\Td P)^{-1}\Td P^*f\ |\ \Td P\gamma g)_M=0,\ \ f, g\in C^\infty(\ol M;\, \Lambda^{0,q}T^*(M')).\]
We have
\begin{align} \label{Be:0801051031}
&(\Td\Pi^{(q)}\Td PT(\Td P^*\Td P)^{-1}\Td P^*u\ |\ \Td\Pi^{(q)}v)_M \nonumber\\
&\quad=(\Td PT(\Td P^*\Td P)^{-1}\Td P^*u\ |\ \Td \Pi^{(q)}v)_M \nonumber \\
&\quad=(\Td PT(\Td P^*\Td P)^{-1}\Td P^*u\ |\ \Td P\gamma\Td\Pi^{(q)}v)_M \nonumber \\
&\quad=(\Td P(\Td P^*\Td P)^{-1}\Td P^*u\ |\ \Td P\gamma\Td\Pi^{(q)}v)_M \nonumber \\
&\quad=(u\ |\ \Td\Pi^{(q)}v)_M.
\end{align}
Thus,
\[(u-\Td\Pi^{(q)}\Td PT(\Td P^*\Td P)^{-1}\Td P^*u\ |\ \Td\Pi^{(q)}v)_M=0.\]
Since ${\rm Dom\,}\dbar^*\bigcap C^\infty(\ol M;\, \Lambda^{0,q}T^*(M'))$ is dense in $L^2(M;\, \Lambda^{0,q}T^*(M'))$, we get
$(u-\Td\Pi^{(q)}\Td PT(\Td P^*\Td P)^{-1}\Td P^*u\ |\ \Td\Pi^{(q)}v)_M=0$,
for all $v\in L^2(M;\, \Lambda^{0,q}T^*(M'))$.
Thus, $\Td\Pi^{(q)}u=\Td\Pi^{(q)}\Td PT(\Td P^*\Td P)^{-1}\Td P^*u$.

Since
\[\Td PT(\Td P^*\Td P)^{-1}\Td P^*u\in{\rm Dom\,}\dbar^*\bigcap C^\infty(\ol M;\, \Lambda^{0,q}T^*(M')),\]
from (\ref{Be:0809102247}) and (\ref{Be:0809071558}), we get the last identity in (\ref{Be:0801042303}).
The proposition follows.
\end{proof}


\section{The principal symbols of $\gamma\dbar \Td P$ and $\gamma\ol{\pr_f}^*\Td P$}

First, we compute the principal symbols of $\dbar$ and $\ol{\pr_f}^*$.
For each point $z_0\in\Gamma$, we can choose an orthonormal frame
$t_1(z),\ldots,t_{n-1}(z)$
for $T^{*,0,1}_z$ varying smoothly with $z$ in a neighborhood of $z_0$, where $T^{*0,1}_z$ is defined by (\ref{Be:0711012320}).
Then (see (\ref{Be:0804010847}))
$t_1(z),\ldots,t_{n-1}(z),t_n(z):=\frac{\dbar r(z)}{\norm{\dbar r(z)}}$
is an orthonormal frame for $\Lambda^{0,1}T^*_z(M')$. Let
\[T_1(z),\ldots,T_{n-1}(z),T_n(z)\]
denote the basis of $\Lambda^{0,1}T_z(M')$ which is dual to
$t_1(z),\ldots,t_n(z)$.
We have (see (\ref{Be:0804010848})) $T_n=\frac{iY+\frac{\pr}{\pr r}}{\norm{iY+\frac{\pr}{\pr r}}}$.
Note that
\begin{equation} \label{Be:0711012340}
T_1(z),\ldots,T_{n-1}(z)\ \
\mbox{is an orthonormal frame for}\ \Lambda^{0,1}T_z(\Gamma),\ z\in\Gamma,
\end{equation}
and
\begin{equation} \label{Be:0711012345}
t_1(z),\ldots, t_{n-1}(z)\ \
\mbox{is an orthonormal frame for}\ \Lambda^{0,1}T^*_z(\Gamma),\ z\in\Gamma.
\end{equation}

We have
$\dbar f=\Bigr(\sum^{n}_{j=1}t_j^\wedge T_j\Bigr)f,\ f\in C^\infty(M')$.
If
\[f(z)t_{j_1}(z)\wedge\cdots\wedge t_{j_q}(z)\in C^\infty(M';\, \Lambda^{0,q}T^*(M'))\]
is a typical term in a general $(0,q)$ form, we have
\begin{align*}
\dbar f&=\sum^{n}_{j=1}(T_jf)t^\wedge_jt_{j_1}\wedge\cdots\wedge t_{j_q} \\
&\quad+\sum^q_{k=1}(-1)^{k-1}f(z)t_{j_1}\wedge\cdots\wedge(\dbar t_{j_k})\wedge\cdots\wedge t_{j_q}.
\end{align*}
So for the given orthonormal frame we have
\begin{align} \label{Be:0711011159}
\dbar&=\sum^n_{j=1}t_j^\wedge\circ T_j+\mbox{lower order terms} \nonumber \\
&=\sum^{n-1}_{j=1}t_j^\wedge\circ T_j+
\frac{(\dbar r)^\wedge}{\norm{\dbar r}}\circ\frac{iY+\frac{\pr}{\pr r}}{\norm{iY+\frac{\pr}{\pr r}}}+\mbox{lower order terms}
\end{align}
and correspondingly
\begin{equation} \label{Be:0711020000}
\ol{\pr_f}^*=\sum^{n-1}_{j=1}t_j^{\wedge, *}\circ T^*_j+
\frac{(\dbar r)^{\wedge, *}}{\norm{\dbar r}}\circ\frac{iY-\frac{\pr}{\pr r}}{\norm{iY+\frac{\pr}{\pr r}}}+\mbox{lower order terms}.
\end{equation}

We consider
\[\gamma\dbar \Td P:C^\infty(\Gamma;\, \Lambda^{0,q}T^*(M'))\To C^\infty(\Gamma;\, \Lambda^{0,q+1}T^*(M'))\]
and
\[\gamma\ol{\pr_f}^*\Td P:C^\infty(\Gamma;\, \Lambda^{0,q+1}T^*(M'))\To C^\infty(\Gamma;\, \Lambda^{0,q}T^*(M')).\]
$\gamma\dbar \Td P$ and $\gamma\ol{\pr_f}^*\Td P$ are classical pseudodifferential operators of order $1$.
From (\ref{Be:0801082257}), we know that
$\norm{\dbar r}=\frac{1}{\sqrt{2}}$ on $\Gamma$.
We can check that
$\norm{iY+\frac{\pr}{\pr r}}=\sqrt{2}$ on $\Gamma$.
Combining this with (\ref{Be:0711011159}), (\ref{Be:0711020000}) and (\ref{Be:0711021317}), we get
\begin{equation} \label{Be:0801052149}
\gamma\dbar \Td P=\sum^{n-1}_{j=1}t^\wedge_j\circ T_j+(\dbar r)^\wedge\circ(iY+\sqrt{-\La_\Gamma})+\mbox{lower order terms}
\end{equation}
and
\begin{equation} \label{Be:0801052150}
\gamma\ol{\pr_f}^*\Td P=\sum^{n-1}_{j=1}t^{\wedge, *}_j\circ T^*_j+(\dbar r)^{\wedge, *}\circ(iY-\sqrt{-\La_\Gamma})+
\mbox{lower order terms}.
\end{equation}
From Lemma~\ref{Bl:0710312345}, it follows that
\begin{equation} \label{Be:0801051840}
\gamma\ol{\pr_f}^*\Td P:C^\infty(\Gamma;\, \Lambda^{0,q+1}T^*(\Gamma))\To C^\infty(\Gamma;\, \Lambda^{0,q}T^*(\Gamma)).
\end{equation}
Put
\begin{align} \label{Be:BP-4.1}
& \Sigma^+=\set{(x, \lambda\omega_0(x))\in T^*(\Gamma)\smallsetminus 0;\, \lambda>0},\nonumber   \\
& \Sigma^-=\set{(x, \lambda \omega_0(x))\in T^*(\Gamma)\smallsetminus 0;\, \lambda<0}.
\end{align}
We recall that $\omega_0=J^t(dr)$. In section $8$, we need the following

\begin{prop} \label{Bp:0801102222}
The map
$\gamma(\dbar r)^{\wedge, *}\dbar \Td P:C^\infty(\Gamma;\, \Lambda^{0,q}T^*(\Gamma))\To C^\infty(\Gamma;\, \Lambda^{0,q}T^*(\Gamma))$
is a classical
pseudodifferential operator of order one from boundary $(0, q)$ forms to boundary $(0, q)$ forms and we have
\begin{equation} \label{Be:0801082320}
\gamma(\dbar r)^{\wedge, *}\dbar \Td P=\frac{1}{2}(iY+\sqrt{-\La_\Gamma})+\mbox{lower order terms}.
\end{equation}
In particular, it is elliptic outside\, $\Sigma^-$.
\end{prop}

\begin{proof}
Note that
\begin{equation} \label{Be:0801082317}
\gamma(\dbar r)^{\wedge, *}\dbar \Td P=\gamma(\dbar r)^{\wedge, *}\dbar \Td P+\gamma\dbar \Td P(\dbar r)^{\wedge, *}
\end{equation}
on the space $C^\infty(\Gamma;\, \Lambda^{0,q}T^*(\Gamma))$. From (\ref{Be:0801052149}), we have
\begin{align*}
\gamma(\dbar r)^{\wedge, *}\dbar \Td P
&=\sum^{n-1}_{j=1}\Bigr((\dbar r)^{\wedge, *}t^\wedge_j\Bigr)\circ T_j+
\Bigr((\dbar r)^{\wedge, *}(\dbar r)^\wedge\Bigr)\circ(iY+\sqrt{-\La_\Gamma})\\
&\quad+\mbox{lower order terms}
\end{align*}
and
\begin{align*}
\gamma\dbar \Td P(\dbar r)^{\wedge, *}
&=\sum^{n-1}_{j=1}\Bigr(t^\wedge_j(\dbar r)^{\wedge, *}\Bigr)\circ T_j+
\Bigr((\dbar r)^\wedge(\dbar r)^{\wedge, *}\Bigr)\circ(iY+\sqrt{-\La_\Gamma})\\
&\quad+\mbox{lower order terms}.
\end{align*}
Thus,
\begin{align} \label{Be:0801082316}
\gamma(\dbar r)^{\wedge, *}\dbar \Td P+\gamma\dbar \Td P(\dbar r)^{\wedge, *}
&=\sum^{n-1}_{j=1}\Bigr(t^\wedge_j(\dbar r)^{\wedge, *}+(\dbar r)^{\wedge, *}t^\wedge_j\Bigr)\circ T_j \nonumber \\
&\quad+\Bigr((\dbar r)^\wedge(\dbar r)^{\wedge, *}+(\dbar r)^{\wedge, *}(\dbar r)^\wedge\Bigr)\circ(iY+\sqrt{-\La_\Gamma})\nonumber\\
&\quad+\mbox{lower order terms}.
\end{align}
Note that
\begin{equation} \label{Be:0804030849}
t^\wedge_j(\dbar r)^{\wedge, *}+(\dbar r)^{\wedge, *}t^\wedge_j=0,\ \ j=1,\ldots,n-1,
\end{equation}
and
\begin{equation} \label{Be:0804030850}
(\dbar r)^\wedge(\dbar r)^{\wedge, *}+(\dbar r)^{\wedge, *}(\dbar r)^\wedge=\frac{1}{2}.
\end{equation}
Combining this with (\ref{Be:0801082316}) and (\ref{Be:0801082317}), we get (\ref{Be:0801082320}).

Note that
$\sigma_{iY+\sqrt{-\La_\Gamma}}(x, \xi)=-\seq{Y, \xi}+\norm{\xi}=\norm{\xi}+(\omega_0\ |\ \xi)\geq0$
with equality precisely when $\xi=-\lambda\omega_0$, $\lambda>0$.
The proposition follows.
\end{proof}

For $z\in\Gamma$, put
\begin{equation} \label{Be:0801052157}
I^{0,q}T^*_z(M')=\set{u\in\Lambda^{0,q}T^*_z(M');\, u=(\dbar r)^\wedge g,\ g\in\Lambda^{0,q-1}T^*_z(M')}.
\end{equation}
$I^{0,q}T^*_z(M)$ is orthogonal to $\Lambda^{0,q}T^*_z(\Gamma)$.
In section $7$, we need the following

\begin{prop} \label{Bp:071231323}
The operator
\[\gamma(\dbar r)^{\wedge}\ol{\pr_f}^*\Td P(\Td P^*\Td P)^{-1}: C^\infty(\Gamma;\, I^{0,q}T^*(M'))\To C^\infty(\Gamma;\, I^{0,q}T^*(M'))\]
is a classical pseudodifferential operator of order one,
\begin{equation} \label{Be:0801082351}
\gamma(\dbar r)^{\wedge}\ol{\pr_f}^*\Td P(\Td P^*\Td P)^{-1}=(iY-\sqrt{-\La_\Gamma})\sqrt{-\La_\Gamma}+\mbox{lower order terms}.
\end{equation}
It is elliptic outside\, $\Sigma^+$.
\end{prop}

\begin{proof}
Note that
\begin{equation} \label{Be:0801082330}
\gamma(\dbar r)^{\wedge}\ol{\pr_f}^*\Td P(\Td P^*\Td P)^{-1}=\gamma(\dbar r)^{\wedge}\ol{\pr_f}^*\Td P(\Td P^*\Td P)^{-1}+
\gamma\ol{\pr_f}^*\Td P(\Td P^*\Td P)^{-1}(\dbar r)^{\wedge}
\end{equation}
on the space $C^\infty(\Gamma;\, I^{0,q}T^*(\Gamma))$. From (\ref{Be:0801052150}) and (\ref{Be:0711021310}), we have
\begin{align} \label{Be:0801082344}
&\gamma(\dbar r)^{\wedge}\ol{\pr_f}^*\Td P(\Td P^*\Td P)^{-1}=
\sum^{n-1}_{j=1}\Bigr((\dbar r)^{\wedge}t^{\wedge, *}_j\Bigr)\circ\Bigr(T^*_j\circ2\sqrt{-\La_\Gamma}\Bigr)\nonumber \\
&\quad+\Bigr((\dbar r)^{\wedge}(\dbar r)^{\wedge, *}\Bigr)\circ\Bigr((iY-\sqrt{-\La_\Gamma})\circ2\sqrt{-\La_\Gamma}\Bigr)
+\mbox{lower order terms}
\end{align}
and
\begin{align} \label{Be:0801082343}
&\gamma\ol{\pr_f}^*\Td P(\Td P^*\Td P)^{-1}(\dbar r)^{\wedge}=
\sum^{n-1}_{j=1}\Bigr(t^{\wedge, *}_j(\dbar r)^{\wedge}\Bigr)\circ\Bigr(T^*_j\circ2\sqrt{-\La_\Gamma}\Bigr)\nonumber \\
&\quad+\Bigr((\dbar r)^{\wedge, *}(\dbar r)^{\wedge}\Bigr)\circ\Bigr((iY-\sqrt{-\La_\Gamma})\circ2\sqrt{-\La_\Gamma}\Bigr)
+\mbox{lower order terms}.
\end{align}
Thus,
\begin{align} \label{Be:0801082345}
&\gamma(\dbar r)^{\wedge}\ol{\pr_f}^*\Td P(\Td P^*\Td P)^{-1}+
\gamma\ol{\pr_f}^*\Td P(\Td P^*\Td P)^{-1}(\dbar r)^{\wedge} \nonumber \\
&\quad=\sum^{n-1}_{j=1}\Bigr(t^{\wedge, *}_j(\dbar r)^{\wedge}+(\dbar r)^{\wedge}t^{\wedge, *}_j\Bigr)\circ\Bigr(T^*_j\circ2\sqrt{-\La_\Gamma}\Bigr)
 \nonumber \\
&\quad+\Bigr((\dbar r)^{\wedge, *}(\dbar r)^{\wedge}+(\dbar r)^{\wedge}(\dbar r)^{\wedge, *}\Bigr)
\circ\Bigr((iY-\sqrt{-\La_\Gamma})\circ2\sqrt{-\La_\Gamma}\Bigr)\nonumber \\
&\quad+\mbox{lower order terms}.
\end{align}
Combining this with (\ref{Be:0801082330}), (\ref{Be:0804030849}) and (\ref{Be:0804030850}), we get (\ref{Be:0801082351}).
The proposition follows.
\end{proof}


\section{The operator $\Box^{(q)}_\beta$}

Put
\begin{equation} \label{Be:0712302232}
\ol{\pr_\beta}=T\gamma\dbar \Td P: C^\infty(\Gamma;\, \Lambda^{0,q}T^*(\Gamma))\To C^\infty(\Gamma;\, \Lambda^{0,q+1}T^*(\Gamma)).
\end{equation}
We recall that (see (\ref{Be:0712302117})) the orthogonal projection $T$ onto ${\rm Ker\,}(\dbar r)^{\wedge, *}$ with respect to
$[\ |\ ]$ is a classical pseudodifferential operator of order $0$ with principal symbol
$2(\dbar r)^{\wedge, *}(\dbar r)^\wedge$.
(See Lemma~\ref{Bl:0712302115}.)
$\ol{\pr_\beta}$ is a classical
pseudodifferential operator of order one from boundary $(0, q)$ forms to boundary $(0, q+1)$ forms.

\begin{lem} \label{Bl:0711302204}
We have $(\ol{\pr_\beta})^2=0$.
\end{lem}

\begin{proof}
Let $u$, $v\in C^\infty(\Gamma;\, \Lambda^{0,q}T^*(\Gamma))$. We claim that
\begin{equation} \label{Be:0712301011}
[T\gamma\dbar \Td P(I-T)\gamma\dbar \Td Pu\ |\ v]=0.
\end{equation}
We have
\begin{align*}
[T\gamma\dbar \Td P(I-T)\gamma\dbar \Td Pu\ |\ v]&=[\gamma\dbar \Td P(I-T)\gamma\dbar \Td Pu\ |\ v]
\ \ \ \ \ \ \ \ \ \ \ (\mbox{since $v\in{\rm Ker\,}(\dbar r)^{\wedge, *}$})\\
&=(\Td P\gamma\dbar \Td P(I-T)\gamma\dbar \Td Pu\ |\ \Td Pv)_M \\
&=(\dbar \Td P(I-T)\gamma\dbar \Td Pu\ |\ \Td Pv)_M \ \ \ \ \ (\mbox{here we used (\ref{Be:0809081416})})\\
&=(\Td P(I-T)\gamma\dbar \Td Pu\ |\ \ol{\pr_f}^*\Td Pv)_M\ \ \mbox{(since $\Td Pv\in{\rm Dom\,}\dbar^*$)}\\
&=[(I-T)\gamma\dbar \Td Pu\ |\ \gamma\ol{\pr_f}^*\Td Pv]\ \ (\mbox{here we used (\ref{Be:0809081416})}).
\end{align*}
From Lemma~\ref{Bl:0710312345}, we have
$\gamma\ol{\pr_f}^*\Td Pv\in{\rm Ker\,}(\dbar r)^{\wedge, *}$.
Thus,
\[[(I-T)\gamma\dbar \Td Pu\ |\ \gamma\ol{\pr_f}^*\Td Pv]=0.\]
We get (\ref{Be:0712301011}), and hence
$T\gamma\dbar \Td P\gamma\dbar \Td Pu=T\gamma\dbar \Td PT\gamma\dbar \Td Pu$, $u\in C^\infty(\Gamma;\, \Lambda^{0,q}T^*(\Gamma))$.

Now,
$(\ol{\pr_\beta})^2=T\gamma\dbar \Td PT\gamma\dbar \Td P=T\gamma\dbar \Td P\gamma\dbar \Td P=T\gamma\dbar^2\Td P=0$.
The lemma follows.
\end{proof}

We pause and recall the tangential Cauchy-Riemann operator.
For $z\in \Gamma$, let
$\pi^{0,q}_z:\Lambda^{0,q}T^*_z(M')\To \Lambda^{0,q}T^*_z(\Gamma)$
be the orthogonal projection map (with respect to $(\ |\ )$). We can check that
$\pi^{0,q}_z=2(\dbar r(z))^{\wedge, *}(\dbar r(z))^\wedge$.
For an open set $U\subset \Gamma$, the tangential
Cauchy-Riemann operator:
$\ol{\pr_b}: C^\infty(U;\, \Lambda^{0,q}T^*(\Gamma))\To C^\infty(U;\, \Lambda^{0,q+1}T^*(\Gamma))$
is now defined as follows: for any $\phi\in C^\infty(U;\, \Lambda^{0,q}T^*(\Gamma))$, let $\Td U$
be an open set in $M'$ with $\Td U\cap \Gamma=U$ and pick
$\phi_1\in C^\infty(\Td U;\, \Lambda^{0,q}T^*(M'))$ that satisfies
$\pi^{0,q}_z(\phi_1(z))=\phi(z)$, for all $z\in U$. Then $\ol{\pr_b}\phi$ is defined to be
a smooth section in $C^\infty(U;\, \Lambda^{0,q+1}T^*(\Gamma))$:
$z\To\pi^{0,q}_z(\gamma\dbar\phi_1(z))$.
It is not difficult to check that the definition of $\ol{\pr_b}$ is independent of the choice of $\phi_1$. Since
$\dbar^2=0$, we have $\ol{\pr_b}^2=0$.
Let $\ol{\pr_b}^*$ be the formal adjoint of $\ol{\pr_b}$ with respect to $(\ |\ )_\Gamma$, that is
$(\dbar_{b}f\ |\ h)_\Gamma=(f\ |\ \ol{\pr_b}^*h)_\Gamma$, $f\in C^\infty_0(U\ ;\Lambda^{0,q}T^*(\Gamma))$,
$h\in C^\infty(U\ ;\Lambda^{0,q+1}T^*(\Gamma))$.
$\ol{\pr_b}^*$ is a differential operator of order one from boundary $(0, q+1)$ forms to boundary $(0, q)$ forms and
$(\ol{\pr_b}^*)^2=0$.

From the definition of $\ol{\pr_b}$, we know that
$\ol{\pr_b}=2(\dbar r)^{\wedge, *}(\dbar r)^{\wedge}\gamma\dbar \Td P$.
Since the principal symbol of $T$ is $2(\dbar r)^{\wedge, *}(\dbar r)^{\wedge}$, it follows that
\begin{equation} \label{Be:0711021453}
\ol{\pr_\beta}=\ol{\pr_b}+\mbox{lower order terms}.
\end{equation}

Let
\begin{equation} \label{Be:0712302233}
\ol{\pr_\beta}^\dagger: C^\infty(\Gamma;\, \Lambda^{0,q+1}T^*(\Gamma))\To C^\infty(\Gamma;\, \Lambda^{0,q}T^*(\Gamma)),
\end{equation}
be the formal adjoint of $\ol{\pr_\beta}$ with respect to $[\ |\ ]$, that is
$[\ol{\pr_\beta} f\ |\ h]=[f\ |\ \ol{\pr_\beta}^\dagger h]$,
$f\in C^\infty(\Gamma;\, \Lambda^{0,q}T^*(\Gamma))$, $h\in C^\infty(\Gamma;\, \Lambda^{0,q+1}T^*(\Gamma))$.
$\ol{\pr_\beta}^\dagger$ is a classical
pseudodifferential operator of order one from boundary $(0, q+1)$ forms to boundary $(0, q)$ forms.

\begin{lem} \label{Bl:0801052333}
We have $\ol{\pr_\beta}^\dagger=\gamma\ol{\pr_f}^*\Td P$.
\end{lem}

\begin{proof}
Let $u\in C^\infty(\Gamma;\, \Lambda^{0,q}T^*(\Gamma))$, $v\in C^\infty(\Gamma;\, \Lambda^{0,q+1}T^*(\Gamma))$. We have
\begin{align*}
[\ol{\pr_\beta} u\ |\ v]&=[T\gamma\dbar \Td Pu\ |\ v]=[\gamma\dbar \Td Pu\ |\ v] \\
&=(\Td P\gamma\dbar \Td Pu\ |\ \Td Pv)_M=(\dbar \Td Pu\ |\ \Td Pv)_M \\
&=(\Td Pu\ |\ \ol{\pr_f}^*\Td Pv)_M=[u\ |\ \gamma\ol{\pr_f}^*\Td Pv],
\end{align*}
and the lemma follows.
\end{proof}

\begin{rem} \label{Br:0801061451}
We can check that on boundary $(0, q)$ forms, we have
\begin{equation} \label{Be:0801061454}
\ol{\pr_\beta}^\dagger=\gamma\ol{\pr_f}^*\Td P=\ol{\pr_b}^*+\mbox{lower order terms}.
\end{equation}
\end{rem}

Set
\begin{equation} \label{Be:BP-2.2}
\Box^{(q)}_\beta=\ol{\pr_\beta}^\dagger\ol{\pr_\beta}+\ol{\pr_\beta}\ \ol{\pr_\beta}^\dagger:
\mathscr D'(\Gamma;\, \Lambda^{0,q}T^*(\Gamma))\To\mathscr D'(\Gamma;\, \Lambda^{0,q}T^*(\Gamma)).
\end{equation}
$\Box^{(q)}_\beta$ is a classical
pseudodifferential operator of order two from boundary $(0, q)$ forms to boundary $(0, q)$ forms. We recall that the
Kohn Laplacian on $\Gamma$ is given by
$\Box^{(q)}_b=\ol{\pr_b}\ \ol{\pr_b}^*+\ol{\pr_b}^*\ol{\pr_b}:
\mathscr D'(\Gamma;\, \Lambda^{0,q}T^*(\Gamma))\To\mathscr D'(\Gamma;\, \Lambda^{0,q}T^*(\Gamma))$.
From (\ref{Be:0711021453}) and (\ref{Be:0801061454}), we see that
$\sigma_{\Box^{(q)}_b}=\sigma_{\Box^{(q)}_\beta}$
and the characteristic manifold of $\Box^{(q)}_\beta$ is
$\Sigma=\Sigma^+\bigcup\Sigma^-$,
where $\Sigma^+$, $\Sigma^-$ are given by (\ref{Be:BP-4.1}).
(See section $3$ of part ${\rm I\,}$.) Moreover, $\sigma_{\Box^{(q)}_\beta}$ vanishes to second order on $\Sigma$ and we have
\begin{equation} \label{Be:0801061515}
\Box^{(q)}_\beta=\Box^{(q)}_b+L_1,
\end{equation}
where $L_1$ is a classical
pseudodifferential operator of order one with
\begin{equation} \label{Be:0801061516}
\sigma_{L_1}=0\ \ \mbox{at each point of\, $\Sigma$}.
\end{equation}

The following is well-known (see the proof of Lemma $3.1$ of part ${\rm I\,}$)

\begin{lem} \label{Bl:0711021655}
$\Sigma$ is a symplectic submanifold of $T^*(\Gamma)$ if and only if the Levi form is non-degenerate at each point of\, $\Gamma$.
(For the precise definition of symplectic manifold, see chapter ${\rm XVIII\,}$ of H\"{o}rmander~\cite{Hor85}.)
\end{lem}

Let $p^s_\beta$ denote the subprincipal symbol of $\Box^{(q)}_\beta$ (invariantly defined on $\Sigma$) and let
$F_\beta(\rho)$ denote the fundamental matrix of $\sigma_{\Box^{(q)}_\beta}$ at $\rho\in\Sigma$. We write $\Td{\rm tr\,}F_\beta(\rho)$ to denote
$\sum\abs{\lambda_j}$, where $\pm i\lambda_j$ are the non-vanishing eigenvalues of $F_\beta(\rho)$. From (\ref{Be:0801061515}) and (\ref{Be:0801061516}),
we see that $p^s_\beta+\frac{1}{2}\Td{\rm tr\,}F_\beta=p^s_b+\frac{1}{2}\Td{\rm tr\,}F_b$ on $\Sigma$,
where $p^s_b$ is the subprincipal symbol of $\Box^{(q)}_b$ and
$F_b$ is the fundamental matrix of $\sigma_{\Box^{(q)}_b}$. (For the precise meanings of
subprincipal symbol and fundamental matrix, see section $3$ of part ${\rm I\,}$.) We have the following

\begin{lem} \label{Bl:0711021707}
Let $\rho=(p, \xi)\in\Sigma$. Then
\begin{equation} \label{Be:0711021710}
\frac{1}{2}\Td{\rm tr\,}F_\beta+p^s_\beta=\sum^{n-1}_{j=1}\abs{\lambda_j}\abs{\sigma_{iY}}+\Bigr(\sum^{n-1}_{j=1}L_p(\ol T_j, T_j)-
\sum^{n-1}_{j, k=1}2t^\wedge_jt^{\wedge, *}_kL_p(\ol T_k, T_j)\Bigr)\sigma_{iY}\ \ \mbox{at}\ \ \rho,
\end{equation}
where $\lambda_j$, $j=1,\ldots,n-1$, are the eigenvalues of $L_p$ and $T_j$, $t_j$, $j=1,\ldots,n-1$, are as in (\ref{Be:0711012340})
and (\ref{Be:0711012345}).
\end{lem}

\begin{proof}
See section $3$ of part ${\rm I\,}$.
\end{proof}

It is not difficult to see that on $\Sigma$ the action of
$\frac{1}{2}\Td{\rm tr\,}F_\beta+p^s_\beta$ on boundary $(0,q)$ forms has the eigenvalues
\begin{equation} \label{Be:h-eigenmore}
\begin{split}
& \sum^{n-1}_{j=1}\abs{\lambda_j}\abs{\sigma_{iY}}+\sum_{j\notin J}\lambda_j\sigma_{iY}-\sum_{j\in J}
  \lambda_j\sigma_{iY}, \; \abs{J}=q,  \\
& J=(j_1,j_2,\ldots,j_q),\; 1\leq j_1<j_2<\cdots<j_q\leq n-1.
\end{split}
\end{equation}
(See section $3$ of part ${\rm I\,}$.)
We assume that the Levi form is non-degenerate at $p\in\Gamma$. Let $(n_-, n_+)$, $n_-+n_+=n-1$, be the signature of $L_p$.
Since $\seq{Y, \omega_0}=-1$, we have $\sigma_{iY}>0$ on $\Sigma^+$, $\sigma_{iY}<0$
on $\Sigma^-$.

Let
\[\inf\, (p^s_\beta+\frac{1}{2}\Td{\rm tr\,}F_\beta)(\rho)=\inf\set{\lambda;\, \lambda:
{\rm eigenvalue\ of\ }(p^s_\beta+\frac{1}{2}\Td{\rm tr\,}F_\beta)(\rho)},\ \rho\in\Sigma.\]
From (\ref{Be:h-eigenmore}), we see that at $(p, \omega_0(p))\in\Sigma^+$,
\begin{equation} \label{Be:h-hsiao1}
\inf\,(p^s_\beta+\frac{1}{2}\Td{\rm tr\,}F_\beta) \left\{ \begin{array}{ll}
=0,   &  q=n_+   \\
>0,   &  q\neq n_+
\end{array}\right..
\end{equation}
At $(p, -\omega_0(p))\in\Sigma^-$,
\begin{equation} \label{Be:h-hsiao2}
\inf\,(p^s_\beta+\frac{1}{2}\Td{\rm tr\,}F_\beta) \left\{ \begin{array}{ll}
=0,   &  q=n_-   \\
>0,   &  q\neq n_-
\end{array}\right..
\end{equation}

\begin{defn} \label{Bd:h-Yofq}
Given $q$, $0\leq q\leq n-1$, the Levi form is said to satisfy condition $Y(q)$ at $p\in\Gamma$
if for any $\abs{J}=q$, $J=(j_1,j_2,\ldots,j_q)$, $1\leq j_1<j_2<\cdots<j_q\leq n-1$, we have
\[\abs{\sum_{j\notin J}\lambda_j-\sum_{j\in J}\lambda_j}<\sum^{n-1}_{j=1}\abs{\lambda_j},\]
where $\lambda_j$, $j=1,\ldots,(n-1)$, are the eigenvalues of $L_p$. If the
Levi form is non-degenerate at $p$, then the condition is equivalent to $q\neq n_+$, $n_-$, where $(n_-, n_+)$, $n_-+n_+=n-1$, is the
signature of $L_p$.
\end{defn}

From now on, we assume that
\begin{ass} \label{Ba:1}
The Levi form is non-degenerate at each point of\, $\Gamma$.
\end{ass}

By classical works of Boutet de Monvel~\cite{Bou74} and Sj\"{o}strand~\cite{Sjo74}, we get
the following

\begin{prop} \label{Bp:h-HorBoutet}
$\Box^{(q)}_\beta$ is hypoelliptic
with loss of one derivative if and only if $Y(q)$ holds at each point of\, $\Gamma$.
\end{prop}


\section{The heat equation for $\Box^{(q)}_\beta$}

In this section, we will apply some results of Menikoff-Sj\"{o}strand~\cite{MS78} to construct approximate orthogonal projection for
$\Box^{(q)}_\beta$. Our presentation is essentially taken from part ${\rm I\,}$. The reader who is familiar with part ${\rm I\,}$ may go
directly to Theorem~\ref{Bt:0801090850}.

Until further notice,
we work with local coordinates $x=(x_1,x_2,\ldots,x_{2n-1})$
defined on a connected open set $\Omega\subset\Gamma$. We identify $T^*(\Omega)$ with $\Omega\times\Real^{2n-1}$.
Thus, the Levi form has constant signature on $\Omega$.
For any $C^\infty$ function $f$, we also write
$f$ to denote an almost analytic extension.
(For the precise meaning of almost analytic functions, we refer the reader to Definition~$1.1$of Melin-Sj\"{o}strand~\cite{MS74}.)
We let the full symbol of $\Box^{(q)}_\beta$ be:
\[\mbox{full symbol of }\ \Box^{(q)}_\beta\sim\sum^\infty_{j=0}q_j(x,\xi),\]
where $q_j(x,\xi)$ is positively homogeneous of order $2-j$.

First, we consider the characteristic equation for $\pr_t+\Box^{(q)}_\beta$. We look for solutions $\psi(t,x,\eta)\in C^\infty
(\ol\Real_+\times T^*(\Omega)\setminus0)$ of the problem
\begin{equation} \label{Be:c-chara}
\left\{ \begin{split}
& \frac{\displaystyle\pr\psi}{\displaystyle\pr t}-iq_0(x,\psi'_x)=
   O(\abs{{\rm Im\,}\psi}^N),\ \forall N\geq0,   \\
& \psi|_{t=0}=\seq{x\ ,\eta} \end{split}\right.
\end{equation}
with ${\rm Im\,}\psi(t, x,\eta)\geq0$.

Let $U$ be an open set in $\Real^n$ and let $f$, $g\in C^\infty(U)$. We write
$f\asymp g$
if for every compact set $K\subset U$ there is a constant $c_K>0$ such that
$f\leq c_Kg$, $g\leq c_Kf$ on $K$.
We have the following

\begin{prop} \label{Bp:c-basis}
There exists $\psi(t,x,\eta)\in C^\infty(\ol\Real_+\times T^*(\Omega)\setminus0)$ such
that ${\rm Im\,}\psi\geq 0$ with equality precisely on $(\set{0}\times T^*(\Omega)\setminus0)
\bigcup(\Real_+\times\Sigma)$ and such that {\rm (\ref{Be:c-chara})} holds where the error term is
uniform on every set of the form $[0,T]\times K$ with $T>0$ and $K\subset T^*(\Omega)\setminus0$ compact.
Furthermore,
\[\psi(t,x,\eta)=\seq{x,\eta}{\;\rm on\;}\Sigma,\; d_{x,\eta}(\psi-\seq{x,\eta})=0{\;\rm on\;}
\Sigma,\]
\[\psi(t, x, \lambda\eta)=\lambda\psi(\lambda t, x, \eta),\ \ \lambda>0,\]
\begin{equation} \label{Be:c-sj1}
{\rm Im\,}\psi(t, x,\eta)\asymp\abs{\eta}\frac{t\abs{\eta}}{1+t\abs{\eta}}{\rm dist\,}
((x, \frac{\eta}{\abs{\eta}}),\Sigma))^2,
\ \ t\geq0,\ \abs{\eta}\geq1.
\end{equation}
\end{prop}

\begin{prop} \label{Bp:c-basislimit}
There exists a function $\psi(\infty,x,\eta)\in C^\infty(T^*(\Omega)\setminus0)$ with a
uniquely determined Taylor expansion at each point of\, $\Sigma$ such that
\[\begin{split}
&\mbox{For every compact set $K\subset T^*(\Omega)\setminus0$ there is a constant $c_K>0$
such that}   \\
&{\rm Im\,}\psi(\infty,x,\eta)\geq c_K\abs{\eta}({\rm dist\,}((x,\frac{\eta}{\abs{\eta}}),\Sigma))^2,  \\
&d_{x,\eta}(\psi(\infty, x, \eta)-\seq{x,\eta})=0\ {\rm on\,}\ \Sigma.
\end{split}\]
If$\ $ $\lambda\in C(T^*(\Omega)\setminus 0)$, $\lambda>0$ and $\lambda|_\Sigma<\min\abs{\lambda_j}$, where
$\pm i\abs{\lambda_j}$ are the non-vanishing eigenvalues of the fundamental matrix of\, $\Box^{(q)}_\beta$,
then the solution $\psi(t,x,\eta)$ of {\rm (\ref{Be:c-chara})} can be chosen so that for every
compact set $K\subset T^*(\Omega)\setminus0$ and all indices $\alpha$, $\beta$, $\gamma$,
there is a constant $c_{\alpha,\beta,\gamma,K}$ such that
\begin{equation} \label{Be:c-************}
\abs{\pr^\alpha_x\pr^\beta_\eta\pr^\gamma_t(\psi(t,x,\eta)-\psi(\infty,x,\eta))}
\leq c_{\alpha,\beta,\gamma,K}e^{-\lambda(x,\eta)t}{\;\rm on\;}\ol\Real_+\times K.
\end{equation}
\end{prop}

For the proofs of Proposition~\ref{Bp:c-basis}, Proposition~\ref{Bp:c-basislimit},
we refer the reader to~\cite{MS78}.

\begin{defn} \label{Bd:hf-heatquasi}
We will say that $a\in C^\infty(\ol\Real_+\times T^*(\Omega))$ is quasi-homogeneous of
degree $j$ if $a(t,x,\lambda\eta)=\lambda^ja(\lambda t,x,\eta)$ for all $\lambda>0$.
\end{defn}

We consider the problem
\begin{equation} \label{Be:hf-heat}
\left\{ \begin{array}{ll}
(\pr_t+\Box^{(q)}_\beta)u(t,x)=0  & {\rm in\;}\Real_+\times\Omega  \\
u(0,x)=v(x) \end{array}\right..
\end{equation}
We shall start by making only a formal construction. We look for an approximate solution of
(\ref{Be:hf-heat}) of the form $u(t,x)=A(t)v(x)$,
\begin{equation} \label{Be:hf-fourierheat}
A(t)v(x)=\frac{1}{(2\pi)^{2n-1}}\int\!\int e^{i(\psi(t,x,\eta)-\seq{y,\eta})}a(t,x,\eta)v(y)dyd\eta
\end{equation}
where formally
\[a(t,x,\eta)\sim\sum^\infty_{j=0}a_j(t,x,\eta),\]
$a_j(t, x, \eta)\in C^\infty(\ol\Real_+\times T^*(\Omega);\,
\mathscr L(\Lambda^{0,q}T^*(\Gamma), \Lambda^{0,q}T^*(\Gamma)))$, $a_j(t,x,\eta)$ is a quasi-homogeneous function of degree $-j$.

We apply $\pr_t+\Box^{(q)}_\beta$ formally
under the integral in (\ref{Be:hf-fourierheat}) and then introduce the asymptotic expansion of
$\Box^{(q)}_\beta(ae^{i\psi})$. Setting $(\pr_t+\Box^{(q)}_\beta)(ae^{i\psi})\sim 0$ and regrouping
the terms according to the degree of quasi-homogeneity. We obtain the transport equations
\begin{equation} \label{Be:hf-heattransport}
\left\{ \begin{array}{ll}
 T(t,x,\eta,\pr_t,\pr_x)a_0=O(\abs{{\rm Im\,}\psi}^N),\; \forall N   \\
 T(t,x,\eta,\pr_t,\pr_x)a_j+l_j(t,x,\eta,a_0,\ldots,a_{j-1})= O(\abs{{\rm Im\,}\psi}^N),\; \forall N.
 \end{array}\right.
\end{equation}
Here
\[T(t,x,\eta,\pr_t,\pr_x)=\pr_t-i\sum^{2n-1}_{j=1}\frac{\pr q_0}{\pr\xi_j}(x,\psi'_x)\frac{\pr}{\pr x_j}+q(t,x,\eta)\]
where
\[q(t,x,\eta)=q_1(x,\psi'_x)+\frac{1}{2i}\sum^{2n-1}_{j,k=1}\frac{\pr^2q_0(x,\psi'_x)}
    {\pr\xi_j\pr\xi_k}\frac{\pr^2\psi(t,x,\eta)}{\pr x_j\pr x_k}\]
and $l_j(t, x, \eta)$ is a linear differential operator acting on $a_0,a_1,\ldots,a_{j-1}$.
We can repeat the method of part ${\rm I\,}$ (see Proposition~$5.6$ of part ${\rm I\,}$) to get the following

\begin{prop} \label{Bp:0711101910}
Let $(n_-, n_+)$, $n_-+n_+=n-1$, be the signature of the Levi form on $\Omega$. We can find solutions
\[a_j(t,x,\eta)\in
C^\infty(\ol\Real_+\times T^*(\Omega);\,
\mathscr L(\Lambda^{0,q}T^*(\Gamma), \Lambda^{0,q}T^*(\Gamma))),\ \ j=0, 1,\ldots,\]
of the system (\ref{Be:hf-heattransport}) with
$a_0(0, x, \eta)=I$, $a_j(0, x, \eta)=0$ when $j>0$,
where $a_j(t,x,\eta)$ is a quasi-homogeneous function of degree $-j$, such that $a_j$ has unique Taylor expansions
on $\Sigma$. Moreover, we can find
\[a_j(\infty,x,\eta)\in C^\infty(T^*(\Omega);\,
\mathscr L(\Lambda^{0,q}T^*(\Gamma), \Lambda^{0,q}T^*(\Gamma))),\ \ j=0,1,\ldots,\]
where $a_j(\infty,x,\eta)$ is a positively homogeneous function of degree $-j$, $\varepsilon_0>0$ such that for
all indices $\alpha$, $\beta$, $\gamma$, $j$, every compact set $K\subset\Omega$,  there
exists $c>0$, such that
\begin{equation} \label{Be:hf-Sjoestimate}
\abs{\pr^\gamma_t\pr^\alpha_x\pr^\beta_\eta(a_j(t,x,\eta)-a_j(\infty,x,\eta))}\leq
ce^{-\varepsilon_0t\abs{\eta}}(1+\abs{\eta})^{-j-\abs{\beta}+\gamma}
\end{equation}
on $\ol\Real_+\times\Bigr((K\times\Real^{2n-1})\bigcap\Sigma\Bigr)$, $\abs{\eta}\geq 1$.

Furthermore, for all $j=0,1,\ldots$,
\begin{equation} \label{Be:hf-******************}
\left\{ \begin{array}{ll}
\textup{all derivatives of}\ a_j(\infty, x, \eta)\ \textup{vanish at}\ \ \Sigma^+ & \textup{if}\ \ q\neq n_+    \\
\textup{all derivatives of}\ a_j(\infty, x, \eta)\ \textup{vanish at}\ \ \Sigma^- & \textup{if}\ \ q\neq n_-
\end{array}\right.
\end{equation}
and
\begin{equation} \label{Be:0801061733}
\left\{ \begin{array}{ll}
a_0(\infty, x, \eta)\neq0\ \textup{at each point of}\ \ \Sigma^+ & \textup{if}\ \ q=n_+    \\
a_0(\infty, x, \eta)\neq0\ \textup{at each point of}\ \ \Sigma^- & \textup{if}\ \ q=n_-
\end{array}\right..
\end{equation}
\end{prop}

\begin{defn} \label{Bd:ss-heatgeneral}
Let $r(x,\eta)$ be a non-negative real continuous function on $T^*(\Omega)$. We assume that
$r(x,\eta)$ is positively homogeneous of degree 1, that is, $r(x,\lambda\eta)=\lambda r(x,\eta)$, for
$\lambda\geq1$, $\abs{\eta}\geq 1$. For $0\leq q\leq n-1$ and $k\in\Real$, we say that
\[a\in\hat S^k_r(\ol\Real_+\times T^*(\Omega);\, \mathscr L(\Lambda^{0,q}T^*(\Gamma), \Lambda^{0,q}T^*(\Gamma)))\]
if $a\in C^\infty(\ol\Real_+\times T^*(\Omega);\, \mathscr L(\Lambda^{0,q}T^*(\Gamma), \Lambda^{0,q}T^*(\Gamma)))$
and for all indices $\alpha$, $\beta$, $\gamma$, every compact set $K\subset\Omega$ and every
$\varepsilon>0$, there exists a constant $c>0$ such that
\[\abs{\pr^\gamma_t\pr^\alpha_x\pr^\beta_\eta a(t,x,\eta)}\leq
 ce^{t(-r(x,\eta)+\varepsilon\abs{\eta})}(1+\abs{\eta})^{k+\gamma-\abs{\beta}},\;x\in K,\;\abs{\eta}\geq 1.\]
\end{defn}

\begin{rem} \label{Br:ss-heatclass}
It is easy to see that we have the following properties:
 \begin{enumerate}
  \item If $a\in\hat S^k_{r_1}$, $b\in\hat S^l_{r_2}$ then $ab\in\hat S^{k+l}_{r_1+r_2}$,
        $a+b\in\hat S^{\max(k,l)}_{\min(r_1,r_2)}$.
  \item If $a\in\hat S^k_r$ then $\pr^\gamma_t\pr^\alpha_x\pr^\beta_\eta a\in
        \hat S^{k-\abs{\beta}+\gamma}_r$.
  \item If $a_j\in\hat S^{k_j}_r$, $j=0,1,2,\ldots$ and $k_j\searrow -\infty$ as $j\To\infty$,
        then there exists $a\in\hat S^{k_0}_r$ such that $a-\sum^{v-1}_0a_j\in
        \hat S^{k_v}_r$, for all $v=1,2,\ldots$. Moreover, if $\hat S^{-\infty}_r$ denotes $\bigcap_{k\in\Real}\hat S^k_r$
        then $a$ is unique modulo $\hat S^{-\infty}_r$.
 \end{enumerate}

If $a$ and $a_j$ have the properties of $(c)$, we write
\[a\sim\sum^\infty_0a_j\ \ \textup{in the symbol space}\ \hat S^{k_0}_r.\]
\end{rem}

From Proposition~\ref{Bp:0711101910} and the standard Borel construction,
we get the following

\begin{prop} \label{Bp:ss-SBmore}
Let $(n_-, n_+)$, $n_-+n_+=n-1$, be the signature of the Levi form on $\Omega$.
We can find solutions
\[a_j(t,x,\eta)\in
C^\infty(\ol\Real_+\times T^*(\Omega);\,
\mathscr L(\Lambda^{0,q}T^*(\Gamma), \Lambda^{0,q}T^*(\Gamma))),\ \ j=0, 1,\ldots,\]
of the system (\ref{Be:hf-heattransport}) with
$a_0(0, x, \eta)=I$, $a_j(0, x, \eta)=0$ when $j>0$,
where $a_j(t,x,\eta)$ is a quasi-homogeneous function of degree $-j$, such that for some $r>0$ as in
Definition~\ref{Bd:ss-heatgeneral},
\[a_j(t,x,\eta)-a_j(\infty,x,\eta)\in\hat S^{-j}_r(\ol\Real_+\times T^*(\Omega);\,
\mathscr L(\Lambda^{0,q}T^*(\Gamma), \Lambda^{0,q}T^*(\Gamma))),\ \ j=0, 1,\ldots,\]
where
\[a_j(\infty,x,\eta)\in C^\infty(T^*(\Omega);\,
\mathscr L(\Lambda^{0,q}T^*(\Gamma), \Lambda^{0,q}T^*(\Gamma))),\ \ j=0, 1,\ldots,\]
and $a_j(\infty, x, \eta)$ is a positively homogeneous function of degree $-j$.

Furthermore, for all $j=0,1,\ldots$,
\[ \left\{ \begin{array}{ll}
a_j(\infty, x, \eta)=0\ \ \mbox{in a conic neighborhood of}\ \ \Sigma^+, & {\rm if\ } q\neq n_+,  \\
a_j(\infty ,x,\eta)=0\ \ \mbox{in a conic neighborhood of}\ \ \Sigma^-, & {\rm if\ } q\neq n_-.
\end{array}\right.\]
\end{prop}

\begin{rem} \label{Br:0801062134}
Let
$b(t,x,\eta)\in\hat S^k_r(\ol\Real_+\times T^*(\Omega);\, \mathscr L(\Lambda^{0,q}T^*(\Gamma),\Lambda^{0,q}T^*(\Gamma)))$
with $r>0$.
We assume that $b(t,x,\eta)=0$ when $\abs{\eta}\leq 1$. Let $\chi\in C^\infty_0(\Real^{2n-1})$ be equal to $1$ near the origin. Put
$B_\eps(x, y)=\int\Bigr(\int^\infty_0e^{i(\psi(t,x,\eta)-\seq{y,\eta})}b(t,x,\eta)dt\Bigr)\chi(\eps\eta)d\eta$.
For $u\in C^\infty_0(\Omega;\, \Lambda^{0,q}T^*(\Gamma))$, we can show that
\[\lim_{\eps\To0}(\int B_{\eps}(x, y)u(y)dy)\in C^\infty(\Omega;\, \Lambda^{0,q}T^*(\Gamma))\]
and
\begin{align*}
B: C^\infty_0(\Omega;\, \Lambda^{0,q}T^*(\Gamma))&\To C^\infty(\Omega;\, \Lambda^{0,q}T^*(\Gamma)) \\
u&\To\lim_{\eps\To0}(\int B_{\eps}(x, y)u(y)dy)
\end{align*}
is continuous. Formally,
\[B(x, y)=\int\Bigr(\int^\infty_0e^{i(\psi(t,x,\eta)-\seq{y,\eta})}b(t,x,\eta)dt\Bigr)d\eta.\]
Moreover, $B$ has a unique continuous extension:
\[B: \mathcal E'(\Omega;\, \Lambda^{0,q}T^*(\Gamma))\To\mathcal D'(\Omega;\, \Lambda^{0,q}T^*(\Gamma))\]
and
$B(x,y)\in C^\infty(\Omega\times\Omega\smallsetminus{\rm diag\,}(\Omega\times\Omega);\,
\mathscr L(\Lambda^{0,q}T^*(\Gamma),\Lambda^{0,q}T^*(\Gamma)))$.
For the details, we refer the reader to Proposition~$6.6$ of part ${\rm I\,}$.
\end{rem}

\begin{rem} \label{Br:ss-oscillmore}
Let
$a(t, x, \eta)\in \hat S^k_0(\ol\Real_+\times T^*(\Omega);\,
\mathscr L(\Lambda^{0,q}T^*(\Gamma),\Lambda^{0,q}T^*(\Gamma)))$.
We assume $a(t, x, \eta)=0$, if $\abs{\eta}\leq 1$ and
\[a(t, x, \eta)-a(\infty, x, \eta)\in \hat S^{k}_r(\ol\Real_+\times T^*(\Omega);\,
\mathscr L(\Lambda^{0,q}T^*(\Gamma),\Lambda^{0,q}T^*(\Gamma)))\]
with $r>0$, where
$a(\infty, x, \eta)\in C^\infty(T^*(\Omega);\,
\mathscr L(\Lambda^{0,q}T^*(\Gamma),\Lambda^{0,q}T^*(\Gamma)))$.
Then we can also define
\begin{equation} \label{Be:0801062201}
A(x, y)=\int\Bigr(\int^{\infty}_0\Bigr(e^{i(\psi(t,x,\eta)-\seq{y,\eta})}a(t,x,\eta)-e^{i(\psi(\infty,x,\eta)-\seq{y,\eta})}
a(\infty,x,\eta)\Bigr)dt\Bigr)d\eta
\end{equation}
as an oscillatory integral by the following formula:
\[A(x, y)=\int\!\!\Bigr(\int^{\infty}_0\!\! e^{i(\psi(t,x,\eta)-
\seq{y,\eta})}(-t)(i\psi^{'}_t(t,x,\eta)a(t,x,\eta)+a'_t(t,x,\eta))dt\Bigr)d\eta.\]
We notice that
$(-t)(i\psi'_t(t,x,\eta)a(t,x,\eta)+a'_t(t,x,\eta))\in\hat S^{k+1}_r$, $r>0$.
\end{rem}

We recall the following

\begin{defn} \label{Bd:ss-symbol}
Let $k\in\Real$. $S^k_{\frac{1}{2},\frac{1}{2}}(T^*(\Omega);\, \mathscr L(\Lambda^{0,q}T^*(\Gamma),\Lambda^{0,q}T^*(\Gamma)))$
is the space of all
$a\in C^\infty(T^*(\Omega);\, \mathscr L(\Lambda^{0,q}T^*(\Gamma),\Lambda^{0,q}T^*(\Gamma)))$
such that for every
compact sets $K\subset\Omega$ and all $\alpha\in\Pstint^{2n-1}$, $\beta\in\Pstint^{2n-1}$, there is a constant
$c_{\alpha,\beta,K}>0$ such that
\[\abs{\pr^\alpha_x\pr^\beta_\xi a(x,\xi)}\leq c_{\alpha,\beta,K}(1+\abs{\xi})^{k-\frac{\abs{\beta}}{2}+
\frac{\abs{\alpha}}{2}},\]
$(x,\xi)\in T^*(\Omega)$, $x\in K$.
$S^k_{\frac{1}{2},\frac{1}{2}}$ is called the space of symbols of order $k$ type
$(\frac{1}{2},\frac{1}{2})$.
\end{defn}

\begin{defn} \label{Bd:ss-pseudomore}
Let $k\in\Real$. A pseudodifferential operator of order $k$ type
$(\frac{1}{2},\frac{1}{2})$ from sections of $\Lambda^{0,q}T^*(\Gamma)$ to sections of $\Lambda^{0,q}T^*(\Gamma)$
is a continuous linear map
$A:C^\infty_0(\Omega;\, \Lambda^{0,q}T^*(\Gamma))\To\mathscr D'(\Omega;\, \Lambda^{0,q}T^*(\Gamma))$
such that the distribution kernel of $A$ is
\[K_A=A(x, y)=\frac{1}{(2\pi)^{2n-1}}\int e^{i\seq{x-y,\xi}}a(x, \xi)d\xi\]
with $a\in S^k_{\frac{1}{2},\frac{1}{2}}(T^*(\Omega);\, \mathscr L(\Lambda^{0,q}T^*(\Gamma),\Lambda^{0,q}T^*(\Gamma)))$.
We call $a(x, \xi)$ the symbol of $A$.
We shall write
$L^k_{\frac{1}{2},\frac{1}{2}}(\Omega;\, \Lambda^{0,q}T^*(\Gamma),\Lambda^{0,q}T^*(\Gamma))$
to denote the space of
pseudodifferential operators of order $k$ type $(\frac{1}{2},\frac{1}{2})$ from sections of
$\Lambda^{0,q}T^*(\Gamma)$ to sections of $\Lambda^{0,q}T^*(\Gamma)$.
\end{defn}

We recall the following classical proposition of Calderon-Vaillancourt. (For a proof, see~\cite{Hor85}.)

\begin{prop} \label{Bp:he-calderon}
If $A\in L^k_{\frac{1}{2},\frac{1}{2}}(\Omega;\, \Lambda^{0,q}T^*(\Gamma),\Lambda^{0,q}T^*(\Gamma))$.
Then, for every $s\in\Real$, $A$ is continuous:
$A:H^s_{\rm comp}(\Omega;\, \Lambda^{0,q}T^*(\Gamma))\To H^{s-k}_{\rm loc}(\Omega;\, \Lambda^{0,q}T^*(\Gamma))$.
\end{prop}

We have the following

\begin{prop} \label{Bp:ss-pseudo}
Let
$a(t, x, \eta)\in \hat S^k_0(\ol\Real_+\times T^*(\Omega);\,
\mathscr L(\Lambda^{0,q}T^*(\Gamma),\Lambda^{0,q}T^*(\Gamma)))$.
We assume $a(t, x, \eta)=0$, if\, $\abs{\eta}\leq 1$ and
\[a(t, x, \eta)-a(\infty, x, \eta)\in \hat S^{k}_r(\ol\Real_+\times T^*(\Omega);\,
\mathscr L(\Lambda^{0,q}T^*(\Gamma),\Lambda^{0,q}T^*(\Gamma)))\]
with $r>0$, where
$a(\infty, x, \eta)\in C^\infty(T^*(\Omega);\,
\mathscr L(\Lambda^{0,q}T^*(\Gamma),\Lambda^{0,q}T^*(\Gamma)))$.
Let
\begin{align*}
A(x, y)=\frac{1}{(2\pi)^{2n-1}}\int\Bigl(\int^{\infty}_0 &\Bigl(e^{i(\psi(t,x,\eta)-\seq{y,\eta})}a(t,x,\eta)-  \\
                                                &e^{i(\psi(\infty,x,\eta)-\seq{y,\eta})}a(\infty,x,\eta)\Bigl)dt\Bigl)d\eta
\end{align*}
be as in (\ref{Be:0801062201}). Then
$A\in L^{k-1}_{\frac{1}{2},
\frac{1}{2}}(\Omega;\, \Lambda^{0,q}T^*(\Gamma),\Lambda^{0,q}T^*(\Gamma))$
with symbol
\[q(x,\eta)=\int^{\infty}_0\Bigr(e^{i(\psi(t,x,\eta)-\seq{x,\eta})}a(t,x,\eta)-
e^{i(\psi(\infty,x,\eta)-\seq{x,\eta})}a(\infty,x,\eta)\Bigr)dt\]
in $S^{k-1}_{\frac{1}{2},\frac{1}{2}}(T^*(\Omega);\, \mathscr L(\Lambda^{0,q}T^*(\Gamma),\Lambda^{0,q}T^*(\Gamma)))$.
\end{prop}

\begin{proof}
See Lemma~$6.14$ and Lemma~$6.16$ of part ${\rm I\,}$.
\end{proof}

From now on, we write
\[\frac{1}{(2\pi)^{2n-1}}\int\Bigl(\int^{\infty}_0\Bigl(e^{i(\psi(t,x,\eta)-\seq{y,\eta})}a(t,x,\eta)-e^{i(\psi(\infty,x,\eta)-\seq{y,\eta})}
a(\infty,x,\eta)\Bigl)dt\Bigl)d\eta \]
to denote the kernel of pseudodifferential operator of order $k-1$ type
$(\frac{1}{2},\frac{1}{2})$ from sections of $\Lambda^{0,q}T^*(\Gamma)$ to sections of
$\Lambda^{0,q}T^*(\Gamma)$.
Here $a(t, x, \eta)$, $a(\infty, x, \eta)$ are as in Proposition~\ref{Bp:ss-pseudo}.

The following is essentially well-known (See page $72$ of ~\cite{MS78}.)

\begin{prop} \label{Bp:he-basismore}
Let $Q$ be a properly supported pseudodifferential operator on $\Omega$ of order
$k>0$ with classical symbol $q(x, \xi)\in C^\infty(T^*(\Omega))$. Let
\[b(t,x,\eta)\in\hat S^m_0(\ol\Real_+\times T^*(\Omega);\, \mathscr L(\Lambda^{0,q}T^*(\Gamma),\Lambda^{0,q}T^*(\Gamma))).\]
We assume that $b(t, x, \eta)=0$ when $\abs{\eta}\leq 1$ and that
\[b(t, x, \eta)-b(\infty, x, \eta)\in
\hat S^{m}_r(\ol\Real_+\times T^*(\Omega);\, \mathscr L(\Lambda^{0,q}T^*(\Gamma),\Lambda^{0,q}T^*(\Gamma)))\]
with $r>0$, where
$b(\infty, x, \eta)\in C^\infty(T^*(\Omega);\, \mathscr L(\Lambda^{0,q}T^*(\Gamma),\Lambda^{0,q}T^*(\Gamma)))$
is a classical symbol of order $m$. Then,
\begin{equation} \label{Be:0801121719}
Q(e^{i(\psi(t,x,\eta)-\seq{y,\eta})}b(t,x,\eta))=e^{i(\psi(t,x,\eta)-\seq{y,\eta})}c(t, x, \eta)+d(t, x, \eta),
\end{equation}
where
$c(t, x, \eta)\in\hat S^{k+m}_0(\ol\Real_+\times T^*(\Omega);\, \mathscr L(\Lambda^{0,q}T^*(\Gamma),\Lambda^{0,q}T^*(\Gamma)))$,
\[c(t, x, \eta)\sim\sum_{\alpha}\frac{1}{\alpha !}q^{(\alpha)}(x,\psi'_x(t,x,\eta))(R_\alpha(\psi,D_x)b)\]
in the symbol space $\hat S^{k+m}_0(\ol\Real_+\times T^*(\Omega);\, \mathscr L(\Lambda^{0,q}T^*(\Gamma), \Lambda^{0,q}T^*(\Gamma)))$,
\[c(t, x, \eta)-c(\infty, x, \eta)\in
\hat S^{k+m}_r(\ol\Real_+\times T^*(\Omega);\, \mathscr L(\Lambda^{0,q}T^*(\Gamma),\Lambda^{0,q}T^*(\Gamma))),\ \ r>0,\]
\[d(t, x, \eta)\in\hat S^{-\infty}_0(\ol\Real_+\times T^*(\Omega);\, \mathscr L(\Lambda^{0,q}T^*(\Gamma),\Lambda^{0,q}T^*(\Gamma))),\]
$d(t, x, \eta)-d(\infty, x, \eta)\in
\hat S^{-\infty}_r(\ol\Real_+\times T^*(\Omega);\, \mathscr L(\Lambda^{0,q}T^*(\Gamma),\Lambda^{0,q}T^*(\Gamma)))$, $r>0$.
Here
\[c(\infty, x, \eta)\in C^\infty(T^*(\Omega);\, \mathscr L(\Lambda^{0,q}T^*(\Gamma),\Lambda^{0,q}T^*(\Gamma)))\]
is a classical symbol of order $k+m$,
\[d(\infty, x, \eta)\in S^{-\infty}_{1, 0}(T^*(\Omega);\, \mathscr L(\Lambda^{0,q}T^*(M'),\Lambda^{0,q}T^*(M')))\]
(For the precise meaning of\, $S^{-\infty}_{1, 0}$, see Definition~\ref{Bd:0712101500}.)
and
\[R_\alpha(\psi,D_x)b=D^\alpha_y\set{e^{i\phi_2(t,x,y,\eta)}
b(t,y,\eta)}\Big|_{y=x},\]
$\phi_2(t,x,y,\eta)=(x-y)\psi'_x(t,x,\eta)-(\psi(t,x,\eta)-\psi(t,y,\eta))$.
Moreover, put
\begin{align*}
B(x, y)=\frac{1}{(2\pi)^{2n-1}}\int\Bigl(\int^{\infty}_0 &\Bigl(e^{i(\psi(t,x,\eta)-\seq{y,\eta})}b(t,x,\eta)-  \\
                                                &e^{i(\psi(\infty,x,\eta)-\seq{y,\eta})}b(\infty,x,\eta)\Bigl)dt\Bigl)d\eta, \\
C(x, y)=\frac{1}{(2\pi)^{2n-1}}\int\Bigl(\int^{\infty}_0 &\Bigl(e^{i(\psi(t,x,\eta)-\seq{y,\eta})}c(t,x,\eta)-   \\
                                                &e^{i(\psi(\infty,x,\eta)-\seq{y,\eta})}c(\infty,x,\eta)\Bigl)dt\Bigl)d\eta.
\end{align*}
We have $Q\circ B\equiv C$.
\end{prop}

As in section $1$, we put
$\Gamma_q=\set{z\in\Gamma;\, Z(q)\ \mbox{fails at}\ z}$
and set
\[\Sigma^-(q)=\set{(x, \xi)\in\Sigma^-;\, Z(q)\ \mbox{fails at}\ x},\ \
\Sigma^+(q)=\set{(x, \xi)\in\Sigma^+;\, Z(q)\ \mbox{fails at}\ x}.\]
From Proposition~\ref{Bp:ss-SBmore} and Proposition~\ref{Bp:he-basismore}, we can repeat the method of part ${\rm I\,}$ to get the following

\begin{thm} \label{Bt:0801090850}
We recall that we work with the Assumption~\ref{Ba:1}.
Given $q$, $0\leq q\leq n-1$. Suppose that $Z(q)$ fails at some point of\, $\Gamma$. Then there exist
\[A\in L^{-1}_{\frac{1}{2},
\frac{1}{2}}(\Gamma;\, \Lambda^{0,q}T^*(\Gamma),\Lambda^{0,q}T^*(\Gamma)), B_-, B_+\in L^{0}_{\frac{1}{2},
\frac{1}{2}}(\Gamma;\, \Lambda^{0,q}T^*(\Gamma),\Lambda^{0,q}T^*(\Gamma))\]
such that
\begin{align} \label{Be:0801121817}
&{\rm WF\,}'(K_{B_-})={\rm diag\,}(\Sigma^-(q)\times\Sigma^-(q)),\nonumber \\
&{\rm WF\,}'(K_{B_+})={\rm diag\,}(\Sigma^+(n-1-q)\times\Sigma^+(n-1-q))
\end{align}
and
\begin{align}
& A\Box^{(q)}_\beta+B_-+B_+\equiv B_-+B_++\Box^{(q)}_\beta A\equiv I, \label{Be:0711160111} \\
&\ol{\pr_\beta} B_-\equiv0,\ \ol{\pr_\beta}^\dagger B_-\equiv0, \label{Be:0801101610} \\
&\ol{\pr_\beta} B_+\equiv0,\ \ol{\pr_\beta}^\dagger B_+\equiv0, \label{Be:0801231556} \\
& B_-\equiv B_-^\dagger\equiv B_-^2, \label{Be:0711160113} \\
& B_+\equiv B_+^\dagger\equiv B_+^2, \label{Be:0801231557}
\end{align}
where $B_-^\dagger$ and $B_+^\dagger$ are the formal adjoints of $B_-$ and $B_+$ with respect to $[\ |\ ]$ respectively and
\[{\rm WF\,}'(K_{B_-})=\set{(x, \xi, y, \eta)\in T^*(\Gamma)\times T^*(\Gamma);\, (x, \xi, y, -\eta)\in{\rm WF}(K_{B_-})}.\]
Here ${\rm WF\,}(K_{B_-})$ is the wave front set of $K_{B_-}$ in the sense of H\"{o}rmander~\cite{Hor71}. See Definition~\ref{Bd:0801202156}
for a review.

Moreover near ${\rm diag\,}(\Gamma_q\times\Gamma_q)$, $K_{B_-}(x, y)$ satisfies
\begin{align*}
K_{B_-}(x, y)\equiv\int^{\infty}_{0} e^{i\phi_-(x, y)t}b(x, y, t)dt
\end{align*}
with
\begin{align}  \label{Be:0801101603}
&b(x, y, t)\in S^{n-1}_{1, 0}(\Gamma\times\Gamma\times]0, \infty[;\, \mathscr L(\Lambda^{0,q}T^*_y(\Gamma), \Lambda^{0,q}T^*_x(\Gamma))),\nonumber \\
&b(x, y, t)\sim\sum^\infty_{j=0}b_j(x, y)t^{n-1-j}\ \ \mbox{in}\ \
S^{n-1}_{1, 0}(\Gamma\times\Gamma\times]0, \infty[;\, \mathscr L(\Lambda^{0,q}T^*_y(\Gamma), \Lambda^{0,q}T^*_x(\Gamma))),\nonumber \\
&b_0(x, x)\neq0\ \ \mbox{if}\ \ x\in\Gamma_q,
\end{align}
(A formula for $b_0(x, x)$ will be given in Proposition~\ref{Bp:i-leading1}.)
where $S^m_{1, 0}$, $m\in\Real$, is the H\"{o}rmander symbol space (see Definition~\ref{Bd:0712101500}),
\[b_j(x, y)\in C^\infty(\Gamma\times\Gamma;\, \mathscr L(\Lambda^{0, q}T^*_y(\Gamma), \Lambda^{0, q}T^*_x(\Gamma))),\ \ j=0,1,\ldots,\]
and
\begin{align}
&\phi_-(x, y)\in C^\infty(\Gamma\times\Gamma),\ \ {\rm Im\,}\phi_-(x, y)\geq0, \label{Be:0709241541} \\
&\phi_-(x, x)=0,\ \ \phi_-(x, y)\neq0\ \ \mbox{if}\ \ x\neq y, \label{Be:0709241542} \\
&d_x\phi_-\neq0,\ \ d_y\phi_-\neq0\ \ \mbox{where}\ \ {\rm Im\,}\phi_-=0, \label{Be:0803271451} \\
&d_x\phi_-(x, y)|_{x=y}=-\omega_0(x), \ \ d_y\phi_-(x, y)|_{x=y}=\omega_0(x),\label{Be:t-i-bis1} \\
&\phi_-(x, y)=-\ol\phi_-(y, x). \label{Be:0709241544}
\end{align}

Similarly, near ${\rm diag\,}(\Gamma_{n-1-q}\times\Gamma_{n-1-q})$,
\begin{align*}
K_{B_+}(x, y)\equiv\int^{\infty}_{0} e^{i\phi_+(x, y)t}c(x, y, t)dt
\end{align*}
with $c(x, y, t)\in S^{n-1}_{1, 0}(\Gamma\times\Gamma\times]0, \infty[;\, \mathscr L(\Lambda^{0,q}T^*_y(\Gamma), \Lambda^{0,q}T^*_x(\Gamma)))$,
\[c(x, y, t)\sim\sum^\infty_{j=0}c_j(x, y)t^{n-1-j}\]
in $S^{n-1}_{1, 0}(\Gamma\times\Gamma\times]0, \infty[;\, \mathscr L(\Lambda^{0,q}T^*_y(\Gamma), \Lambda^{0,q}T^*_x(\Gamma)))$,
where
\[c_j(x, y)\in C^\infty(\Gamma\times\Gamma;\, \mathscr L(\Lambda^{0, q}T^*_y(\Gamma), \Lambda^{0, q}T^*_x(\Gamma))),\ \ j=0,1,\ldots,\]
and $-\ol\phi_+(x, y)$ satifies (\ref{Be:0709241541})-(\ref{Be:0709241544}).
\end{thm}

We only give the outline of the proof of Theorem~\ref{Bt:0801090850}.
For all the details, we refer the reader to section $7$ and section $8$ of part ${\rm I\,}$.
Let
\[a_j(t, x, \eta)\in\hat S^{-j}_0(\ol\Real_+\times T^*(\Omega);\, \mathscr L(\Lambda^{0,q}T^*(\Gamma),\Lambda^{0,q}T^*(\Gamma))),\ \
j=0,1,\ldots,\]
and
$a_j(\infty, x, \eta)\in C^\infty(T^*(\Omega);\, \mathscr L(\Lambda^{0,q}T^*(\Gamma),\Lambda^{0,q}T^*(\Gamma)))$, $j=0,1,\ldots$,
be as in Proposition~\ref{Bp:ss-SBmore}. We recall that for some $r>0$
\[a_j(t,x,\eta)-a_j(\infty,x,\eta)\in\hat S^{-j}_r(\ol\Real_+\times T^*(\Omega);\,
\mathscr L(\Lambda^{0,q}T^*(\Gamma),\Lambda^{0,q}T^*(\Gamma))),\ \ j=0,1,\ldots.\]
Let
$a(\infty, x, \eta)\sim\sum^\infty_{j=0}a_j(\infty, x, \eta)$ in $S^{0}_{1,0}(T^*(\Omega);\,
\mathscr L(\Lambda^{0,q}T^*(\Gamma),\Lambda^{0,q}T^*(\Gamma)))$.
Let
\[a(t,x,\eta)\sim\sum^\infty_{j=0}a_j(t, x, \eta)\]
in $\hat S^{0}_0(\ol\Real_+\times T^*(\Omega);\,
\mathscr L(\Lambda^{0,q}T^*(\Gamma),\Lambda^{0,q}T^*(\Gamma)))$.
We take $a(t, x, \eta)$ so that for every compact set $K\subset\Omega$ and all indices $\alpha$, $\beta$, $\gamma$, $k$,
there exists $c>0$, $c$ is independent of $t$, such that
\begin{equation} \label{Be:allerallerallergo2}
\abs{\pr^\gamma_t\pr^\alpha_x\pr^\beta_\eta(a(t, x, \eta)-\sum^k_{j=0}a_j(t, x, \eta))}
\leq c(1+\abs{\eta})^{-k-1+\gamma-\abs{\beta}},
\end{equation}
where $t\in\ol\Real_+$, $x\in K$, $\abs{\eta}\geq1$, and
\[a(t, x, \eta)-a(\infty, x, \eta)\in\hat S^0_r(\ol\Real_+\times T^*(\Omega);\,
\mathscr L(\Lambda^{0,q}T^*(\Gamma),\Lambda^{0,q}T^*(\Gamma)))\]
with $r>0$.

Choose $\chi\in C^\infty_0
(\Real^{2n-1})$ so that $\chi(\eta)=1$ when $\abs{\eta}<1$ and $\chi(\eta)=0$ when $\abs{\eta}>2$.
Set
\begin{align} \label{Be:allerimportanthome1}
A(x, y) &= \frac{1}{(2\pi)^{2n-1}}\int\Bigl(\int^{\infty}_0\Bigl(e^{i(\psi(t,x,\eta)-\seq{y,\eta})}a(t,x,\eta)-\nonumber  \\
  &\quad\quad\quad e^{i(\psi(\infty,x,\eta)-\seq{y,\eta})}a(\infty,x,\eta)\Bigr)(1-\chi(\eta))dt\Bigl)d\eta.
\end{align}
Put
\begin{equation} \label{Be:allerimportanthome2}
B(x, y)=\frac{1}{(2\pi)^{2n-1}}\int e^{i(\psi(\infty,x,\eta)-\seq{y,\eta})}a(\infty,x,\eta)d\eta.
\end{equation}
Since $a_j(t, x, \eta)$, $j=0,1,\ldots$, solve the transport equations (\ref{Be:hf-heattransport}), we can check that
$B+\Box^{(q)}_\beta A\equiv I$,
$\Box^{(q)}_\beta B\equiv0$.
From the global theory of Fourier integral operators (see~\cite{MS74}),
we get
$K_{B}\equiv K_{B_-}+K_{B_+}$,
wher $K_{B_-}$ and $K_{B_+}$ are as in Theorem~\ref{Bt:0801090850}.
By using a partition of unity we get the global result.

\begin{rem} \label{Br:0803122143}
For more properties of the phase $\phi_-(x, y)$, see Theorem $1.5$ and Remark $1.6$ of part ${\rm I\,}$.
\end{rem}

We can repeat the computation of the leading term of the Szeg\"{o} projection (see section $9$ of part ${\rm I\,}$), to
get the following

\begin{prop} \label{Bp:i-leading1}
Let $p\in\Gamma_q$, $q=n_-$. Let
$U_1(x),\ldots,U_{n-1}(x)$
be an orthonormal frame of $\Lambda^{1,0}T_x(\Gamma)$, for which
the Levi form is diagonalized at $p$. Let $e_j(x)$, $j=1,\ldots,n-1$,
denote the basis of $\Lambda^{0,1}T^*_x(\Gamma)$, which is dual to $\ol U_j(x)$, $j=1,\ldots,n-1$. Let
$\lambda_j(x)$, $j=1,\ldots,n-1$, be the eigenvalues of the Levi form $L_x$. We assume that
$\lambda_j(p)<0$ if $1\leq j\leq n_-$.
Then
\[b_0(p, p)=\frac{1}{2}\abs{\lambda_1(p)}\cdots
\abs{\lambda_{n-1}(p)}\pi^{-n}\prod_{j=1}^{j=n_-}e_j(p)^\wedge e_j(p)^{\wedge, *},\]
where $b_0$ is as in (\ref{Be:0801101603}).
\end{prop}

In section $8$, we need the following

\begin{prop} \label{Be:0801101606}
Suppose that $Z(q)$ fails at some point of\, $\Gamma$.
Let $B_-$ be as in Theorem~\ref{Bt:0801090850}. Then,
\begin{equation} \label{Be:0801101608}
\gamma\dbar \Td PB_-\equiv 0.
\end{equation}
\end{prop}

\begin{proof}
In view of Theorem~\ref{Bt:0801090850}, we know that
\[T\gamma\dbar \Td PB_-=\ol{\pr_\beta} B_-\equiv0,\ \ \gamma\ol{\pr_f}^*\Td PB_-=\ol{\pr_\beta}^\dagger B_-\equiv0.\]
Combining this with $\gamma(\dbar\ol{\pr_f}^*+\ol{\pr_f}^*\dbar)\Td P\equiv0$, we have
\[\gamma\ol{\pr_f}^*\Td P\gamma\dbar \Td PB_-\equiv-\gamma\dbar \Td P\gamma\ol{\pr_f}^*\Td PB_-\equiv0\]
and
\begin{equation} \label{Be:0712310012}
\gamma\ol{\pr_f}^*\Td P(I-T)\gamma\dbar \Td PB_-=\gamma\ol{\pr_f}^*\Td P\gamma\dbar \Td PB_--\gamma\ol{\pr_f}^*\Td PT\gamma\dbar \Td PB_-\equiv0.
\end{equation}
Combining this with (\ref{Be:0801201632}), we get
$\gamma\ol{\pr_f}^*\Td P(\Td P^*\Td P)^{-1}(\dbar r)^\wedge R\gamma\dbar \Td PB_-\equiv0$.
Thus,
\[\gamma(\dbar r)^{\wedge}\ol{\pr_f}^*\Td P(\Td P^*\Td P)^{-1}(\dbar r)^\wedge R\gamma\dbar \Td PB_-\equiv0.\]
In view of Proposition~\ref{Bp:071231323}, we know that
\[\gamma(\dbar r)^{\wedge}\ol{\pr_f}^*\Td P(\Td P^*\Td P)^{-1}: C^\infty(\Gamma;\, I^{0,q}T^*(M'))\To C^\infty(\Gamma;\, I^{0,q}T^*(M'))\]
is elliptic near\, $\Sigma^-$, where $I^{0,q}T^*_z(M')$ is as in (\ref{Be:0801052157}). Since
\[{\rm WF\,}'(K_{B_-})\subset{\rm diag\,}(\Sigma^-\times\Sigma^-),\]
we get $(\dbar r)^\wedge R\gamma\dbar \Td PB_-\equiv0$.
(See Proposition~\ref{Bp:0801202219} and Proposition~\ref{Bp:0801202234}.)
Thus, by (\ref{Be:0801201632}),
$(I-T)\gamma\dbar \Td PB_-\equiv0$.
The proposition follows.
\end{proof}


\section{The Bergman projection}

Given $q$, $0\leq q\leq n-1$.
In this section, we assume that $Z(q)$ fails at some point of\, $\Gamma$ and that
$Z(q-1)$ and $Z(q+1)$ hold at each point of\, $\Gamma$. In view of Proposition~\ref{Bp:0801042202}, we know that
$\Td\Pi^{(q)}:C^\infty(\ol M;\, \Lambda^{0,q}T^*(M'))\To D^{(q)}$.
Put
\begin{equation} \label{Be:0711121605}
K=\gamma\Td\Pi^{(q)}\Td P:C^\infty(\Gamma;\, \Lambda^{0,q}T^*(\Gamma))\To C^\infty(\Gamma;\, \Lambda^{0,q}T^*(\Gamma)).
\end{equation}
Let $K^\dagger$ be the formal adjoint of $K$ with respect to $[\ |\ ]$. That is,
\begin{align*}
&K^\dagger:\mathscr D'(\Gamma;\, \Lambda^{0,q}T^*(\Gamma))\To\mathscr D'(\Gamma;\, \Lambda^{0,q}T^*(\Gamma)) \\
&[K^\dagger u\ |\ v]=[u\ |\ K v],\ u\in\mathscr D'(\Gamma;\, \Lambda^{0,q}T^*(\Gamma)),\
v\in C^\infty(\Gamma;\, \Lambda^{0,q}T^*(\Gamma)).
\end{align*}

\begin{lem} \label{Bl:0711121610}
We have $K^\dagger v=K v$,
$v\in C^\infty(\Gamma;\, \Lambda^{0,q}T^*(\Gamma))$.
\end{lem}

\begin{proof}
For $u$, $v\in C^\infty(\Gamma;\, \Lambda^{0,q}T^*(\Gamma))$, we have
\begin{align*}
[K u\ |\ v]&=[\gamma\Td\Pi^{(q)}\Td Pu\ |\ v] \\
&=(\Td\Pi^{(q)}\Td Pu\ |\ \Td Pv)_M \\
&=(\Td Pu\ |\ \Td\Pi^{(q)}\Td Pv)_M  \\
&=[u\ |\ K v].
\end{align*}
Thus, $K^\dagger v=K v$.
The lemma follows.
\end{proof}

We can extend $K$ to
$\mathscr D'(\Gamma;\, \Lambda^{0,q}T^*(\Gamma))\To\mathscr D'(\Gamma;\, \Lambda^{0,q}T^*(\Gamma))$
by the following formula:
$[K u\ |\ v]=[u\ |\ K^\dagger v]$,
$u\in\mathscr D'(\Gamma;\, \Lambda^{0,q}T^*(\Gamma))$, $v\in C^\infty(\Gamma;\, \Lambda^{0,q}T^*(\Gamma))$.

\begin{lem} \label{Bl:0801102218}
Let $u\in\mathscr D'(\Gamma;\, \Lambda^{0,q}T^*(\Gamma))$. We have
${\rm WF\,}(K u)\subset\Sigma^-$.
\end{lem}

\begin{proof}
Let $u\in\mathscr D'(\Gamma;\, \Lambda^{0,q}T^*(\Gamma))$. We have
$(\gamma(\dbar r)^{\wedge, *}\dbar \Td P)(K u)=0$.
In view of Proposition~\ref{Bp:0801102222}, we know that $\gamma(\dbar r)^{\wedge, *}\dbar \Td P$
is elliptic outside\, $\Sigma^-$. The lemma follows. (See Definition~\ref{Bd:0801202156}.)
\end{proof}

\begin{lem} \label{Bl:0801102229}
Let $B_-$ be as in Theorem~\ref{Bt:0801090850}.
We have
$B_-K\equiv K B_-\equiv K$.
\end{lem}

\begin{proof}
Let $A$, $B_-$ and $B_+$ be as in Theorem~\ref{Bt:0801090850}.
In view of Theorem~\ref{Bt:0801090850}, we have
\[B_-+B_++A\Box^{(q)}_\beta\equiv I.\]
We may replace $B_+$ by $I-A\Box^{(q)}_\beta-B_-$ and get
$B_-+B_++A\Box^{(q)}_\beta=I$.
It is easy to see that
$\Box^{(q)}_\beta K=0$.
Thus,
\begin{equation} \label{Be:s-in}
K=(B_-+B_++A\Box^{(q)}_\beta)K=(B_-+B_+)K.
\end{equation}
Let $u\in\mathscr D'(\Gamma;\, \Lambda^{0,q}T^*(\Gamma))$. From Lemma~\ref{Bl:0801102218}, we know that
${\rm WF\,}(K u)\subset\Sigma^-$.
Note that
${\rm WF\,}'(K_{B_+})\subset{\rm diag\,}(\Sigma^+\times\Sigma^+)$.
Thus, $B_+K u\in C^\infty$, so $B_+K$ is smoothing (see Lemma~\ref{Bl:0801202214} and Proposition~\ref{Bp:0801202234}) and
$(B_-+B_+)K\equiv B_-K$.
From this and (\ref{Be:s-in}), we get
$K\equiv B_-K$
and
$K=K^\dagger\equiv K^\dagger B_-^\dagger\equiv K B_-$.
The lemma follows.
\end{proof}

We pause and introduce some notations.
Let $X$ and $Y$ be $C^\infty$ vector bundles over $M'$ and $\Gamma$ respectively. Let
$C, D: C^\infty(\Gamma;\, Y)\To \mathscr D'(M;\, X)$
with distribution kernels
$K_{C}(z, y)$, $K_{D}(z, y)\in\mathscr D'(M\times\Gamma;\, \mathscr L(Y_y, X_z))$.
We write
$C\equiv D$ mod $C^\infty(\ol M\times\Gamma)$
if $K_{C}(z, y)=K_{D}(z, y)+F(z, y)$,
where $F(z, y)\in C^\infty(\ol M\times\Gamma;\, \mathscr L(Y_y, X_z))$.

\begin{lem} \label{Bl:0801231658}
We have
\begin{equation} \label{Be:0801231710}
\Td\Pi^{(q)}\Td PB_-\equiv\Td\Pi^{(q)}\Td P\ \ {\rm mod\,}C^\infty(\ol M\times\Gamma).
\end{equation}
\end{lem}

\begin{proof}
From Lemma~\ref{Bl:0801102229}, we have
$K=\gamma\Td\Pi^{(q)}\Td P\equiv KB_-=\gamma\Td\Pi^{(q)}\Td PB_-$.
Thus,
$\Td\Pi^{(q)}\Td P=\Td P\gamma\Td\Pi^{(q)}\Td P\equiv \Td P\gamma\Td\Pi^{(q)}\Td PB_-\ \ {\rm mod\,}C^\infty(\ol M\times\Gamma)$.
We get (\ref{Be:0801231710}).
\end{proof}

Put
\begin{equation} \label{Be:0801131847}
Q=\Td PB_-T(\Td P^*\Td P)^{-1}\Td P^*:C^\infty(\ol M;\, \Lambda^{0,q}T^*(M'))\To C^\infty(\ol M;\, \Lambda^{0,q}T^*(M')),
\end{equation}
where $T$ is as in (\ref{Be:0712302117}).

\begin{prop} \label{Bp:0801131851}
We have
$Q\equiv\Td\Pi^{(q)}\ \ {\rm mod\,}C^\infty(\ol M\times\ol M)$.
From (\ref{Be:0809071558}), it follows that
$Q\equiv\Pi^{(q)}\ \ {\rm mod\,}C^\infty(\ol M\times\ol M)$.
\end{prop}

\begin{proof}
We have
\begin{equation} \label{Be:0801132144}
\Td\Pi^{(q)}Q=\Td\Pi^{(q)}\Td PB_-T(\Td P^*\Td P)^{-1}\Td P^*\equiv\Td\Pi^{(q)}\Td PT(\Td P^*\Td P)^{-1}\Td P^*\ \ {\rm mod\,}C^\infty(\ol M\times\ol M).
\end{equation}
Here we used (\ref{Be:0801231710}).
From (\ref{Be:0801132144}) and the first part of (\ref{Be:0801042303}), we get
\begin{equation} \label{Be:0801132148}
\Td\Pi^{(q)}Q\equiv\Td\Pi^{(q)}\ \ {\rm mod\,}C^\infty(\ol M\times\ol M).
\end{equation}
From Theorem~\ref{Bt:0710301306}, we have
$\Td\Pi^{(q)}Q=(I-K^{(q)})(I-\dbar^*N^{(q+1)}\dbar-\dbar N^{(q-1)}\dbar^*)Q$,
where $N^{(q+1)}$ and $N^{(q-1)}$ are as in Theorem~\ref{Bt:0710301305}. From (\ref{Be:0801101608}), (\ref{Be:0801101610}) and
Lemma~\ref{Bl:0801052333}, we see that
$\dbar Q\equiv 0$, $\dbar^*Q\equiv0$ mod $C^\infty(\ol M\times\ol M)$.
Thus,
$\Td\Pi^{(q)}Q=(I-K^{(q)})(I-\dbar^*N^{(q+1)}\dbar-\dbar N^{(q-1)}\dbar^*)Q\equiv Q$ mod $C^\infty(\ol M\times\ol M)$.
From this and (\ref{Be:0801132148}), the proposition follows.
\end{proof}

Let $x=(x_1,\ldots,x_{2n-1})$ be a system of local coordinates on $\Gamma$ and extend the functions
$x_1,\ldots,x_{2n-1}$
to real smooth functions in some neighborhood of $\Gamma$. We write
$(\xi_1,\ldots,\xi_{2n-1}, \theta)$
to denote the dual variables of $(x, r)$. We write
$z=(x_1,\ldots,x_{2n-1}, r)$, $x=(x_1,\ldots,x_{2n-1},0)$,
$\xi=(\xi_1,\ldots, \xi_{2n-1})$, $\zeta=(\xi, \theta)$.
Until further notice, we work with the local coordinates $z=(x, r)$ defined on some neighborhood of $p\in\Gamma$.

We represent the Riemannian metric on $T(M')$ by
\[h=\sum^{2n}_{j,k=1}h_{j,k}(z)dx_j\otimes dx_k,\ \ dx_{2n}=dr,\]
where $h_{j,k}(z)=h_{k,j}(z)$, $j$, $k=1,\ldots,n$, and $\left(h_{j,k}(z)\right)_{1\leq j,k\leq 2n}$ is positive definite at each point of $M'$. Put
$\left(h_{j,k}(z)\right)^{-1}_{1\leq j,k\leq 2n}=\left(h^{j,k}(z)\right)_{1\leq j,k\leq 2n}$.
It is well-known (see page $99$ of Morrow-Kodaira~\cite{MK71}) that
\begin{equation} \label{Be:BK-har.070613.1}
\Box^{(q)}_f=-\frac{1}{2}\Bigl(h^{2n,2n}(z)\frac{\pr^2}{\pr r^2}+2\sum^{2n-1}_{j=1}h^{2n,j}(z)\frac{\pr^2}{\pr r\pr x_j}+
T(r)\Bigl)+\mbox{lower order terms},
\end{equation}
where
\begin{equation} \label{Be:0804051017}
T(r)=\sum^{2n-1}_{j,k=1}h^{j,k}(z)\frac{\pr^2}{\pr x_j\pr x_k}.
\end{equation}
Note that $T(0)=\triangle_\Gamma$+lower order terms
and
\[h^{2n,2n}(x)=1,\ \ h^{2n,j}(x)=0,\ \ j=1,\ldots,2n-1.\]
We let the full symbol of $\Td\Box^{(q)}_f$ be:
\[\mbox{full symbol of }\ \Td\Box^{(q)}_f=\sum^2_{j=0}q_j(z,\zeta)\]
where $q_j(z,\zeta)$ is a homogeneous polynomial of order $2-j$ in $\zeta$ (we recall that $\Td\Box^{(q)}_f=\Box^{(q)}_f+K^{(q)}$).
We have the following

\begin{prop} \label{Bp:BK-WKB1}
Let $\phi_-\in C^\infty(\Gamma\times\Gamma)$ be as in Theorem~\ref{Bt:0801090850}. Then, in
some neighborhood $U$ of ${\rm diag\,}(\Gamma_q\times\Gamma_q)$ in $M'\times M'$,
there exists a smooth function
$\Td\phi(z, y)\in C^\infty((\ol M\times\Gamma)\bigcap U)$
such that
\begin{align}  \label{Be:0803122151}
&\Td\phi(x, y)=\phi_-(x, y),\ \ {\rm Im\,}\Td\phi\geq0, \nonumber  \\
& d_z\Td\phi\neq0,\ d_y\Td\phi\neq0\ \ \mbox{where ${\rm Im\,}\Td\phi=0$}, \nonumber \\
&{\rm Im\,}\Td\phi>0\ \ \mbox{if}\ \ r\neq0,
\end{align}
and $q_0(z, \Td\phi'_z)$
vanishes to infinite order on $r=0$. We write $\frac{\pr}{\pr r(z)}$ to denote $\frac{\pr}{\pr r}$ acting in the $z$ variables. We have
\begin{equation} \label{Be:0711062350}
\frac{\pr}{\pr r(z)}\Td\phi(z, y)|_{r=0}=-i\sqrt{-\sigma_{\La_\Gamma}(x,  (\phi_-)'_x)}
\end{equation}
in some neighborhood of $x=y$, where ${\rm Re\,}\sqrt{-\sigma_{\La_\Gamma}(x, (\phi_-)'_x)}>0$.
\end{prop}

\begin{proof}
From (\ref{Be:BK-har.070613.1}) and (\ref{Be:0804051017}),
we have
\begin{align} \label{Be:BK-har.070613.3}
&q_0(z,\zeta)=\frac{1}{2}h^{2n,2n}(z)\theta^2+\sum^{2n-1}_{j=1}h^{2n,j}(z)\theta\xi_j+g(z, \xi),\nonumber \\
&g(x, \xi)=-\frac{1}{2}\sigma_{\La_\Gamma},
\end{align}
where $g(z, \xi)$ is the principal symbol of $-\frac{1}{2}T(r)$.

We consider the Taylor expansion of $q_0(z, \zeta)$ with respect to $r$,
\begin{equation} \label{Be:BK-WKB***}
q_0(z, \zeta)=\frac{1}{2}\theta^2-\frac{1}{2}\sigma_{\La_\Gamma}+\sum^\infty_{j=1}g_j(x, \xi)r^j+\sum^\infty_{j=1}s_j(x, \zeta)\theta r^j.
\end{equation}

We introduce the Taylor expansion of $\Td\phi(z, y)$ with respect to $r$,
\[\Td\phi(z, y)=\phi_-(x, y)+\sum^\infty_1\phi_j(x, y)r^j.\]
Let $\phi_1(x, y)=-i\sqrt{-\sigma_{\La_\Gamma}(x, (\phi_-)'_x)}$.
Since $(\phi_-)'_x|_{x=y}=-\omega_0(x)$ is real,
we choose the branch of $\sqrt{-\sigma_{\La_\Gamma}(x, (\phi_-)'_x)}$ so that
${\rm Re\,}\sqrt{-\sigma_{\La_\Gamma}(x, (\phi_-)'_x)}>0$
in some neighborhood of $x=y$, $r=0$.
Put $\Td\phi_1(z, y)=\phi_-(x, y)+r\phi_1(x, y)$.
We have $q_0(z, (\Td\phi_1)'_z)=O(r)$.
Similarly, we can find $\phi_2(x, y)$ so that
$q_0(z, (\Td\phi_2)'_z)=O(r^2)$,
where $\Td\phi_2(z, y)=\phi_-(x, y)+r\phi_1(x, y)+r^2\phi_2(x, y)$.
Continuing in this way we get the phase  $\Td\phi(z, y)$ such that
$\Td\phi(x, y)=\phi_-(x, y)$
and $q_0(z, \Td\phi'_z)$
vanishes to infinite order on $r=0$.
The proposition follows.
\end{proof}

\begin{rem} \label{Br:0801141650}
Let $\Td\phi(z, y)$ be as in Proposition~\ref{Bp:BK-WKB1} and let
\[d(z, y, t)\in S^{m}_{1, 0}(\ol M\times\Gamma\times]0, \infty[;\, \mathscr L(\Lambda^{0,q}T^*_y(M'), \Lambda^{0,q}T^*_z(M')))\]
with support in some neighborhood of ${\rm diag\,}(\Gamma_q\times\Gamma_q)$. (For the meaning of the space
$S^{m}_{1, 0}(\ol M\times\Gamma\times]0, \infty[;\, \mathscr L(\Lambda^{0,q}T^*_y(M'), \Lambda^{0,q}T^*_z(M')))$,
see Definition~\ref{Bd:0712101500}.)
Choose a cut-off function $\chi(t)\in C^\infty(\Real)$
so that $\chi(t)=1$ when $\abs{t}<1$ and $\chi(t)=0$ when $\abs{t}>2$. For all $u\in C^\infty(\Gamma;\, \Lambda^{0,q}T^*(M'))$, set
\[(D_\eps u)(z)=\int\int^\infty_0e^{i\Td\phi(z, y)t}d(z, y, t)\chi(\eps t)u(y)dtdy.\]
Since ${\rm Im\,}\Td\phi\geq0$ and $d_y\Td\phi\neq0$ where ${\rm Im\,}\Td\phi=0$, we can integrate by parts in $y$, $t$ and obtain
$\lim_{\eps\To0}(D_\eps u)(z)\in C^\infty(\ol M;\, \Lambda^{0, q}T^*(M'))$.
This means that
$D=\lim_{\eps\To0}D_\eps: C^\infty(\Gamma;\, \Lambda^{0,q}T^*(M'))\To C^\infty(\ol M;\, \Lambda^{0, q}T^*(M'))$
is continuous. Formally,
\[D(z, y)=\int^\infty_0e^{i\Td\phi(z, y)t}d(z, y, t)dt.\]
\end{rem}

\begin{prop} \label{Bp:BK-WKB2}
Let
\[B_-(x, y)=\int^\infty_0e^{i\phi_-(x, y)t}b(x, y, t)dt\]
be as in Theorem~\ref{Bt:0801090850}. We have
\[\Td PB_-(z, y)\equiv\int^\infty_0e^{i\Td\phi(z, y)t}\Td b(z, y, t)dt \ \ {\rm mod\,}C^\infty(\ol M\times\Gamma)\]
with
$\Td b(z, y, t)\in S^{n-1}_{1, 0}(\ol M\times\Gamma\times]0, \infty[;\, \mathscr L(\Lambda^{0,q}T^*_y(\Gamma), \Lambda^{0,q}T^*_z(M')))$,
\[\Td b(z, y, t)\sim\sum^\infty_{j=0}\Td b_j(z, y)t^{n-1-j}\]
in $S^{n-1}_{1, 0}(\ol M\times\Gamma\times]0, \infty[;\, \mathscr L(\Lambda^{0,q}T^*_y(\Gamma), \Lambda^{0,q}T^*_z(M')))$,
where
\[\Td b_j(z, y)\in C^\infty(\ol M\times\Gamma;\, \mathscr L(\Lambda^{0,q}T^*_y(\Gamma), \Lambda^{0,q}T^*_z(M'))),\ \ j=0,1,\ldots.\]
\end{prop}

\begin{proof}
Put
$b(x, y, t)\sim\sum^\infty_{j=0}b_j(x, y)t^{n-1-j}$
and formally set
\[\Td b(z, y, t)\sim\sum^\infty_{j=0}\Td b_j(z, y)t^{n-1-j}.\]
We notice that
$B_-(x, y)\in
C^\infty(\Gamma\times\Gamma\setminus{\rm diag\,}(\Gamma_q\times\Gamma_q);\, \mathscr L(\Lambda^{0,q}T^*_y(\Gamma), \Lambda^{0,q}T^*_z(\Gamma)))$.
For simplicity, we may assume that $b(x, y, t)=0$ outside some small neighborhood of ${\rm diag\,}(\Gamma_q\times\Gamma_q)\times\ol\Real_+$.
Put $\Td\Box^{(q)}_f(\Td b(z, y, t)e^{i\Td\phi t})=\Td c(z, y, t)e^{i\Td\phi t}$.
From (\ref{Be:0709241542}) and (\ref{Be:0803122151}), we know that near ${\rm diag\,}(\Gamma_q\times\Gamma_q)$,
$\Td\phi(z, y)=0$ if and only if $x=y$, $r=0$. From this
observation, we see that if $\Td c(z, y, t)$ vanishes to infinite order on
${\rm diag\,}(\Gamma_q\times\Gamma_q)\times\ol\Real_+$, we can integrate by parts and obtain
\[\lim_{\eps\To 0}\int^\infty_0e^{i\Td\phi t}\Td c(z, y, t)\chi(\eps t)dt\equiv 0\ \ {\rm mod\,}C^\infty(\ol M\times\Gamma),\]
where $\chi(t)$ is as in Remark~\ref{Br:0801141650}. Thus, we only need to consider the Taylor expansion of
$\Td b(z, y, t)$ on $x=y$, $r=0$.
We introduce the asymptotic expansion of
$\Td\Box^{(q)}_f(\Td be^{i\Td\phi t})$. Setting
$\Td\Box^{(q)}_f(\Td be^{i\Td\phi t})\sim 0$ and
regrouping the terms according to the degree of homogeneity. We obtain the transport equations
\begin{equation} \label{Be:BK-heattransport}
\left\{ \begin{array}{ll}
 T(z, y,\pr_z)\Td b_0(z, y)=0  \\
 T(z, y,\pr_z)\Td b_j(z, y)+l_j(z,y,\Td b_0(z, y),\ldots,\Td b_{j-1}(z, y))=0,\ j=1,2,\ldots.
 \end{array}\right.
\end{equation}
Here
\begin{align*}
T(z, y, \pr_z)&=-i\sum^{2n-1}_{j=1}\frac{\pr q_0}{\pr\xi_j}(z,
\Td\phi'_z)\frac{\pr}{\pr x_j}-
i\frac{\pr q_0}{\pr\theta}(z, \Td\phi'_z)\frac{\pr}{\pr r}  \\
                            &\quad+R(z, y),
\end{align*}
where
\[R(z, y)=q_{1}(z, \Td\phi'_z)+\frac{1}{2i}\sum^{2n}_{j,k=1}\frac{\partial^2
q_0(z, \Td\phi'_z)}{\partial\xi_j\partial\xi_k}\frac{\partial^2\Td\phi}{\partial x_j\partial x_k},\ \ x_{2n}=r,\ \ \xi_{2n}=\theta,\]
and $l_j$ is a linear differential operator acting on $\Td b_0(z, y),\ldots,\Td b_{j-1}(z, y)$.

We introduce the Taylor expansion of $\Td b_0(z, y)$ with respect to $r$,
\[\Td b_0(z, y)=b_0(x, y)+\sum^\infty_1b^j_0(x, y)r^j.\]
Since $\frac{\pr q_0}{\pr\theta}|_{r=0}=\theta$
and
$\Td\phi'_r|_{r=0}=-i\sqrt{-\sigma_{\La_\Gamma}(x, (\phi_-)'_x)}$,
we have $\frac{\pr q_0}{\pr\theta}(z, \Td\phi'_z)|_{r=0}\neq0$
in some neighborhood of $x=y$.
Thus, we can find $b^1_0(x, y)r$  such that
\[T(z, y, \pr_z)(b_0(x, y)+b^1_0(x, y)r)=O(\abs{r})\]
in some neighborhood of $r=0$, $x=y$.
We can repeat the procedure above to find $b^2_0(x, y)$ such that
$T(z, y, \pr_z)(b_0(x, y)+\sum^2_{k=1}b^k_0(x, y)r^k)=O(\abs{r}^{2})$
in some neighborhood of $r=0$, $x=y$.
Continuing in this way we solve the first transport equation to infinite order at $r=0$, $x=y$.

For the second transport equation, we can repeat the method above to solve the second transport equation to infinite order
at $r=0$, $x=y$. Continuing in this way we solve (\ref{Be:BK-heattransport}) to infinite order at $r=0$, $x=y$.

Put $\Td B(z, y)=\int^\infty_0e^{i\Td\phi(z, y)t}\Td b(z, y, t)dt$.
From the construction above, we see that
\begin{equation} \label{Be:0804051421}
\Td\Box^{(q)}_f\Td B\equiv0\ \ {\rm mod\,}C^\infty(\ol M\times\Gamma),\ \ \gamma\Td B\equiv B_-.
\end{equation}
It is well-known (see chapter XX of~\cite{Hor85}) that there exists
\[G:C^\infty(\ol M;\, \Lambda^{0,q}T^*(M'))\To C^\infty(\ol M;\, \Lambda^{0,q}T^*(M'))\]
such that
\begin{equation} \label{Be:0804051422}
G\Td\Box^{(q)}_f+\Td P\gamma=I\ \ \mbox{on}\ \ C^\infty(\ol M;\, \Lambda^{0,q}T^*(M')).
\end{equation}
From this and (\ref{Be:0804051421}), we have
$\Td B=(G\Td\Box^{(q)}_f+\Td P\gamma)\Td B\equiv \Td PB_-$ mod $C^\infty(\ol M\times\Gamma)$.
The proposition follows.
\end{proof}

From Proposition~\ref{Bp:BK-WKB2}, we have
\[C(z, y):=\Td PB_-T(\Td P^*\Td P)^{-1}(z, y)\equiv\int^\infty_0e^{i\Td\phi(z, y)t}c(z, y, t)dt \ \ {\rm mod\,}C^\infty(\ol M\times\Gamma)\]
with $c(z, y, t)\in S^{n}_{1, 0}(\ol M\times\Gamma\times]0, \infty[;\, \mathscr L(\Lambda^{0,q}T^*_y(M'), \Lambda^{0,q}T^*_z(M')))$,
\[c(z, y, t)\sim\sum^\infty_{j=0}c_j(z, y)t^{n-j}\]
in the space $S^{n}_{1, 0}(\ol M\times\Gamma\times]0, \infty[;\, \mathscr L(\Lambda^{0,q}T^*_y(M'), \Lambda^{0,q}T^*_z(M')))$.
Let
\[C^*:C^\infty(\ol M;\, \Lambda^{0,q}T^*(M'))\To\mathscr D'(\Gamma;\, \Lambda^{0,q}T^*(M'))\]
be the operator defined by
$(C^*u\ |\ v)_\Gamma=(u\ |\ Cv)_M$, $u\in C^\infty(\ol M;\, \Lambda^{0,q}T^*(M'))$,
$v\in C^\infty(\Gamma;\, \Lambda^{0,q}T^*(M'))$.
The distribution kernel of $C^*$ is
\begin{equation} \label{Be:0804062203}
C^*(y, z)\equiv\int^\infty_0e^{-i\ol{\Td\phi}(z, y)t}c^*(y, z, t)dt\ \ {\rm mod\,}C^\infty(\Gamma\times\ol M)
\end{equation}
where
\begin{align*}
&c^*(y, z, t)\in S^{n}_{1, 0}(\Gamma\times\ol M\times]0, \infty[;\, \mathscr L(\Lambda^{0,q}T^*_z(M'), \Lambda^{0,q}T^*_y(M'))), \\
&(c^*(y, z, t)\mu\ |\ \nu)=(\mu\ |\ c(z, y, t)\nu),\ \ \mu\in\Lambda^{0,q}T^*_z(M'),\ \nu\in\Lambda^{0,q}T^*_y(M'), \\
&c^*(y, z, t)\sim\sum^\infty_{j=0}c^*_j(y, z)t^{n-j}
\end{align*}
in $S^{n}_{1, 0}(\Gamma\times\ol M\times]0, \infty[;\, \mathscr L(\Lambda^{0,q}T^*_z(M'), \Lambda^{0,q}T^*_y(M')))$. The integral
(\ref{Be:0804062203}) is defined as follows: Let $u\in C^\infty(\ol M;\, \Lambda^{0,q}T^*(M'))$. Set
\[(C^*_\eps u)(y)=\int\int^\infty_0e^{-i\ol{\Td\phi(z, y)}t}c^*(y, z, t)\chi(\eps t)u(z)dtdz,\]
where $\chi$ is as in Remark~\ref{Br:0801141650}.
Since $d_x\Td\phi\neq0$ where ${\rm Im\,}\Td\phi=0$,
we can integrate by parts in $x$ and $t$ and obtain
$\lim_{\eps\To0}(C^*_\eps u)(y)\in C^\infty(\Gamma;\, \Lambda^{0, q}T^*(M'))$.
This means that
$C^*=\lim_{\eps\To0}C^*_\eps: C^\infty(\ol M;\, \Lambda^{0,q}T^*(M'))\To C^\infty(\Gamma;\, \Lambda^{0, q}T^*(M'))$
is continuous.

We also write $w=(y_1,\dots,y_{2n-1},r)$.
We can repeat the proof
of Proposition~\ref{Bp:BK-WKB1} to find
$\phi(z, w)\in C^\infty(\ol M\times\ol M)$
such that $\phi(z, y)=\Td\phi(z, y)$,
${\rm Im\,}\phi\geq0$,
${\rm Im\,}\phi>0$ if $(z, w)\notin\Gamma\times\Gamma$
and $q_0(w, -\ol\phi'_w)$
vanishes to infinite order on $r=0$. Since $\phi_-(x, y)=-\ol\phi_-(y, x)$, we can take $\phi(z, w)$ so that
$\phi(z, w)=-\ol\phi(w, z)$.
As in the proof of Proposition~\ref{Bp:BK-WKB2}, we can find
\begin{align*}
&a^*(w, z, t)\in S^{n}_{1, 0}(\ol M\times\ol M\times[0, \infty[;\, \mathscr L(\Lambda^{0,q}T^*_z(M'), \Lambda^{0,q}T^*_w(M'))), \\
&a^*(w, z, t)\sim\sum^\infty_{j=0}a^*_j(w, z)t^{n-j}
\end{align*}
in $S^{n}_{1, 0}(\ol M\times\ol M\times]0, \infty[;\, \mathscr L(\Lambda^{0,q}T^*_z(M'), \Lambda^{0,q}T^*_w(M')))$, such that
\[a^*(y, z, t)=c^*(y, z, t)\]
and $\Td\Box^{(q)}_f(a^*(w, z, t)e^{-i\ol\phi(z, w)t})$
vanishes to infinite order on ${\rm diag\,}(\Gamma_q\times\Gamma_q)\times\ol\Real_+$. From (\ref{Be:0804051422}), we have
$\Td PC^*(w, z)\equiv\int^\infty_0 e^{-i\ol\phi(z, w)t}a^*(w, z, t)dt\ \ {\rm mod\,}C^\infty(\ol M\times\ol M)$.
Thus,
\begin{align*}
&C\Td P^*(z, w)\equiv\int^\infty_0 e^{i\phi(z, w)t}a(z, w, t)dt\ \ {\rm mod\,}C^\infty(\ol M\times\ol M),\\
&a(z, w, t)\in S^{n}_{1, 0}(\ol M\times\ol M\times[0, \infty[;\, \mathscr L(\Lambda^{0,q}T^*_w(M'), \Lambda^{0,q}T^*_z(M'))), \\
&a(z, w, t)\sim\sum^\infty_{j=0}a_j(z, w)t^{n-j}
\end{align*}
in the space $S^{n}_{1, 0}(\ol M\times\ol M\times]0, \infty[;\, \mathscr L(\Lambda^{0,q}T^*_w(M'), \Lambda^{0,q}T^*_z(M')))$.
Note that $C\Td P^*=\Td PB_-T(\Td P^*\Td P)^{-1}\Td P^*$.
From this and Proposition~\ref{Bp:0801131851}, we get the main result of this work

\begin{thm} \label{Bt:0801141751}
Given $q$, $0\leq q\leq n-1$.
Suppose that $Z(q)$ fails at some point of\, $\Gamma$ and that $Z(q-1)$ and $Z(q+1)$ hold at each point of\, $\Gamma$.
Then
\[K_{\Pi^{(q)}}(z, w)\in
C^\infty(\ol M\times\ol M\setminus{\rm diag\,}(\Gamma_q\times\Gamma_q);\, \mathscr L(\Lambda^{0,q}T^*_w(M'),\Lambda^{0,q}T^*_z(M'))).\]
Moreover, in a neighborhood $U$ of ${\rm diag\,}(\Gamma_q\times\Gamma_q)$, $K_{\Pi^{(q)}}(z, w)$ satisfies
\begin{equation}
K_{\Pi^{(q)}}(z, w)\equiv\int^\infty_0e^{i\phi(z, w)t}a(z, w, t)dt\ \ {\rm mod\,}\ \ C^\infty(U\bigcap(\ol M\times\ol M))
\end{equation}
with
$a(z, w, t)\in S^{n}_{1, 0}(U\bigcap(\ol M\times\ol M)\times]0, \infty[;\, \mathscr L(\Lambda^{0,q}T^*_w(M'), \Lambda^{0,q}T^*_z(M')))$,
\[a(z, w, t)\sim\sum^\infty_{j=0}a_j(z, w)t^{n-j}\]
in the space $S^{n}_{1, 0}(U\bigcap(\ol M\times\ol M)\times]0, \infty[;\, \mathscr L(\Lambda^{0,q}T^*_w(M'), \Lambda^{0,q}T^*_z(M')))$,
\[a_0(z, z)\neq0,\ z\in\Gamma_q,\]
where
$a_j(z, w)\in C^\infty(U\bigcap(\ol M\times\ol M);\, \mathscr L(\Lambda^{0,q}T^*_w(M'), \Lambda^{0,q}T^*_z(M')))$, $j=0,1,\ldots$,
and
\begin{align}
&\phi(z, w)\in C^\infty(U\bigcap(\ol M\times\ol M)), \ \ {\rm Im\,}\phi\geq0, \label{Be:0711180902bis} \\
&\phi(z, z)=0,\ \ z\in \Gamma_q,\ \ \phi(z, w)\neq0\ \ \mbox{if}\ \ (z, w)\notin{\rm diag\,}(\Gamma_q\times\Gamma_q), \label{Be:0711180903bis} \\
&{\rm Im\,}\phi(z, w)>0\ \ \mbox{if}\ \ (z, w)\notin\Gamma\times\Gamma, \label{Be:0711180905bis} \\
&\phi(z, w)=-\ol\phi(w, z) \label{Be:0711180906bis}.
\end{align}
For $p\in\Gamma_q$, we have $\sigma_{\Box^{(q)}_f}(z, d_z\phi(z, w))$ vanishes to infinite order at $z=p$, where
$(z, w)$ is in some neighborhood of\, $(p, p)$ in $M'$.

For $z=w$, $z\in\Gamma_q$, we have $d_z\phi=-\omega_0-idr$, $d_w\phi=\omega_0-idr$.
\end{thm}

As before, we put $B_-(x, y)\equiv\int^\infty_0 e^{i\phi_-(x, y)t}b(x, y, t)dt$,
\[b(x, y, t)\sim\sum^\infty_{j=0}b_j(x, y)t^{n-1-j}\]
and
\[K_{\Pi^{(q)}}(z, w)\equiv\int^\infty_0e^{i\phi(z, w)t}a(z, w, t)dt,\]
$a(z, w, t)\sim\sum^\infty_{j=0}a_j(z, w)t^{n-j}$.
Since
\[\Pi^{(q)}\equiv \Td PB_-T(\Td P^*\Td P)^{-1}\Td P^*,\]
$(\Td P^*\Td P)^{-1}=2\sqrt{-\La_\Gamma}$+lower order terms and
\[T=2(\dbar r)^{\wedge, *}(\dbar r)^\wedge+\mbox{lower order terms},\]
we have
\[a_0(x, x)=2\sigma_{\sqrt{-\La_\Gamma}}(x, (\phi_-)'_y(x, x))b_0(x, x)2(\dbar r(x))^{\wedge, *}(\dbar r(x))^\wedge ,\ \ x\in\Gamma.\]
Since $(\phi_-)'_y(x, x)=\omega_0(x)$ and $\norm{\omega_0}=1$ on $\Gamma$, it follows that
\begin{equation} \label{Be:0711131222}
a_0(x, x)=4b_0(x, x)(\dbar r(x))^{\wedge, *}(\dbar r(x))^\wedge.
\end{equation}
From this and Proposition~\ref{Bp:i-leading1}, we get the following

\begin{prop} \label{Bp:0711170831}
Under the assumptions of Theorem~\ref{Bt:0801141751}, let $p\in\Gamma_q$, $q=n_-$. Let
$U_1(z),\ldots,U_{n-1}(z)$
be an orthonormal frame of $\Lambda^{1,0}T_z(\Gamma)$, $z\in\Gamma$, for which
the Levi form is diagonalized at $p$. Let $e_j(z)$, $j=1,\ldots,n-1$
denote the basis of $\Lambda^{0,1}T^*_z(\Gamma)$, $z\in\Gamma$, which is dual to $\ol U_j(z)$, $j=1,\ldots,n-1$. Let
$\lambda_j(z)$, $j=1,\ldots,n-1$ be the eigenvalues of the Levi form $L_z$, $z\in\Gamma$. We assume that
$\lambda_j(p)<0$ if $1\leq j\leq n_-$.
Then
\begin{equation} \label{Be:0711171545bis}
a_0(p, p)=\abs{\lambda_1(p)}\cdots
\abs{\lambda_{n-1}(p)}\pi^{-n}
2\Bigr(\prod_{j=1}^{j=n_-}e_j(p)^\wedge e_j^{\wedge, *}(p)\Bigr)\circ(\dbar r(p))^{\wedge, *}(\dbar r(p))^\wedge.
\end{equation}
\end{prop}

\newpage

\addcontentsline{toc}{part}{\textsf Appendices}
\part*{\mbox{} \\ \mbox{} \\ \mbox{} \mbox{} \\ \mbox{} \\ \mbox{} \mbox{} \\ \mbox{} \\ \mbox{}
\quad\quad\quad\quad\quad\quad\quad Appendices}
\mbox{}
\newpage

\appendix

\section{Almost analytic manifolds, functions and vector fields}

We will give a brief discussion of almost analytic manifolds,
functions and vector fields in a setting
appropriate for our purpose. For more details on the subject, see~\cite{MS74}.

Let $W\subset\Complex^n$ be an open set and let $f\in C^\infty(W)$. In this section, we will use the following notations:
$\dbar f=\sum^n_{j=1}\frac{\pr f}{\pr\ol z_j}d\ol z_j$ and $\pr f=\sum^n_{j=1}\frac{\pr f}{\pr z_j}dz_j$.

\begin{defn} \label{d:c-almost}
Let $W\subset\Complex^n$ be an open set and let $\phi(z)$ be a positive continuous function on $W$.
If $f\in C^\infty(W)$, we say that $f$ is almost analytic with respect to the weight function
$\phi$ if, given any compact subset $K$ of $W$ and any integer $N\geq 0$, there is a constant
$c>0$ such that
$\abs{\dbar f(z)}\leq c\phi(z)^N$, $\forall z\in K$.
When $\phi(z)=\abs{{\rm Im\,}z}$ we simply say that $f$ is almost analytic.
\end{defn}

\begin{defn} \label{d:c-aleq}
Let $f_1,f_2\in C^\infty(W)$ with $W$, $\phi$ as above. We say that $f_1$ and $f_2$ are equivalent
with respect to the weight function $\phi$ if, given any compact subset $K$ of $W$ and any integer
$N>0$, there is a constant $c>0$ such that
$\abs{(f_1-f_2)(z)}\leq c\phi(z)^N$, $\forall z\in K$.
When $\phi(z)=\abs{{\rm Im\,}z}$ we simply say that they are equivalent and we write
$f_1\sim f_2$.
\end{defn}

The following proposition is due to H\"{o}rmander. For a proof, see~\cite{MS74}.

\begin{prop} \label{p:c-almost}
Let $W\subset\Complex^n$ be an open set and let $W_\Real=W\bigcap\Real^n$.
If $f\in C^\infty(W_\Real)$ then $f$ has an almost analytic extension, uniquely determined up to
equivalence.
\end{prop}

\begin{defn} \label{d:c-am}
Let $U$ be an open subset of $\Complex^n$ and let $\Lambda$ be a $C^\infty$ submanifold of codimension $2k$ of $U$.
We say that $\Lambda$ is an almost analytic manifold
if for every point $z_0$ of $\Lambda\bigcap\Real^n$ there exist an open
neighborhood $V$ of $z_0$ in $U$ and $k$ complex $C^{\infty}$ almost analytic functions
$f_1,\ldots,f_k$ defined on $V$ such that
\[\begin{split}
&\mbox{$\Lambda\bigcap V$ is defined by the equations $f_1(z)=\cdots=f_k(z)=0$,}   \\
&\mbox{$\pr f_1,\ldots,\pr f_k$ are linearly independent over $\Complex$ at every point of $V$}.
\end{split}\]
\end{defn}

\begin{defn} \label{d:c-ame}
Let $\Lambda_1$ and $\Lambda_2$ be two $C^\infty$ closed submanifolds of an open set $U\subset\Complex^n$.
We say that
$\Lambda_1$ and $\Lambda_2$ are equivalent (and we write $\Lambda_1\sim\Lambda_2$) if they have the same
intersection with $\Real^n$ and the same dimension and if for every open set $V\subset\subset U$ and
$N\in\mathbb N$ we
have
${\rm dist\,}(z, \Lambda_2)\leq c_{N,V}\abs{{\rm Im\,}z}^N$, $z\in V\bigcap\Lambda_1$, $c_{N,V}>0$.
\end{defn}

It is trivial that $\sim$ is an equivalence relation and that $\Lambda_1$ and $\Lambda_2$ are tangential to
infinite order in the real points when $\Lambda_1\sim \Lambda_2$.

In section $8$ of part ${\rm I\,}$, we need the following

\begin{prop} \label{p:0709201613}
Let $W$ be an open neighborhood of the origin in $\Complex^{n+m}$. Let $f_j(z, w)$, $g_j(z, w)$,
$j=1,\ldots,n$, be almost analytic functions on $W$ with $f_j(0, 0)=0$, $g_j(0, 0)=0$, $j=1,\ldots,n$,
${\rm det\,}\left(\frac{\pr f_j(0, 0)}{\pr z_k}\right)^n_{j,k=1}\neq0$
and $\pr g_1,\ldots,\pr g_n$ are linearly independent over $\Complex$ at the origin. Let
\[\Lambda_1=\set{(z, w)\in W;\, f_1(z, w)=\ldots=f_n(z, w)=0}\]
and
$\Lambda_2=\set{(z, w)\in W;\, g_1(z, w)=\ldots=g_n(z, w)=0}$.
If $\Lambda_1$ coincides to infinite order with $\Lambda_2$ at $(0, 0)$ we can then find
$a_{j, k}(z, w)\in C^\infty$ in a neighborhood of $(0, 0)$, $j, k=1,\ldots,n$, with
${\rm det\,}\left(a_{j, k}(0, 0)\right)^n_{j, k=1}\neq0$
so that $g_j-\sum^n_{k=1}a_{j, k}f_k$ vanishes to infinite order at $(0, 0)$ in $\Complex^{n+m}$, for all $j$.
\end{prop}

\begin{proof}
We write $z=x+iy$, $w=u+iv$, where $x$, $y\in\Real^n$, $u$, $v\in\Real^m$. From the Malgrange preparation theorem
(see Theorem~$7.57$ of~\cite{Hor03}), it follows that
$g_j(x, u)=\sum^n_{k=1}a_{j,k}(x, u)f_k(x, u)+r(u)$, $j=1,\ldots,n$,
in a real neighborhood of $(0, 0)$, where
$a_{j, k}(x, u)\in C^\infty$ in a real neighborhood of $(0, 0)$, $j, k=1,\ldots,n$. Since
$\Lambda_1$ coincides to infinite order with $\Lambda_2$ at $(0, 0)$, it follows that $r(u)$ vanishes to infinite order
at $0$. Since $f_k(0, 0)=0$, $k=1,\ldots,n$, we have
$dg_j(0, 0)=\sum^n_{k=1}a_{j,k}(0, 0)df_k(0, 0)$, $j=1,\ldots,n$.
Hence
${\rm det\,}\left(a_{j, k}(0, 0)\right)^n_{j,k=1}\neq0$.
Let $a_{j, k}(z, w)$ be an almost analytic extension of $a_{j, k}(x, u)$ to a complex neighborhood of $(0, 0)$, where
$j$, $k=1,\ldots,n$. Then
$g_j(z, w)-\sum^n_{k=1}a_{j, k}(z, w)f_k(z, w)$
also vanishes to infinite order at $(0, 0)$, for all $j$.
\end{proof}

\begin{lem} \label{l:0709152203}
Let\, $\Omega$, $\Omega_\kappa$ be open sets in $\Real^n$. Let
$\kappa: \Omega\To \Omega_\kappa$
be a diffeomorphism. Let\, $\Omega^\Complex$ and $\Omega^\Complex_\kappa$ be open sets in $\Complex^n$ with
$\Omega^\Complex\bigcap\Real^n=\Omega$, $\Omega^\Complex_\kappa\bigcap\Real^n=\Omega_\kappa$.
Let $\Td\kappa: \Omega^\Complex\To \Omega^\Complex_\kappa$
be an almost analytic extension of\, $\kappa$. We take $\Omega^\Complex$ and $\Omega^\Complex_\kappa$ so
that $\Td\kappa$ is a diffeomorphism. If\, $\Lambda$ is an almost analytic manifold of\, $\Omega^\Complex$,
then $\Td\kappa(\Lambda)$ is an almost analytic manifold of\, $\Omega^\Complex_\kappa$. Moreover, if\,
$\hat\kappa: \Omega^\Complex\To \Omega^\Complex_\kappa$
is another almost analytic extension of\, $\kappa$ and $\hat\kappa$ is a diffeomorphism, then
$\hat\kappa(\Lambda)\sim\Td\kappa(\Lambda)$.
Furthermore, if\, $\Lambda_1\sim\Lambda_2$, then
$\Td\kappa(\Lambda_1)\sim\Td\kappa(\Lambda_2)$,
where $\Lambda_1$ and $\Lambda_2$ are almost analytic manifolds of\, $\Omega^\Complex$.
\end{lem}

We shall now generalize the notion of almost analytic manifolds. We have the following

\begin{defn} \label{d:0709152255}
Let $X$ be a $n$ dimensional real paracompact $C^\infty$ manifold. An almost analytic manifold $\Lambda$
associated to $X$ is given by
 \begin{enumerate}
 \item A locally closed set $\Lambda_\Real$. (Locally closed means that every point of $\Lambda_\Real$ has
an neighborhood $\omega$ in $X$ such that $\Lambda_\Real\bigcap\omega$ is closed in $\omega$.)
 \item A covering of $\Lambda_\Real$ by open coordinate patches
$\kappa_\alpha:X\supset X_\alpha\To\Omega_\alpha\subset\Real^n$, $\alpha\in J$
and almost analytic manifolds $\Lambda_\alpha\subset\Omega^\Complex_\alpha$ with
$\Lambda_{\alpha}\bigcap\Real^n:=\Lambda_{\alpha\Real}=\kappa_\alpha(X_\alpha\bigcap\Lambda_\Real)$.
Here $\Omega^\Complex_\alpha\subset\Complex^n$
is some open set with $\Omega^\Complex_\alpha\bigcap\Real^n=\Omega_\alpha$ and the $\Lambda_\alpha$ shall satisfy
the following compatibility conditions: If
$\kappa_{\beta\alpha}=\kappa_\beta\circ\kappa^{-1}_\alpha:
\kappa_\alpha(X_\alpha\bigcap X_\beta)\To\kappa_\beta(X_\alpha\bigcap X_\beta)$
and if $\Td\kappa_{\beta\alpha}$ is an almost analytic extension of $\kappa_{\beta\alpha}$, then
$\Td\kappa_{\beta\alpha}(\Lambda_\alpha)$ and $\Lambda_\beta$ are equivalent near all points of
$\kappa_\beta(X_\alpha\bigcap X_\alpha\bigcap\Lambda_\Real)$.
 \end{enumerate}
The $\Lambda_\alpha$ are called local representatives of $\Lambda$ and we shall say that two almost analytic manifolds
$\Lambda$, $\Lambda'$ associated to $X$ are equivalent (and we write $\Lambda\sim\Lambda'$) if
$\Lambda_\Real=\Lambda'_\Real$ and if the corresponding local representatives are equivalent as in $(b)$.
\end{defn}

Similarly we extend the notion of almost analytic functions and equivalence of almost analytic functions.

\begin{defn} \label{d:0709101418}
Let $W$ be an open subset of $\Complex^n$ and let $V$ be a complex $C^\infty$ vector field on $W$. We say that
$V$ is almost analytic if $V(f)$ is almost analytic and $V(\ol f)\sim0$ for all almost analytic functions $f$ on $W$.
\end{defn}

We identify $\Complex^n$ with $\Real^{2n}$. We shall denote the real coordinates
by $x_j$, $y_j$, $j=1,\ldots,n$, and the complex coordinates by $z_j=x_j+iy_j$, $j=1,\ldots,n$.

\begin{defn} \label{d:0709101411}
Let $W$ be an open subset of $\Complex^n$ and let
$U=\sum^n_{j=1}a_j(z)\frac{\pr}{\pr z_j}+\sum^n_{j=1}b_j(z)\frac{\pr}{\pr\ol z_j}$,
$V=\sum^n_{j=1}c_j(z)\frac{\pr}{\pr z_j}+\sum^n_{j=1}d_j(z)\frac{\pr}{\pr\ol z_j}$,
be complex $C^\infty$ vector fields on $W$, where
$a_j(z), b_j(z), c_j(z), d_j(z)\in C^\infty(W),\ j=1,\ldots,n$.
We say that $U$ and $V$ are equivalent if
$a_j(z)-c_j(z)\sim0,\ b_j(z)-d_j(z)\sim0$
for all $j$.
If $U$ and $V$ are equivalent, we write
$U\sim V$.
\end{defn}

Clearly $U$ is almost analytic if and only if
$U\sim\sum^n_{j=1}a_j(z)\frac{\pr}{\pr z_j}$,
where $a_j$, $j=1,\ldots,n$, are almost analytic. We have the following easy lemma

\begin{lem} \label{l:0709161519}
Let $W$ be an open subset of $\Complex^n$ and let $V$ be an almost analytic vector field on $W$. Then
$V(f)\sim(V+\ol V)(f)$ for all almost analytic functions $f$.
If $U$ is a real vector field on $W$ and
$U(f)\sim V(f)$ for all almost analytic functions $f$,
then $U\sim V+\ol V$.
\end{lem}

\begin{prop} \label{p:0709161529}
Let $W$ be an open subset of $\Complex^n$ and let $\Sigma$ be a $C^\infty$ closed submanifold of\, $\Real^n$.
Let $V$ be an almost analytic vector field on $W$. We assume that
$V=0$ on $\Sigma$.
Let $\Phi(t, \rho)$ be the $V+\ol V$ flow. Let $U$ be a real vector field on $W$ such that
$U\sim V+\ol V$.
Let $\hat\Phi(t, \rho)$ be the $U$ flow. Then, for every compact set $K\subset W$, $N\geq0$, there is a
$c_{N,K}(t)>0$,
such that
$\abs{\Phi(t, \rho)-\hat\Phi(t, \rho)}\leq c_{N,K}(t){\rm dist\,}(\rho, \Sigma)^N$, $\rho\in K$.
\end{prop}

\begin{proof}
We have the following well-known fact: If $Z$ is a smooth vector field on an open set $\Omega\subset\Real^n$ with $Z(x_0)=0$,
$\Psi(t, x)={\rm exp\,}(tZ)(x)$, then $\Psi(t, x_0)=x_0$ and $\pr^\alpha_x\Psi(t, x)|_{x=x_0}$, $\alpha\in\Pstint^n$, only depend on
$(\pr^\beta_xZ)(x_0)$, $\beta\in\Pstint^n$.
In our situation, we therefore have that $\Phi(t, \rho)$, $\hat\Phi(t, \rho)$ have the same Taylor expansion at every
point of $\Sigma$.
\end{proof}

The following proposition is useful (see section $2$ of~\cite{MS74})

\begin{prop} \label{p:0709171438}
Assume that $f(x, w)$ is a $C^\infty$ complex function in a neighborhood of $(0, 0)$
in $\Real^{n+m}$ and that
\begin{equation} \label{e:0709171048}
{\rm Im\,}f\geq0,\ \ {\rm Im\,}f(0, 0)=0,\ \ f'_x(0, 0)=0,\ \ {\rm det\,}f''_{xx}(0, 0)\neq0.
\end{equation}
Let $\Td f(z, w)$, $z=x+iy$, $w\in\Complex^m$, denote an almost analytic extension of $f$ to a complex neighborhood
of\, $(0, 0)$ and let $z(w)$ denote the solution of $\frac{\pr\Td f}{\pr z}(z(w), w)=0$
in a neighborhood of\, $0$ in $\Complex^m$. Then,
\begin{equation} \label{e:0709201408}
\frac{\pr}{\pr w}(\Td f(z(w), w))-\frac{\pr}{\pr w}\Td f(z, w)|_{z=z(w)},\ \ w\ \mbox{is real},
\end{equation}
vanishes to infinite order at $0\in\Real^m$.
Moreover, there is a constant $c>0$ such that near the origin we have
\begin{equation} \label{e:0709171222}
{\rm Im\,}\Td f(z(w), w)\geq c\abs{{\rm Im\,}z(w)}^2,\ w\in\Real^m
\end{equation}
and
\begin{equation} \label{e:0709171223}
{\rm Im\,}\Td f(z(w), w)\geq c\inf_{x\in\Omega}\Bigr({\rm Im\,}f(x, w)+\abs{d_xf(x, w)}^2\Bigr),\ w\in\Real^m,
\end{equation}
where $\Omega$ is some open set of the origin in $\Real^n$.

We call $\Td f(z(w), w)$ the corresponding critical value.
\end{prop}

\begin{proof}
For a proof of (\ref{e:0709171222}), see~\cite{MS74}. We only prove (\ref{e:0709171223}).
In view of the proof of (\ref{e:0709171222}) (see p.$147$ of~\cite{MS74}), we see that
\begin{equation} \label{e:0710271153}
{\rm Im\,}\Td f(z(w), w)\geq c\Bigr(\inf_{t\in\Real^n, \abs{t}\leq1}{\rm Im\,}f({\rm Re\,}z(w)-t\abs{{\rm Im\,}z(w)}, w)+\abs{{\rm Im\,}z(w)}^2\Bigr)
\end{equation}
for $w$ small, where $w$ is real and $c$ is a positive constant.
Using the almost analyticity, we get by Taylor's formula:
\begin{equation} \label{e:0709171352}
f'_x(x, w)=f''_{zz}(z(w), w)(x-z(w))+O(\abs{x-z(w)}^2+\abs{{\rm Im\,}z(w)}^2)
\end{equation}
for $x$, $w$ small, where $w$ is real. Since $f''_{zz}$ is invertible near the origin, we have that when $w\in\Real^m$ is close to the origin,
\[\abs{{\rm Im\,}z(w)}^2\geq c\abs{f'_x({\rm Re\,}z(w)-t\abs{{\rm Im\,}z(w)}, w)}^2\]
for all $t\in\Real^n$, $\abs{t}\leq1$, where $c$ is a positive constant. From this and (\ref{e:0710271153}),
we get (\ref{e:0709171223}).
\end{proof}

In the following, we let $z=z(w)$ be the point defined as above.
We recall the stationary phase formula of Melin and Sj\"{o}strand

\begin{prop} \label{p:0709171439}
Let $f(x, w)$ be as in Proposition~\ref{p:0709171438}. Then there are neighborhoods $U$ and $V$ of the origin in
$\Real^n$ and $\Real^m$ respectively and differential operators $C_{f, j}$ in $x$ of order $\leq2j$ which are $C^\infty$
functions of $w\in V$ such that
\begin{align} \label{e:0709171445}
&\abs{\int\!\!e^{itf}udx-\!\!\Bigr({\rm det\,}\left(\frac{t\Td f''_{zz}(z(w), w)}{2\pi i}\right)
\Bigr)^{-\frac{1}{2}}\!\!\!e^{it\Td f(z(w), w)}
\sum^{N-1}_0(C_{f, j}\Td u)(z(w), w)t^{-j}}\nonumber \\
&\quad\leq c_Nt^{-N-\frac{n}{2}},\ \ t\geq1,
\end{align}
where $u\in C^\infty_0(U\times V)$.
Here $\Td f$ and $\Td u$ are almost analytic extensions of\, $f$ and $u$ respectively. The function
$\Bigr({\rm det\,}\left(\frac{t\Td f''_{zz}(z(w), w)}{2\pi i}\right)\Bigr)^{-\frac{1}{2}}$
is the branch of the square root of
$\Bigr({\rm det\,}\left(\frac{t\Td f''_{zz}(z(w), w)}{2\pi i}\right)\Bigr)^{-1}$
which is continuously deformed into $1$ under the homotopy
$s\in [0, 1]\To i^{-1}(1-s)\Td f''_{zz}(z(w), w)+sI\in {\rm GL\,}(n, \Complex)$.
\end{prop}

We need the following asymptotic formula (see~\cite{MS74})

\begin{prop} \label{p:0709171623}
Let $P$ be a classical pseudodifferential operator on $\Real^n$. Let $\varphi(x)\in C^\infty(\Real^n)$ satisfy ${\rm Im\,}\varphi\geq0$
and $d\varphi\neq0$ where ${\rm Im\,}\varphi=0$. Let $u(x)\in C^\infty_0(\Real^n)$.
If $p(x, \xi)$ is the full symbol of $P$, then
\begin{equation} \label{e:0709171647}
P(e^{it\varphi(x)}u(x))\sim
e^{it\varphi(x)}\sum_{\alpha}\frac{1}{\alpha !}
\Td p^{(\alpha)}(x,t\varphi'_x(x))\frac{1}{\alpha!}D^\alpha_y(u(y)e^{it\rho(x, y)})|_{y=x}
\end{equation}
with asymptotic convergence in $S^m_{0, 1}(\Real^n\times\Real_+)$
(for the precise meaning of the space $S^m_{0, 1}$, see page $5$ of~\cite{GS94}), where
$\Td p$ is an almost analytic extension of $p$ and
$\rho(x, y)=\varphi(y)-\varphi(x)-\seq{y-x, \varphi'_x(x)}$.
\end{prop}

\begin{defn} \label{d:m-non-degeneratephase}
The $C^\infty$ function $\varphi(x, \theta)$ defined in an open conic set $V\subset\Real^n\times\Real^N\setminus0$
is called a  non-degenerate complex phase function if
\begin{enumerate}
 \item $d\varphi\neq0$.
 \item $\varphi(x, \theta)$ is positively homogeneous of degree $1$.
 \item Put $C=\set{(x, \theta)\in V;\, \varphi'_\theta(x, \theta)=0}$.
The differentials
$d(\frac{\pr\varphi}{\pr\theta_j})$, $j=1,\ldots,N$, are linearly independent over the complex numbers on $C$.
\item ${\rm Im\,}\varphi\geq0$.
 \end{enumerate}
\end{defn}

Let $\phi(x,\theta)$ be a non-degenerate complex phase function in a conic open subset $\Gamma$ of
$\Real^n\times\dot\Real^N$. Let
$C_\phi=\set{(x,\theta)\in \Gamma;\, \phi'_\theta(x,\theta)=0}$.
By Euler's homogeneity relation, we have $\phi(x,\theta)=\theta.\phi'_\theta(x,\theta)=0$ on $C_\phi$ and
therefore ${\rm Im\,}\phi$ vanishes on $C_\phi$. So does $d({\rm Im\,}\phi)$, for otherwise there would be a
change of sign of ${\rm Im\,}\phi$. Thus, $\set{(x, \phi'_x(x, \theta));\, (x, \theta)\in C_\phi}\subset\Real^n\times\dot{\Real^n}$.

Let $\Td\phi$ be an almost analytic extension of $\phi$ in a conic open set
$\Td\Gamma\subset\Complex^n\times\dot\Complex^N$, $\Td\Gamma\bigcap(\Real^n\times\dot\Real^N)=\Gamma$.
We can choose $\Td\phi$ such that $\Td\phi$
is homogeneous of degree 1. Set
$\Td\phi'_{\Td\theta}=(\pr_{\Td\theta_1}\Td\phi,\ldots,\pr_{\Td\theta_N}\Td\phi)$
and
$\Td\phi'_{\Td x}=(\pr_{\Td x_1}\Td\phi,\ldots,\pr_{\Td x_n}\Td\phi)$.
Let
\begin{align}
C_{\Td\phi}       &= \set{(\Td x, \Td\theta)\in \Td\Gamma;\,
                            \Td\phi'_{\Td\theta}(\Td x, \Td\theta)=0},   \\
\Lambda_{\Td\phi} &= \set{(\Td x,\Td\phi'_{\Td x}(\Td x, \Td\theta));\,
                            (\Td x, \Td\theta)\in C_{\Td\phi}}. \label{e:0710281327}
\end{align}
$\Lambda_{\Td\phi}$ is an almost analytic manifold and
\[(\Lambda_{\Td\phi})_\Real=
\Lambda_{\Td\phi}\bigcap(\Real^n\times\dot\Real^n)=\set{(x, \phi'_x(x, \theta));\, (x, \theta)\in C_\phi}.\]
Let $\Td\phi_1$ be another almost analytic extension of $\phi$ in $\Td\Gamma$. We have
$\Lambda_{\Td\phi}\sim\Lambda_{\Td\phi_1}$.
(For the details, see~\cite{MS74}.) Moreover, we have the following

\begin{prop} \label{p:c-mesj}
For every point $(x_0,\xi_0)\in (\Lambda_{\Td\phi})_\Real$ and after suitable change of local $x$ coordinates,
$\Lambda_{\Td\phi}$ is equivalent to a manifold
$-\Td x=\frac{\pr h(\Td\xi)}{\pr \Td\xi},\ \Td\xi\in \Complex^N$
in some open neighborhood of $(x_0,\xi_0)$, where $h$ is an almost analytic function and ${\rm Im\,}h\geq0$ on
$\Real^N$ with equality at $\xi_0$.
\end{prop}

\begin{defn} \label{d:c-mesj}
An almost analytic manifold satisfying the conditions of Proposition~\ref{p:c-mesj} at every real
point for some real symplectic coordinate $(x,\xi)$ is called a positive Lagrangean manifold.
\end{defn}

Let $\varphi$ and $\psi$ be non-degenerate phase functions
defined in small conic neighborhoods of $(x_0,\theta_0)\in\Real^n\times\dot\Real^N$ and $(x_0,w_0)\in
\Real^n\times\dot\Real^M$ respectively. We assume that $\varphi'_\theta(x_0, \theta_0)=0$, $\psi'_w(x_0, w_0)=0$
and that
$\varphi'_x(x_0, \theta_0)=\psi'_x(x_0, w_0)=\xi_0$,
where the last equation is a definition. Put $\lambda_0=(x_0, \xi_0)$. Let $S^m_{{\rm cl\,}}(\Real^n\times\Real^N)$ denot the space of classical symbols of
order $m$ on $\Real^n\times\Real^N$ (see page $35$ of~\cite{GS94}, for the precise meaning).
We have the following definition

\begin{defn} \label{d:0709171814}
We say that $\varphi$ and $\psi$ are equivalent
at $\lambda_0$ for classical symbols if there is a conic neighborhood $\Lambda$ of $(x_0, \theta_0)$ such that for every distribution
$A=\int e^{i\varphi(x,\theta)}a(x,\theta)d\theta$,
where $a(x, \theta)\in S^m_{{\rm cl\,}}(\Real^n\times\Real^N)$ with support in $\Lambda$,
there exists
$b(x,w)\in S^{m+\frac{N-M}{2}}_{{\rm cl\,}}(\Real^n\times\Real^M)$
with support in a conic neighborhood of $(x_0, \omega_0)$ such that
$A-B\in C^\infty$,
where $B=\int e^{i\psi(x,w)}b(x,w)dw$ and vise versa.
\end{defn}

The global theory of Fourier integral operators with complex phase is the following

\begin{prop} \label{p:c-mesjmore}
Let $\varphi$ and $\psi$ be non-degenerate phase functions
defined in small conic neighborhoods of $(x_0,\theta_0)\in\Real^n\times\dot\Real^N$ and $(x_0,w_0)\in
\Real^n\times\dot\Real^M$ respectively. We assume that $\varphi'_\theta(x_0, \theta_0)=0$, $\psi'_w(x_0, w_0)=0$
and that
\[\varphi'_x(x_0, \theta_0)=\psi'_x(x_0, w_0)=\xi_0.\]
Then $\varphi$ and $\psi$ are equivalent
at $(x_0, \xi_0)$ for classical symbols if and only if $\Lambda_{\Td\varphi}$ and $\Lambda_{\Td\psi}$
are equivalent in some neighborhood of\, $(x_0,\xi_0)$, where $\Td\varphi$ and $\Td\psi$ are almost analytic extensions
of $\varphi$ and $\psi$ respectively and $\Lambda_{\Td\varphi}$, $\Lambda_{\Td\psi}$ are as in (\ref{e:0710281327}).
\end{prop}

\section{The wave front set of a distribution, a review}

We will give a brief discussion of wave front set in a setting
appropriate for our purpose. For more details on the subject, see~\cite{Hor85},
~\cite{Hor03} and~\cite{GS94}. Our presentation is
essentially taken from~\cite{GS94}. For all the proofs of this section, we refer the reader to chapter $7$ of~\cite{GS94},
chapter ${\rm VIII\,}$ of~\cite{Hor85} and chapter ${\rm XVIII\,}$ of~\cite{Hor03}.

We will asume the reader is familiar with some basic notions and facts of microlocal analysis such as:
H\"{o}rmander symbol spaces, pseudodifferential operators. Nevertheless we recall briefly some of this notions.

Let $\Omega\subset\Real^n$ be an open set. From now on, we write $x^\alpha=x^{\alpha_1}_1\cdots
x^{\alpha_n}_n$, $\pr^\alpha_x=\pr^{\alpha_1}_{x_1}\cdots \pr^{\alpha_n}_{x_n}$,
$D^\alpha_x=D^{\alpha_1}_{x_1}\cdots D^{\alpha_n}_{x_n}$ and $\abs{\alpha}=\alpha_1+\cdots+\alpha_n$,
where $x=(x_1,\ldots,x_n)$, $D_{x_j}=-i\pr_{x_j}$. We have the following

\begin{defn} \label{Bd:m-symbol}
Let $m\in\Real$. $S^m_{1, 0}(\Omega\times\Real^N)$ is the space of all $a\in C^\infty(\Omega\times\Real^N)$ such that for all
compact sets $K\subset\Omega$ and all $\alpha\in\Pstint^n$, $\beta\in\Pstint^N$, there is a
constant $c>0$ such that
\[\abs{\pr^\alpha_x\pr^\beta_\xi a(x,\xi)}\leq c(1+\abs{\xi})^{m-\abs{\beta}},
\; (x,\xi)\in K\times\Real^N.\]
$S^m_{1,0}$ is called the space of symbols of order $m$. We write $S^{-\infty}_{1,0}=\bigcap S^m_{1, 0}$,
$S^\infty_{1, 0}=\bigcup S^m_{1, 0}$.
\end{defn}

In order to simplify the discussion of composition of some operators,
it is convenient to introduce the notion of properly supported operators. Let $C$ be a closed
subset of $Y\times Z$. We say that $C$ is proper if the two projections
\begin{align*}
& \Pi_y: (y, z)\in C\To y\in Y  \\
& \Pi_z: (y, z)\in C\To z\in Z
\end{align*}
are proper, that is the inverse image of every compact subset of $Y$ and $Z$ respectively is
compact.

A continuous linear operator $A: C^\infty_0(Z)\To\mathscr D'(Y)$
is said to be properly supported if ${\rm supp\,}K_A\subset Y\times Z$ is proper. If $A$ is
properly supported, then $A$ is continuous $C^\infty_0(Z)\To\mathscr E'(Y)$
and $A$ has a unique continuous extension $C^\infty(Z)\To\mathscr D'(Y)$.

\begin{defn} \label{Bd:ss-pseudo}
Let $k\in\Real$. A pseudodifferential operator of order $k$ type
$(1, 0)$ is a continuous linear map $A:C^\infty_0(\Omega)\To\mathscr D'(\Omega)$
such that the distribution kernel of $A$ is
\[K_A=A(x, y)=\frac{1}{(2\pi)^n}\int e^{i\seq{x-y,\xi}}a(x, \xi)d\xi\]
with $a\in S^k_{1, 0}(T^*(\Omega))$. We shall write $L^k_{1, 0}(\Omega)$
to denote the space of pseudodifferential operators of order $k$ type $(1, 0)$.
\end{defn}

\begin{defn} \label{Bd:m-cha}
Let
\[A(x, y)=\frac{1}{(2\pi)^n}\int e^{i(x-y)\xi}a(x, \xi)d\xi\in L^m_{1, 0}(\Omega),\]
$a\in S^m_{1, 0}(T^*(\Omega))$. Then $A$ is
said to be elliptic at $(x_0,\xi_0)\in T^*(\Omega)\smallsetminus 0$ if
$ab-1\in S^{-1}_{1, 0}(T^*(\Omega))$
in a conic neighborhood of $(x_0,\xi_0)$ for some $b\in S^{-m}_{1, 0}(T^*(\Omega))$.
\end{defn}

From now on, all pseudofifferential operators in this section will be assumed properly supported.

\begin{defn} \label{Bd:0801202156}
Let $u\in\mathscr D'(\Omega)$, $(x_0, \xi_0)\in T^*(\Omega)\setminus0$. We say that $u$ is $C^\infty$ near
$(x_0, \xi_0)$ if there exists $A\in L^0_{1, 0}(\Omega)$ elliptic at $(x_0, \xi_0)$, such that $Au\in C^\infty(\Omega)$. We
let ${\rm WF\,}(u)$ be the set of points in $T^*(\Omega)\setminus0$, where $u$ is not $C^\infty$.
\end{defn}

\begin{lem} \label{Bl:0801202214}
Let $u\in\mathscr D'(\Omega)$. Then $u\in C^\infty(\Omega)$ if and only if\, ${\rm WF\,}(u)=\emptyset$.
\end{lem}

Let $K\in\mathscr D'(\Omega\times\Omega)$. Put
${\rm WF\,}'(K)=\set{(x, \xi, y, \eta);\, (x, \xi, y, -\eta)\in{\rm WF\,}(K)}$.

\begin{prop} \label{Bp:0801202219}
Let
$A=\frac{1}{(2\pi)^n}\int e^{i(x-y)\xi}a(x, \xi)d\xi\in L^m_{1, 0}(\Omega)$, $a\in S^m_{1, 0}(T^*(\Omega))$.
Let $\Lambda$ be the smallest closed cone in $T^*(\Omega)\setminus0$ such that for any $\chi\in C^\infty(T^*(\Omega))$,
$\chi(x, \lambda\xi)=\chi(x, \xi)$ and $\chi=0$ in some conic neighborhood of\, $\Lambda$, we have
$\chi a\in S^{-\infty}_{1, 0}(T^*(\Omega))$.
Then ${\rm WF\,}'(K_A)={\rm diag\,}(\Lambda\times\Lambda)$.
Moreover, let $u\in\mathscr D'(\Omega)$. Then,
${\rm WF\,}(Au)\subset\Lambda\bigcap{\rm WF\,}(u)$.
\end{prop}

\begin{prop} \label{Bp:0801202234}
Let $\mathscr K:C^\infty_0(\Omega)\To\mathscr D'(\Omega)$
with distribution kernel $K\in\mathscr D'(\Omega\times\Omega)$. We assume that
${\rm WF\,}'(K)\subset\set{(x, \xi, x, \xi);\, (x, \xi)\in T^*(\Omega)\setminus0}$.
Then, there is a unique way of defining $\mathscr Ku$ for every $u\in\mathscr E'(\Omega)$ so that the map
$u\in\mathscr E'(\Omega)\To\mathscr K u\in\mathscr D'(\Omega)$
is continuous. Moreover, we have
\begin{align*}
{\rm WF\,}(\mathscr Ku)\subset&\{(x, \xi);\, (x, \xi, x, \xi)\in{\rm WF\,}'(K)\\
&\mbox{for some}\ \  (x, \xi)\in{\rm WF\,}(u)\}.
\end{align*}
\end{prop}

\end{document}